\documentclass[11pt]{amsart}
\usepackage{color, amsmath,amssymb, amsfonts, amstext,amsthm, latexsym,mathtools}

\usepackage[active]{srcltx}
\allowdisplaybreaks

\usepackage{cite}
\usepackage{upgreek}
\usepackage{mathrsfs}

\usepackage{ bbold }
\usepackage{graphicx}
\usepackage[show]{ed} 
\usepackage[T1]{fontenc}
\usepackage[english]{babel}
\usepackage{enumerate}
\usepackage{babel}
\usepackage[pdftex,breaklinks,colorlinks, citecolor=blue,urlcolor=blue]{hyperref}
\usepackage{verbatim}
\numberwithin{equation}{section}
\usepackage{enumitem}

\usepackage[a4paper,top=30mm,bottom=30mm,inner=25mm,outer=25mm]{geometry}

\usepackage{relsize}

\newcommand{\1}{\mbox{1\hspace{-.15em}l}}

\newcommand{\im}{\mathrm{Im}}

\newcommand{\re}{\mathrm{Re}}

\newcommand{\R}{{\mathbb R}}
\newcommand{\Z}{{\mathbb Z}}

\newcommand{\C}{{\mathbb C}}

\newcommand{\HH}{ \mathbb{H} }

\newcommand{\sh}{\operatorname{sh}}

\newcommand{\trho}{\Tilde{\rho}}

\newcommand{\PPsi}{\widetilde{\Psi}}

\newcommand{\UUppsi}{\widetilde{\Uppsi}}
\newcommand{\tc}{\Tilde{c}}
\newcommand{\tr}{\Tilde{r}}
\newcommand{\tR}{\Tilde{R}}

\newcommand{\calW}{\mathcal{W}}
\newcommand{\calL}{\mathcal{L}}
\newcommand{\calD}{\mathcal{D}}
\def\spec{\mathrm{spec}}

\newcommand{\FL}{\mathfrak{L}}
\newcommand{\FT}{\mathfrak{T}}

 \newcommand{\tvarepsilon}{\Tilde{\varepsilon}}

\numberwithin{equation}{section}
\newtheorem{theorem}{Theorem}[section]
\newtheorem{defi}[theorem]{Definition}

\newtheorem{lemma}[theorem]{Lemma}
\newtheorem{remark}[theorem]{Remark}
\newtheorem{prop}[theorem]{Proposition}
\newtheorem{coro}[theorem]{Corollary}
\newtheorem{claim}[theorem]{Claim}

\newtheorem{assumption}{Assumption}

\begin{document}
\title[The distorted Fourier transform for the linearized Gross-Pitaevskii equation]{The distorted Fourier transform for the linearized Gross-Pitaevskii equation 
in the Hyperbolic plane}


\begin{abstract}
 Motivated by the stability problem for Ginzburg-Landau vortices on the hyperbolic plane, we develop the distorted Fourier transform for a general class of radial non-self-adjoint matrix Schr\"odinger operators on the hyperbolic plane. This applies in particular to the operator obtained by linearizing the equivariant Ginzburg-Landau equation on the hyperbolic plane around the degree one vortex. We systematically
 construct the distorted Fourier transform by writing the Stone formula for complex energies and taking the limit as the energy tends to the spectrum of the operator on the real line. This approach entails a careful analysis of the resolvent for complex energies in a neighborhood of the real line. It is the analogue of the approaches used in \cite{KS,ES2, LSS25}, where the limiting operator as $r\to\infty$ is not self-adjoint and which we carry out for all energies. Our analysis serves as the starting point for the study of the stability of the Ginzburg-Landau vortex under equivariant perturbations.
\end{abstract}

\author[O. Landoulsi]{Oussama Landoulsi }
\address{Department of Mathematics \\ University of Massachusetts \\ Amherst, MA 01003, USA}
\email{olandoulsi@umass.edu}
 
 \author[S. Shahshahani]{Sohrab Shahshahani}
\address{Department of Mathematics \\ University of Massachusetts \\ Amherst, MA 01003, USA}
\email{sshahshahani@umass.edu}

\thanks{The second author was partially supported by the Simons Foundation grant 639284. The authors thank Jonas L\"uhrmann and Wilhelm Schlag for helpful discussions.}
\maketitle 
 \tableofcontents

\section{Introduction}
In this paper, we consider the Gross-Pitaevskii equation on the real hyperbolic plane $\mathbb{H}^2$:
\begin{equation} 
\label{eq:GP2}
    i\partial_t\Psi+\Delta_{\HH^2} \Psi+(1-|\Psi|^2)\Psi=0.
\end{equation}
It is well-known that the analogous problem on $\mathbb{R}^2$ admits stationary solutions, or vortices, of the form $V_n(r,\theta)=\rho_n(r)e^{in\theta}$ with $\rho_n(0)=0$ and $\lim_{r\to\infty}\rho_n(r)=1$. See for instance \cite{HerveHerve94, ChenElliott94}. The integer $n$ is called the degree of the vortex. Formally, $V_n$ is a critical point of the Ginzburg-Landau energy functional 
\begin{align*}
    \int \Big(\frac{1}{2}|\nabla \Psi|^2+\frac{1}{4}(1-|\Psi|^2)^2\Big)\,d x.
\end{align*}
However, due to the asymptotic behavior of $\rho_n$ at infinity, these stationary solution have logarithmically divergent $\dot{H}^1$ norm, and in particular do not have finite energy. In contrast, in view of the exponential growth of the volume of the ball of radius $r$ in $\mathbb{H}^2$, the corresponding problem on the hyperbolic plane admits finite energy solutions. See for instance Appendix~\ref{sec:appOde}. We continue to refer to these solutions as vortices.  Our concern in this paper is to initiate the study of the asymptotic stability of the degree one vortex on $\mathbb{H}^2$. As a first step we consider the problem under the additional equivariance ansatz that the perturbations are of the form $\phi(t,r)e^{in\theta}$.  An important challenge, which is present for both $\mathbb{R}^2$ and $\mathbb{H}^2$, is that the linearized operator about the vortex is not self-adjoint, even asymptotically as $r\to\infty$. Despite this difficulty, the linear analysis in the Euclidean case was recently taken up in \cite{collotGermainEliot25,LSS25}. In particular, these papers constructed a distorted Fourier transform for the linearized operator, which serves as a first natural step in studying the asymptotic stability of vortices. The main result of the current work is a complete construction of the corresponding distorted Fourier transform for the problem on $\mathbb{H}^2$. As a corollary, we establish an $L^2$-bound of the operator norm, in contrast to the Euclidean case where a growth rate of the $L^2$-norm was obtained. On the hyperbolic space, the spectrum of the linearized operator is real but exhibits a spectral gap near $0$, which makes the spectral analysis delicate. Nevertheless, we are able to precisely determine the essential spectrum and prove the absence of threshold or embedded eigenvalues. However, a complete classification of the intermediate behavior of solutions at the edge of the spectrum remains an open problem. For this reason, we consider a more general class of operators and, in particular, allow for different intermediate behaviors of threshold solutions at the spectral edge. See Section~\ref{subsec:DS-iL} below for more precise statements. Our motivation for this work is two-fold: On the one hand dispersive equations and soliton stability problems on hyperbolic spaces have received considerable attention in recent years. See for instance \cite{banicaCarlesDuyckaerts08,banicaCarlesStaffilani08,banicaDuyckaerts14, LOS1,IS2013}. On the other hand, the problem at hand provides a natural context to generalize the methods developed in \cite{LSS25}. Specifically, in this instance the main differences are the presence of a spectral gap and the different (hyperbolic rather than Euclidean) large $r$ asymptotics of the operators.\\

Before proceeding with a more precise description of the results, we make a few remarks about the geometry of the hyperbolic plane. First, in this work we choose to work with the hyperbolic plane of constant curvature $-\frac{1}{2}$. The motivation for this choice comes from the corresponding gauged self-dual Ginzburg-Landau problem. In that context an $SO(3)$ cylindrical symmetry reduction of the four dimensional Yang-Mills equations on the Euclidean space gives rice to the Ginzburg-Landau problem on the hyperbolic plane of curvature $-\frac{1}{2}$. Moreover, for this choice of the curvature the resulting elliptic equations can be integrated to give explicit formulas for the vortices. We refer the reader to \cite[Chapter~7]{mantonSutcliffe04}, \cite{OS25,jaffeeTaubes80,witten77} for a more complete discussion of the gauged-problem and these observations. Even though we are not aware of such remarkable structures in the non-gauged, or global, problem, we have decided to work with this choice of the curvature. Throughout the paper we will use polar coordinates $(r,\theta)\in (0,\infty)\times [0,2\pi)$ to describe the hyperbolic plane. Then  the Riemannian metric in these coordinates takes the form
\[
2dr^2 + 2(\sinh r)^2 \, d\theta^2.
\]
The corresponding Riemannian volume element is given by
\[
dv = 2 (\sinh r) \, dr \, d\theta,
\]
and the Laplace–Beltrami operator in $\HH^2$ with curvature $-\frac{1}{2}$ reads
\[
\Delta_{\mathbb{H}^2} = \frac{1}{2}\partial_r^2 + \frac{1}{2} \coth r \, \partial_r + \frac{1}{2}(\sinh r)^{-2} \,\partial_{\theta}^2 .
\]
It is well-known that $\spec\,\Delta_{\HH^2}=[\frac{1}{8},\infty)$. See for instance \cite{Helgason94}.

\subsection{The linearized operator around the degree-one vortex}
Consider the complex Ginzburg-Landau equation in the hyperbolic space $\HH^2,$
\begin{equation} 
\label{eq:complex_GL}
 \Delta_{\HH^2} \Psi + (1 - |\Psi|^2)\Psi = 0, 
\end{equation}
under the equivariance condition $\Psi_n(r,\theta)=e^{ in \theta} \rho_n(r)$, $n \in \Z \backslash\{0\}$. For $\Psi_n$ to be a solution $\rho_n$ have to satisfy the following ordinary differential equation
\begin{align}
\label{eq:ode2}
&\frac{1}{2} \partial_r^2\rho_n  + \frac{1}{2} \coth(r) \partial_r \rho_n  - \frac{n^2}{2 \sinh^2( r)} \rho_n + \rho_n - \rho_n^3=0.
\end{align}

\begin{theorem}
\label{th:ode}
    For any real $n > 0,$ the equation \eqref{eq:ode2}, 
    admits a unique smooth solution satisfying
    \begin{align}
          \label{eq:boundary_condition-ode2}
&\lim_{r \to 0} \rho_n(r)=0, \quad \lim_{r \to \infty} \rho_n(r) =1 , \quad \rho_n(r) \geq 0, \quad \forall r >0.  
    \end{align}
This solution satisfies
\begin{align}
 & 0 < \rho_n(r) < 1, \quad \rho_n'(r) >0, \quad \forall r>0, \\
  \label{eq:rho_n-asymptotics-near-0}
 & \rho_n(r) = a r^n   - \frac{a r^{n+2}}{2n+2}+ o(r^{2n+2}) \quad \text{as} \quad r \to 0,  \\ \qquad 
  \label{eq:rho_n-asymptotics-near-infty}
 & \rho_n(r) = 1 - 2 n^2 e^{-2r} + o(e^{-2r}) \quad \text{as} \quad r \to \infty.
\end{align}
\end{theorem}

\begin{remark}
  The proof of the existence and uniqueness is similar to one in the Euclidean space, see \cite{HerveHerve94,ChenElliott94}. In particular, the asymptotics near the origin are the almost the same. For completeness, we give a sketch of the proof in Appendix \ref{sec:appOde}. 
\end{remark}

 The solutions of the form $\Psi_n(r,\theta)=e^{ in \theta} \rho_n(r),$ are time independent solution to the  Gross-Pitaevskii equation \eqref{eq:GP2}.  They are prototypical examples of topological solitons and understanding their stability is a central question in the analysis of the equation’s dynamics. 
As a first step we study the Gross-Pitaevskii equation linearized around $\rho_1(r) e^{i \theta}$.  Consider an equivariant solution $\Psi$ of the nonlinear Gross-Pitaevskii equation \eqref{eq:GP2} close to the degree-one vortex $\rho_1(r) e^{i \theta}.$ 
We decompose $\Psi(t,x)$ as 
\begin{align*}
    \Psi(t,x) = \big( \rho_1(r) + h(t,r) \big) e^{i\theta},
 \end{align*}
Plugging this decomposition to the equation \eqref{eq:GP2}, we obtain the linearized evolution equation 
\begin{align*}
  \big(  \partial_t - \mathbf{L} \big) \mathbf{h}(t,r) = N(h), 
\end{align*}
where $N(h)$ contains the nonlinear terms on $h,$ 
 $ \mathbf{h}(t,r):= \begin{pmatrix}
     \mathrm{Re} \,h(t,r) \\
     \mathrm{Im} \, h(t,r) 
 \end{pmatrix} $ 
 and 
 \begin{align*}
\mathbf{L} := \begin{pmatrix}
   0  & -\Delta_{\mathrm{rad}} + \frac{1}{\sinh^2( r)}  + \rho_1^2 - 1 \\ 
   - \big( -\Delta_{\mathrm{rad}} + \frac{1}{\sinh^2( r)} + 3 \rho_1^2 - 1 \big) & 0
 \end{pmatrix}.
\end{align*}
Conjugating the corresponding linearized operator to the half-line by the weight $\sinh^{\frac12}$ gives rise to the operator
\begin{align*}
  \mathcal{L}_{GP}:= \sinh^{\frac{1}{2}}r\cdot\mathbf{L}\cdot\sinh^{-\frac{1}{2}}r=\begin{pmatrix} 
  0 & \mathfrak{L}_1 \\
  -\mathfrak{L}_2 & 0 
   \end{pmatrix},
\end{align*}
where
\begin{equation*}
\begin{aligned}
  \mathfrak{L}_1 := -\frac{1}{2}\partial_r^2 + \frac{3}{8 \sinh^2(r)} + \rho_1^2(r) - 1+\frac{1}{8}, \quad \mathfrak{L}_2 := - \frac{1}{2}\partial_r^2 + \frac{3}{8 \sinh^2(r)} + 3\rho_1^2(r) - 1 + \frac{1}{8}.
 \end{aligned}
\end{equation*}
Note that the scalar Schr\"odinger operators $\mathfrak{L}_1$ and $\mathfrak{L}_2$ can be written as
\begin{align}\label{eq:L012def1}
    \mathfrak{L}_1 = \mathfrak{L}_0+\rho_1^2-1, \quad \mathfrak{L}_2 = \mathfrak{L}_0+3\rho_1^2-1, \quad \text{with} \quad \mathfrak{L}_0 := -\frac{1}{2 }\partial_r^2+\frac{3}{8  \sinh^2(r)} + \frac{1}{8}.
\end{align}
It is standard that $\mathfrak{L}_0$ is limit point at both $r=0$ and $r=\infty$, and that it is essentially self-adjoint.
We denote the domain of $\mathfrak{L}_0$ by $\mathcal{D},$
\begin{equation*}
    \label{eq:defDomainL_0}
    \mathcal{D}:=\{ f \in L^2_r(0,\infty) \,|\;  \mathfrak{L}_0f \in L^2_r(\R_+) \}.
\end{equation*}
The operator $ \mathcal{L}_{GP}$ can be decomposed as
\begin{align} \label{eq:defL_GP-V_GP}
  \mathcal{L}_{GP}=\mathcal{L}_0  +V_{GP}, \quad \mathcal{L}_0:=\begin{pmatrix} 
  0 & \mathfrak{L}_0 \\
  -(\mathfrak{L}_0+2) & 0 
   \end{pmatrix}, \quad V_{GP}:=\begin{pmatrix}
       0  & \rho_1^2-1 \\
       -3(\rho_1^2-1) & 0 
   \end{pmatrix}.
\end{align}

\subsection{Main results}

In this work, we study the time evolution generated by the semigroup $e^{t\mathcal{L}}$ on the Hilbert space $L^2_r(\mathbb{R}_+) \times L^2_r(\mathbb{R}_+)$, where $\mathcal{L}$ is of the same form as $\mathcal{L}_{GP}$. Specifically, we consider an operator of the form  
\begin{align} 
\label{eq:def1-general-L}
\mathcal{L} := 
\begin{pmatrix} 
0 & \mathfrak{L}_0 \\
-(\mathfrak{L}_0 + 2) & 0 
\end{pmatrix} 
+ V, 
\qquad 
V := 
\begin{pmatrix}
0 & V_1(r) \\
V_2(r) & 0
\end{pmatrix},
\end{align}
where we assume that the potentials $V_1(r)$ and $V_2(r)$ satisfy the following conditions:

\begin{assumption}[\textbf{Decay and Regularity}] \label{asmp:1}
   The functions $V_1(r)$ and $V_2(r)$ are smooth, bounded and decay exponentially at infinity. Specifically,
$$
|V_j(r)| \lesssim C e^{-2 r}, \quad \text{for } j = 1, 2. 
$$
\end{assumption}

\begin{assumption}[\textbf{Spectral Condition}] \label{asmp:2}
The essential spectrum of $i \mathcal{L}$ is $(-\infty, -\frac{\sqrt{17}}{8}] \cup [\frac{\sqrt{17}}{8},\infty)$ and its the spectrum is real,  $\spec(i\mathcal{L}) \subseteq \R.$ Moreover, $ i \mathcal{L} $ has an empty discrete spectrum, in particular, no gap eigenvalues.
\end{assumption}

\begin{assumption}[\textbf{Absence of Embedded Eigenvalues}] \label{asmp:3}
$ i \mathcal{L} $ has no threshold or embedded eigenvalues. Moreover,
 $L_1 f=0$ and $L_2 f =0$ have a positive solution that is $L^2(0,1)$ but not in $L^2(1,\infty).$
\end{assumption}

\begin{remark}
In Section~\ref{subsec:DS-L_GP}, we will prove that $i\mathcal{L}_{GP}$ satisfies Assumptions~\ref{asmp:1},~\ref{asmp:2},~\ref{asmp:3}, with the exception that its discrete spectrum is empty. Although, we explicitly determine the essential spectrum of $i\mathcal{L}_{GP}$, the precise location of its discrete spectrum is more delicate and remains 
undetermined at this stage. Nevertheless, we can show that, if the discrete spectrum is non-empty, then it is contained in $\left(-\tfrac{\sqrt{17}}{8}, -\eta\right] \cup \left[\eta, \tfrac{\sqrt{17}}{8}\right),$ for some $\eta>0$, see also Remark~\ref{rem:DS-empty} for further discussion.  Moreover, $i \mathcal{L}$ represents a general class of operators that  may admit a \textit{threshold resonance}, in the sense that will be made precise later, see Definition \ref{def:resonance}.\footnote{Note that our use of the terminology \emph{threshold resonance} is not necessarily standard, see Remark~\ref{lem:resterminlogy}.} While the precise status of a threshold resonance for $i\mathcal{L}_{GP}$ cannot be conclusively determined at present, our analysis is designed to accommodate both possibilities. In particular, we construct the distorted Fourier transform for $i \calL$ and establish the corresponding $L^2$-bounds in both scenarios: when a resonance is present and when it is absent. 
\end{remark}
\begin{remark}
\label{rem:DS-empty}
The assumption that the discrete spectrum is empty is made simply for convenience. Although gap eigenvalues within the interval  $
\left(-\tfrac{\sqrt{17}}{8}, -\eta\right] \cup \left[\eta, \tfrac{\sqrt{17}}{8}\right) ,$ for $\eta>0$ can play an important role in the nonlinear dynamics, they can be handled relatively easily in the linear analysis. In fact, in the presence of gap eigenvalues, the linear analysis of this paper still applies to the essential spectrum; see, for example, \cite{ES2}, \cite{KS}.\\
The assumption of the zero energy solution of $L_1$ and $L_2$, which is consistent with the linearized operator $\mathcal{L}_{GP}$ for the  Gross-Pitaevskii equation, is made in order to obtain a classification of the behavior of the solution to the threshold problem $i \mathcal{L}\Phi=\pm \frac{\sqrt{17}}{8} \Phi.$ Instead, it would suffice to assume that this classification holds. See sections \ref{subsec:non-resonance} and \ref{subsec:resonance}. In fact, the assumption on $L^1$ and $L^2$ can be used to rule out threshold and embedded eigenvalues. Indeed, this is precisely how we prove absence of threshold and embedded eigenvalues for $\mathcal{L}_{GP}$. This argument, which is inspired by \cite{collotGermainEliot25}, uses only the existence of positive solutions of $L_1f=0$ and $L_2f=0$ that are $L^2$ near $r=0$ but not globally. See Section~\ref{subsec:DS-L_GP}.
\end{remark}

The main achievement of this paper is the construction of the distorted Fourier transform for the operator $\mathcal{L}$. This construction is a key point in establishing decay estimates for the linearized operator and for  initiating the analysis of the nonlinear asymptotic stability problem, which is our ultimate objective. In the present work, we restrict attention to the $L^2$ estimates, leaving the study of other related questions to future work.

\begin{theorem} \label{Main-theorem}
   The operator $\mathcal{L}$ is closed on $\mathcal{D}\times \mathcal{D} \subseteq L^2_r(\R^{+}) \times L^2_r(\R^{+}). $ The semi-group $\{   e^{t\mathcal{L} }\}_{ t \in \R} $ is defined as a bounded operator on $ L^2_r(\R^{+}) \times L^2_r(\R^{+})$ via the Hille-Yosida theorem. 
   For any $\psi \in L^2,$ and for all $t\in \R$, we have
    \begin{equation}
        \left\| e^{ t \mathcal{L}} \psi \right\|_{L^2} \lesssim \left\| \psi \right\|_{L^2}. 
    \end{equation}
Moreover, there exists a function $\mathbf{\Theta}(r,\lambda)=\begin{pmatrix}
    {\uptheta}_1(r,\lambda) \\  {\uptheta}_2(r,\lambda)
\end{pmatrix}$ such that for all $t>0$ and for any $\Phi,\Psi \in L^2_r(0,\infty)$ and with $\sigma_1$ denoting the Pauli matrix $\begin{pmatrix}
       0 &1 \\1 & 0
   \end{pmatrix}$, 
we have 
    \begin{align*}
      <  e^{t \mathcal{L}} \Phi, \Psi> = \frac{1}{2 \pi i } \int_{(-\infty, -\frac{\sqrt{17}}{8}] \cup [\frac{\sqrt{17}}{8},\infty)} e^{it \lambda} \left<  \mathbf{\Theta}(\cdot,\lambda)  , \sigma_1 \Phi(\cdot) \right>_{L^2_r} \left< \mathbf{\Theta}(\cdot,\lambda) , \Psi(\cdot) \right>_{L^2_r} d \lambda.
\end{align*}
Here  ${\uptheta}_1$ is an even function in $\lambda$, ${\uptheta}_2$ is an odd function in $\lambda$.
\end{theorem}

\begin{remark}
    The main result in the theorem is the construction of the distorted Fourier transform. The $L^2$ bound is proved as a corollary of this construction. Note that using energy estimates one can show that the $H^1$ norm remains bounded in the evolution. In view of the spectral gap for the hyperbolic Laplacian, this also proves that $\left\| e^{ t \mathcal{L}} \psi \right\|_{L^2}\lesssim 1$ for all $t$, but with a constant that may a priori depend on $\|\psi\|_{H^1}$.
\end{remark}
Our construction follows the general strategy of \cite{LSS25}, which in turn builds on \cite{ES2,KS}. The only difference arises in the proof of the absence of threshold and embedded eigenvalues for the operator $i\mathcal{L}_{GP}$, where our approach is inspired by \cite{collotGermainEliot25}. Let us note that in our hyperbolic setting the construction of the distorted Fourier transform allows us to establish an $L^2$-bound for the evolution, in contrast with the Euclidean case, where a sharp growth rate of the $L^2$ norm was derived in \cite{collotGermainEliot25,LSS25}. The main steps in the construction of the distorted Fourier transform as as follows:

\begin{enumerate}
    \item Using the Laplace transform we represent $e^{t\mathcal{L}}$ as a sum of two integrals on $\im \, z=\pm b$, with $b>0$ sufficiently large. The convergence of this integral relies on the Hille-Yosida bound for the evolution. We then deform the contour to pass the region of integration to the spectrum of $i\calL$ which, by assumption (or inspection in the case of $\mathcal{L}_{GP}$), is contained in the real line. This will naturally lead us to study the jump of the resolvent across the spectrum between positive and negative imaginary parts. 
    \item To justify passing the limit above to the real line we need two ingredients. One is limiting absorption principle. This is used to justify that the vertical lines $\re \,z=\pm R$ in the deformed contour of integration do not contribute as $R\to\infty$. The second ingredient, which requires the bulk of the work, is to analyze the resolvent and justify the exchange of limits and integrals. The hardest part is the analysis for small energies, that is, near the thresholds $\pm\frac{\sqrt{17}}{8}$. For intermediate frequencies we rely on the absence of embedded eigenvalues to show that there can be no singularities that prevent us from passing to the limit.
    \item After passing to the limit onto the spectrum, we can study the structure of the jump of the resolvent. This, which goes hand in hand with the previous step, requires the construction of the Green's function for large and small energies. In both cases this is done using the Lyapunov-Perron method separately for $r$ near zero and $r$ large.  It is important to be able to extend the constructions from $r=0$ and $r\to\infty$ to an overlapping region, where the Wronskians between the solutions can be computed. This is the connection problem. With the aid of the Wronskians we can then obtain our global Green's function. The following is a curious technical aspect of the construction for large $r$: The study of the equation at large $r$ naturally leads to an ansatz of the form $e^{ik_j(z)r}$, $j=1,2,3,4$, where $z$ denotes the energy and $k_j(z)$ are the roots of a fourth order polynomial. It turns out that the natural domain in the complex plane where the roots $k_j$ can be continued analytically excludes the spectrum $(-\infty,-\frac{\sqrt{17}}{8}]\cup[\frac{\sqrt{17}}{8},\infty)$, which is the region of interest. Therefore, to understand the behavior of the Weyl-Titchmarsh solutions for large $r$, one has to carefully compute the limits of these roots as $z$ approaches the spectrum from the upper and lower half planes to observe favorable cancellations.
 \end{enumerate}

\subsection{References} The Ginzburg-Landau and Gross-Pitaevskii equations arise in the study of superconductivity and superfluids respectively. See \cite{Chen21,CL1,CG23,DelMas20,GP20,GermPusZhang22,LL2,LS1}. These problems on $\mathbb{R}^2$ have been studied extensively. Global existence for finite energy solutions of the Gross-Piraevskii equation was proved in \cite{Gerard}. The asymptotic stability of small solutions (that is, close to a constant) was studied in the works \cite{GuoHaniNakanishi18,GNT1,GNT2}. The existence of vortices was proved in \cite{ChenElliott94,HerveHerve94}. Orbital stability using a renormalized energy is proved in \cite{GPS21}, see also \cite{dPFK,Wein}. The construction of the radial distorted Fourier transform for the linearized operator was carried out in \cite{collotGermainEliot25,LSS25}.  The work \cite{collotGermainEliot25} also contains dispersive estimates for the linearized flow. For other related works see \cite{BJS,BS,CP2,CP1,CP3,CJ,JS,JS2,JOST,LX,OS97,OS98,OS2002,PP24} and their references. Our construction of the distorted Fourier transform builds on insights from \cite{BP,ES2,GZ,KMS,KS,KST}. We also mention \cite{CL1,CG23,KS,Li23,LL2} as other examples of works that involve distorted Fourier theory for non-self-adjoint matrix operators. Finally we refer the reader to \cite{AnPi09, banica07,BaDu07, LLOS2,IPS, BorMar1, MT11, MTay12, MWYZ1, wilsonYu25, ZeLi1} and their references for related works on dispersive equations on hyperbolic spaces.

\subsection{Organization of the paper}
The paper is organized as follows: In section~\ref{sec:spectrum},
we study the spectral properties of $i\calL$ and $i \calL_{GP}.$ First, we classify the behavior of the solutions at the spectrum threshold for $i \calL.$ We distinguish two possible behaviors that lead to the characterization of the resonance and non resonance cases. Next, we define the evolution $e^{t \calL}$ via the Hille–Yosida theorem and establish a Stone-type formula. We also prove that the spectrum of $i\calL_{GP}$ is real and its essential spectrum is $(-\infty, - \frac{\sqrt{17}}{8}]\cup [\frac{\sqrt{17}}{8},\infty).$ We also show that $i\calL_{GP}$ has no threshold or embedded eigenvalues in its essential spectrum. In Sections \ref{sec:Sol-near-0-non-resonance} and \ref{sec:Sol-near-0-resonance}, we construct the fundamental matrix system of solutions to $i \calL F =z F$ near $r=0,$ for small and large $|z|,$ in the non-resonance and resonance case respectively. In Sections \ref{sec:nearinfty-small-xi} and \ref{sec:nearinfty-large-xi}, we construct the Weyl-Titchmarsh solutions to $i \calL \Psi^{\pm} =z \Psi^{\pm}$ near $r=\infty,$ for small and large $|z|,$ respectively. In Sections \ref{sec:GreenKernel-non-resonance} and \ref{sec:GreenKernel-resonance}, we compute the connection coefficients of solutions constructed near the origin and at $\infty.$ Then, we study the Green's kernel of $i \calL -z .$ Moreover, we derive an integral representation of the evolution $e^{t \calL}$ in terms of the resolvent. This representation relies on the Stone-type formula obtained on in \S \ref{subsec:STF}, the fact that the coefficients appearing in the resolvent kernel expressions are sufficiently regular (by analysis obtained in \S \ref{sec:GreenKernel-non-resonance}-\ref{sec:GreenKernel-resonance}) and the absence of embedded eigenvalues. Finally, we extract the distorted Fourier transform representation of evolution $e^{t \calL}$ for all energies. In Section \ref{sec:L2bounds}, we establish the $L^2$-bound for the operator norm $\|e^{t \calL}\|$ and conclude the proof of the main Theorem~\ref{Main-theorem}. In Appendix \ref{sec:appOde}, we present the proof of Theorem \ref{th:ode}. In Appendix \ref{sec:resol-iL0}, we compute the resolvent kernel for $(i\calL_0 -z)^{-1}. $

\subsection{Notation and conventions} For simplicity, we write $\rho$ in place of $\rho_1$ in what follows. We denote by $C>0$ an absolute constant that may change from line to line. For non-negative $f$ and $g$ we write $f \lesssim  g,$ if $f \leq C g,$ and if $f \ll g$ to indicate that the implicit constant is small. We use the notation $f=O(g)$ if $|f| \lesssim g. $ The notation $\mathbf{J} $ and $J_2$ stands for the matrices

\begin{align*}
\mathbf{J}:=\begin{pmatrix}
    0 & -1 \\ -1 & 0
\end{pmatrix}, \qquad 
J_2:= \begin{pmatrix}
    1 & 1 \\ 1 &1 
\end{pmatrix} 
\end{align*} 
and we denote  by $\sigma_1,\sigma_2,\sigma_3$ Pauli matrices 
   \begin{align*}
   \sigma_1=\begin{pmatrix}
       0 &1 \\1 & 0
   \end{pmatrix} 
   \quad    \sigma_2=\begin{pmatrix}
       0 &-i \\i & 0
   \end{pmatrix} 
   \quad 
      \sigma_3=\begin{pmatrix}
       1 &0 \\0 & -1
   \end{pmatrix} 
\end{align*}
For real-valued scalar functions $f$ and $g$ we use the notation
\begin{align*}
    \langle f,g\rangle:=\int_0^\infty f(r)g(r) d r.
\end{align*}
For $\R^2$-valued functions $\mathbf{f} =\begin{pmatrix}
    f_1 \\ f_2
\end{pmatrix}$ and $\mathbf{g} =\begin{pmatrix}
    g_1 \\ g_2
\end{pmatrix}, $ we use the notation $\langle \mathbf{f} ,\mathbf{g}\rangle:=\langle f_1,g_1\rangle+\langle f_2,g_2\rangle$ for real vector inner product. We denote the Wronskian between two vectors $\mathbf{f}$ and $\mathbf{g}$ by  $
    W(\mathbf{f},\mathbf{g}) := \langle \mathbf{f}, \sigma_3 \mathbf{g}^{\prime} \rangle,$ so that $W(\mathbf{f},\mathbf{f}) = 0$.  Next, we define the matrix Wronskian:
\begin{align*}
    W(F,G) :&= F^t \sigma_3 G^{\prime} - (F^{\prime})^t \sigma_3 G \\
    &= \begin{pmatrix}
        W(F_1,G_1) & W(F_1,G_2) \\ 
        W(F_2,G_1) & W(F_2,G_2)
    \end{pmatrix},
\end{align*}
where $F = \begin{bmatrix} F_1 & F_2 \end{bmatrix}$ and 
$G = \begin{bmatrix} G_1 & G_2 \end{bmatrix}$ are $2 \times 2$ matrices, with $F_i$ and $G_i$ denoting their respective columns.

The norm $\left\| \cdot \right\|_{L^2_r}$ denotes the  $L^2$ norm on the half-line with respect to the measure $d r$. We also denote by $X_0:=L^2(\R^{+})$, $\R^{+}=[0,\infty),$ with respect to the measure $dr$ and let $X_1:=X_0 \times X_0.$ We also denote by  $   \mathcal{D}:=\{ f \in L^2_r(0,\infty) \,|\;  \mathfrak{L}_0f \in L^2_r(\R_+) \}.$

Throughout the paper, the expression “large \(|\xi|\)” (or “large \(|z|\)”) will refer to the regime where the real part is large while the imaginary part remains bounded. We also use the expression “small \(|\xi|\)” (or “small \(|z|\)”) to refer to energies close to the spectral threshold $\pm \frac{\sqrt{17}}{8}.$\\
Note that, throughout the rest of the paper, we denote $\rho_1(r)$ by $\rho(r)$ and $\trho(r)=\sinh^{\frac{1}{2}}(r)\rho(r)$

\section{Spectral properties of the linearized operator} \label{sec:spectrum}
In this section, we establish basic spectral properties of $i\mathcal{L}$ at the spectrum threshold and show that the linearized operator $\mathcal{L}_{GP}$ associated with the Gross–Pitaevskii equation satisfies Assumptions~\ref{asmp:1}–\ref{asmp:3}, except for the emptiness of the discrete spectrum. In \S\ref{subsec:DS-iL}, we classify the behavior of solutions to the threshold problem, $i \calL \Phi=\pm \frac{\sqrt{17}}{8} \Phi$. We distinguish two possible behaviors of solutions, which lead to the characterization of the threshold resonance and non-resonance cases. In \S\ref{subsec:HY}, we define the evolution operator $e^{t\mathcal{L}}$ and derive an exponential bound via the Hille–Yosida theorem. In \S\ref{subsec:STF}, we obtain an integral representation of the evolution. Finally, in \S\ref{subsec:DS-L_GP}, we study the spectrum of the linearized operator $i\mathcal{L}_{GP}$, proving that it is real, that its essential spectrum is $(-\infty,-\frac{\sqrt{17}}{8}] \cup [\frac{\sqrt{17}}{8}, \infty)$, and that there are no threshold or embedded eigenvalues in the essential spectrum.\\
Throughout the rest of the paper, we denote $\rho_1(r)$ by $\rho(r)$ and $\trho(r)=\sinh^{\frac{1}{2}}(r)\rho(r)$. \\

We first define the following operators in the Hilbert space 
$L^2(\R_+)$ relative to the Lebesgue measure $dr$:
\begin{align*}
    L_1 &:= \mathfrak{L}_0 + V_1 , \quad L_2:=  \mathfrak{L}_0 +2+ V_2
\end{align*}
with domain $C^2_c(\R_+).$ We also define the matrix operator
\begin{align*}
    \mathcal{L} &:=  \begin{pmatrix}
        0 & L_1 \\ -L_2 & 0 
    \end{pmatrix},
\end{align*} 
with domain $C^2_c(\R_+) \times C^2_c(\R_+)$. For the linearized operator $\mathcal{L}_{GP}$ associated with the Gross–Pitaevskii equation, we define
\begin{align*}
    \mathcal{L}_{GP} &:=  \begin{pmatrix}
        0 & \mathfrak{L}_1 \\ -\mathfrak{L}_2 & 0 
    \end{pmatrix},
\end{align*}
where
\begin{align*}
      \mathfrak{L}_1 = \mathfrak{L}_0+\rho^2-1, \quad \mathfrak{L}_2 =   \mathfrak{L}_1 + 2 \rho^2 =\mathfrak{L}_0+3\rho^2-1, \quad \text{with} \quad \mathfrak{L}_0 := -\frac{1}{2 }\partial_r^2+\frac{3}{8  \sinh^2(r)} + \frac{1}{8}. 
\end{align*}

\subsection{Discrete spectrum of $i\mathcal{L}$}
\label{subsec:DS-iL}

In this section, we establish the basic discrete spectrum properties at the spectrum threshold using Assumptions \ref{asmp:1}, \ref{asmp:2}, \ref{asmp:3}. Let $\lambda= \pm \frac{\sqrt{17}}{8},$ and $\Phi := \begin{pmatrix} \phi \\ \psi \end{pmatrix} .$ Then a solution to  $ i \mathcal{L} \Phi= \lambda \Phi,$ satisfies 
 \begin{equation}
\begin{cases}
\label{sys:L_1=lambda-L2=lambda}
    i L_1 \psi &= \lambda \phi ,\\
    - i L_2 \phi &= \lambda \psi    
\end{cases}    
 \end{equation}

Note that, for small $r$, the operator 
$
-\frac{1}{2} \partial_r^2 + \frac{3}{8 r^2}
$ 
provides a good approximation for both $L_1$ and $L_2$.  Therefore, by a perturbative analysis, the possible asymptotics of the fundamental system for the full system \eqref{sys:L_1=lambda-L2=lambda} near $0$ are
\begin{equation} \label{eq:behav-near-0} 
\left\{
\begin{pmatrix}
    r^{\frac{3}{2}}(1+O(r^2))  \\ O(r^{\frac{7}{2}}) 
\end{pmatrix},\quad
\begin{pmatrix}
    O(r^{\frac{7}{2}}) \\ r^{\frac{3}{2}}(1+O(r^2)) 
\end{pmatrix},\quad
\begin{pmatrix}
    r^{-\frac{1}{2}}(1+O(r^2\log r) \\ O(r^{\frac{3}{2}}\log r) 
\end{pmatrix},\quad
\begin{pmatrix}
    O(r^{\frac{3}{2}}\log r) \\ r^{-\frac{1}{2}}(1+O(r^2\log r) 
\end{pmatrix} 
\right\}.
\end{equation}
For instance a solution with the first asymptotics can be obtained by writing $\phi=r^{\frac{3}{2}}+\tilde{\phi}$ and solving the Voltera integral fixed point for $r\in [0,r_0]$, $r_0\ll1$ (with suitable constant $c_1$ and $c_2$)
\begin{align*}
    \tilde{\phi}(r)=c_1\int_0^r(r^{\frac{3}{2}}s^{-\frac{1}{2}}-r^{-\frac{1}{2}}s^{\frac{3}{2}})(\lambda\psi(s)+iV_2(s)(s^{\frac{3}{2}}+\tilde{\phi}(s)))d s,\\
     \psi(r)=c_2\int_0^r(r^{\frac{3}{2}}s^{-\frac{1}{2}}-r^{-\frac{1}{2}}s^{\frac{3}{2}})(\lambda(s^{\frac{3}{2}}+\tilde{\phi}(s))-iV_1(s)\psi(s))d s.
\end{align*}
Similarly the for the third displayed asymptotics we write $\phi=r^{-\frac{1}{2}}+\tilde{\phi}$ and use
\begin{align*}
    \tilde{\phi}(r)&=c_1\int_r^{r_0} r^{\frac{3}{2}}s^{-\frac{1}{2}}(\lambda\psi(s)+iV_2(s)(s^{-\frac{1}{2}}+\tilde{\phi}(s)))d s\\
    &\quad+c_1\int_0^r r^{-\frac{1}{2}}s^{\frac{3}{2}}(\lambda\psi(s)+iV_2(s)(s^{-\frac{1}{2}}+\tilde{\phi}(s)))d s,\\
     \psi(r)&=c_2\int_r^{r_0} r^{\frac{3}{2}}s^{-\frac{1}{2}}(\lambda((s^{-\frac{1}{2}}+\tilde{\phi}(s)))-iV_1(s)\psi(s))d s\\
     &\quad+c_2\int_0^r r^{-\frac{1}{2}}s^{\frac{3}{2}}(\lambda((s^{-\frac{1}{2}}+\tilde{\phi}(s)))-iV_1(s)\psi(s))d s.
\end{align*}

 For large $r$  first observe that applying $i L_1$ to the second equation in \eqref{sys:L_1=lambda-L2=lambda}, we obtain a fourth-order equation 
\begin{equation}
\label{eq:L_1L2=lambda-square}
   L_1 L_2 \phi = \lambda^2 \phi. 
\end{equation}
Here the operators $\tilde{L}_1:=( -\frac{1}{2} \partial_r^2 + \frac{1}{8}) $ and $\tilde{L}_2:=( -\frac{1}{2} \partial_r^2 + \frac{17}{8})$ provide a good approximation for $L_1$ and $L_2,$ respectively. Therefore, the equation \eqref{eq:L_1L2=lambda-square} is, near $r = \infty$, approximately equivalent to
\begin{align*}
   ( -\frac{1}{2} \partial_r^2 + \frac{1}{8})   ( -\frac{1}{2} \partial_r^2 + \frac{17}{8}) \phi =\frac{17}{64}\phi
\end{align*}
Thus, the the possible asymptotics for the fourth-order equation near $\infty$ are given by  
$$\{e^{\frac{3}{\sqrt{2}} r},e^{-\frac{3}{\sqrt{2}} r}, 1 , r \}.$$ 
Therefore, the possible asymptotics of the fundamental system for the full system \eqref{sys:L_1=lambda-L2=lambda} for $r\gg1$ are
$$\bigg\{ f_{1,\infty}^{\pm}:=\begin{pmatrix}
    1 \\  \mp i \sqrt{17}
\end{pmatrix} , f_{2,\infty}^{\pm}:=\begin{pmatrix}
   e^{-\frac{3}{\sqrt{2}} r} \\ \pm i \frac{1}{\sqrt{17}}   e^{-\frac{3}{\sqrt{2}} r}
\end{pmatrix},f_{3,\infty}^{\pm}:=\begin{pmatrix}
  r \\  \mp i \sqrt{17}r
\end{pmatrix}, f_{4,\infty}^{\pm}:=\begin{pmatrix}
   e^{\frac{3}{\sqrt{2}} r} \\ \pm i \frac{1}{\sqrt{17}}   e^{\frac{3}{\sqrt{2}} r}
\end{pmatrix}
\bigg\}. $$
Indeed this can be seen by setting up a fixed point problem for the integral representations as before. The formula for the Green's functions are a bit more complicated compared to $r\ll1$ due to the coupling at the linear level. We will derive such matrix Green's functions systematically in Section~\ref{sec:Sol-near-0-non-resonance} below. The final result here can be written as follows. Let $\mathcal{V}$ be defined by $i\mathcal{L}=\begin{pmatrix}0&i\tilde{L_1}\\-i\tilde{L}_2&0\end{pmatrix}+\mathcal{V}$. Let $d_1=\frac{1}{18}$ and $d_2=\frac{17\sqrt{2}}{108}$. These are the nonzero Wronskians among the $f_{j,\infty}^{\pm}.$ Then we define $f_j^{\pm}$, $j=1,2,3,4$, by solving the following fixed point problems for $r>r_\infty\gg1$ (with suitable constants $c_j$): 
\begin{align*}
    f_1^{\pm}(r)&=c_1\int_{r_\infty}^r(d_1 f_{1,\infty}^{\pm}(r)f_{3,\infty}^{\pm}(s)^t +d_2 f_{2,\infty}^{\pm}(r) f_{4,\infty}^{\pm}(s)^t) \sigma_1 \mathcal{V} (s)f_1^{\pm} (s)ds\\
    &\quad+c_1\int_r^\infty(d_1 f_{3,\infty}^{\pm}(r)f_{1,\infty}^{\pm}(s)^t +d_2 f_{4,\infty}^{\pm}(r) f_{2,\infty}^{\pm}(s)^t) \sigma_1 \mathcal{V}(s) f_1^{\pm}(s) ds,
\end{align*}
\begin{align*}
    f_2^{\pm}(r)&=c_2\int_r^\infty(d_1 f_{1,\infty}^{\pm}(r)f_{3,\infty}^{\pm}(s)^t +d_2 f_{2,\infty}^{\pm}(r) f_{4,\infty}^{\pm}(s)^t) \sigma_1 \mathcal{V}(s) f_2^{\pm} (s)ds\\
    &\quad-c_2\int_r^\infty (d_1 f_{3,\infty}^{\pm}(r)f_{1,\infty}^{\pm}(s)^t +d_2 f_{4,\infty}^{\pm}(r) f_{2,\infty}^{\pm}(s)^t) \sigma_1 \mathcal{V} (s)f_2^{\pm}(s) ds,
\end{align*}
\begin{align*}
    f_3^{\pm}(r)&=c_3\int_{r_\infty}^r(d_1 f_{1,\infty}^{\pm}(r)f_{3,\infty}^{\pm}(s)^t +d_2 f_{2,\infty}^{\pm}(r) f_{4,\infty}^{\pm}(s)^t) \sigma_1 \mathcal{V}(s) f_3^{\pm}(s) ds\\
    &\quad+c_3\int_r^\infty (d_1 f_{3,\infty}^{\pm}(r)f_{1,\infty}^{\pm}(s)^t +d_2 f_{4,\infty}^{\pm}(r) f_{2,\infty}^{\pm}(s)^t) \sigma_1 \mathcal{V}(s) f_3^{\pm}(s) ds,
\end{align*}
\begin{align*}
    f_4^{\pm}(r)&=c_4\int_{r_\infty}^r(d_1 f_{1,\infty}^{\pm}(r)f_{3,\infty}^{\pm}(s)^t +d_2 f_{2,\infty}^{\pm}(r) f_{4,\infty}^{\pm}(s)^t) \sigma_1 \mathcal{V}(s) f_4^{\pm}(s) ds\\
    &\quad-c_4\int_{r_\infty}^r(d_1 f_{3,\infty}^{\pm}(r)f_{1,\infty}^{\pm}(s)^t +d_2 f_{4,\infty}^{\pm}(r) f_{2,\infty}^{\pm}(s)^t) \sigma_1 \mathcal{V}(s) f_4^{\pm}(s) ds,
\end{align*}

Under Assumption \ref{asmp:2}, we know that $\lambda = \pm \frac{\sqrt{17}}{8}$ is not an eigenvalue. Therefore, any exponentially decaying solution, i.e., $L^2(1,\infty)$  must exhibit singular behavior near the origin, i.e., not an $L^2(0,1)$-solution. For instance, any solution that behaves like $e^{- \frac{3}{\sqrt{2}}r}$ near infinity  must grow  like $r^{-\frac{1}{2}}$ as $r \to 0$.
Moreover, any solution that behaves like $r^{\frac{3}{2}}$ near the origin can either grow at infinity or be a constant.  Next, observe that under Assumption \ref{asmp:3}, using an argument similar to that in section \ref{subsec:DS-L_GP}, particularly Lemma \ref{lem:no-eigenvalue-singular-r0}, one can show that  there exists a solution that behaves like \(r^{\frac{3}{2}}\) near zero and has exponential growth at infinity.\\

On the other hand, solutions that exhibit intermediate behavior at infinity (i.e., bounded like $1$ or growing polynomially like $r$) may behave near the origin either like $r^{\frac{3}{2}}$, which is  $L^2(0,1)$, or like $r^{-\frac{1}{2}}$, which is not square-integrable on $(0,1)$.

\begin{defi} \label{def:resonance}
We distinguish two possible behaviors of bounded solutions near $\infty$.
\begin{itemize}
    \item \textbf{Case 1 (Non-resonance case):} If the constant behavior at infinity connects to the $r^{-\frac{1}{2}}$ behavior near the origin, the solution does not lie in $L^\infty(0,\infty).$
    
    \item \textbf{Case 2 (resonance case):} If the constant behavior at infinity connects to the $r^{\frac{3}{2}}$ behavior near the origin (i.e., square-integrable on $(0,1)$), hence the solution lies in $L^\infty(0,\infty).$ 
\end{itemize}
    
\end{defi}

\begin{remark}\label{lem:resterminlogy}
The terminology \emph{resonance} and \emph{non-resonance} is not  standard here. In particular we make no claim about mapping properties of the resolvent on weighted spaces. For scalar operators resonances can occur when there is a (non-$L^2$) solution connecting the subordinate behaviors near $r=0$ and $r=\infty$. Here there are four possible behaviors  for large $r$. A solution connecting the subordinate behaviors would be in $L^2$ due to the exponential decay. The resonant/non-resonant terminology is simply used to distinguish between the two intermediate behaviors at large $r$.
\end{remark}

\subsubsection{Case $1$ :No Resonance}
\label{subsec:non-resonance}
Based on the above observation, we conclude that, for $\lambda=\frac{\sqrt{17}}{8}:$ \\

There exist two solutions $F_{1}^{+}(r)$ and $F_{2}^{+}(r)$ solutions to $i \mathcal{L} F_j^{+}(r)=  \frac{\sqrt{17}}{8}F_j^{+}(r),$ for $j=1,2,$ such that $F_{1}^{+}(r) \in L^2(0,1)$ and it is unique fundamental matrix solution of the spectral problem, with this property, up to the transformation $F_1 \mapsto F_1 B,$ where $ B \in \mathrm{GL}(2, \mathbb{C}).$ \\

Moreover, $F_{1}^{+}(r)=\begin{bmatrix}
    \varphi_1^{+}(r)& \varphi_2^{+}(r)
\end{bmatrix}$ and $F_{2}^{+}(r)=\begin{bmatrix}
   \varphi_3^{+}(r)& \varphi_4^{+}(r) 
\end{bmatrix},$ where $\varphi_i^{+}$ satisfy the following asymptotics 
\begin{align}
\label{eq:asumptotic-phi_j(r)-non-resonance}
   \varphi_1^{+}(r)&= \begin{cases} \nonumber
        \begin{pmatrix}
         0 \\ 
           c_1^{2} r^{\frac{3}{2}}
         \end{pmatrix} + O(r^{\frac{7}{2}}), \; r \to 0  \\ \\
          \begin{pmatrix}
           \tc_1^1 e^{\frac{3}{\sqrt{2} }r} \\ 
           \tc_1^2 e^{\frac{3}{\sqrt{2}}r}
         \end{pmatrix}(1+O(r^{-1})) , \; r \to \infty 
   \end{cases} 
 \quad \varphi_2^{+}(r)= \begin{cases}
        \begin{pmatrix}
         c_2^{1} r^{\frac{3}{2}} \\ 
         c_2^2 r^{\frac{3}{2}}
         \end{pmatrix}  (1+ O(r^2)) , \; r \to 0  \\ \\
          \begin{pmatrix}
           \tc_2^1 r  \\ 
          \tc_2^2 r 
         \end{pmatrix} (1+O(r^{-1})) , \; r \to \infty 
   \end{cases} \\ 
    \varphi_3^{+}(r)&= \begin{cases}
        \begin{pmatrix}
         c_3^1 r^{-\frac 12} \\ 
            c_3^2 r^{-\frac{1}{2}}
         \end{pmatrix} (1+ O(r^2\log(r)), \; r \to 0  \\ \\
          \begin{pmatrix}
          \tc_3^1 e^{-\frac{3}{\sqrt{2} }r}  \\ 
           \tc_3^2 e^{-\frac{3}{\sqrt{2}}r}
         \end{pmatrix}  (1+O(r^{-1})) , \; r \to \infty 
   \end{cases} 
 \quad \varphi_4^{+}(r)= \begin{cases}
        \begin{pmatrix}
         c_4^1 r^{-\frac{1}{2}} \\ 
         0
         \end{pmatrix} + O(r^{\frac{3}{2}}\log(r)) , \; r \to 0  \\ \\
          \begin{pmatrix}
          \tc_4^1  \\ 
           \tc_4^2 
         \end{pmatrix} (1+O(r^{-1})), \; r \to \infty 
   \end{cases}
\end{align}

\begin{remark}
Note that the set $\{ \varphi_1^{+}, \varphi_2^{+}, \varphi_3^{+},\varphi_3^{+}\} $ forms a fundamental system for the equation $i \mathcal{L} F^{+}_j = \frac{\sqrt{17}}{8} F^{+}_j.$
Our choice of the behavior of $ \varphi_1^{+} $ and $ \varphi_1^{+} $ near the origin is made so that it matches the one for $  \mathcal{L}_{GP} ,$ see \S \ref{subsec:DS-L_GP}. 
\end{remark}
Next, we define $F^{-}_1(r)=\begin{bmatrix}
    \varphi_1^{-}(r) & \varphi_2^{-}(r)
\end{bmatrix}$ and $F^{-}_2(r)= \begin{bmatrix}
    \varphi_3^{-}(r) & \varphi_4^{-}(r)
\end{bmatrix}
$ solutions to  $i \mathcal{L} F_j^{-}(r)= - \frac{\sqrt{17}}{8}F_j^{-}(r),$ for $j=1,2,$ such that 
\begin{align*}
   F_1^{-}(r)&=-\sigma_3 F_1^{+}(r) \sigma_3\\
    F_2^{-}(r)&=\sigma_3 F_2^{+}(r) \sigma_3
\end{align*}
where $F^{+}_1(r)$ and $F^{+}_2(r)$ are solutions to the equation $i \mathcal{L} F_j^{+}(r)=  \frac{\sqrt{17}}{8}F_j^{-}(r).$ Moreover, $\varphi_j^{-}(r)$ satisfy the same asymptotics as $\varphi_j^{+}(r)$ for $r$ near $0$ and infinity, i.e., 
\begin{align*}
    \varphi_1^{-}(r)&=-\sigma_3 \varphi_1^{+}(r), \; \; \qquad \varphi_3^{-}(r)=-\sigma_3 \varphi_3^{+}(r), \\
   \varphi_2^{-}(r)&=\sigma_3 \varphi_2^{+}(r), \quad \; \qquad \varphi_4^{-}(r)=\sigma_3 \varphi_4^{+}(r). 
\end{align*}
\subsubsection{Case 2: Resonance Case}
\label{subsec:resonance}
In this case, we note that $ \varphi_1$ and $ \varphi_3 $ exhibit the same behavior as in the previous case. However, the asymptotic behavior of $ \varphi_2 $ and $ \varphi_4 $ at infinity is reversed. Recall that, in this case the constant behavior at infinity corresponds to a smooth behavior near the origin $r^{\frac{3}{2}}$. Consequently, the solution $ \varphi_2$ must belong to $ L^\infty(0,\infty) $. We adopt the notation $ \varphi_j^R $ to indicate that the solution arises in the resonance case. This convention will be used consistently throughout the paper.
Based on the above observation, we conclude that, for $\lambda=\frac{\sqrt{17}}{8}:$ \\

There exist two solutions $F_{1}^{R,+}(r)$ and $F_{2}^{R,+}(r)$ solutions to $i \mathcal{L} F_j^{R,+}(r)=  \frac{\sqrt{17}}{8}F_j^{R,+}(r),$ for $j=1,2,$ such that $F_{1}^{R,+}(r) \in L^2(0,1)$ and it is a unique fundamental matrix solution of the spectral problem, with this property, up to the transformation $F_1 \mapsto F_1 B,$ where $ B \in \mathrm{GL}(2, \mathbb{C}).$ \\

Moreover, $F_{1}^{R,+}(r)=\begin{bmatrix}
    \varphi_1^{R,+}(r) &  \varphi_2^{R,+}(r)
\end{bmatrix}$ and $F_{2}^{R,+}(r)=\begin{bmatrix}
    \varphi_3^{R,+}(r) & \varphi_4^{R,+}(r)
\end{bmatrix},$ where $\varphi_i^{R,+}$ satisfy the following asymptotics 

\begin{align*}
   \varphi_1^{R,+}(r)&= \begin{cases}
        \begin{pmatrix}
         0 \\ 
           c_1^{2} r^{\frac{3}{2}}
         \end{pmatrix} + O(r^{\frac{7}{2}}), \; r \to 0  \\ \\
          \begin{pmatrix}
           \tc_1^1 e^{\frac{3}{\sqrt{2} }r} \\ 
           \tc_1^2 e^{\frac{3}{\sqrt{2}}r}
         \end{pmatrix}(1+O(r^{-1})) , \; r \to \infty 
   \end{cases} 
 \quad \varphi_2^{R,+}(r)= \begin{cases}
        \begin{pmatrix}
         c_2^{1} r^{\frac{3}{2}} \\ 
         c_2^2 r^{\frac{3}{2}}
         \end{pmatrix} (1+ O(r^2)) , \; r \to 0  \\ \\
          \begin{pmatrix}
           \tc_2^1   \\ 
          \tc_2^2  
         \end{pmatrix} (1+O(r^{-1})) , \; r \to \infty 
   \end{cases} \\ 
    \varphi_3^{R,+}(r)&= \begin{cases}
        \begin{pmatrix}
         c_3^1 r^{-\frac 12} \\ 
            c_3^2 r^{-\frac{1}{2}}
         \end{pmatrix} (1+ O(r^2\log(r)), \; r \to 0  \\ \\
          \begin{pmatrix}
          \tc_3^1 e^{-\frac{3}{\sqrt{2} }r}  \\ 
           \tc_3^2 e^{-\frac{3}{\sqrt{2}}r}
         \end{pmatrix}  (1+O(r^{-1})) , \; r \to \infty 
   \end{cases} 
 \quad \varphi_4^{R,+}(r)= \begin{cases}
        \begin{pmatrix}
         c_4^1 r^{-\frac{1}{2}} \\ 
         0
         \end{pmatrix} + O(r^{\frac{3}{2}}\log(r) , \; r \to 0  \\ \\
          \begin{pmatrix}
          \tc_4^1 r \\ 
           \tc_4^2 r
         \end{pmatrix} (1+O(r^{-1})), \; r \to \infty 
   \end{cases}
\end{align*}

Similarly to the non-resonance case, we define $F^{R,-}_1(r)=[\varphi_1^{R,-}(r) \; \varphi_2^{R,-}(r)]$ and $F^{-}_2(r)=[\varphi_3^{R,-}(r) \; \varphi_4^{R,-}(r)]$ solutions to  $i \mathcal{L} F_j^{R,-}(r)= - \frac{\sqrt{17}}{8}F_j^{R,-}(r),$ for $j=1,2,$ such that 
\begin{align*}
   F_1^{R,-}(r)&=-\sigma_3 F_1^{R,+}(r) \sigma_3\\
    F_2^{R,-}(r)&=\sigma_3 F_2^{R,+}(r) \sigma_3
\end{align*}
where $F^{R,+}_1(r)$ and $F^{R,+}_2(r)$ are solutions to the equation $i \mathcal{L} F_j^{R,+}(r)=  \frac{\sqrt{17}}{8}F_j^{R,-}(r).$ Moreover, $\varphi_j^{R,-}(r)$ satisfy the same asymptotics as $\varphi_j^{R,+}(r)$ for $r$ near $0$ and infinity, i.e., 

\begin{align*}
    \varphi_1^{R,-}(r)&=-\sigma_3 \varphi_1^{R,+}(r), \; \; \qquad \varphi_3^{R,-}(r)=-\sigma_3 \varphi_3^{R,+}(r), \\
   \varphi_2^{R,-}(r)&=\sigma_3 \varphi_2^{R,+}(r), \quad \; \qquad \varphi_4^{R,-}(r)=\sigma_3 \varphi_4^{R,+}(r). 
\end{align*}

\subsection{Hille-Yosida bound} \label{subsec:HY}
The main goal of this section is to define the evolution $e^{t\calL }$ via the Hille-Yosida theorem. 

\begin{lemma} \label{lem:HY}
The matrix operator $\mathcal{L}$  is closed in $\mathcal{D} \times \mathcal{D} \subset X_1$ and it generates a group of bounded operator on $X_1$ satisfying $\left\| e^{t \mathcal{L}}\right\|_{X_1} \leq e^{ a t} $ for $t \geq 0,$ and for some $a>0.$
\end{lemma}
\begin{proof}
The proof follows similar argument to that of Lemma 2.5 in \cite{LSS25} which is in turn based on \cite{ES2,KS}. Since $L_1$ and $L_2$ are self-adjoint, and hence closed, it follows that $\mathcal{L}$ is closed on $\calD \times \calD$. Next, it is more convenient to consider a unitarily equivalent operator to $\calL.$ Let $U= \frac{1}{\sqrt{2}}\begin{pmatrix}
    1 & i \\ 1 & -i
\end{pmatrix}$ and define 
\begin{align*}
   i \mathcal{H}= - U \mathcal{L} U^{-1}=i \begin{pmatrix}
       \mathbf{L} & \frac{V_2(r)-V_1(r)}{2}+1 \\
       \frac{V_1(r)-V_2(r)}{2}-1 & -\mathbf{L} 
   \end{pmatrix} , 
\end{align*}
where $\mathbf{L}:=\mathfrak{L}_0 + 1 +   \frac{V_1(r)+V_2(r)}{2}.$ We have the following resolvent identity 
\begin{align*}
    ( i \mathcal{H} - \lambda)^{-1} = R(\lambda) ( \mathrm{Id} + \mathfrak{V} R(\lambda) )^{-1}, 
\end{align*}
where \begin{align*}
  R(\lambda):= \begin{pmatrix}
      (i \mathbf{L}-\lambda)^{-1} & 0 \\ 
      0 &  (i \mathbf{L}-\lambda)^{-1}
  \end{pmatrix}  , \qquad \mathfrak{V}:= i \begin{pmatrix}
      0 & \frac{V_2(r)-V_1(r)}{2}+1  \\
     \frac{V_1(r)-V_2(r)}{2}-1  & 0   
  \end{pmatrix}.                                                                       
\end{align*}
Using the fact $\mathbf{L}$ is positive and self-adjoint operator, we have  for $\lambda > a, $ for some $a>0$
\begin{align*}
   \left\| ( i \mathcal{H} - \lambda)^{-1} \right\| \leq \frac{1}{\lambda} \left( 1- \frac{\left\| \mathfrak{V} \right\|_{\infty}}{\lambda} \right)^{-1} \leq \frac{1}{\lambda -a}
\end{align*}
Applying the Hille-Yosida theorem to $i \mathcal{H},$ we obtain the existence of the group of bounded operators $e^{it \mathcal{H}}$ and the exponential bound $\left\| e^{i t \mathcal{H}} \right\| \leq e^{a t}.$ Since $i \mathcal{H}$ and $\calL$ are unitary equivalent, we obtain $\left\| e^{ t \mathcal{L}} \right\| \leq e^{a t}.$
\end{proof}
\begin{remark}
Note that, if $\calL=\calL_{GP},$ we can take $a=1$ by inspection of the $L^{\infty}$-norm of the matrix $\mathfrak{V}.$
\end{remark}

\subsection{A Stone-type formula.} \label{subsec:STF}
The main goal of this section is to derive a Stone-type formula for the evolution $e^{t \calL}$. For this purpose, we first study the resolvent of $(i \calL_0-z)^{-1}$ of the free operator 
\begin{align*}
  \calL_0:= \begin{pmatrix}
      0 & \mathfrak{L}_0 \\
      -(\mathfrak{L}_0 +2) & 0 
  \end{pmatrix},
\end{align*}
where $\mathfrak{L}_0= -\frac{1}{2} \partial_r^2 + \frac{3}{8 \sh^2(r) } + \frac{1}{8}.$
Consider the following differential equation
\begin{align}
    \label{eq:ODE-L0}
-\partial_r^2 h + \frac{3}{4 \sh(r)^2} h = z^2 h   . 
\end{align}  
Let $\varphi_0(r,z)$ be the elementary spherical functions in $\HH^4$ satisfying $(-\Delta_{\HH^4}-\frac{9}{4}) \varphi_0(r,z)= z^2 \varphi_0(r,z).$ Then $\phi_0(r,z)=\sh^{\frac{3}{2}}(r) \varphi_0(r,z) ,$ is a solution to the equation \eqref{eq:ODE-L0}. It well known that in view of the asymptotics of $\varphi_0$ that $\phi_0(r,z)=c \,r^{\frac{3}{2}} + O(r^{\frac{5}{2}})$ for small $r$ and $\phi_0(r,z) \sim c_1 \, e^{i z r} + c_2 \, e^{-i z r}$ for large $r$.
Let $\psi_0(r,z)$ be a solution to \eqref{eq:ODE-L0} such that $(\phi_0,\psi_0)$ forms a fundamental system of solutions to \eqref{eq:ODE-L0}. Notice that only possible behavior for large $r$ are $e^{i z r}$ and $e^{-i z r}.$ We consider $\psi_0(r,z)$ to be solution with asymptotic 
$e^{i z r}$ for large $r.$ Therefore, in view of the possible behaviors near the origin we must have $\psi_0(r,z)=\tc_1 r^{-\frac{1}{2}} + \tc_2 r^{\frac{3}{2}} + o(1)$ for small $r. $ \\

Next, suppose that $\Phi:= \begin{pmatrix}
    u \\ v
\end{pmatrix}$ solves the $i\mathcal{L}_0\Phi=z \Phi.$ Then, $u(r,z)$ and $v(r,z)$ must satisfy \begin{align}
\begin{cases}
     \qquad  \quad   i \mathfrak{L}_0 v = z u, \\  
       -i (\mathfrak{L}_0 +2) u = z v. 
\end{cases}
\end{align}
Therefore, $u(r,z)$ must satisfy $\mathfrak{L}_0(\mathfrak{L}_0 +2) u =z^2 u$ and 
corresponding $v(r,z)= -\frac{i}{z} (\mathfrak{L}_0 +2) u(r,z).$ The ansatz $u(r,z)=\psi_0(r,k(z))$ leads to the equation 
\begin{align}
\label{eq:P(k,z)-1}
P(k,z):&=    \frac{1}{4} k^4 + \frac{9}{8} k^2 +\frac{17}{64}  - z^2 =0, \quad \text{with } \; k \equiv k(z)
\end{align} 
This equation admits four distinct roots, which we denote $k_j(z)$ for $j=1,2,3,4.$ See section~\ref{sec:nearinfty-small-xi} and Lemma~\ref{lem:k_j-behavior} for the derivation and behavior of these complex roots. In view of the behavior of $k_j,$ we will only use $k_1(z)$ and $k_2(z).$ Then, in view of Lemma \ref{lem:k_j-behavior}, we obtain the Weyl-Titchmarsh matrix solutions of the operator $i \calL_0:$
\begin{align*}
    \Psi_{(0)}^{\pm}(r,z):=\begin{bmatrix}
        \Psi_{(0,1)}^{\pm}(r,z) & \Psi_{(0,2)}^{\pm}(r,z)
    \end{bmatrix} = \begin{pmatrix}
        \psi_0(r,k_1(z))  &  \psi_0(r,k_2(z)) \\
        c_1(z) \psi_0(r,k_1(z))  &     c_2(z) \psi_0(r,k_2(z))
    \end{pmatrix}  
\end{align*}
and \begin{align*}
   \PPsi_{(0)}^{\pm}(r,z):=\begin{bmatrix}
        \Psi_{(0,3)}^{\pm}(r,z) & \Psi_{(0,4)}^{\pm}(r,z)
       \end{bmatrix} =  \begin{pmatrix}
        \phi_0(r,k_1(z))  &  \phi_0(r,k_2(z)) \\
        c_1(z) \phi_0(r,k_1(z))  &     c_2(z) \phi_0(r,k_2(z))
    \end{pmatrix} 
\end{align*}
where,  \begin{align*}
   c_j(z)= c(k_j(z),z)&:=     - i     \frac{ ( \frac{1}{2} k_j(z)^2 + \frac{17}{8}) }{ z} \qquad \text{for} \; j=1,2.
\end{align*}
Note that, $c_j(-z)=-c_j(z)$ and $ \Psi_{(0)}^{\pm} $ is $L^2$ near $\infty$ and  $\PPsi_{(0)}^{\pm}$ is $L^2$ near $r=0.$  

We define the Green's function as 
\begin{align*}
     \mathfrak{R}_{ij}^{\pm}(r,s,z):=i\bigg( \Psi_{(0,i)}^{\pm}(r,z) \Psi_{(0,j)}^{\pm }(s,z)^t \mathbb{1}_{\{ 0\leq s \leq r\} } + \Psi_{(0,j)}^{\pm}(r,z) \Psi_{(0,i)}^{\pm  }(s,z)^t \mathbb{1}_{\{ r\leq s \leq \infty \} } \bigg) \sigma_1
\end{align*}
and we denote by 
\begin{align*}
   \mu_{ij}^{0,\pm}(z)=W(\Psi_{(0,i)}^{\pm}(\cdot,z),\Psi_{(0,j)}^{\pm}(\cdot,z))
\end{align*}

\begin{lemma}
\label{lem::kernel-resol-L0}
The kernel of the resolvent $(i \calL_0 -z)^{-1}$ is given by 
\begin{align}
\label{eq:kernelR0}
    \mathcal{R}_0^{\pm}(r,s,z):=\frac{1}{\mu_{13}^{0,\pm}(z)}  \mathfrak{R}_{13}^{\pm}(r,s,z) + \frac{1}{\mu_{24}^{0,\pm}(z)}  \mathfrak{R}_{24}^{\pm}(r,s,z),
\end{align}
Moreover, for large $|z|$ we have\begin{align}
\label{eq:mu_ij-0}
     |\mu_{13}^{0,\pm}(z)|\simeq|\mu_{24}^{0,\pm}(z)|\simeq \sqrt{|z|}
\end{align}
\end{lemma}
\begin{proof}
The proof of this lemma relies on details and techniques introduced in later sections, in particular in sections \ref{sec:nearinfty-small-xi}-\ref{sec:GreenKernel-non-resonance}. Since it is independent of the rest of the material in this section, we postpone the proof to Appendix~\ref{sec:resol-iL0}.
\end{proof}

\begin{lemma} \label{lem:ST1}
Let $t\in \R.$   For any $\phi,\psi \in \calD \times \calD,$ we have 
\begin{align}
\label{eq:ST1}
    \langle e^{t \mathcal{L}} \Phi, \Psi \rangle&= \lim_{R \to \infty} \lim_{b \to 0^{+}}  \frac{1}{2\pi i} \int_{ib-R}^{i b+R} e^{it \lambda} \langle \left( e^{-bt}(i\mathcal{L}-(\lambda + i b))^{-1} - e^{bt}(i\mathcal{L}-(\lambda - i b))^{-1}  \right) , \Phi, \Psi \rangle d \lambda
\end{align}
\end{lemma}
\begin{proof}
The proof very similar to \cite[Lemma 6.8]{KS}, \cite[Lemma 12]{ES2}, and \cite[Lemma 2.7]{LSS25}. First notice that, using the Laplace transform with the Hille-Yosida operator, we can write the resolvent for $\re(z)>a$ ($a$ as in Lemma \ref{lem:HY}) as
\begin{align} \label{eq:Resol-int}
    (i\calL-z)^{-1} \phi=-\int_0^{\infty} e^{-sz} e^{s \calL} \phi ds.
\end{align}
Then arguing as the works mentioned above, we can invert the Laplace transform to conclude that for $t\in \mathbb{R}$ and for any $\phi,\psi \in \calD \times \calD$ and any $b>a$ we have 
\begin{align}
\label{eq:Weak-represt-L-R}
      \langle e^{t \calL} \phi ,\psi \rangle=\frac{1}{2\pi i} \lim_{R \to \infty} \bigg[ \int_{ib-R}^{ib+R} e^{-it z} \langle(i \calL -z)^{-1}\phi, \psi \rangle dz  - \int_{-ib-R}^{-ib+R} e^{-it z} \langle(i \calL -z)^{-1}\phi, \psi \rangle dz   \bigg].
\end{align}
Next, we consider the first integral in \eqref{eq:Weak-represt-L-R} over the closed curve $\Gamma^{+}_R,$  which is the rectangle with vertices $\pm R+ib$ and $\pm R+i\varepsilon,$ for some $\varepsilon>0.$ Similarly, the second integral in \eqref{eq:Weak-represt-L-R} is taken over the closed curve $\Gamma^{-}_R,$ which is the reflection of $\Gamma^+_R$ below the real axis. Write $\Gamma^{\pm}_R=\bigcup_{1 \leq j \leq 4} \gamma^{j,\pm}_R,$ where $\gamma^{1,\pm}_R$ and $\gamma^{2,\pm}_R$ are the top and bottom horizontal curves, and $\gamma^{3,\pm}_R$ and $\gamma^{4,\pm}_R$ are the right and left sides of the rectangles, respectively. Applying the residue theorem to the contour integrals over the closed curves $\Gamma^{+}_R$ and $\Gamma^{-}_R,$ one can see the integral over $\gamma^{1,\pm}_R$ in \eqref{eq:Weak-represt-L-R} is equal to contour integral over $\gamma^{j,\pm}_R,$ for $j=2,3,4.$ 
Therefore, we obtain 
\begin{align*}
    \langle e^{t \mathcal{L}} \Phi, \Psi \rangle&= \lim_{R \to \infty} \lim_{\varepsilon \to 0^{+}}  \frac{1}{2\pi i} \int_{\varepsilon-iR}^{\varepsilon+iR} e^{it \lambda} \langle \left( e^{-\varepsilon t}(i\mathcal{L}-(\lambda + i \varepsilon))^{-1} - e^{\varepsilon t}(i\mathcal{L}-(\lambda - i \varepsilon))^{-1}  \right) , \Phi, \Psi \rangle d \lambda \\
    &+ \lim_{R \to \infty} \lim_{\varepsilon \to 0^{+}} J_{R,\varepsilon}^3(\Phi, \Psi)+\lim_{R \to \infty} \lim_{\varepsilon \to 0^{+}} J_{R,\varepsilon}^4(\Phi, \Psi),
\end{align*}
where $J_{R,\varepsilon}^3$ and $J_{R,\varepsilon}^3$ are the contour integral over the horizontal segments $\gamma^{3,\pm}_R$ and $\gamma^{4,\pm}_R$.

Next, we will prove that the contribution of the horizontal segments $\gamma^{3,\pm}_R$ and $\gamma^{4,\pm}_R$ vanishes as $R \to \infty.$ This is achieved via limiting absorption principle, which will be established in  
 the following Claim. 
\begin{claim} \label{clm:absorption-prinp}
   Let $\Lambda_0 \gg 1,$ and $\sigma>\frac{1}{2},$ then we have 
   \begin{align}
   \label{eq:absorption-prinp}
       \sup_{ |\im(z)|>0|, |\re(z)| >\Lambda_0} |z|^{\frac{1}{2}} \left\| \langle \cdot \rangle^{-\sigma} (i\calL-z)^{-1} f \right\|_{L^2_r(0,\infty)} \lesssim \left\| \langle \cdot \rangle^{\sigma} f \right\|_{L^2_r(0,\infty)}. 
   \end{align}
\end{claim}
\begin{proof}
We only give the proof for $\im(z)>0$ and the case $\im(z)<0$ follows similarly. We first prove the limiting absorption estimate for $i \calL_0,$ i.e., we show that for $\sigma>\frac{1}{2},$
   \begin{align}
    \label{eq:absorption-prinp-L-0}
       \sup_{\im(z)>0, |\re(z)| >a} |z|^{\frac{1}{2}} \left\| \langle \cdot \rangle^{-\sigma} (i\calL_0-z)^{-1} f \right\|_{L^2_r(0,\infty)} \lesssim \left\| \langle \cdot \rangle^{\sigma} f \right\|_{L^2_r(0,\infty)}. 
   \end{align}
In view of the kernel of the resolvent $ \mathcal{R}_0^{+}(r,s,z)$ in \eqref{eq:kernelR0} and \eqref{eq:mu_ij-0}, it suffices to show that 
\begin{align*}
   \sup_{\im(z)>0, |\re(z)| >a}  \left\| \langle \cdot \rangle^{-\sigma}  \mathcal{T}_{ij}^{(0) }f \right\|_{L^2_r} \lesssim \left\| \langle \cdot \rangle^{\sigma} f \right\|_{L^2_r}, \qquad (i,j) \in \{ (1,3), (2,4) \}
\end{align*}
where 
\begin{align*}
 \mathcal{T}_{ij}^{(0) }f(r,z):=\int_0^{\infty} \mathfrak{R}_{ij}^{+}(r,s,z) f(s) ds.   
\end{align*}
Recall that, 
\begin{align*}
     \mathfrak{R}_{ij}^{\pm}(r,s,z):=i\bigg( \Psi_{(0,i)}^{\pm}(r,z) \Psi_{(0,j)}^{\pm }(s,z)^t \mathbb{1}_{\{ 0\leq s \leq r\} } + \Psi_{(0,j)}^{\pm}(r,z) \Psi_{(0,i)}^{\pm }(s,z)^t \mathbb{1}_{\{ r\leq s \leq \infty \} } \bigg) \sigma_1
\end{align*}
We only give the proof for $(i,j)=(1,3)$ and the other case follows similarly. Denote by 
\begin{align*}
 \mathfrak{R}_{13}^{+}(r,s,z)= {R}_{13}^{1}(r,s,z)+ {R}_{13}^{2}(r,s,z),
\end{align*}
where 
\begin{align*}
    {R}_{13}^{1}(r,s,z) &:= i  \Psi_{(0,1)}^{+}(r,z) \Psi_{(0,3)}^{+ }(s,z)^t \sigma_1  \mathbb{1}_{\{ 0\leq s \leq r\} } \\
  {R}_{13}^{2}(r,s,z) &:= i \Psi_{(0,3)}^{+}(r,z) \Psi_{(0,1)}^{+  }(s,z)^t \mathbb{1}_{\{ r\leq s \leq \infty \} } \sigma_1
\end{align*}

In view of the asymptotics of $\Psi_{(0,j)}^{+}$, we divide into $3$ regions depending on $r,s,|z|^{-\frac{1}{2}}$ as follows 
\begin{align*}
  T_l^j:= \int_0^\infty \left( \mathbb{1}_{\{ 0\leq s \leq r\} }  K_l^j(r,s,z) + \mathbb{1}_{\{ r \leq s \} } \widetilde{K}_l^j(r,s,z)  \right)  f(s) ds, \qquad l=1,2,3, \quad j=1,2, 
\end{align*}
where 
\begin{align*}
  K_1^j(r,s,z)&=  {R}_{13}^{j}(r,s,z)  \mathbb{1}_{\{   s \leq r \leq |z|^{-\frac{1}{2}} \}}\; \qquad  \widetilde{K}_1^j(r,s,z)=  {R}_{13}^{j}(r,s,z)  \mathbb{1}_{\{   r \leq s \leq |z|^{-\frac{1}{2}} \}} \\
  K_2^j(r,s,z)&=  {R}_{13}^{j}(r,s,z) \mathbb{1}_{\{  s \leq |z|^{-\frac{1}{2}}  \leq r \}} \qquad  \widetilde{K}_2^j(r,s,z)=  {R}_{13}^{j}(r,s,z) \mathbb{1}_{\{  r \leq |z|^{-\frac{1}{2}}  \leq s \}}, \\
    K_3^j(r,s,z) &=  {R}_{13}^{j}(r,s,z)  \mathbb{1}_{\{ |z|^{-\frac{1}{2}} \leq s \leq r \}}\,  \qquad    \widetilde{K}_3^j(r,s,z) =  {R}_{13}^{j}(r,s,z)  \mathbb{1}_{\{ |z|^{-\frac{1}{2}} \leq r \leq s \}}  .
\end{align*}

Therefore, it is suffices to prove that \begin{align*}
    \left\| \langle \cdot \rangle^{-\sigma}   T_l^j f \right\|_{L^2_r} \lesssim \left\| \langle \cdot \rangle^{\sigma} f \right\|_{L^2_r}, \qquad l=1,2,3, \quad j=1,2.
\end{align*}
Using the asymptotic of $\psi_0(r,k(z))$ and $\phi_0(r,k(z)),$ i.e., 
\begin{align*}
    \psi_0(r,z) &\sim c \, e^{i k(z) r} , \qquad \qquad \qquad   \psi_0(r,z) \sim    (k(z)r)^{-\frac{1}{2}}, \\
    \phi_0(r,z) &\sim c_1 e^{ik(z) r } + c_2 e^{-i k(z) r} , \quad \phi_0(r,z) \sim (k(z) r)^{\frac{3}{2}}.
\end{align*}
 we obtain 
 \begin{align*}
  \sum_{j=1}^2 |K_1^j(r,s,z) | + |\widetilde{K}_1^j(r,s,z) | &\lesssim  \mathbb{1}_{\{ s \leq r \leq |z|^{-\frac{1}{2}}  \}} (|z|^{\frac{1}{2}} s )^{\frac{3}{2}} r^{-\frac{1}{2}} + \mathbb{1}_{\{ r \leq s \leq |z|^{-\frac{1}{2}}   \}} (|z|^{\frac{1}{2}} r )^{\frac{3}{2}} s^{-\frac{1}{2}} \lesssim 1\\ 
 \sum_{j=1}^2 |K_2^j(r,s,z) | + |\widetilde{K}_2^j(r,s,z) | &\lesssim  \mathbb{1}_{\{ s \leq |z|^{-\frac{1}{2}} \leq r  \}} (|z|^{\frac{1}{2}} s )^{\frac{3}{2}} e^{-\im(k_1(z) r } + \mathbb{1}_{\{ r \leq |z|^{-\frac{1}{2}} \leq s  \}} (|z|^{\frac{1}{2}} r )^{\frac{3}{2}} e^{-\im(k_1(z) s}  \lesssim 1\\ 
  \sum_{j=1}^2 |K_3^j(r,s,z) | + |\widetilde{K}_3^j(r,s,z) |  &\lesssim e^{-\im(k_1(z)|r-s|} \lesssim 1  
 \end{align*}
 Thus, for $\sigma > \frac{1}{2}$ we have  \begin{align*}
    \left\| \langle \cdot \rangle^{-\sigma}   T_l^j f \right\|_{L^2_r}^2  \lesssim \int_0^{\infty}   \langle r \rangle^{-2 \sigma} \left( \int_0^\infty |f(s)| ds \right)^2 dr  \lesssim \left\| \langle \cdot \rangle^{\sigma} f \right\|_{L^2_r}^2.
\end{align*}
To obtain \eqref{eq:absorption-prinp}, we write $i\calL:=i\calL_0 + V .$ Since $|V(r)| \lesssim e^{-2r}$  then by \eqref{eq:absorption-prinp-L-0} for $\Lambda_0$ large we have
\begin{align*}
    \sup_{ \im(z)>0, \re(z)> \Lambda_0} \left\| V (i \calL_0 -z)^{-1} \right\| \leq \frac{1}{2},
\end{align*}
where $\|\cdot\|$ here denotes the operator nom for a map from $\langle \cdot \rangle^{-\sigma} L^2 \times \langle \cdot \rangle^{-\sigma} L^2   \to \langle \cdot \rangle^{\sigma} L^2 \times \langle \cdot \rangle^{\sigma} L^2 .$ This proves the existence of $(\mathrm{Id}+V (i \calL_0 -z)^{-1})$ as bounded operator from $\langle \cdot \rangle^{\sigma} L^2 \times \langle \cdot \rangle^{\sigma} L^2 $ to $\langle \cdot \rangle^{-\sigma} L^2 \times \langle \cdot \rangle^{-\sigma} L^2   .$
Finally, we obtain \eqref{eq:absorption-prinp}  as a consequence of the resolvent identity 
\begin{align*}
    (i \calL -z)^{-1} :=   (i \calL_0 -z)^{-1} (\mathrm{Id}+V  (i \calL_0 -z)^{-1})^{-1},
\end{align*}
which holds as an identity between operators from $\langle \cdot \rangle^{\sigma} L^2 \times \langle \cdot \rangle^{\sigma} L^2   \to \langle \cdot \rangle^{-\sigma} L^2 \times \langle \cdot \rangle^{-\sigma} L^2 ,$ for $\sigma>\frac{1}{2}.$ This concludes the proof of Claim \ref{clm:absorption-prinp}.
\end{proof}
This completes the proof of Lemma \ref{lem:ST1}.
\end{proof}
 
\subsection{Spectrum of $i\mathcal{L}_{GP}$}
\label{subsec:DS-L_GP}   
In this section, we study the spectrum of the linearized operator $\mathcal{L}_{GP}$ associated with the Gross–Pitaevskii equation.  Firstm notice that from the definition of $V_{GP}$ in \eqref{eq:defL_GP-V_GP} and the asymptotic behavior of $\rho$ in \eqref{eq:rho_n-asymptotics-near-infty}, it follows that $i\mathcal{L}_{GP}$ satisfies Assumption~\ref{asmp:1}. We further show that the spectrum of $i\mathcal{L}_{GP}$ is real, in particular, $\spec(i\calL_{GP})\subset (-\infty,-\eta]\cup[\eta,\infty), $ for some $\eta>0.$ In addition, we prove that its essential spectrum is $(-\infty,-\tfrac{\sqrt{17}}{8}] \cup [\tfrac{\sqrt{17}}{8},\infty),$ and it contains no threshold or embedded eigenvalues. Hence, if the discrete spectrum is non-empty, it must be contained in $
\left(-\tfrac{\sqrt{17}}{8}, -\eta\right] \cup \left[\eta, \tfrac{\sqrt{17}}{8}\right).$
Consequently, Assumptions~\ref{asmp:1}–\ref{asmp:3} are verified for $i\mathcal{L}_{GP}$, apart from the discrete spectrum condition. \\

Let  $ \Phi = \begin{pmatrix} \phi \\ \psi \end{pmatrix} $ 
be a solution to the equation
$ i \mathcal{L}_{GP}  \Phi = \lambda \Phi. $
We decompose the components into their real and imaginary parts by setting
$ \phi = \phi_1 + i \phi_2 $  and  $ \psi = \psi_1 + i \psi_2, $ where $\phi_1, \phi_2, \psi_1, \psi_2$ are real-valued functions. Then, we have 

\begin{align*}
 i \mathcal{L}_{GP} \begin{pmatrix}  i \phi_2 \\  \psi_1    \end{pmatrix}  = \lambda \begin{pmatrix}  i \phi_2 \\  \psi_1    \end{pmatrix}  \quad \text{ and } \quad  
 i \mathcal{L}_{GP} \begin{pmatrix}  \phi_1 \\  i \psi_2    \end{pmatrix}  = \lambda \begin{pmatrix}   \phi_1 \\  i \psi_2    \end{pmatrix} .
\end{align*}

Thus, if $\phi $ is real  then $\psi$ is pure imaginary. Assume $f= i \psi,  $ then 
\begin{align} \label{eq:sys-L-GP}
\begin{cases}
\mathfrak{L}_1 f = \lambda \phi \\
\mathfrak{L}_2 \phi = \lambda f  ,
\end{cases}
\end{align}
where 
\begin{align*}
    \mathfrak{L}_1 &= -\frac{1}{2} \partial_r^2 + \frac{3}{8 \sh^2(r)}  + (\rho^2 - 1) + \frac{1}{8} \\ 
    \mathfrak{L}_2 &= -\frac{1}{2} \partial_r^2 + \frac{3}{8 \sh^2(r)} + 3 (\rho^2 - 1) +  \frac{17}{8}
\end{align*}

\begin{lemma}
\label{L1=0}
The  two linearly independent solutions of $  \mathfrak{L}_1 f=0,$ are $\trho(r):=\sinh^{\frac{1}{2}}(r) \rho(r) $ and 
$$ g(r) = 2 \trho(r) \int_r^\infty \frac{1}{\trho(s)^2} ds.$$ 
Moreover, they satisfy 
\begin{align*}
    \trho(r)= a r^{\frac{3}{2}}(1 + O(r^{2})), \quad g(r)= c r^{-\frac{1}{2}}(1 +O(r^{2}\log(r))), 
\end{align*}
when $r \to 0 ,$ for some $a,c>0,$ and 
\begin{align*}
    \trho(r)=e^{\frac{1}{2}r} ( 1 -2 e^{-2 r} +O(e^{-\frac{5}{2} r})), \quad g(r)= e^{-\frac{1}{2}r}(1 + O(e^{-2r})) ,  \text{ as }r \to \infty
\end{align*}

\end{lemma}
\begin{proof}
We have $ \mathfrak{L}_1 \trho =\sinh^{\frac{1}{2}}(r) ( \frac{1}{2 }\Delta_{{\mathrm{rad}}} + \frac{1}{2 \sinh^2(r)}  +  (\rho^2-1))\rho  = 0,$ where $\Delta_{{\mathrm{rad}}}:=\frac{1}{2} \partial_r^2 + \frac{1}{2}\coth(r) \partial_r.$ Using the ordinary differential equation \eqref{eq:ode2}, we get $ \mathfrak{L}_1 \trho =0 .$ Next, we seek another linearly independent solution to the equation $\mathfrak{L}_1 f = 0$, of the form $g = \tilde{\rho} \, h$, where $h$ satisfies the equation 
$$-\frac{1}{2} \Tilde{\rho}  \partial_r^2 h  - \partial_r \Tilde{\rho} \partial_r h =0 .$$
This yields, $  \partial_r( -\frac{1}{2} \trho^2 \partial_r h ) = 0 ,$ and therefore 
\begin{align*}
    g(r) =2 \trho(r) \int_r^\infty \frac{1}{\trho(s)^2} ds.
\end{align*}

The asymptotics of $\trho$ and $g(r)$ follow from their definitions, along with the asymptotics of $\trho$ given in \eqref{eq:rho_n-asymptotics-near-0} near $0$ and in \eqref{eq:rho_n-asymptotics-near-infty} as $r \to \infty$. Thus, in view of the asymptotic behavior of $\trho(r)$ and $g(r)$ near the origin and at infinity, they form a linearly independent pair of solutions to the second-order homogeneous equation $\mathfrak{L}_1 f = 0$.

\end{proof}

\begin{lemma}
\label{L_1f=F}
Consider two  functions $f$ and $F$ solving the equation 
\begin{equation*}
  \mathfrak{L}_1 f = F,  
\end{equation*}

 \begin{itemize}
    \item If $f$ and $F$ are bounded near $0,$ then there exists $C_1 \in \R$ such that \begin{equation*}
  f(r)=  C_1\trho(r)  -2 \trho(r)  \int_0^r \frac{1}{\trho^2(t)} \int_0^t \trho(s) F(s) ds dt 
  \end{equation*}
  \item If $f$ converges to $0,$ and $F$ decays exponentially when $r \to \infty$ then 
  \begin{equation*}
     f(r)= -2 \trho(r) \int_r^{\infty} \frac{1}{\trho^2(t)} \int_t^{\infty} \trho(s) F(s) ds.  
  \end{equation*}
\end{itemize}
\end{lemma}
\begin{proof}

Consider the equation $\mathfrak{L}_1 f =F,$ with $f$ and $F$ bounded near $0.$ Let $f=\trho y, $ then $y $ satisfies $ \trho \partial_r^2 y +2 \partial_r \trho \partial_r y =- 2 F ,$ i.e., $ \partial_r ( \trho^2 \partial_r y) = - 2 F \trho.$ Since $F$ is bounded near $0,$ and $\trho(r) \sim r^{\frac{3}{2}}$, we have  \begin{align*}
   \partial_r   y(r) = \frac{-2}{\trho^2(r)} \int_0^r \trho(s) F(s) ds
\end{align*}
which yields, \begin{align*}
   y(r) =-2 \int_0^r  \frac{1}{\trho^2(t)} \int_0^t \trho(s) F(s) dsdt
\end{align*}
Thus, if  $f$ and $F$ are bounded near $0,$  then there exists $C_1 $ such that 
\begin{align*}
    f(r)=   C_1\trho(r)  -2 \trho(r)  \int_0^r \frac{1}{\trho^2(t)} \int_0^t \trho(s) F(s) ds dt .   
\end{align*}
If $f$ converges to $0,$ and $F$ decays exponentially when $r \to \infty$ then 
  \begin{equation*}
     f(r)=  -2 \trho(r) \int_r^{\infty} \frac{1}{\trho^2(t)} \int_t^{\infty} \trho(s) F(s) ds.  
  \end{equation*}
  \end{proof}

\begin{lemma} \label{lem:FL_2=0}
 There exist two linear independent solutions $u_1$ and $u_2$ of $$\mathfrak{L}_2 u = 0,$$ 
such that 
 \begin{align*}
     u_1(r)= r^{\frac{3}{2}} (1 + O(r)), \qquad u_2(r)=r^{-\frac{1}{2}}(1 + O(r))
 \end{align*}
 when $r \to 0$ and 
  \begin{align*}
     u_1(r)= e^{\frac{\sqrt{17}}{2} r} (1 + O(r^{-1})), \qquad u_2(r)=e^{-\frac{\sqrt{17}}{2} r}(1 + O(r^{-1}))
 \end{align*}
    when $r \to \infty$. Moreover, both $u_1(r)$ and $u_2(r)$ are strictly positive functions. 
\end{lemma}
\begin{proof}
Recall that $\mathfrak{L}_2 = -\frac{1}{2} \partial_r^2 + \frac{3}{8 \sh^2(r)} + 3 (\rho^2 - 1) +  \frac{17}{8}.$ Then the operators $ \mathfrak{L}_2^0=-\frac{1}{2} \partial_r^2 + \frac{3}{8 r^2} $ and $\mathfrak{L}_2^{\infty}= -\frac{1}{2}\partial_r^2 + \frac{17}{8}$ are a good approximation of $\mathfrak{L}_2$ near $r=0$ and $r=\infty,$ respectively. The fundamental system is given by $\{r^{\frac{3}{2} }, r^{-\frac{1}{2}}  \}$ near zero and $\{e^{\frac{\sqrt{17}}{2} r},e^{-\frac{\sqrt{17}}{2} r}\} $ at infinity. Note that
the spectrum of $\mathfrak{L}_2$ is contained in $[\eta_2,\infty),$ for some $\eta_2 >0,$ see Lemma \ref{lem:FL_1-L_2-Spec}. In particular zero is not an eigenvalue and hence
    \begin{align*}
     u_1(r)\sim
 \begin{cases}
        r^{\frac{3}{2}}, \quad r \longrightarrow 0 ,  \\
        e^{\frac{\sqrt{17}}{2} r}, \quad r \longrightarrow \infty ,
    \end{cases} 
\quad \text{and} \quad 
  u_2 (r)\sim
 \begin{cases}
       r^{-\frac{1}{2}}, \quad r \longrightarrow 0 ,\\
        e^{-\frac{\sqrt{17}}{2} r}, \quad r \longrightarrow \infty.
    \end{cases}  
\end{align*}
Next, we prove that these solutions are strictly positive. For that, we consider solution of the form $u_1(r)=\trho(r) v(r),$ which lead to the equation $\partial_r ( \partial_rv \, \trho)=4 \rho^2 \trho^2 v.$ Consider this equation with the condition $\partial_rv(0)=0$ and $v(r)=C>0,$ for some constant $C,$ leads to equation $$\partial_rv(r) =\frac{1}{\trho^2(r)}\int_0^r 4 \rho^2(s) \trho^2(s) v(s) ds.$$
This implies that $v$ and $\partial_r v$ is positive. Then $u_1(r)$ is strictly positive.  Finally, consider solution of the form $u_2(r)=u_1(r) g(r),$ implies that $u_2(r)=u_1(r) \int_r^{\infty} \frac{1}{u_1(s)^2} ds,$ which is strictly positive.
\end{proof}

\begin{lemma}
\label{L_2u=U}
Consider two  functions $u$ and $U$ solving the equation 
\begin{equation*}
  \mathfrak{L}_2 u = U,  
\end{equation*}

 \begin{itemize}
    \item If $u$ and $U$ are bounded near $0,$ then there exists $C_2 \in \R$ such that \begin{equation*}
  u(r)= C_2 u_1(r) -2  u_1(r) \int_0^r \frac{1}{u_1(t)^2} \int_0^t u_1(s) U(s) ds dt.  
  \end{equation*}
  \item If $u$ and $U$ decays exponentially, when $r \to \infty$ then 
  \begin{equation*}
     u(r)= -2  u_1(r) \int_r^{\infty} \frac{1}{u_1(t)^2} \int_t^{\infty} u_1(s) U(s) dsdt.  
  \end{equation*}
\end{itemize}
\end{lemma}
\begin{proof}
  The proof follows by similar arguments as in Lemma~\ref{L_1f=F}, and we omit the details.
\end{proof}

\begin{lemma}\label{lem:no-eigenvalue-expGrowth}
Let $\lambda=\pm \frac{\sqrt{17}}{8},$ then there exist $f$ and $\phi$ solution to 
    \begin{align}
    \label{sys-L_1f_L2phi}
\begin{cases}
\mathfrak{L}_1 f = \lambda \phi \\
\mathfrak{L}_2 \phi = \lambda f  
\end{cases}
\end{align}
satisfying 
\begin{align*}
    f(r)= c_1 r^{\frac{3}{2}} (1 + O(r^{2})), \qquad \phi(r)=c_2 r^{\frac{7}{2}}(1 + O(r))
 \end{align*}
 when $r \to 0,$ for some constants  $c_1,c_2$   and 
  \begin{align*}
     f(r)= \tc_1 e^{\frac{3}{\sqrt{2}} r} (1 + O(r^{-1})), \qquad \phi(r)=\tc_2 e^{\frac{3}{\sqrt{2}} r}(1 + O(r^{-1}))
 \end{align*}
    when $r \to \infty,$ for some constants $\tc_1,\tc_2 .$ 
\end{lemma}

\begin{proof}
We will only prove the case $\lambda=\frac{\sqrt{17}}{8},$ and similarly one can obtain the same results for $\lambda=-\frac{\sqrt{17}}{8}.$ 
By Lemma \ref{L_1f=F} and \ref{L_2u=U}, we consider the solution with $C_1=0$ and $C_2=-1$

\begin{align}
\label{eq_f}
 f(r)&= - \trho(r) -2 \lambda \trho(r) \int_0^r \frac{1}{\trho^2(s)} \int_0^t \trho(s) \phi(s) ds dt. \\
 \label{eq_phi}
\phi(r) &=   - 2 \lambda u_1(r) \int_0^r \frac{1}{u_1(s)^2} \int_0^t u_1(s) f(s) ds dt.
\end{align}

One can prove that $(f,\phi)$ is a solution to the system \eqref{sys-L_1f_L2phi}.
using the  fixed point Theorem on the space $(X,\left\| \cdot \right\|_X),$ where $\left\|(f,\phi) \right\|_X = \displaystyle \sup_{r<1} {r^{-\frac{3}{2}}} ( | f(r)|+| \phi(r)|  ).$   Define $T(f,\phi): X  \longrightarrow X $, $(f,\phi) \longrightarrow (T_1(f,\phi), T_2(f,\phi)) ,$ where $T_1(f,\phi)$ and $T_2(f,\phi)$ are given by \eqref{eq_f} and \eqref{eq_phi} respectively. Using the fact that $\trho(r)\sim u_1(r) \sim r^{\frac{3}{2}}$ for $r$ small, one can check that $T$ is a contraction on $X$ and therefore $T(f,\phi)=(f,\phi) $ is a solution to \eqref{sys-L_1f_L2phi}.\\ 

Using the behavior of $\phi$ and $\trho$ for small $r,$ we have  $ f(r)<0$ near $0$. Indeed, 
\begin{align*}
   f(r) = - r^{\frac{3}{2}} - 2 \lambda r^{\frac{3}{2}} \int_0^r ( t^{-3} + O(t^{-2}) ) \int_0^t s^{\frac{3}{2}} s^{\frac{3}{2}} ds dt = - r^{\frac{3}{2}} - 2 \lambda ( r^{\frac{3}{2} + 2 }  + O(r^{\frac{9}{2}}))  
\end{align*}
Similarly, we have $\phi(r) >0$ for small $r$, indeed, 

\begin{align*}
   \phi(r) =  - 2 \lambda r^{\frac{3}{2}} \int_0^r ( t^{-3} + O(t^{-2}) ) \int_0^t s^{\frac{3}{2}} (-s)^{\frac{3}{2}} ds dt =  2 \lambda ( r^{\frac{7}{2} }  +  O(r^{\frac{9}{2}}))
\end{align*}
Thus, $f(r)<0$ and $\phi(r)>0$ for $r>0$ small. Assume by contradiction, there exists $r_0$ such that either $f(r)$ or $\phi(r)$ changes sign for $r>r_0$. Let $r_0$ be the smallest such value, that is, $r_0 := \inf \{ r > 0 : \phi(r) = 0 \text{ or } f(r) = 0 \}.$ Then $f(r)<0$ and $\phi(r)>0$ for all $r<r_0$. Since $u_1(r)>0$ and $\trho (r)>0$  for all $r,$
 then by the definition \eqref{eq_phi} of $f,$  we have $f(r_0) <0,$ and by \eqref{eq_phi}, $\phi(r_0)<0,$ which is a contradiction. First, notice that, $|f(r)|>|\trho(r)| \sim e^{\frac{1}{2}r} ,$ and by \eqref{eq_phi} we have $|\phi(r)|> e^{\frac{\sqrt{17}}{2}r}  $. Using \eqref{eq_f}, we also have $|f(r)|> 2 \lambda e^{\frac{1}{2} r} \int_0^r e^{- t } \int_0^t e^{ ( \frac{1}{2}  + \frac{\sqrt{17}}{2} )s} ds dt \sim r e^{\frac{\sqrt{17}}{2}r}$ and this implies $|f(r)| \sim e^{\frac{3}{\sqrt{2}}r} $ and $|\phi(r)|\sim e^{\frac{3}{\sqrt{2}}r}$ for large $r.$ 
\end{proof}

\begin{lemma} \label{lem:no-eigenvalue-singular-r0}
Let $\lambda=\pm \frac{\sqrt{17}}{8},$ and let $f$ and $\phi$ be two exponentially decaying functions near $\infty$ satisfying  
    \begin{align}
    \label{sys-L_1f_L2phi-2}
\begin{cases}
\mathfrak{L}_1 f = \lambda \phi \\
\mathfrak{L}_2 \phi = \lambda f  
\end{cases}
\end{align}
then   
\begin{align*}
    f(r)= c_1 r^{-\frac{1}{2}} (1 + O(r)), \qquad \phi(r)=c_2 r^{-\frac{1}{2}} (1 + O(r)),
 \end{align*}
 when $r \to 0,$ for some constants  $c_1,c_2$ and 
  \begin{align*}
     f(r)= \tc_1 e^{-\frac{3}{\sqrt{2}} r}  (1 + O(r^{-1})), \qquad \phi(r)= \tc_2 
     e^{-\frac{3}{\sqrt{2}} r}(1 + O(r^{-1}))
 \end{align*}
    when $r \to \infty,$ for some constants $\tc_1,\tc_2 .$ 
\end{lemma}

\begin{proof}
We will only prove the case $\lambda=\frac{\sqrt{17}}{8},$ and similarly one can obtain the same results for $\lambda=-\frac{\sqrt{17}}{8}.$  
By Lemma \ref{L_1f=F} and \ref{L_2u=U}, we have     
\begin{align}
\label{eq_f_infty}
f(r)&= -2  \lambda \trho(r) \int_r^{\infty} \frac{1}{\trho^2(t)} \int_t^{\infty} \trho(s) \phi(s) ds.\\
\label{eq_phi_infty}
\phi(r) &= -2 \lambda  u_1(r) \int_r^{\infty} \frac{1}{u_1(t)^2} \int_t^{\infty} u_1(s) f(s) dsdt.  
\end{align}
One can prove that $(f,\phi)$ is a solution to the system \eqref{sys-L_1f_L2phi-2}.
using the  fixed point Theorem on the space $(X,\left\| \cdot \right\|_X),$ where $\left\|(f,\phi) \right\|_X = \displaystyle \sup_{r>1} e^{\frac{3}{ \sqrt{2}} r} ( | f(r)|+| \phi(r)|  ).$   Define $T(f,\phi): X  \longrightarrow X $, $(f,\phi) \longrightarrow (T_1(f,\phi), T_2(f,\phi)) ,$ where $T_1(f,\phi)$ and $T_2(f,\phi)$ are given by \eqref{eq_f_infty} and \eqref{eq_phi_infty} respectively. Using the fact that $\trho(r) \sim e^{\frac{1}{2}r}$ and $u_1(r) \sim e^{\frac{\sqrt{17}}{2}r}$ for $r$ large. One can check that $T$ is a contraction on $X$ and therefore $T(f,\phi)=(f,\phi) $ is a solution to \eqref{sys-L_1f_L2phi-2}.\\

We have \begin{align*} 
    f(r) &= - 2 \lambda e^{\frac{1}{2}r}(1+O(e^{-2r})) \int_r^{\infty} e^{-t} (1+O(e^{-2t}) \int_t^{\infty} e^{\frac{1}{2} s}(1 +O(e^{-2s}) e^{-\frac{3}{\sqrt{2}}s}  dsdt \\
    &= 2 \lambda \frac{1}{\frac{1}{2}+ \frac{3}{\sqrt{2}}} \frac{1}{\frac{1}{2}- \frac{3}{\sqrt{2}}} e^{-\frac{3}{\sqrt{2}}r}(1 + O(r^{-1}))= 2 \lambda \frac{-4}{5} e^{-\frac{3}{\sqrt{2}}r} (1 + O(r^{-1})).
\end{align*}

Therefore, 
\begin{align*}
    \phi(r)& = - 2 \lambda e^{\frac{\sqrt{17}}{2}r} (1+O(r^{-1} ))  \int_r^{\infty} e^{-\sqrt{17}t} (1+O(t^{-1})) \int_t^{\infty} e^{\frac{\sqrt{17}}{2} s}  \lambda \frac{-8}{5} e^{-\frac{3}{\sqrt{2}}s}(1+O(s^{-1})) dsdt \\
        &= \frac{16}{5} \lambda^2 4 e^{-\frac{3}{\sqrt{2}} r}(1 + O(r^{-1}))  
\end{align*}
Therefore, we have $f(r)<0$ and $\phi(r)>0$ for large $r.$ Assume by contradiction, there exists $r_0$ such that either $f(r)$ or $\phi(r)$ changes sign for $r>r_0$. Let $r_0$ be the smallest such value, that is, $r_0 := \sup_{r>1} \{ r > 0 : \phi(r) = 0 \text{ or } f(r) = 0 \}.$ Then $\phi(r)>0$ and $f(r)<0$ for all $r>r_0.$ Since $\trho$ and $u_1$ are positive for all $r,$ then by \eqref{eq_phi_infty} we have $\phi(r_0)>0,$ and by \eqref{eq_f_infty} we have $f(r_0)<0,$ which is a contraction. \\

For the behavior of $\phi$ and $f$ near $0,$ we use \eqref{eq_f_infty} and \eqref{eq_phi_infty} with the fact that $\trho(r)= a r^{\frac{3}{2}}(1 + O(r^{2}))$ and $  u_1(r)= r^{\frac{3}{2}} (1 + O(r))$ for small $r.$ Then, as $r \to 0$  we have 
\begin{align*}
 \phi(r)= - 2 \lambda r^{\frac{3}{2}} (1 + O(r)) \int_r^{1} (t^{-3} + O(t^{-2})) \int_0^{\infty} u_1(s) f(s) ds  dt + O(1) = c_2 r^{-\frac{1}{2}}(1+O(r))
\end{align*}
Similarly, we obtain $f(r)=c_1 r^{-\frac{1}{2}}(1+O(r))$ for small $r.$
\end{proof}

\begin{lemma}
\label{lem:GP-thres-emb-EV}
    The matrix operator $i\mathcal{L}_{GP}$ of the linearized equation associated with the Gross–Pitaevskii equation has no eigenvalues in $(-\infty, -\frac{\sqrt{17}}{8}]\cup [\frac{\sqrt{17}}{8},\infty).$
\end{lemma}
\begin{proof}
Using Lemma \ref{lem:no-eigenvalue-expGrowth}, and \ref{lem:no-eigenvalue-singular-r0},  one can see that the operator $\mathcal{L}_{GP}$ has no eigenvalue at $\lambda= \pm \frac{\sqrt{17}}{8}.$ Next, we prove that $\calL_{GP}$ has no eigenvalues for $\lambda>\frac{\sqrt{17}}{8}$.  Let 
$\Phi=\begin{pmatrix}
    \phi \\ \psi 
\end{pmatrix}$ be solution to $i \calL_{GP} \Phi=\pm \lambda \Phi, $ where $\lambda>\frac{\sqrt{17}}{8},$ i.e., satisfies \eqref{eq:sys-L-GP}. Similar computation as in Subsection \ref{subsec:DS-iL},  lead us to the fourth-order equation 
\begin{align*}
    \mathfrak{L}_1\mathfrak{L}_2 \phi = \lambda^2 \phi.
\end{align*}
Therefore, we obtain the same asymptotics near $0$ as in \eqref{eq:behav-near-0}. However,  the possible behaviors near $\infty$ are $  \{ e^{ \alpha_1 r },e^{- \alpha_1 r }, e^{i \alpha_2 r },e^{ - i \alpha_2 r }    \}$ 
where 
\begin{align*}
    \alpha_1&:= \sqrt{2\sqrt{\lambda^2 +1} + \frac{9}{4}},  \quad    \alpha_2:= \sqrt{2\sqrt{\lambda^2 +1} - \frac{9}{4}}.
\end{align*}
Observe that near $\infty$, the possible behaviors are exponentially growing and decaying solutions, as well as oscillating solutions. The only possibility to have an embedded eigenvalue is if an $L^2(0,1)$ behavior near $0$, i.e., $r^{\tfrac{3}{2}}$, connects to the exponentially decaying behavior near $\infty$. However, since $\alpha_1 > \tfrac{3}{\sqrt{2}}$, using the same argument as in Lemma~\ref{lem:no-eigenvalue-singular-r0}, we obtain that any exponentially decaying solution near $\infty$ must connect to the $r^{-\frac{1}{2}}$ asymptotic near $0$. Then, the operator $i \mathcal{L}_{GP}$ has no eigenvalues for $\lambda > \frac{\sqrt{17}}{8}.$ Similarly, we can prove that $i \mathcal{L}_{GP}$ has no eigenvalues for $\lambda <- \frac{\sqrt{17}}{8}.$
\end{proof}

\begin{lemma} \label{lem:FL_1-L_2-Spec}
The operators $\FL_1$ and $\FL_2$ are self-adjoint, positive operators with domain $\calD$. Moreover, the essential spectrum of $\FL_1$ is $\spec_{ess}(\FL_1)=[\frac{1}{8},\infty)$ and $\spec(\FL_1) \subset [\eta_1,\infty),$ for some $\eta_1>0$. The spectrum of $\FL_2$ 
lies in $[\eta_2,\infty)$ for some $\eta_2>0$, and its essential spectrum satisfies $ \spec_{ess}(\FL_2)=[\tfrac{17}{8},\infty).$
\end{lemma}

\begin{proof}
The assertions about the essential spectrum follow from the Weyl criterion. Next we show that the spectrum of $\mathfrak{L}_2$ is contained in $[\eta_2, \infty), $ where $\eta_2>0$. Indeed $\mathfrak{L}_2\tilde{\rho}=2\rho^2\tilde{\rho}>0$ and a comparison argument can be used to show that zero is not an eigenvalue. Since the essential spectrum of $\mathfrak{L}_2$ is $[\frac{17}{8},\infty)$ this proves the claim for $\FL_2$. The argument for $\FL_1$ is similar where we use that $\FL_1\tilde{\rho}=0$ to prove that zero is not an eigenvalue.
\end{proof}

\begin{lemma}  \label{lem:GP-spec}
The spectrum of the operator $i\mathcal{L}_{GP}$ is real and its essential spectrum  is $(-\infty, -\frac{\sqrt{17}}{8}] \cup [\frac{\sqrt{17}}{8},\infty).$ 
\end{lemma}

\begin{proof}
First, we show that the spectrum of $i \mathcal{L}_{GP}$ is real, in particular,  $\spec(i \mathcal{L}_{GP}) \subset (-\infty, -\eta] \cup [\eta,\infty),$ for some $\eta>0.$ For that, we write 
\begin{align*}
    \calL_{GP} -z =\begin{pmatrix}
        1 & z \mathfrak{L}_2^{-1} \\
        0 & 1 
    \end{pmatrix}
    \begin{pmatrix}
        0 & \FT(z) \\
        - \FL_2 & -z 
    \end{pmatrix},
\end{align*}
where  $\mathfrak{T}(z):=z^2 \mathfrak{L}_2^{-1}+ \mathfrak{L}_1,$ which closed on $\calD \subset X_0.$ Next, we prove that for any $z \in \mathfrak{X} := \C \setminus \big( (-i\infty, -i\eta] \cup [i\eta, i\infty) \big),$ where $\eta$ will be determined below, $(\calL_{GP} - z)^{-1}$ defines a bounded operator on $X_0 \times X_0$. It suffices to show that for all $z \in \mathfrak{X},$ the operator $\FT(z)$ admits a bounded inverse $\FT(z)^{-1} : X_0 \to X_0.$ Indeed, if this is the case, then 
\begin{align*}
 (\calL_{GP}-z)^{-1}= \begin{pmatrix}
        -z \FL_2^{-1} \FT(z)^{-1} & z^2  \FL_2^{-1} \FT(z)^{-1} \FL_2^{-1}- \FL_2^{-1}  \\ 
        \FT(z)^{-1} & -z \FT(z)^{-1} \FL_2^{-1}
    \end{pmatrix}
\end{align*}
To see the invertibility of $\FT$, we write 
$\FT(z)=\FL_2^{-\frac{1}{2}} (z^2+ \FT_0) \FL_2^{-\frac{1}{2}},$ where $\FT_0:=\FL_2^{\frac{1}{2}} \FL_1 \FL_2^{\frac{1}{2}}=\FL_2^{2}-2 \FL_2^{\frac{1}{2}} \rho^2 \FL_2^{\frac{1}{2}}.$ 
Note that $\FT_0$ is symmetric. Recall that from  Lemma \ref{lem:FL_1-L_2-Spec}, we have $\spec(\FL_1)\subseteq[\eta_1,\infty)$ and $ \spec(\FL_2)\subseteq[\eta_2,\infty)$ for some $\eta_1,\eta_2>0$.
Then we have $\FT_0 - \eta >0$ for  $\eta=\eta_1\eta_2>0$. Indeed, 
\begin{align*}
    \langle (\FT_0 - \eta_1 \eta_2) f,f\rangle=    \langle (\FL_1-\eta_1)  \FL_2^{\frac{1}{2}} f,\FL_2^{\frac{1}{2}}f\rangle  + \eta_1  \langle  (\FL_2 - \eta_2) f,f  \rangle \geq 0
\end{align*}
Next, we show that $\FT_0$ is self-adjoint with domain $\FL_2^{-1} \calD=\FL_2^{-2} X_0.$ Let $g,h \in X_0$ satisfy 
\begin{align*}
       \langle \FT_0f,g\rangle= \langle f,h \rangle, \quad \forall f \in \FL_2^{-2} X_0.
 \end{align*}
Our aim is to show that $g \in \FL_2^{-2}X_0.$ Therefore, we write $f=\FL_2^{-2} \tilde{f }$ with $\tilde{f} \in X_0,$ which yields 
\begin{align} \label{eq:adj1}
     \langle \tilde{f},g \rangle =   \langle \tilde{f},\FL_2^{-2}h \rangle + 2   \langle \FL_2^{\frac{1}{2}} \rho^2 \FL_2^{-\frac{1}{2}} \FL_2^{-1} \tilde{f},g \rangle.
\end{align}
Next, we show that $B_{\frac{1}{2}}:=\FL_2^{\frac{1}{2}} \rho^2 \FL_2^{-\frac{1}{2}}$ is bounded on $X_0.$ Using complex interpolation, it suffices to show that $B_{\alpha}:=\FL_2^{\alpha} \rho^2 \FL_2^{-\alpha}$ is bounded on the lines $\re(\alpha)=1$ and $\re(\alpha)=-1.$ By the spectral theorem we have $\FL_2^{i \alpha}$ is unitary for $\alpha \in \R.$ Therefore, it remains only to check that $B_1$ and $B_{-1}$ are bounded on $X_0$. To see this, we write $$B_1=\rho^2 - \frac{1}{2}( \partial_r^2 \rho^2 + 2 \partial_r  \rho^2 \partial_r ) \FL_2^{-1}. $$  By Lemma \ref{lem:FL_2=0}, we have 
\begin{align} \label{eq:FL2-1}
    (\FL_2 f)^{-1}(r):=c_1 \int_0^r u_2(r) u_1(s) f(s) ds + c_2 \int_r^{\infty} u_1(r) u_2(s) f(s) ds,
\end{align}
for some suitable constants $c_1$ and $c_2.$ Thus, we have 
\begin{align*}
 \partial_r   (\FL_2 f)^{-1}(r):=c_1 \int_0^r \partial_ru_2(r) u_1(s) f(s) ds + c_2 \int_r^{\infty}\partial_r u_1(r) u_2(s) f(s) ds. 
\end{align*}
Using Cauchy-Schwarz inequality and the asymptotics of $u_1(r)$ and $u_2(r)$ from Lemma \ref{lem:FL_2=0}, we obtain $\| r  \partial_r   (\FL_2 f)^{-1} \|_{L^2(0,1)}  + \| e^{-2r}  \partial_r   (\FL_2 f)^{-1} \|_{L^2(1,\infty)}\lesssim \| f\|_{L^2}.$ Hence, $B_1$ is bounded on $X_0.$
For $\alpha=-1$ one can see that $B_{-1}f=\FL_{2}^{-1} \rho^2 \FL_{2}f=\rho^2f - \frac{1}{2}\FL_2^{-1}(  \partial_r^2 \rho^2 f  + 2 \partial_r  \rho^2 \partial_r f ) .$ Using \eqref{eq:FL2-1} and integration by parts together with Cauchy-Schwarz together we obtain that $B_{-1}$ is bounded on $X_0.$ Taking the adjoint in the first term in \eqref{eq:adj1}, and using the fact that $B_{-\frac{1}{2}}$ is bounded on $X_0,$ we obtain 
\begin{align*}
    g= 2 \FL_2^{-\frac{3}{2}} \rho^2 \FL_2^{\frac{1}{2}} g + \FL_2^{-2} h= \FL_2^{-1} \left(2  B_{-\frac{1}{2}} g + \FL_2^{-1} h  \right)\in \FL_{2}^{-1}X_0 .
\end{align*}
Using the fact that $B_{\frac{1}{2}}$ is bounded on $X_0$, we obtain 
\begin{align*}
    g=  \FL_2^{-2} \left(2  B_{\frac{1}{2}} \FL_2 g  +  h \right)  \in \FL_{2}^{-2}X_0 .
\end{align*}
Thus, $\FT_0$ is self-adjoint. Moreover, for all $z^2 \in (-\infty,-\eta^2),$ the bounded inverse operator $(z^2+ \FT_0)^{-1}:X_0 \to X_0$ exists. In order to obtain that $\FT(z)^{-1}$ is bounded on $X_0,$ we have to show that \begin{align} \label{eq:FT0-X0--FL_2}
    (z^2+\FT_0)^{-1}: \FL_2^{\frac{1}{2}} X_0 \to \FL_2^{-\frac{1}{2}} X_0 .
\end{align} 
The proof of \eqref{eq:FT0-X0--FL_2} is similar to that of Lemma~2.5 in \cite{LSS25} and we therefore omit the details. Therefore, we have for all $z \in \mathfrak{X}, $ we have $(\calL_{GP} -z)^{-1}$ is a bounded operator on $X_0 \times X_0.$ We conclude that the spectrum of $i \mathcal{L}_{GP}$ is real, and, $\spec(i \mathcal{L}_{GP}) \subset (-\infty, -\eta] \cup [\eta,\infty),$ for some $\eta>0.$  Thus, it remains to prove that $\spec_{ess}(i\calL_{GP})=(-\infty, -\frac{\sqrt{17}}{8}] \cup [\frac{\sqrt{17}}{8},\infty).$ \\
Let $i\calL_{GP}=i\calL_{GP}^{\infty}+ \mathcal{V}_{GP}(r) ,$ where
\begin{align*}
   \calL_{GP}^{\infty}:=\begin{pmatrix}
    0 & -\frac{1}{2} \partial_r^2+ \frac{3}{8\sinh^2(r)} + \frac{1}{8} \\
   \frac{1}{2} \partial_r^2 -\frac{3}{8\sinh^2(r)}- \frac{17}{8} & 0  
\end{pmatrix} , \quad \mathcal{V}_{GP}(r):=\begin{pmatrix}
    0 &   (\rho^2-1) \\
    -3(\rho^2-1) & 0
\end{pmatrix}
\end{align*}
Observe that, using similar arguments as above with $\eta_1=\frac{1}{8}$ and $\eta_2=\frac{17}{8}$ one can show that $\spec(i\calL_{GP}^{\infty})=\spec_{ess}(i\calL_{GP}^{\infty})=(-\infty, -\frac{\sqrt{17}}{8}] \cup [\frac{\sqrt{17}}{8},\infty). $ Note that for $z$ in the resolvent set of $i\calL_{GP}^{\infty}$ and $i\calL_{GP},$ we have $(i\calL_{GP}^{\infty}-z)^{-1}-(i\calL_{GP}-z)^{-1}$ is compact. Since $\spec(i \mathcal{L}_{GP}) \subset (-\infty, -\eta] \cup [\eta,\infty)$  then $0 $ is in the resolvent set of $i \mathcal{L}_{GP}$ and by the above observation  $0$ is resolvent set of $i \mathcal{L}_{GP}^{\infty}.$ At this stage, we can invoke the argument from the proof of the Weyl criterion in \cite[Theorem~XIII.14]{ReSi.T4} to obtain that $\spec_{ess}(i\calL_{GP})=(-\infty, -\frac{\sqrt{17}}{8}] \cup [\frac{\sqrt{17}}{8},\infty).$ Indeed, let $A=\calL_{GP}^\infty$, $B=\calL_{GP}$, $D=A^{-1}$, $E=B^{-1}$. In view of \cite[Lemma~XIII.4.2]{ReSi.T4}, if we can show that $\spec_{ess}(E)=\spec_{ess}(D)$ then $\spec_{ess}(A)=\spec_{ess}(B)$. But the equality  $\spec_{ess}(E)=\spec_{ess}(D)$ follows from \cite[Lemma~XIII.4.3]{ReSi.T4}.
\end{proof}

\section{Solution to $i\mathcal{L}-z$ near $0$ in the non-resonance case}
\label{sec:Sol-near-0-non-resonance}
In this section, we start the construction of the distorted Fourier transform associated with the linearized operator $\mathcal{L},$ in the absence of the resonance. The main goal is to construct the fundamental matrix solutions $F_1(\cdot,z)$ and $F_2(\cdot, z)$ to 
\begin{equation}
\label{eq:sec3:iLF=zF}
   i \mathcal{L} F(\cdot, z) = z F(\cdot, z),  
\end{equation}
for small and large $|z|.$ Here we denote by $F_1(\cdot,z)$ a solution branch that is $L^2$ near $r=0.$  \\

We define three Green's functions as 
\begin{align}
\label{Greens-function-F_j}
\begin{split}
   \mathcal{G}_1^{+}(r,s)&:= F_1^{+}(r) S_1^{+}(s) \mathbb{1}_{\{ 0\leq s \leq r\} } + F_2^{+}(r) T_1^{+}(s) \mathbb{1}_{  \{ 0\leq s \leq r \} }  \\
     \mathcal{G}_2^{+}(r,s)&:= F_1^{+}(r) S_2^{+}(s) \mathbb{1}_{\{ r\leq s \leq r_0 \} } + F_2^{+}(r) T_2^{+}(s) \mathbb{1}_{  \{ 0\leq s \leq r \} } \\
    \mathcal{G}_3^{+}(r,s)&:= F_1^{+}(r) S_3^{+}(s) \mathbb{1}_{\{ r\leq s \leq r_0 \} } + F_2^{+}(r) T_3^{+}(s) \mathbb{1}_{  \{ r\leq s \leq r_0 \} }        
\end{split}
\end{align}
where for $i=1,2,3$ we require the matrices $S_i^{+}(r)$ and $T_i^{+}(r),$ for $\mathcal{G}_i^{+}(r,s)$   to satisfy

\begin{align}
\label{eq:S_1T_1condition}
\begin{pmatrix}
F_1^{+}(r) &     F_2^{+}(r)\\
  \partial_r F_1^{+}(r) & \partial_r F_2^{+}(r) 
\end{pmatrix} 
\begin{pmatrix}
    S_1^{+}(r) \\ T_1^{+}(r)
\end{pmatrix}
&= \begin{pmatrix}
    0 \\ \sigma_2
\end{pmatrix} , \\ 
\label{eq:S_2T_2condition}
 \begin{pmatrix}
F_1^{+}(r) &     F_2^{+}(r)\\
  \partial_r F_1^{+}(r) & \partial_r F_2^{+}(r) 
\end{pmatrix} 
\begin{pmatrix}
   - S_2^{+}(r) \\ T_2^{+}(r)
\end{pmatrix}
&= \begin{pmatrix}
    0 \\ \sigma_2
\end{pmatrix}, \\ 
\label{eq:S_3T_3condition}
 \begin{pmatrix}
F_1^{+}(r) &     F_2^{+}(r)\\
  \partial_r F_1^{+}(r) & \partial_r F_2^{+}(r) 
\end{pmatrix} 
\begin{pmatrix}
   - S_3^{+}(r) \\- T_3^{+}(r)
\end{pmatrix}
&= \begin{pmatrix}
    0 \\ \sigma_2
\end{pmatrix}.
\end{align}

Denote by \begin{align*}
   F_1(r,z)= \begin{bmatrix}
       \upvarphi_1(r,z) &  \upvarphi_2(r,z)
   \end{bmatrix}    \quad \text{and} \quad  F_2(r,z)= \begin{bmatrix}
        \upvarphi_3(r,z) & \upvarphi_4(r,z)
   \end{bmatrix}  
\end{align*} 
and  
\begin{align*}
 \Gamma_i^{+}(r,z,\upvarphi (r,z))&:= \int_0^\infty \mathcal{G}^{+}_i(r,s) \, z \upvarphi (s,z)  ds  \qquad \text{for } i=1,2, 3.
\end{align*}
\\

Then a solution to \eqref{eq:sec3:iLF=zF} for $z$ near $\frac{\sqrt{17}}{8}$ and $-\frac{\sqrt{17}}{8}$ is given by the following two scenarios:\\

\textbf{Scenario I:} $z$ near $\frac{\sqrt{17}}{8}$: Let $z= \frac{\sqrt{17}}{8}+\xi,$ $\re (z)>0$, then we will define $F_1(r,z)$ and $F_2(r,z)$ as a solution of the fixed point problem  
\begin{align}
\label{eq:defF_1(r,z)-scenarioI}
   F_1(r,z)&:=  C_1  F_1^{+} (r) + \mathcal{T}_1^{+}(r,z, F_1(r,z))  \\
\label{eq:defF_2(r,z)-scenarioI}
   F_2(r,z)&:=  C_2  F_2^{+} (r) + \mathcal{T}_2^{+}(r,z, F_2(r,z))  
\end{align}

where, $F_j^{+}(r)$ are the solutions to $ i \mathcal{L}F_j^{+}(r)=\frac{\sqrt{17}}{8}F_j^{+}(r)$ from Section~\ref{subsec:non-resonance}, and 
\begin{align*}
 \mathcal{T}_1^{+}(r,z, F_1(r,z))  &:= \Gamma_1^{+}(r,z, \upvarphi_1(r,z))+ \Gamma_2^{+}(r,z, \upvarphi_2(r,z))  \\
  & =\int_0^{\infty} \mathcal{G}_1^{+}(r,s) \, \xi  \upvarphi_1(s,z) ds  + \int_0^{\infty} \mathcal{G}_2^{+}(r,s)  \, \xi  \upvarphi_2(s,z) ds  
\end{align*}
\begin{align*}
  \mathcal{T}_2^{+}(r,z, F_2(r,z))  &:= \Gamma_3^{+}(r,z,\upvarphi_3 (r,z))+ \Gamma_2^{+}(r,z,\upvarphi_4 (r,z))  \\
  & =\int_0^{\infty} \mathcal{G}_3^{+}(r,s)  \, \xi  \upvarphi_3(s,z) ds  + \int_0^{\infty} \mathcal{G}_2^{+}(r,s)  \, \xi  \upvarphi_4(s,z) ds  
\end{align*}
\textbf{Scenario II:} $z$ near $\frac{-\sqrt{17}}{8}$: Let $z= -\frac{\sqrt{17}}{8}+\xi,$ $\re(z)<0$, then we define, 
\begin{align*}
   F_1(r,z)&:=  -\sigma_3 F_1(r,-z) \sigma_3  \\
   F_2(r,z)&:=  -\sigma_3 F_2(r,-z) \sigma_3 
\end{align*}
Therefore,  $\upvarphi_j(r,z)$ satisfy 
\begin{align}
\label{eq:sym-varphi_j}
\begin{split}
    \upvarphi_1(r,z)&=-\sigma_3 \upvarphi_1(r,-z), \; \;
    \qquad \upvarphi_3(r,z)=-\sigma_3 \upvarphi_3(r,-z), \\
   \upvarphi_2 (r,z)&=\sigma_3 \upvarphi_2(r,-z), \qquad  \qquad \upvarphi_4(r,z)=\sigma_3 \varphi_4(r,-z).  
\end{split}
\end{align}

We will apply a Lyapunov–Perron-type argument to construct a solution via the associated integral equations defined below, and show that the integral operator is a contraction in a suitable space for small and large $|\xi|.$ \\

First, we point out the following inversion identity from \cite[Lemma 5.4]{LSS25}. 

\begin{lemma}[\cite{LSS25}]
\label{inver-2-matrix}
    Let $F(r)$ and $G(r)$ be two ($r$-dependent) $2 \times 2 $ matrices. Suppose that $\mathcal{W}(F,F)=\mathcal{W}(G,G)=0.$ Moreover, suppose that $D:=\mathcal{W}(F,G)$ is invertible. Then,
we have 
\begin{align*}
 \begin{pmatrix}
        F & G \\
        F^{\prime} & G^{\prime} 
    \end{pmatrix}^{-1} = 
    \begin{pmatrix}
        (D^t)^{-1} & 0 \\ 
        0 & D^{-1}
    \end{pmatrix} \begin{pmatrix}
        0 & -I \\
        I & 0 
    \end{pmatrix}
    \begin{pmatrix}
        F^t & (F^{\prime})^t \\
       G^t & (G^{\prime})^t
    \end{pmatrix}
   \begin{pmatrix}
       0 & \sigma_3 \\
       - \sigma_3 & 0 
   \end{pmatrix} .
\end{align*} 
\end{lemma}
\begin{proof}
 Under the assumption that $\mathcal{W}[F,F] = \mathcal{W}[G,G] = 0$, we have that
 \begin{align*}
  \begin{pmatrix} 0 & - I \\ I & 0 \end{pmatrix} \begin{pmatrix} F^t & F'^t \\ G^t & G'^t \end{pmatrix} \begin{bmatrix} 0 & \sigma_3 \\ -\sigma_3 & 0 \end{bmatrix} \begin{pmatrix} F & G \\ F' & G' \end{pmatrix} = \begin{pmatrix} - \mathcal{W}[G,F] & 0 \\ 0 & \mathcal{W}[F,G] \end{pmatrix}.
 \end{align*}
Using the fact that $D^t = -\mathcal{W}[G,F]$, and the above identity we obtain the desired result.
\end{proof}

In order to compute $S_i^{+}(r)$ and $T_i^{+}(r)$, we first need to invert the matrix defined above. To do this, we begin by computing the Wronskians between $F_1^{\pm}(r)$ and $F_2^{\pm}(r)$.

\begin{lemma}
\label{D_0}
    Let $D_0^{+}=\mathcal{W}(F_1^{+}(\cdot),F_2^{+}(\cdot)).$ Then, we have  
\begin{align*}D_0^{+}=
\begin{pmatrix}
      2c_1^2 c_3^2 & 0 \\ \\
      0  & -2c_2^1 c_4^1
   \end{pmatrix}
\qquad \text{and} \qquad 
   (D_0^{+})^{-1}  =(D_0^{+})^{-t}= \begin{pmatrix}
 \frac{1}{2c_1^2 c_3^2}      &  0 \\ \\
0       & -\frac{1}{2c_2^1 c_4^1}
   \end{pmatrix}
  : =  \begin{pmatrix}
d_1    &  0 \\ \\
0      &  d_2
   \end{pmatrix}
\end{align*}
\end{lemma}
\begin{proof}
    Let
$  D^{+}= \begin{pmatrix}
       \delta^{+} & \gamma^{+} \\
       \beta^{+} & \alpha^{+}
   \end{pmatrix} 
$, then we have 
\begin{align*}
    \delta^{+}&= W(\varphi_1^{+}(\cdot), \varphi_3^{+}(\cdot)) = W\left( \begin{pmatrix} 0\\   c_1^2 r^{\frac{3}{2}} \end{pmatrix} + O(r^{\frac{7}{2}}), \begin{pmatrix}
        c_3^1 r^{-\frac{1}{2}} \\
         c_3^2 r^{-\frac{1}{2}}
    \end{pmatrix} (1+O(r^{2} \log(r))) \, \right)=2c_1^2 c_3^2, \\
    \gamma^{+}&= W(\varphi_1^{+}, \varphi_4^{+}) = W\left( \begin{pmatrix} 0\\   c_1^2 r^{\frac{3}{2}} \end{pmatrix}+O(r^{\frac{7}{2}}), \begin{pmatrix}
         c_4^1 r^{-\frac{1}{2}}\\
       0
    \end{pmatrix} + O(r^{\frac{3}{2}} \log(r))  \, \right)=0, \\
    \alpha^{+} &=  W(\varphi_2^{+}, \varphi_4^{+}) = W\left(\begin{pmatrix}
           c_2^1 r^{\frac{3}{2}}  \\ 
          c_2^2 r^{\frac{3}{2}} 
         \end{pmatrix} (1+O(r^{2})) , \begin{pmatrix}
        c_4^1r^{-\frac{1}{2}} \\
        0
    \end{pmatrix}+O(r^{\frac{3}{2}}\log(r) ) \right)=-2c_2^1 c_4^1,  \\
    \beta^{+} &= W(\varphi_2^{+}, \varphi_3^{+}) = W\left( \begin{pmatrix}
           \tc_2^1 r  \\ 
          \tc_2^2 r 
         \end{pmatrix} (1+O(r^{-2})),  \begin{pmatrix}
          \tc_3^1 e^{-\frac{3}{\sqrt{2}}r}\\
           \tc_3^2 e^{-\frac{3}{\sqrt{2}}r}
    \end{pmatrix} (1 + O(r^{-1})) \right)=0. 
\end{align*}

Since, the Wronskians are independent of $r$ and can be evaluated from the asymptotics behavior of $\varphi_j^{+}(r).$ Therefore, we computed the first three Wronskians for small $r$ and last one for large $r.$ 
Thus \begin{align*}
   D_0^{+} = \begin{pmatrix}
      2c_1^2 c_3^2 & 0 \\ \\
      0  & -2c_2^1 c_4^1
   \end{pmatrix}
\quad \text{and}
\quad
    (D_0^{+})^{-t}= \frac{1}{d}  \begin{pmatrix}
 -2c_2^1 c_4^1     & 0  \\ \\
   0    &  2c_1^2 c_3^2
   \end{pmatrix}  : =  \begin{pmatrix}
d_1    &  0 \\ \\
0      &  d_2
   \end{pmatrix}
\end{align*}
where, 
 $ d^{+}={\rm{det}}(D)= -4 c_1^2 c_3^2 c_2^1 c_4^1\neq 0, $  $d_1:=\frac{1}{2c_1^2 c_3^2}$ and $d_2:=-\frac{1}{2 c_2^1 c_4^1}$
\end{proof}

 Next, we compute $S_i^{+}(r)$ and $T_i^{+}(r)$, using Lemma \ref{inver-2-matrix} and \ref{D_0}.
\begin{lemma} 
\label{lem:def_S-T-near 0}
We have
\begin{align*}
&\begin{cases}
S_1^{+}(r)&= - (D_0^{+})^{-t} F_2^{+}(r)^t \sigma_3 \sigma_2 \\
    T_1^{+}(r)&= (D_0^{+})^{-1}F_1^{+}(r)^t \sigma_3 \sigma_2   
\end{cases}
\qquad 
\begin{cases}
S_2^{+}(r)&=  (D_0^{+})^{-t} F_2^{+}(r)^t \sigma_3 \sigma_2 \\
T_2^{+}(r)&= (D_0^{+})^{-1}F_1^{+}(r)^t \sigma_3 \sigma_2   
\end{cases}\\
&\begin{cases}
S_3^{+}(r)&=  (D_0^{+})^{-t} F_2^{+}(r)^t \sigma_3 \sigma_2 \\
T_3^{+}(r)&= -(D_0^{+})^{-1}F_1^{+}(r)^t \sigma_3 \sigma_2   
\end{cases}
\end{align*}

 and 
 \begin{align*}
\mathcal{G}_1^{+}(r,s):&= \begin{cases}
    i   F_1^{+}(r) (D_0^{+})^{-t} F_2^{+}(s)^t \sigma_1, \qquad 0 < s \leq r, \\ \\
  - i  F_2^{+}(r) (D_0^{+})^{-1} F_1^{+}(s) ^t \sigma_1 , \qquad 0 < s \leq r.
\end{cases} \\ \\
\mathcal{G}_2^{+}(r,s):&= \begin{cases}
  -  i   F_1^{+}(r) (D_0^{+})^{-t} F_2^{+}(s)^t \sigma_1, \qquad r \leq s \leq r_0, \\ \\
  - i  F_2^{+}(r) (D_0^{+})^{-1} F_1^{+}(s) ^t \sigma_1 , \qquad 0 < s \leq r.
\end{cases} \\ \\
\mathcal{G}_3^{+}(r,s):&= \begin{cases}
   - i   F_1^{+}(r) (D_0^{+})^{-t} F_2^{+}(s)^t \sigma_1, \qquad r \leq s \leq r_0, \\ \\
   i  F_2^{+}(r) (D_0^{+})^{-1} F_1^{+}(s) ^t \sigma_1 , \qquad r \leq s \leq r_0.
\end{cases}
\end{align*}
\end{lemma}
\begin{proof}
Using \eqref{eq:S_1T_1condition} and Lemma \ref{inver-2-matrix}, we obtain
 \begin{align*}
  \begin{pmatrix}
    S_1^{+}(r) \\ T_1(r)
\end{pmatrix} &=    \begin{pmatrix}
F_1^{+}(r) &     F_2^{+}(r)\\
  \partial_r F_1^{+}(r) & \partial_r F_2^{+}(r) 
\end{pmatrix}^{-1} 
\begin{pmatrix}
    0 \\ \sigma_2
\end{pmatrix} \\
&=  \begin{pmatrix}
       (D_0^{+})^{-t}  & 0 \\ 
        0 & (D_0^{+})^{-1} 
    \end{pmatrix} \begin{pmatrix}
        0 & -I \\
        I & 0 
    \end{pmatrix}
    \begin{pmatrix}
        F_1^{+}(r)^t & (\partial_r F_1^{+}(r) )^t \\
       F_2^{+}(r)^t & (\partial_r F_2^{+}(r))^t
    \end{pmatrix}
   \begin{pmatrix}
       0 & \sigma_3 \\
       - \sigma_3 & 0 
   \end{pmatrix} 
   \begin{pmatrix}
    0 \\ \sigma_2
\end{pmatrix} \\
& = \begin{pmatrix}
    -(D_0^{+})^{-t}  F_2^{+}(r)^t \sigma_3 \sigma _2 \\
   (D_0^{+})^{-1}  F_1^{+}(r)^t \sigma_3 \sigma_2 
\end{pmatrix}
 \end{align*}
 Using the fact that $S_2^{+}(r)=-S_1^{+}(r)$, $T_2^{+}(r)=T_1^{+}(r)$ and $S_3^{+}(r)=-S_1^{+}(r)$, $T_3^{+}(r)=-T_1^{+}(r),$ we get 
\begin{align*}
     \begin{pmatrix}
    S_2^{+}(r) \\ T_2^{+}(r)
\end{pmatrix} = \begin{pmatrix}
   (D_0^{+})^{-t}  F_2^{+}(r)^t \sigma_3 \sigma _2 \\
  (D_0^{+})^{-1}  F_1^{+}(r)^t \sigma_3 \sigma_2 
\end{pmatrix}
\qquad \text{ and }
\qquad 
 \begin{pmatrix}
    S_3^{+}(r) \\ T_3^{+}(r)
\end{pmatrix} = \begin{pmatrix}
   (D_0^{+})^{-t}  F_2^{+}(r)^t \sigma_3 \sigma _2 \\
   - (D_0^{+})^{-1}  F_1^{+}(r)^t \sigma_3 \sigma_2 
\end{pmatrix}   
\end{align*}
Using the fact that $\sigma_3 \sigma_2= -i \sigma_1,$ we obtain the desired formula for $\mathcal{G}^{+}_j(r,s).$
\end{proof}

\subsection{Small $|\xi|$}

Using the Lyapunov–Perron method, we construct solutions to $i \mathcal{L}F=(\frac{\sqrt{17}}{8}+ \xi ) F ,$ for small $\xi,$ via fixed point argument. In particular, we show that the associated integral operators, defined in \eqref{eq:defF_1(r,z)-scenarioI} and \eqref{eq:defF_2(r,z)-scenarioI}, are contractions on the following function spaces.

\begin{defi}
\label{def:SpaceX_j}
Let $r_0>0$, and $F=\begin{bmatrix}
 F^{(1)}&   F^{(2)}
\end{bmatrix} , $ where $F^{(i)}$ denote the columns of a matrix $F.$ We define the following norms:
\begin{align*}
   \left\| F \right\|_{\mathcal{X}_1} &:= \| F^{(1)} \|_{X_1}  +  \| F^{(2)} \|_{X_2} \quad \text{and} \quad \left\| F \right\|_{\mathcal{X}_1^{\prime}} := \| F^{(1)} \|_{X_1^{\prime}}  +  \| F^{(2)} \|_{X_2^{\prime}}, \\
     \left\| F \right\|_{\mathcal{X}_2} &:= \| F^{(1)} \|_{X_3}  +  \| F^{(2)} \|_{X_4} \quad \text{and} \quad  \left\| F \right\|_{\mathcal{X}_2^{\prime}} := \| F^{(1)} \|_{X_3^{\prime}}  +  \| F^{(2)} \|_{X_4^{\prime}},
\end{align*} 
where
\begin{align*}
  \| F^{(1)} \|_{X_1} &:=  \sup_{0 < r < 1}    | r^{-\frac{3}{2}} F^{(1)} (r) |   + \sup_{1 \leq r\leq r_{0}} | e^{-\frac{3}{\sqrt{2}}r} F^{(1)}(r) | ,  \\
  \| F^{(2)} \|_{X_2} &:=  \sup_{0 < r < 1}    | r^{-\frac{3}{2}} F^{(2)} (r) |   + \sup_{1 \leq r\leq r_{0}} |r^{-1}  F^{(2)}(r) | , \\
  \| F^{(1)} \|_{X_3} &:=  \sup_{0 < r < 1}    | r^{\frac{1}{2}} F^{(1)} (r) |   + \sup_{1 \leq r\leq r_{0}} | e^{\frac{3}{\sqrt{2}}r} F^{(1)}(r) | , \\
  \| F^{(2)} \|_{X_4} &:=  \sup_{0 < r < 1}    | r^{\frac{1}{2}} F^{(2)} (r) |   + \sup_{1 \leq r\leq r_{0}} | F^{(2)}(r) |  .
\end{align*}
and 
\begin{align*}
  \| F^{(1)} \|_{X_1^{\prime}} &:=  \sup_{0 < r < 1}    | r^{-\frac{1}{2}} F^{(1)} (r) |   + \sup_{1 \leq r\leq r_{0}} | e^{-\frac{3}{\sqrt{2}}r} F^{(1)}(r) | ,  \\
  \| F^{(2)} \|_{X_2^{\prime}} &:=  \sup_{0 < r < 1}    | r^{-\frac{1}{2}} F^{(2)} (r) |   + \sup_{1 \leq r\leq r_{0}} |  F^{(2)}(r) | , \\
  \| F^{(1)} \|_{X_3^{\prime}} &:=  \sup_{0 < r < 1}    | r^{\frac{3}{2}} F^{(1)} (r) |   + \sup_{1 \leq r\leq r_{0}} | e^{\frac{3}{\sqrt{2}}r} F^{(1)}(r) | , \\
  \| F^{(2)} \|_{X_4^{\prime}} &:=  \sup_{0 < r < 1}    | r^{\frac{3}{2}} F^{(2)} (r) |   + \sup_{1 \leq r\leq r_{0}} |r^{2} F^{(2)}(r) |  .
\end{align*}

\end{defi}

\begin{prop}
\label{Prop:contraction-T_j(F_j)-xi-small}
Let $0<|\xi| \leq \delta_0 < 1 ,$ and $r_0 :=\frac{\varepsilon_0}{\sqrt{|\xi|}},$ where $0<\delta_0\ll \varepsilon_0\ll1$. Then $\mathcal{T}^{+}_j$ is a contraction on $\mathcal{X}_j,$ for $j=1,2,$  i.e., $$ \|\mathcal{T}^{+}_j(F(\cdot,z)) \|_{\mathcal{X}_j} \lesssim \varepsilon_0 \left\| F(\cdot,z) \right\|_{\mathcal{X}_j} . $$
Moreover, the fixed points $F_j(\cdot,z)$ of \eqref{eq:defF_1(r,z)-scenarioI} and \eqref{eq:defF_2(r,z)-scenarioI} satisfy $$ \|  \mathcal{T}^{+}_j(F_j(\cdot,z)) \|_{\mathcal{X}_j}+ \| \partial_r \mathcal{T}^{+}_j(F_j(\cdot,z)) \|_{\mathcal{X}_j^{\prime}} \lesssim \varepsilon_0.$$ 
\end{prop}
\begin{proof}
We will only proof that $\mathcal{T}^{+}_j$ is a contraction on $\mathcal{X}_j,$ and the derivative bounds can be obtained using similar estimates and required conditions \eqref{eq:S_1T_1condition},\eqref{eq:S_2T_2condition} and \eqref{eq:S_3T_3condition}. Recall that 
\begin{align*}
 \mathcal{T}_1^{+}(r,z, F(r,z))  &:= \Gamma_1^{+}(r,z, \upvarphi_1(r,z))+ \Gamma_2^{+}(r,z, \upvarphi_2(r,z))  \\
  & =\int_0^{\infty} \mathcal{G}_1^{+}(r,s) \, \xi  \upvarphi_1(s,z) ds  + \int_0^{\infty} \mathcal{G}_2^{+}(r,s)  \, \xi  \upvarphi_2(s,z) ds  
\end{align*}
\begin{align*}
  \mathcal{T}_2^{+}(r,z, F(r,z))  &:= \Gamma_3^{+}(r,z,\upvarphi_3 (r,z))+ \Gamma_2^{+}(r,z,\upvarphi_4 (r,z))  \\
  & =\int_0^{\infty} \mathcal{G}_3^{+}(r,s)  \, \xi  \upvarphi_3(s,z) ds  + \int_0^{\infty} \mathcal{G}_2^{+}(r,s)  \, \xi  \upvarphi_4(s,z) ds  
\end{align*}
where we have denoted the columns of $F$ by $\varphi_1$, $\varphi_2$ in the case of $\mathcal{T}_1^+$ and by $\varphi_3$, $\varphi_4$ in the case of $\mathcal{T}_2^+$.
Therefore, we will only provide the estimates for $\Gamma_j^{+}$, from which the corresponding bounds for $\mathcal{T}_j^{+}$ can be deduced. In view of the asymptotic behavior of $F_1^{+}(r)$ and $F_2^{+}(r)$, we divide the analysis into two regions: the small-$r$ regime ($r < 1$) and the large-$r$ regime ($r \geq 1$).

\begin{claim}
\label{Claim:Gamma_1-and-2-small-r-non-resonance}
    For small $r<1,$ we have 
\begin{align*}
   \sup_{0 < r < 1}    | r^{-\frac{3}{2}} \Gamma_1^{+}(r,z, \upvarphi_1(r,z)) |
& \lesssim \xi \sup_{0 < r < 1}    | r^{-\frac{3}{2}}  \upvarphi_1(r,z)) | \\
   \sup_{0 < r < 1} | r^{-\frac{3}{2}}   \Gamma_2^{+}(r,z, \upvarphi_2(r,z)) |  &\lesssim \xi \sup_{0 < r < 1} | r^{-\frac{3}{2}}    \upvarphi_2(r,z) | + r_{0}^{2} \, \xi \sup_{1 \leq r \leq r_0 } | r^{-1}    \upvarphi_2(r,z) |
\end{align*}
\end{claim}

\begin{proof}

We have 
    \begin{align*}
        \Gamma_1^{+}(r,z, \upvarphi_1(r,z))&=\int_0^{\infty} \mathcal{G}_1^{+}(r,s) \xi  \upvarphi_1(s,z) ds \\
        &= \int_0^r  i   F_1^{+}(r) (D_0^{+})^{-t} F_2^{+}(s)^t \sigma_1 \xi  \upvarphi_1(s,z) ds  \\
        & - \int_0^r  i  F_2^{+}(r) (D_0^{+})^{-1} F_1^{+}(s)^t \sigma_1 \xi  \upvarphi_1(s,z) ds \\
 &= i \int_0^r \begin{bmatrix} \varphi_1^{+}(r)  & \varphi_2^{+} (r) \end{bmatrix}   \begin{pmatrix}
        d_1 & 0 \\
        0 & d_2
    \end{pmatrix}
    \begin{bmatrix} (\varphi_3^{+})^{t} (s)  \\ (\varphi_4^{+})^{t} (s) \end{bmatrix}  \sigma_1 \xi  \upvarphi_1(s,z) ds  \\
 &  - i \int_0^r \begin{bmatrix} \varphi_3^{+}(r)  & \varphi_4^{+} (r) \end{bmatrix}   \begin{pmatrix}
        d_1 & 0 \\
        0 & d_2
    \end{pmatrix}
    \begin{bmatrix} (\varphi_1^{+})^{t} (s)  \\ (\varphi_2^{+})^{t} (s) \end{bmatrix}  \sigma_1 \xi  \upvarphi_1(s,z) ds \\
&= i \int_0^r   \left(  d_1 \varphi_1^{+}(r)   (\varphi_3^{+}(s))^{t}  +  d_2 \varphi_2^{+}(r)   (\varphi_4^{+} (s)) ^{t} \right)  \sigma_1 \xi  \upvarphi_1(s,z) ds \\
&- i \int_0^r    \left(  d_1 \varphi_3^{+}(r)   (\varphi_1^{+}(s))^{t} +  d_2 \varphi_4^{+}(r)   (\varphi_2^{+}(s)) ^{t}  \right)  \sigma_1\xi  \upvarphi_1(s,z) ds \\
\end{align*}


Using the asymptotics of $\varphi_j^{+}(r)$ near $0,$ we obtain 

\begin{align}
\label{Gamma_1-varphi-1}
  \Gamma_1^{+}(r,z, \upvarphi_1(r,z))   & \lesssim r^{\frac{3}{2}} \left|  \int_0^{r}   s^{- \frac{1}{2}}  \xi  \upvarphi_1(s,z) ds   \right| + r^{-\frac{1}{2}} \left|  \int_0^{r}   s^{ \frac{3}{2}}  \xi  \upvarphi_1(s,z) ds   \right| 
\end{align}
Thus,

\begin{align}
   \sup_{0 < r < 1} |   r^{-\frac{3}{2}} \Gamma_1^{+}(r,z, \upvarphi_1(r,z)) | \lesssim \xi \sup_{0 < r < 1} | r^{-\frac{3}{2}}    \upvarphi_1(r,z) | .
\end{align}

Next, we estimate $\Gamma_2^{+}(r,z, \upvarphi_2(r,z)).$ Similarly, we obtain 
\begin{align*}
    \Gamma_2^{+}(r,z, \upvarphi_2(r,z))&= \int_0^{\infty} \mathcal{G}_2^{+}(r,s) 
\, \xi  \upvarphi_2(s,z) ds \\
  &= -  i  \int_r^{r_0}   F_1^{+}(r) (D_0^{+})^{-t} F_2^{+}(s)^t \sigma_1 \xi  \upvarphi_2(s,z) ds  \\
        &  - i \int_0^r   F_2^{+}(r) (D_0^{+})^{-1} F_1^{+}(s)^t \sigma_1 \xi  \upvarphi_2(s,z) ds \\
 &= -  i  \int_r^{1}  \left(  d_1 \varphi_1^{+}(r)   (\varphi_3^{+}(s))^{t}  +  d_2 \varphi_2^{+}(r)   (\varphi_4^{+} (s)) ^{t} \right)  \sigma_1 \xi  \upvarphi_2(s,z) ds\\
 & -  i  \int_{1}^{r_0}  \left(  d_1 \varphi_1^{+}(r)   (\varphi_3^{+}(s))^{t}  +  d_2 \varphi_2^{+}(r)   (\varphi_4^{+} (s)) ^{t} \right)  \sigma_1 \xi  \upvarphi_2(s,z) ds \\
&- i \int_0^r    \left(  d_1 \varphi_3^{+}(r)   (\varphi_1^{+}(s))^{t} +  d_2 \varphi_4^{+}(r)   (\varphi_2^{+}(s)) ^{t}  \right)  \sigma_1\xi  \upvarphi_2(s,z) ds 
\end{align*}

Using the asymptotics of $\varphi_j^{+}(r)$ near $0,$ and for $r \geq 1$, we obtain 
\begin{align*}
    \Gamma_2^{+}(r,z, \upvarphi_2(r,z)) & \lesssim  r^{\frac{3}{2}} \left|  \int_r^{1}   s^{- \frac{1}{2}}  \xi  \upvarphi_2(s,z) ds   \right|  + r^{\frac{3}{2}} \left|   \int_{1}^{r_0}  (e^{-\frac{3}{\sqrt{2}}s}+ 1) \xi  \upvarphi_2(s,z) ds  \right| \\
&+ r^{-\frac{1}{2}} \left| \int_0^r   s^{\frac{3}{2}}  \sigma_1 \xi  \upvarphi_2(s,z) ds \right| 
\end{align*}
Thus, 

\begin{align*}
  \sup_{0 < r < 1} | r^{-\frac{3}{2}}   \Gamma_2^{+}(r,z, \upvarphi_2(r,z)) |  \lesssim \xi \sup_{0 < r < 1} | r^{-\frac{3}{2}}    \upvarphi_2(r,z) | + r_{0}^{2} \, \xi \sup_{1 \leq r \leq r_0 } | r^{-1}    \upvarphi_2(r,z) |
\end{align*}
\end{proof}

\begin{claim}
\label{Claim:Gamma_3-and-2-small-r-non-resonance}
    For small $r \leq 1,$ we have
    \begin{align*}
 \sup_{0 < r\leq  1 } | r^{\frac{1}{2}}   \Gamma_3^{+}(r,z,\upvarphi_3 (r,z))  | & \lesssim   \xi  \sup_{0 < r < 1} | r^{\frac{1}{2}}  \upvarphi_3(r,z) | + \xi  r_0  \sup_{1 \leq r \leq r_0 } | e^{\frac{3}{\sqrt{2}}r}  \upvarphi_3(r,z) |   \\
   \sup_{0 < r < 1} | r^{\frac{1}{2}}   \Gamma_2^{+}(r,z, \upvarphi_4(r,z))  |  &\lesssim   \xi  \sup_{0 < r < 1} | r^{\frac{1}{2}}    \upvarphi_4(r,z) |+  r_0 \,  \xi    \sup_{1 \leq r \leq r_0 } |     \upvarphi_4(r,z) |.
    \end{align*}
\end{claim}
\begin{proof}
We have 
\begin{align*}
    \Gamma_3^{+}(r,z, \upvarphi_3(r,z))&= \int_0^{\infty} \mathcal{G}_3^{+}(r,s) 
\, \xi  \upvarphi_3(s,z) ds \\
  &= - i  \int_r^{r_0}    F_1^{+}(r) (D_0^{+})^{-t} F_2^{+}(s)^t \sigma_1 \xi  \upvarphi_3(s,z) ds  \\
        & + i \int_r^{r_0}   F_2^{+}(r) (D_0^{+})^{-1} F_1^{+}(s)^t \sigma_1 \xi  \upvarphi_3(s,z) ds \\
 &= -  i  \int_r^{1}  \left(  d_1 \varphi_1^{+}(r)   (\varphi_3^{+}(s))^{t}  +  d_2 \varphi_2^{+}(r)   (\varphi_4^{+} (s)) ^{t} \right)  \sigma_1 \xi  \upvarphi_3(s,z) ds\\
 & -  i  \int_{1}^{r_0}  \left(  d_1 \varphi_1^{+}(r)   (\varphi_3^{+}(s))^{t}  +  d_2 \varphi_2^{+}(r)   (\varphi_4^{+} (s)) ^{t} \right)  \sigma_1 \xi  \upvarphi_3(s,z) ds \\
&+  i  \int_r^{1}     \left(  d_1 \varphi_3^{+}(r)   (\varphi_1^{+}(s))^{t} +  d_2 \varphi_4^{+}(r)   (\varphi_2^{+}(s)) ^{t}  \right)  \sigma_1\xi  \upvarphi_3(s,z) ds \\
&+  i  \int_{1}^{r_0}      \left(  d_1 \varphi_3^{+}(r)   (\varphi_1^{+}(s))^{t} +  d_2 \varphi_4^{+}(r)   (\varphi_2^{+}(s)) ^{t}  \right)  \sigma_1\xi  \upvarphi_3(s,z) ds 
\end{align*}

Using the asymptotics of $\varphi_j^{+}(r)$ near $0,$ and for $r \geq 1$, we obtain 
\begin{align*}
 \Gamma_3^{+}(r,z, \upvarphi_3(r,z))   & \lesssim  r^{\frac{3}{2}} \left|  \int_r^{1}   s^{- \frac{1}{2}}  \xi  \upvarphi_3(s,z) ds   \right|  + r^{\frac{3}{2}} \left|   \int_{1}^{r_0}  (e^{-\frac{3}{\sqrt{2}}s}+ 1) \xi  \upvarphi_3(s,z) ds  \right| \\
&+ r^{-\frac{1}{2}} \left| \int_r^1   s^{\frac{3}{2}}  \sigma_1 \xi  \upvarphi_3(s,z) ds \right| + r^{-\frac{1}{2}} \left|   \int_{1}^{r_0}  (e^{\frac{3}{\sqrt{2}}s}+ s) \xi  \upvarphi_3(s,z) ds  \right| \\
& \lesssim r^{\frac{3}{2}} |\log(r)| \xi  \sup_{0 < r < 1} | r^{\frac{1}{2}}  \upvarphi_3(r,z) |  + r^{\frac{3}{2}} \xi    \sup_{1 \leq r \leq r_0 }  | e^{\frac{3}{\sqrt{2}}r}  \upvarphi_3(r,z) |  \\
&+ r^{-\frac{1}{2}}  \xi  \sup_{0 < r < 1} | r^{\frac{1}{2}}  \upvarphi_3(r,z) | +  r^{-\frac{1}{2}} r_0    \xi    \sup_{1 \leq r \leq r_0 }  | e^{\frac{3}{\sqrt{2}}r}  \upvarphi_3(r,z) |
\end{align*}
Thus, 
\begin{align*}
 \sup_{0 < r\leq  1 } | r^{\frac{1}{2}}   \Gamma_3^{+}(r,z,\upvarphi_3 (r,z))  | & \lesssim   \xi  \sup_{0 < r < 1} | r^{\frac{1}{2}}  \upvarphi_3(r,z) | + \xi  r_0  \sup_{1 \leq r \leq r_0 } | e^{\frac{3}{\sqrt{2}}r}  \upvarphi_3(r,z) | 
\end{align*}

Similarly, to the estimate for $ \Gamma_2^{+}(r,z, \upvarphi_2(r,z)),$ we have
\begin{align*}
    \Gamma_2^{+}(r,z, \upvarphi_4(r,z))&= \int_0^{\infty} \mathcal{G}_2(r,s) 
\xi  \upvarphi_4(s,z) ds \\
 &= -  i  \int_r^{1}  \left(  d_1 \varphi_1^{+}(r)   (\varphi_3^{+}(s))^{t}  +  d_2 \varphi_2^{+}(r)   (\varphi_4^{+} (s)) ^{t} \right)  \sigma_1 \xi  \upvarphi_4(s,z) ds\\
 & -  i  \int_{1}^{r_0}  \left(  d_1 \varphi_1^{+}(r)   (\varphi_3^{+}(s))^{t}  +  d_2 \varphi_2^{+}(r)   (\varphi_4^{+} (s)) ^{t} \right)  \sigma_1 \xi  \upvarphi_4(s,z) ds \\
&- i \int_0^r    \left(  d_1 \varphi_3^{+}(r)   (\varphi_1^{+}(s))^{t} +  d_2 \varphi_4^{+}(r)   (\varphi_2^{+}(s)) ^{t}  \right)  \sigma_1\xi  \upvarphi_4(s,z) ds \\
 & \lesssim  r^{\frac{3}{2}} \left|  \int_r^{1}   s^{- \frac{1}{2}}  \xi  \upvarphi_4(s,z) ds   \right|  + r^{\frac{3}{2}} \left|   \int_{1}^{r_0}  (e^{-\frac{3}{\sqrt{2}}s}+ 1) \xi  \upvarphi_4(s,z) ds  \right| \\
&+ r^{-\frac{1}{2}} \left| \int_0^r   s^{\frac{3}{2}}  \sigma_1 \xi  \upvarphi_4(s,z) ds \right| \\
& \lesssim r^{\frac{3}{2}} |\log(r)|  \xi  \sup_{0 < r < 1} | r^{\frac{1}{2}}  \upvarphi_4(r,z) | + r^{\frac{3}{2}} r_0 \xi   \sup_{1 \leq r \leq r_0 } |  \upvarphi_4(r,z) | \\
& + r^{\frac{3}{2}} \xi  \sup_{0 < r < 1} | r^{\frac{1}{2}}  \upvarphi_4(r,z) | 
\end{align*}
Then 
\begin{align*}
  \sup_{0 < r < 1} | r^{\frac{1}{2}}   \Gamma_2^{+}(r,z, \upvarphi_4(r,z))  |  \lesssim   \xi  \sup_{0 < r < 1} | r^{\frac{1}{2}}    \upvarphi_4(r,z) |+  r_0 \,  \xi    \sup_{1 \leq r \leq r_0 } |     \upvarphi_4(r,z) |.
\end{align*}

\end{proof}

\begin{claim}
\label{Claim:Gamma_1-and-2-large-r-non-resonance}

    For large $r>1,$ we have 

\begin{align}
\label{eq:Gamma1-large-r-small-xi}
  \sup_{1 \leq r\leq r_{0}} | e^{-\frac{3}{\sqrt{2}}r}  \Gamma_1^{+}(r,z, \upvarphi_1(r,z)) | &\lesssim r_0 \xi \sup_{0< r \leq 1 } | r^{-\frac{3}{2}} \upvarphi_1(r,z) | + r_0 \xi \sup_{1 \leq r\leq r_{0}} |e^{-\frac{3}{\sqrt{2}}r}  \upvarphi_1(r,z)|   \\
 \label{eq:Gamma2-large-r-small-xi} 
\sup_{1 \leq r \leq r_0 }  |     r^{-1}     \Gamma_2^{+}(r,z, \upvarphi_2(r,z)) | &\lesssim  \xi    \sup_{0 < r < 1} |r^{\frac{3}{2}}  \upvarphi_2(r,z)| +  r_0^2 \xi \sup_{1 \leq r \leq r_0 } |r^{-1}  \upvarphi_2(r,z)|
\end{align}
\end{claim}
\begin{proof}

We have 
    \begin{align*}
        \Gamma_1^{+}(r,z, \upvarphi_1(r,z))&=\int_0^{\infty} \mathcal{G}_1^{+}(r,s) \xi  \upvarphi_1(s,z) ds \\
&= i \int_0^r   \left(  d_1 \varphi_1^{+}(r)   (\varphi_3^{+}(s))^{t}  +  d_2 \varphi_2^{+}(r)   (\varphi_4^{+} (s)) ^{t} \right)  \sigma_1 \xi  \upvarphi_1(s,z) ds \\
&- i \int_0^r    \left(  d_1 \varphi_3^{+}(r)   (\varphi_1^{+}(s))^{t} +  d_2 \varphi_4^{+}(r)   (\varphi_2^{+}(s)) ^{t}  \right)  \sigma_1\xi  \upvarphi_1(s,z) ds \\
&= i \int_0^1   \left(  d_1 \varphi_1^{+}(r)   (\varphi_3^{+}(s))^{t}  +  d_2 \varphi_2^{+}(r)   (\varphi_4^{+} (s)) ^{t} \right)  \sigma_1 \xi  \upvarphi_1(s,z) ds \\
&+i \int_1^r   \left(  d_1 \varphi_1^{+}(r)   (\varphi_3^{+}(s))^{t}  +  d_2 \varphi_2^{+}(r)   (\varphi_4^{+} (s)) ^{t} \right)  \sigma_1 \xi  \upvarphi_1(s,z) ds \\
&- i \int_0^1    \left(  d_1 \varphi_3^{+}(r)   (\varphi_1^{+}(s))^{t} +  d_2 \varphi_4^{+}(r)   (\varphi_2^{+}(s)) ^{t}  \right)  \sigma_1\xi  \upvarphi_1(s,z) ds \\
&- i \int_1^r    \left(  d_1 \varphi_3^{+}(r)   (\varphi_1^{+}(s))^{t} +  d_2 \varphi_4^{+}(r)   (\varphi_2^{+}(s)) ^{t}  \right)  \sigma_1\xi  \upvarphi_1(s,z) ds \\
\end{align*}

Using the asymptotics of $\varphi_j^{+}(r)$ near $0,$ and for $r \geq 1$, we obtain 
 \begin{align*}
        \Gamma_1^{+}(r,z, \upvarphi_1(r,z))& \lesssim  e^{\frac{3}{\sqrt{2}}r} \left| \int_0^1  s^{-\frac{1}{2}}     \xi      \upvarphi_1(s,z) ds \right|  +  r \left| \int_0^1  s^{-\frac{1}{2}}  \xi  \upvarphi_1(s,z) ds \right|   \\ 
 & + e^{-\frac{3}{\sqrt{2}}r} \left|  \int_0^1 s^{\frac{3}{2}}   \xi      \upvarphi_1(s,z) ds \right| + \left|   \int_0^1  s^{\frac{3}{2}} \;  \xi  \upvarphi_1(s,z) ds \right|   \\
 &+e^{\frac{3}{\sqrt{2}}r} \left|  \int_1^r      e^{-\frac{3}{\sqrt{2}}s}  \xi     \upvarphi_1(s,z) ds \right| + r \left|  \int_1^r   \xi  \upvarphi_1(s,z) ds \right|  \\ 
 & + e^{-\frac{3}{\sqrt{2}}r}  \left| \int_1^r  e^{\frac{3}{\sqrt{2}}s}  \xi      \upvarphi_1(s,z) ds  \right| + \left|   \int_1^r  s \;  \xi  \upvarphi_1(s,z) ds \right| \\
 & \lesssim  e^{\frac{3}{\sqrt{2}}r} \xi \sup_{0< r \leq 1 } | r^{-\frac{3}{2}} \upvarphi_1(r,z) | + r \xi \sup_{0< r \leq 1 } | r^{-\frac{3}{2}} \upvarphi_1(r,z) | \\
& +e^{\frac{3}{\sqrt{2}}r} \xi \sup_{0< r \leq 1 } | r^{-\frac{3}{2}} \upvarphi_1(r,z) | +  \xi \sup_{0< r \leq 1 } | r^{-\frac{3}{2}} \upvarphi_1(r,z) | \\
& + e^{\frac{3}{\sqrt{2}}r} \xi \sup_{1 \leq r\leq r_{0}} |e^{-\frac{3}{\sqrt{2}}r}  \upvarphi_1(r,z)| + r e^{\frac{3}{\sqrt{2}}r} \xi \sup_{1 \leq r\leq r_{0}} |e^{-\frac{3}{\sqrt{2}}r}  \upvarphi_1(r,z)| 
\end{align*}

Therefore, 
\begin{align}
     \sup_{1 \leq r\leq r_{0}} | e^{-\frac{3}{\sqrt{2}}r}  \Gamma_1^{+}(r,z, \upvarphi_1(r,z)) | \lesssim r_0 \xi \sup_{0< r \leq 1 } | r^{-\frac{3}{2}} \upvarphi_1(r,z) | + r_0 \xi \sup_{1 \leq r\leq r_{0}} |e^{-\frac{3}{\sqrt{2}}r}  \upvarphi_1(r,z)| 
\end{align}

Next, we estimate $\Gamma_2^{+}(r,z, \upvarphi_2(r,z)).$ Similarly, we obtain 

\begin{align*}
\Gamma_2^{+}(r,z, \upvarphi_2(r,z))&= \int_0^{\infty} \mathcal{G}_2^{+}(r,s) 
\xi  \upvarphi_2(s,z) ds \\
&= - i \int_r^{r_0}   \left(  d_1 \varphi_1^{+}(r)   (\varphi_3^{+}(s))^{t}  +  d_2 \varphi_2^{+}(r)   (\varphi_4^{+} (s)) ^{t} \right)  \sigma_1 \xi  \upvarphi_2(s,z) ds \\
&- i \int_0^r    \left(  d_1 \varphi_3^{+}(r)   (\varphi_1^{+}(s))^{t} +  d_2 \varphi_4^{+}(r)   (\varphi_2^{+}(s)) ^{t}  \right)  \sigma_1\xi  \upvarphi_2(s,z) ds \\
&= - i \int_r^{r_0}   \left(  d_1 \varphi_1^{+}(r)   (\varphi_3^{+}(s))^{t}  +  d_2 \varphi_2^{+}(r)   (\varphi_4^{+} (s)) ^{t} \right)  \sigma_1 \xi  \upvarphi_2(s,z) ds \\
&- i \int_0^1    \left(  d_1 \varphi_3^{+}(r)   (\varphi_1^{+}(s))^{t} +  d_2 \varphi_4^{+}(r)   (\varphi_2^{+}(s)) ^{t}  \right)  \sigma_1\xi  \upvarphi_2(s,z) ds \\
&- i \int_1^r    \left(  d_1 \varphi_3^{+}(r)   (\varphi_1^{+}(s))^{t} +  d_2 \varphi_4^{+}(r)   (\varphi_2^{+}(s)) ^{t}  \right)  \sigma_1\xi  \upvarphi_2(s,z) ds \\
\end{align*}
Using the asymptotics of $\varphi_j^{+}(r)$ near $0,$ and for $r \geq 1$, we obtain 
\begin{align*}
  \Gamma_2^{+}(r,z, \upvarphi_2(r,z))& \lesssim    e^{-\frac{3}{\sqrt{2}}r}  
\left|  \int_0^1 s^{\frac{3}{2}}   \xi      \upvarphi_2(s,z) ds \right|  +  \left|  \int_0^1  s^{\frac{3}{2}} \;  \xi  \upvarphi_2(s,z) ds \right|  \\
 &+ e^{\frac{3}{\sqrt{2}}r} \left|   \int_r^{r_0}      e^{-\frac{3}{\sqrt{2}}s}  \xi      \upvarphi_2(s,z) ds  \right| +   r \left|   \int_r^{r_0}  \xi   \upvarphi_2(s,z) ds \right|  \\ 
 & +e^{-\frac{3}{\sqrt{2}}r}  \left| \int_1^r  e^{\frac{3}{\sqrt{2}}s}  \xi      \upvarphi_2(s,z) ds\right|  +  \left|   \int_1^r  s \;  \xi  \upvarphi_2(s,z) ds \right| \\
 & \lesssim  e^{-\frac{3}{\sqrt{2}}r}  \xi    \sup_{0 < r < 1} |r^{\frac{3}{2}}  \upvarphi_2(r,z)| + \xi  \sup_{0 < r < 1} |r^{\frac{3}{2}}  \upvarphi_2(r,z)| \\ 
 & +  r \xi \sup_{1 \leq r \leq r_0 } |r^{-1}  \upvarphi_2(r,z)|  +  r r_0^2 \xi \sup_{1 \leq r \leq r_0 } |r^{-1}  \upvarphi_2(r,z)| \\
& +  r \xi \sup_{1 \leq r \leq r_0 } |r^{-1}  \upvarphi_2(r,z)| + r^3 \xi \sup_{1 \leq r \leq r_0 } |r^{-1}  \upvarphi_2(r,z)|
\end{align*}

\begin{align}
\sup_{1 \leq r \leq r_0 }  |     r^{-1}     \Gamma_2^{+}(r,z, \upvarphi_2(r,z)) | \lesssim  \xi    \sup_{0 < r < 1} |r^{\frac{3}{2}}  \upvarphi_2(r,z)| +  r_0^2 \xi \sup_{1 \leq r \leq r_0 } |r^{-1}  \upvarphi_2(r,z)|
\end{align}
 \end{proof}

\begin{claim}
\label{Claim:Gamma_3-and-2-large-r-non-resonance}

    For large $r>1,$ we have  \\
    \begin{align*}
     \sup_{1 \leq r \leq r_0}  | e^{\frac{3}{\sqrt{2}}r}  \Gamma_3^{+}(r,z, \upvarphi_3(r,z)) | & \lesssim r_0   \xi  \sup_{1 \leq r \leq r_0} |e^{\frac{3}{\sqrt{2}}r}  \upvarphi_3(r,z)|   \\
       \sup_{1 \leq r \leq r_0 } | \Gamma_2^{+}(r,z, \upvarphi_4(r,z))| &\lesssim 
   \xi  \sup_{0 < r < 1} |r^{\frac{1} {2}}  \upvarphi_4(r,z)| +  r_0^2 \xi \sup_{1 \leq r \leq r_0 } |  \upvarphi_4(r,z)| 
        \end{align*}
\end{claim}
\begin{proof}
We have 
\begin{align*}
    \Gamma_3^{+}(r,z, \upvarphi_3(r,z))&= \int_0^{\infty} \mathcal{G}_3^{+}(r,s) 
\, \xi  \upvarphi_3(s,z) ds \\
 &=  - i  \int_r^{r_0}  \left(  d_1 \varphi_1^{+}(r)   (\varphi_3^{+}(s))^{t}  +  d_2 \varphi_2^{+}(r)   (\varphi_4^{+} (s)) ^{t} \right)  \sigma_1 \xi  \upvarphi_3(s,z) ds\\
&+  i \int_r^{r_0}     \left(  d_1 \varphi_3^{+}(r)   (\varphi_1^{+}(s))^{t} +  d_2 \varphi_4^{+}(r)   (\varphi_2^{+}(s)) ^{t}  \right)  \sigma_1 \xi  \upvarphi_3(s,z) ds \\
& \lesssim e^{\frac{3}{\sqrt{2}}r}  \left|  \int_r^{r_0}      e^{-\frac{3}{\sqrt{2}}s}  \xi      \upvarphi_3(s,z) ds \right| + r \left| \int_r^{r_0}  \xi  \upvarphi_3(s,z) ds  \right| \\ 
 & + e^{-\frac{3}{\sqrt{2}}r}  \left| \int_r^{r_0}  e^{\frac{3}{\sqrt{2}}s}  \xi      \upvarphi_3(s,z) ds \right| +  \left|   \int_r^{r_0}  s \;  \xi  \upvarphi_3(s,z) ds \right| 
\end{align*}

Using the asymptotics of $\varphi_j^{+}(r)$ near $0,$ and for $r \geq 1$, we obtain
\begin{align*}
\Gamma_3^{+}(r,z, \upvarphi_3(r,z))   & \lesssim  e^{-\frac{3}{\sqrt{2}}r}  \xi  \sup_{1 \leq r \leq r_0} |e^{\frac{3}{\sqrt{2}}r}  \upvarphi_3(r,z)| + r e^{-\frac{3}{\sqrt{2}}r}  \xi  \sup_{1 \leq r \leq r_0} |e^{\frac{3}{\sqrt{2}}r}  \upvarphi_3(r,z)| \\
& +  r_0 e^{-\frac{3}{\sqrt{2}}r}  \xi  \sup_{1 \leq r \leq r_0} |e^{\frac{3}{\sqrt{2}}r}  \upvarphi_3(r,z)| + r e^{-\frac{3}{\sqrt{2}}r}  \xi  \sup_{1 \leq r \leq r_0} |e^{\frac{3}{\sqrt{2}}r}  \upvarphi_3(r,z)|
\end{align*}
Therefore, we have 

\begin{align*}
    \sup_{1 \leq r \leq r_0}  | e^{\frac{3}{\sqrt{2}}r}  \Gamma_3^{+}(r,z, \upvarphi_3(r,z)) | \lesssim r_0   \xi  \sup_{1 \leq r \leq r_0} |e^{\frac{3}{\sqrt{2}}r}  \upvarphi_3(r,z)| 
\end{align*}

Next, we estimate $\Gamma_2^{+}(r,z, \upvarphi_4(r,z)).$ Similarly, we have 

\begin{align*}
\Gamma_2^{+}(r,z, \upvarphi_4(r,z))&= \int_0^{\infty} \mathcal{G}_2(r,s) 
\xi  \upvarphi_4(s,z) ds \\
&= - i \int_r^{r_0}   \left(  d_1 \varphi_1^{+}(r)   (\varphi_3^{+}(s))^{t}  +  d_2 \varphi_2^{+}(r)   (\varphi_4^{+} (s)) ^{t} \right)  \sigma_1 \xi  \upvarphi_4(s,z) ds \\
&- i \int_0^r    \left(  d_1 \varphi_3^{+}(r)   (\varphi_1^{+}(s))^{t} +  d_2 \varphi_4^{+}(r)   (\varphi_2^{+}(s)) ^{t}  \right)  \sigma_1\xi  \upvarphi_4(s,z) ds \\
&= - i \int_r^{r_0}   \left(  d_1 \varphi_1^{+}(r)   (\varphi_3^{+}(s))^{t}  +  d_2 \varphi_2^{+}(r)   (\varphi_4^{+} (s)) ^{t} \right)  \sigma_1 \xi  \upvarphi_4(s,z) ds \\
&- i \int_0^1    \left(  d_1 \varphi_3^{+}(r)   (\varphi_1^{+}(s))^{t} +  d_2 \varphi_4^{+}(r)   (\varphi_2^{+}(s)) ^{t}  \right)  \sigma_1\xi  \upvarphi_4(s,z) ds \\
&- i \int_1^r    \left(  d_1 \varphi_3^{+}(r)   (\varphi_1^{+}(s))^{t} +  d_2 \varphi_4^{+}(r)   (\varphi_2^{+}(s)) ^{t}  \right)  \sigma_1\xi  \upvarphi_4(s,z) ds \\
\end{align*}

Using the asymptotics of $\varphi_j^{+}(r)$ near $0,$ and for $r \geq 1$, we obtain

\begin{align*}
 \Gamma_2^{+}(r,z, \upvarphi_4(r,z))& \lesssim  e^{-\frac{3}{\sqrt{2}}r}    \left|  \int_0^1 s^{\frac{3}{2}}   \xi     \upvarphi_4(s,z) ds  \right| + \left|  \int_0^1  s^{\frac{3}{2}} \;  \xi  \upvarphi_4(s,z) ds \right|  \\
 &+e^{\frac{3}{\sqrt{2}}r}   \left|  \int_r^{r_0}      e^{-\frac{3}{\sqrt{2}}s}  \xi   \upvarphi_4(s,z) ds \right| +  r \left|  \int_r^{r_0}  \xi   \upvarphi_4(s,z) ds \right|  \\ 
 & +e^{-\frac{3}{\sqrt{2}}r}  \left|  \int_1^r  e^{\frac{3}{\sqrt{2}}s}  \xi      \upvarphi_4(s,z) ds \right|+  \left|  \int_1^r  s \; \xi  \upvarphi_4(s,z) ds\right|  \\
 & \lesssim  e^{-\frac{3}{\sqrt{2}}r}   \xi  \sup_{0 < r < 1} |r^{\frac{1} {2}}  \upvarphi_4(r,z)| + \xi  \sup_{0 < r < 1} |r^{\frac{1} {2}}  \upvarphi_4(r,z)| \\
 & +\xi   \sup_{1 \leq r \leq r_0 } |  \upvarphi_4(r,z)|   + r^2 \xi   \sup_{1 \leq r \leq r_0 } |  \upvarphi_4(r,z)|  
\end{align*}
Thus, 
\begin{align*}
    \sup_{1 \leq r \leq r_0 } | \Gamma_2^{+}(r,z, \upvarphi_4(r,z))| \lesssim 
   \xi  \sup_{0 < r < 1} |r^{\frac{1} {2}}  \upvarphi_4(r,z)| +  r_0^2 \xi \sup_{1 \leq r \leq r_0 } |  \upvarphi_4(r,z)|    
\end{align*}
\end{proof}
Combining the estimates from Claim \ref{Claim:Gamma_1-and-2-small-r-non-resonance} and \ref{Claim:Gamma_1-and-2-large-r-non-resonance} one can be deduce the corresponding bounds for $\mathcal{T}_1^+$ and similarly for $\mathcal{T}_2^+$ by using Claim \ref{Claim:Gamma_3-and-2-small-r-non-resonance} and \ref{Claim:Gamma_3-and-2-large-r-non-resonance}.
\end{proof}

\subsection{Large $|\xi|$}
Similarly to the case of small $|\xi|$, we construct solutions to $i \mathcal{L}F=(\frac{\sqrt{17}}{8}+ \xi ) F ,$ for large $\xi,$ using the Banach fixed-point theorem. Specifically, we show that the integral equations defined in \eqref{eq:defF_1(r,z)-scenarioI} and \eqref{eq:defF_2(r,z)-scenarioI} yield solutions as fixed points in the following function spaces. Note that by large $|\xi|$ (or equivalently large $|z|$), we mean the regime where the real part is large while the imaginary part remains small.

\begin{defi}
\label{def:Y_j-space}
Let $\tr_0>0$, and $F=\begin{bmatrix}
 F^{(1)}&   F^{(2)}
\end{bmatrix} , $ where $F^{(i)}$ denote the columns of a matrix $F.$ We define the following norms:
\begin{align*}
   \left\| F \right\|_{\mathcal{Y}_1} &:= \| F^{(1)} \|_{Y_1}  +  \| F^{(2)} \|_{Y_1} \quad \text{and} \quad \left\| F \right\|_{\mathcal{Y}_1^{\prime}} := \| F^{(1)} \|_{Y_1^{\prime}}  +  \| F^{(2)} \|_{Y_1^{\prime}}, \\
     \left\| F \right\|_{\mathcal{Y}_2} &:= \| F^{(1)} \|_{Y_2}  +  \| F^{(2)} \|_{Y_2} \quad \text{and} \quad  \left\| F \right\|_{\mathcal{Y}_2^{\prime}} := \| F^{(1)} \|_{Y_2^{\prime}}  +  \| F^{(2)} \|_{Y_2^{\prime}},
\end{align*} 
where
\begin{align*}
  \| f \|_{Y_1} &:=  \sup_{0 < r < \tr_0}    | r^{-\frac{3}{2}} f (r) | \quad \text{and} \quad   \| f \|_{Y_1^{\prime}} :=  \sup_{0 < r < \tr_0}    | r^{-\frac{1}{2}}f  (r) |  \\
  \| f \|_{Y_2} &:=  \sup_{0 < r < \tr_0}   | r^{\frac{1}{2}} f (r) |   \quad \text{and} \quad   \|f \|_{Y_2^{\prime}} :=  \sup_{0 < r < \tr_0}    | r^{\frac{3}{2}} f (r) | .
\end{align*}

\end{defi}

\begin{prop}
\label{prop:contraction-large-xi}
Let $ |\xi| > \Lambda_0>1 ,$ and $\tr_0 :=\frac{\tvarepsilon_0}{\sqrt{|\xi|}},$ where $0<\tvarepsilon_0\ll1$ and $\Lambda_0\gg {\tilde{\varepsilon}}_0^{-1}$. Then $\mathcal{T}^{+}_j$ is a contraction on $\mathcal{Y}_j,$ for $j=1,2,$  i.e., $$ \|\mathcal{T}^{+}_j(F(\cdot,z)) \|_{\mathcal{Y}_j} \lesssim \varepsilon_0 \left\| F(\cdot,z) \right\|_{\mathcal{Y}_j} . $$
Moreover, the fixed points $F_j(\cdot,z)$ of \eqref{eq:defF_1(r,z)-scenarioI} and \eqref{eq:defF_2(r,z)-scenarioI} satisfy $$ \|  \mathcal{T}^{+}_j(F_j(\cdot,z)) \|_{\mathcal{Y}_j}+ \| \partial_r \mathcal{T}^{+}_j(F_j(\cdot,z)) \|_{\mathcal{Y}_j^{\prime}} \lesssim \varepsilon_0 .$$ 

\end{prop}
\begin{proof}

We will only proof that $\mathcal{T}^{+}_j$ is a contraction on $\mathcal{Y}_j,$ and the derivative bounds can be obtained using similar estimates and required conditions \eqref{eq:S_1T_1condition},\eqref{eq:S_2T_2condition} and \eqref{eq:S_3T_3condition}. In view of the definitions of $\mathcal{T}_j^{+},$ we will only provide the estimates for $\Gamma_j^{+}$, from which one can obtain the corresponding bounds for $\mathcal{T}_j^{+}.$  

\begin{claim}
\label{Gamma_j-beahvior-large-xi}
    \begin{align}
    \label{eq:Gamma1-phi1-large-xi}
   \sup_{0 < r \leq \tr_0 } |   r^{-\frac{3}{2}} \Gamma_1^{+}(r,z, \upvarphi_1(r,z)) |  &\lesssim \tr_0^2 \xi \sup_{0 < r \leq \tr_0 } | r^{-\frac{3}{2}}    \upvarphi_1(r,z) | , \\
      \label{eq:Gamma2-phi2-large-xi}
   \sup_{0 < r \leq \tr_0 } | r^{-\frac{3}{2}}   \Gamma_2^{+}(r,z, \upvarphi_2(r,z)) |  & \lesssim \tr_0^2 \xi \sup_{0 < r \leq \tr_0 } | r^{-\frac{3}{2}}    \upvarphi_2(r,z) | , \\
   \label{eq:Gamma3-phi3-large-xi}
    \sup_{0 < r\leq  \tr_0 } | r^{\frac{1}{2}}   \Gamma_3^{+}(r,z,\upvarphi_3 (r,z))  | & \lesssim  \tr_0^2 \xi  \sup_{0 < r \leq \tr_0 } | r^{\frac{1}{2}}  \upvarphi_3(r,z) | \\
     \label{eq:Gamma2-phi4-large-xi}
    \sup_{0 < r \leq \tr_0 }   | r^{\frac{1}{2}}   \Gamma_2^{+}(r,z,\upvarphi_4 (r,z))  | & \lesssim  \tr_0^2 \xi    \sup_{0 < r \leq \tr_0 } | r^{\frac{1}{2}}  \upvarphi_4(r,z) |
\end{align} 
\end{claim}
\begin{proof}
Recall that by \eqref{Gamma_1-varphi-1}, we have
\begin{align*}
  \Gamma_1^{+}(r,z, \upvarphi_1(r,z))   & \lesssim r^{\frac{3}{2}} \left|  \int_0^{r}   s^{- \frac{1}{2}}  \xi  \upvarphi_1(s,z) ds   \right| + r^{-\frac{1}{2}} \left|  \int_0^{r}   s^{ \frac{3}{2}}  \xi  \upvarphi_1(s,z) ds   \right|  
\end{align*}
Thus,

\begin{align*}
   \sup_{0 < r \leq \tr_0 } |   r^{-\frac{3}{2}} \Gamma_1^{+}(r,z, \upvarphi_1(r,z)) | \lesssim \tr_0^2 \xi \sup_{0 < r \leq \tr_0 } | r^{-\frac{3}{2}}    \upvarphi_1(r,z) | .
\end{align*}

Next, we estimate $\Gamma_2^{+}(r,z, \upvarphi_2(r,z)).$ We have
\begin{align*}
    \Gamma_2^{+}(r,z, \upvarphi_2(r,z))&= \int_0^{\infty} \mathcal{G}_2^{+}(r,s) 
\, \xi  \upvarphi_2(s,z) ds \\
  &= -  i  \int_r^{\tr_0}   F_1^{+}(r) (D_0^{+})^{-t} F_2^{+}(s)^t \sigma_1 \xi  \upvarphi_2(s,z) ds  \\
        &  - i \int_0^r   F_2^{+}(r) (D_0^{+})^{-1} F_1^{+}(s)^t \sigma_1 \xi  \upvarphi_2(s,z) ds \\
 &= -  i  \int_r^{\tr_0}  \left(  d_1 \varphi_1^{+}(r)   (\varphi_3^{+}(s))^{t}  +  d_2 \varphi_2^{+}(r)   (\varphi_4^{+} (s)) ^{t} \right)  \sigma_1 \xi  \upvarphi_2(s,z) ds\\
&- i \int_0^r    \left(  d_1 \varphi_3^{+}(r)   (\varphi_1^{+}(s))^{t} +  d_2 \varphi_4^{+}(r)   (\varphi_2^{+}(s)) ^{t}  \right)  \sigma_1\xi  \upvarphi_2(s,z) ds 
\end{align*}

Using the asymptotics of $\varphi_j^{+}(r)$ near $0,$ we obtain 
\begin{align}
\label{Gamma_2-varphi-2-xi-large}
    \Gamma_2^{+}(r,z, \upvarphi_2(r,z)) & \lesssim  r^{\frac{3}{2}} \left|  \int_r^{\tr_0}   s^{- \frac{1}{2}}  \xi  \upvarphi_2(s,z) ds   \right|  + r^{-\frac{1}{2}} \left| \int_0^r   s^{\frac{3}{2}}  \sigma_1 \xi  \upvarphi_2(s,z) ds \right| 
\end{align}
Thus, 

\begin{align*}
  \sup_{0 < r \leq \tr_0 } | r^{-\frac{3}{2}}   \Gamma_2^{+}(r,z, \upvarphi_2(r,z)) |  \lesssim \tr_0^2 \xi \sup_{0 < r \leq \tr_0 } | r^{-\frac{3}{2}}    \upvarphi_2(r,z) | .
\end{align*}

Next, we estimate $ \Gamma_3^{+}(r,z, \upvarphi_3(r,z)).$ We have 
\begin{align*}
    \Gamma_3^{+}(r,z, \upvarphi_3(r,z))&= \int_0^{\infty} \mathcal{G}_3^{+}(r,s) 
\, \xi  \upvarphi_3(s,z) ds \\
  &= - i  \int_r^{\tr_0}    F_1^{+}(r) (D_0^{+})^{-t} F_2^{+}(s)^t \sigma_1 \xi  \upvarphi_3(s,z) ds  \\
        & + i \int_r^{\tr_0}   F_2^{+}(r) (D_0^{+})^{-1} F_1^{+}(s)^t \sigma_1 \xi  \upvarphi_3(s,z) ds \\
 & =-  i  \int_{r}^{\tr_0}  \left(  d_1 \varphi_1^{+}(r)   (\varphi_3^{+}(s))^{t}  +  d_2 \varphi_2^{+}(r)   (\varphi_4^{+} (s)) ^{t} \right)  \sigma_1 \xi  \upvarphi_3(s,z) ds \\
&+  i  \int_{r}^{\tr_0}      \left(  d_1 \varphi_3^{+}(r)   (\varphi_1^{+}(s))^{t} +  d_2 \varphi_4^{+}(r)   (\varphi_2^{+}(s)) ^{t}  \right)  \sigma_1\xi  \upvarphi_3(s,z) ds 
\end{align*}

Using the asymptotics of $\varphi_j^{+}(r)$ near $0,$  we obtain 
\begin{align*}
 \Gamma_3^{+}(r,z, \upvarphi_3(r,z))   & \lesssim  r^{\frac{3}{2}} \left|  \int_r^{\tr_0}   s^{- \frac{1}{2}}  \xi  \upvarphi_3(s,z) ds   \right|   + r^{-\frac{1}{2}} \left| \int_r^{\tr_0}   s^{\frac{3}{2}}  \sigma_1 \xi  \upvarphi_3(s,z) ds \right| \\
& \lesssim r^{\frac{3}{2}} |\log(\frac{\tr_0}{r})| \xi  \sup_{0 < r \leq \tr_0 } | r^{\frac{1}{2}}  \upvarphi_3(r,z) |   + r^{-\frac{1}{2}} \tr_0^2 \xi  \sup_{0 < r \leq \tr_0 } | r^{\frac{1}{2}}  \upvarphi_3(r,z) |
\end{align*}
Thus, 
\begin{align*}
   \sup_{0 < r \leq \tr_0 }   | r^{\frac{1}{2}}   \Gamma_3^{+}(r,z,\upvarphi_3 (r,z))  | & \lesssim  \tr_0^2 \xi    \sup_{0 < r \leq \tr_0 } | r^{\frac{1}{2}}  \upvarphi_3(r,z) |
\end{align*}

By \eqref{Gamma_2-varphi-2-xi-large}, we have 
\begin{align*}
    \Gamma_2^{+}(r,z, \upvarphi_4(r,z)) & \lesssim  r^{\frac{3}{2}} \left|  \int_r^{\tr_0}   s^{- \frac{1}{2}}  \xi  \upvarphi_4(s,z) ds   \right|  + r^{-\frac{1}{2}} \left| \int_0^r   s^{\frac{3}{2}}  \sigma_1 \xi  \upvarphi_4(s,z) ds \right| \\
    & \lesssim r^{\frac{3}{2}} |\log(\frac{\tr_0}{r})| \xi  \sup_{0 < r \leq \tr_0 } | r^{\frac{1}{2}}  \upvarphi_4(r,z) |   + r^{-\frac{1}{2}} \tr_0^2 \xi  \sup_{0 < r \leq \tr_0 } | r^{\frac{1}{2}}  \upvarphi_4(r,z) |
\end{align*}
Thus, 
\begin{align*}
   \sup_{0 < r \leq \tr_0 }   | r^{\frac{1}{2}}   \Gamma_2^{+}(r,z,\upvarphi_4 (r,z))  | & \lesssim  \tr_0^2 \xi    \sup_{0 < r \leq \tr_0 } | r^{\frac{1}{2}}  \upvarphi_4(r,z) |
\end{align*}

\end{proof} 
 Using the estimates \eqref{eq:Gamma1-phi1-large-xi} and \eqref{eq:Gamma2-phi2-large-xi}, we deduce that $\mathcal{T}_1^{+}$ is a contraction on $\mathcal{Y}_1.$ Similarly, using estimates \eqref{eq:Gamma3-phi3-large-xi} and \eqref{eq:Gamma2-phi4-large-xi}, we deduce that $\mathcal{T}_2^{+}$ is a contraction on $\mathcal{Y}_2.$ The estimates for the derivatives can be obtained using the similar argument with conditions \eqref{eq:S_1T_1condition},\eqref{eq:S_2T_2condition} and \eqref{eq:S_3T_3condition}. This concludes the proof of Proposition \ref{prop:contraction-large-xi}.
\end{proof}

\section{Solution to $i\mathcal{L}-z$ near $0$ in the resonance case}
\label{sec:Sol-near-0-resonance}
In this section, we start the construction of the distorted Fourier transform associated with the linearized operator $\mathcal{L},$ in the presence of resonance. The main goal is to construct the fundamental matrix solutions $F_1^{R}(\cdot,z)$ and $F_2^{R}(\cdot, z)$ to 
\begin{equation}
\label{eq:iLF=zF-resonance}
   i \mathcal{L} F^{R}(\cdot, z) = z F^{R}(\cdot, z),  
\end{equation}
for small and large $|z|.$ We denote by $F_1^{R}(\cdot,z)$ a solution branch that is the $L^2$  near $r=0.$  \\k

We define the Green's functions exactly as \eqref{Greens-function-F_j} using $F_1^{R,\pm}(r)$ and $F_2^{R,\pm}(r):$

\begin{align}
\label{Greens-function-F_j-resonance}
\begin{split}
   \mathcal{G}_1^{R,+}(r,s)&:= F_1^{R,+}(r) S_1^{R,+}(s) \mathbb{1}_{\{ 0\leq s \leq r\} } + F_2^{R,+}(r) T_1^{R,+}(s) \mathbb{1}_{  \{ 0\leq s \leq r \} }  \\
     \mathcal{G}_2^{R,+}(r,s)&:= F_1^{R,+}(r) S_2^{R,+}(s) \mathbb{1}_{\{ r\leq s \leq R_0 \} } + F_2^{R,+}(r) T_2^{R,+}(s) \mathbb{1}_{  \{ 0\leq s \leq r \} } \\
    \mathcal{G}_3^{R,+}(r,s)&:= F_1^{R,+}(r) S_3^{R,+}(s) \mathbb{1}_{\{ r\leq s \leq R_0 \} } + F_2^{R,+}(r) T_3^{R,+}(s) \mathbb{1}_{  \{ r\leq s \leq R_0 \} }        
\end{split}
\end{align}
where for $i=1,2,3$ we require the matrices $S_i^{R,+}(r)$ and $T_i^{R,+}(r),$ for $\mathcal{G}_i^{R,+}(r,s)$   to satisfy the following conditions 

\begin{align}
\label{eq:S_1T_1condition-resonance}
\begin{pmatrix}
F_1^{R,+}(r) &     F_2^{R,+}(r)\\
  \partial_r F_1^{R,+}(r) & \partial_r F_2^{R,+}(r) 
\end{pmatrix} 
\begin{pmatrix}
    S_1^{R,+}(r) \\ T_1^{R,+}(r)
\end{pmatrix}
&= \begin{pmatrix}
    0 \\ \sigma_2
\end{pmatrix} , \\ 
\label{eq:S_2T_2condition-resonance}
 \begin{pmatrix}
F_1^{R,+}(r) &     F_2^{R,+}(r)\\
  \partial_r F_1^{R,+}(r) & \partial_r F_2^{R,+}(r) 
\end{pmatrix} 
\begin{pmatrix}
   - S_2^{R,+}(r) \\ T_2^{R,+}(r)
\end{pmatrix}
&= \begin{pmatrix}
    0 \\ \sigma_2
\end{pmatrix}, \\ 
\label{eq:S_3T_3condition-resonance}
 \begin{pmatrix}
F_1^{R,+}(r) &     F_2^{R,+}(r)\\
  \partial_r F_1^{R,+}(r) & \partial_r F_2^{R,+}(r) 
\end{pmatrix} 
\begin{pmatrix}
   - S_3^{R,+}(r) \\- T_3^{R,+}(r)
\end{pmatrix}
&= \begin{pmatrix}
    0 \\ \sigma_2
\end{pmatrix}.
\end{align}

Denote by \begin{align*}
   F_1^{R}(r,z)= \begin{bmatrix}
       \upvarphi_1^{R}(r,z) &  \upvarphi_2^{R}(r,z)
   \end{bmatrix} , \qquad \text{and} \qquad  F_2^{R}(r,z)= \begin{bmatrix}
       \upvarphi_3^{R}(r,z) &  \upvarphi_4^{R}(r,z)
   \end{bmatrix}  
\end{align*} 
and  
\begin{align*}
 \Gamma_i^{R,+}(r,z,\upvarphi_j^{R} (r,z))&:= \int_0^\infty \mathcal{G}^{R,+}_i(r,s) \, z \upvarphi_j^{R} (s,z)  ds  \qquad \text{for } i=1,2, 3 \text{ and } j=1,2.
\end{align*}
where $\upvarphi_j^{R}(r,z) $ is the $j$-th column of the matrix $F^{R}(r,z).$ \\

Then a solution to \eqref{eq:iLF=zF-resonance} for $z$ near $\frac{\sqrt{17}}{8}$ and $-\frac{\sqrt{17}}{8}$ is given by the following two scenarios:\\

\textbf{Scenario I:} $z$ near $\frac{\sqrt{17}}{8}$: Let $z= \frac{\sqrt{17}}{8}+\xi,$ for small $\xi$  we have 
\begin{align*}
   F_1^{R}(r,z)&:=  C_1  F_1^{R,+} (r) + \mathcal{T}_1^{R,+}(r,z, F_1^{R}(r,z))  \\
   F_2^{R}(r,z)&:=  C_2  F_2^{R,+} (r) + \mathcal{T}_2^{R,+}(r,z, F_2^{R}(r,z))  
\end{align*}

where, $F_j^{R,+}(r)$ are solutions to $ i \mathcal{L}F_j^{R,+}(r)=\frac{\sqrt{17}}{8}F_j^{R,+}(r) $ and 
\begin{align}
\label{eq:def-T_1^R}
\begin{split}
 \mathcal{T}_1^{R,+}(r,z, F_1^{R}(r,z))  &:= \Gamma_1^{R,+}(r,z, \upvarphi_1^{R}(r,z))+ \Gamma_2^{R,+}(r,z, \upvarphi_2^{R}(r,z))  \\
  & =\int_0^{\infty} \mathcal{G}_1^{R,+}(r,s) \, \xi  \upvarphi_1^{R}(s,z) ds  + \int_0^{\infty} \mathcal{G}_2^{R,+}(r,s)  \, \xi  \upvarphi_2^{R}(s,z) ds  
\end{split}
\end{align}
\begin{align}
\label{eq:def-T_2^R}
\begin{split}
  \mathcal{T}_2^{R,+}(r,z, F_2^{R}(r,z))  &:= \Gamma_3^{R,+}(r,z,\upvarphi_3 (r,z))+ \Gamma_2^{R,+}(r,z,\upvarphi_4^{R} (r,z))  \\
  & =\int_0^{\infty} \mathcal{G}_3^{R,+}(r,s)  \, \xi  \upvarphi_3^{R}(s,z) ds  + \int_0^{\infty} \mathcal{G}_2^{R,+}(r,s)  \, \xi  \upvarphi_4^{R}(s,z) ds  
\end{split}
\end{align}
\textbf{Scenario II:} $z$ near $\frac{-\sqrt{17}}{8}$: Let $z= -\frac{\sqrt{17}}{8}+\xi,$ for small $\xi$ we define, 
\begin{align*}
   F_1^{R}(r,z)&:=  -\sigma_3 F_1^{R}(r,-z) \sigma_3  \\
   F_2^{R}(r,z)&:=  -\sigma_3 F_2^{R}(r,-z) \sigma_3 
\end{align*}
Therefore,  $\upvarphi_j^{R}(r,z)$ satisfy 
\begin{align} \label{eq:sys-varphi_j-resonance}
\begin{split}
    \upvarphi_1^{R}(r,z)&=-\sigma_3 \upvarphi_1^{R}(r,-z), \; \;
    \qquad \upvarphi_3^{R}(r,z)=-\sigma_3 \upvarphi_3^{R}(r,-z), \\
   \upvarphi_2^{R} (r,z)&=\sigma_3 \upvarphi_2^{R}(r,-z), \qquad \quad \;  \upvarphi_4^{R}(r)=\sigma_3 \varphi_4^{R}(r,-z). 
   \end{split}
\end{align}

\begin{lemma}
\label{D_0-resonance}
    Let $D_0^{\pm}=\mathcal{W}(F_1^{R,\pm}(\cdot),F_2^{R,\pm}(\cdot)).$ Then, we have  where 
\begin{align*}D_0^{\pm}=
\begin{pmatrix}
      2c_1^2 c_3^2 & 0 \\ 
      0  & 2c_2^1 c_4^1
   \end{pmatrix}
\qquad \text{and} \qquad 
   (D_0^{\pm})^{-1}  =(D_0^{\pm})^{-t}= \begin{pmatrix}
 \frac{1}{2c_1^2 c_3^2}      &  0 \\ 
0       & \frac{1}{2c_2^1 c_4^1}
   \end{pmatrix}
  : =  \begin{pmatrix}
d_1    &  0 \\ 
0      &  d_2
   \end{pmatrix}.
\end{align*}
Moreover, we have
\begin{align*}
\mathcal{G}_1^{R,+}(r,s):&= \begin{cases}
    i   F_1^{R,+}(r) (D_0^{+})^{-t} F_2^{R,+}(s)^t \sigma_1, \qquad \quad 0 < s \leq r, \\ 
  - i  F_2^{R,+}(r) (D_0^{+})^{-1} F_1^{R,+}(s) ^t \sigma_1 , \qquad 0 < s \leq r.
\end{cases} \\ \\
\mathcal{G}_2^{R,+}(r,s):&= \begin{cases}
  -  i   F_1^{R,+}(r) (D_0^{+})^{-t} F_2^{R,+}(s)^t \sigma_1, \qquad r \leq s \leq R_0, \\ 
  - i  F_2^{R,+}(r) (D_0^{+})^{-1} F_1^{R,+}(s) ^t \sigma_1 , \qquad 0 < s \leq r.
\end{cases} \\ \\
\mathcal{G}_3^{R,+}(r,s):&= \begin{cases}
   - i   F_1^{R,+}(r) (D_0^{+})^{-t} F_2^{R,+}(s)^t \sigma_1, \qquad r \leq s \leq R_0, \\ 
   i  F_2^{R,+}(r) (D_0^{+})^{-1} F_1^{R,+}(s) ^t \sigma_1 , \qquad \;
   \; r \leq s \leq R_0.
\end{cases}
\end{align*}
\end{lemma}
\begin{proof}
Let
$  D^{\pm}= \begin{pmatrix}
       \delta^{\pm} & \gamma^{\pm} \\
       \beta^{\pm} & \alpha^{\pm}
   \end{pmatrix} $. Since the Wronskians are independent of $r$, they can be evaluated using the asymptotic behavior of $\varphi_j^{R,\pm}(r)$. In particular, the coefficients $\delta^{\pm}$, $\gamma^{\pm}$, and $\alpha^{\pm}$ can be computed in exactly the same way as in Lemma~\ref{D_0} for the non-resonance case, based on the asymptotic behavior near $r = 0$, which is similar for both $\varphi_j^{\pm}(r)$ and $\varphi_j^{R,\pm}(r)$. Therefore, we will only include the computation of Wronskians involving the asymptotic behavior as $r \to \infty$.

\begin{align*}
    \beta^{\pm} &= W(\varphi_2^{\pm}, \varphi_3^{\pm}) = W\left( \begin{pmatrix}
           \tc_2^1   \\ 
          \tc_2^2  
         \end{pmatrix} (1+O(r^{-1})),  \begin{pmatrix}
          \tc_3^1 e^{-\frac{3}{\sqrt{2}}r}\\
           \tc_3^2 e^{-\frac{3}{\sqrt{2}}r}
    \end{pmatrix} (1 + O(r^{-1})) \right)=0. 
\end{align*}
 Moreover, the expression for the Green's function can be obtained in the same manner as in the proof of Lemma~\ref{lem:def_S-T-near 0}.
\end{proof}

\subsection{Small $|\xi|$}
Similarly to the resonance case, we show that the integral equations, defined in \eqref{eq:defF_1(r,z)-scenarioI} and \eqref{eq:defF_1(r,z)-scenarioI}, are contractions on the following space  $X_j^R$

\begin{defi}
\label{def:SpaceX_j-resonnance}
Let $R_0>0$, and $F=\begin{bmatrix}
 F^{(1)}&   F^{(2)}
\end{bmatrix} , $ where $F^{(i)}$ denote the columns of matrix $F.$ We define the following norms:
\begin{align*}
   \left\| F \right\|_{\mathcal{X}_1^R} &:= \| F^{(1)} \|_{X_1^R}  +  \| F^{(2)} \|_{X_2^R} \quad \text{and} \quad \left\| F \right\|_{\mathcal{X}_1^{\prime,R}} := \| F^{(1)} \|_{X_1^{\prime,R}}  +  \| F^{(2)} \|_{X_2^{\prime,R}} \\
     \left\| F \right\|_{\mathcal{X}_2^R} &:= \| F^{(1)} \|_{X_3^R}  +  \| F^{(2)} \|_{X_4^R} \quad \text{and} \quad  \left\| F \right\|_{\mathcal{X}_2^{\prime,R}} := \| F^{(1)} \|_{X_3^{\prime,R}}  +  \| F^{(2)} \|_{X_4^{\prime,R}} 
\end{align*} 
where
\begin{align*}
  \| F^{(1)} \|_{X_1^R} &:=  \sup_{0 < r < 1}    | r^{-\frac{3}{2}} F^{(1)} (r) |   + \sup_{1 \leq r\leq R_{0}} | e^{-\frac{3}{\sqrt{2}}r} F^{(1)}(r) | ,  \\
  \| F^{(2)} \|_{X_2^R} &:=  \sup_{0 < r < 1}    | r^{-\frac{3}{2}} F^{(2)} (r) |   + \sup_{1 \leq r\leq R_{0}} |  F^{(2)}(r) | , \\
  \| F^{(1)} \|_{X_3^R} &:=  \sup_{0 < r < 1}    | r^{\frac{1}{2}} F^{(1)} (r) |   + \sup_{1 \leq r\leq R_{0}} | e^{\frac{3}{\sqrt{2}}r} F^{(1)}(r) | , \\
  \| F^{(2)} \|_{X_4^R} &:=  \sup_{0 < r < 1}    | r^{\frac{1}{2}} F^{(2)} (r) |   + \sup_{1 \leq r\leq R_{0}} | r^{-1} F^{(2)}(r) |  .
\end{align*}
and 
\begin{align*}
  \| F^{(1)} \|_{X_1^{\prime,R}} &:=  \sup_{0 < r < 1}    | r^{-\frac{1}{2}} F^{(1)} (r) |   + \sup_{1 \leq r\leq R_{0}} | e^{-\frac{3}{\sqrt{2}}r} F^{(1)}(r) | ,  \\
  \| F^{(2)} \|_{X_2^{\prime,R}} &:=  \sup_{0 < r < 1}    | r^{-\frac{1}{2}} F^{(2)} (r) |   + \sup_{1 \leq r\leq R_{0}} | r^{2} F^{(2)}(r) | , \\
  \| F^{(1)} \|_{X_3^{\prime,R}} &:=  \sup_{0 < r < 1}    | r^{\frac{3}{2}} F^{(1)} (r) |   + \sup_{1 \leq r\leq R_{0}} | e^{\frac{3}{\sqrt{2}}r} F^{(1)}(r) | , \\
  \| F^{(2)} \|_{X_4^{\prime,R}} &:=  \sup_{0 < r < 1}    | r^{\frac{3}{2}} F^{(2)} (r) |   + \sup_{1 \leq r\leq R_{0}} | F^{(2)}(r) |  .
\end{align*}
\end{defi}

\begin{prop}
\label{prop:contraction-small-xi-resonance}
Let $0<|\xi| \leq \delta_0 < 1 ,$ and $R_0 :=\frac{\varepsilon_0}{\sqrt{|\xi|}},$ where $0<\delta_0 \ll \varepsilon_0 \ll1.$ Then $\mathcal{T}^{R,+}_j$ is a contraction on $\mathcal{X}_j^{R},$ for $j=1,2,$  i.e., $$ \|\mathcal{T}^{R,+}_j(F_j^R(\cdot,z)) \|_{\mathcal{X}_j^{R}} \lesssim \varepsilon_0 \left\| F_j^R(\cdot,z) \right\|_{\mathcal{X}_j^{R}} . $$
Moreover, we have $$ \| \partial_r \mathcal{T}^{R,+}_j(F_j^{R}(\cdot,z)) \|_{\mathcal{X}_j^{\prime,R}} \lesssim \varepsilon_0 \left\|  \partial_r F_j^{R}(\cdot,z) \right\|_{\mathcal{X}_j^{\prime,R}} . $$

\end{prop}
\begin{proof}
We will only prove that $\mathcal{T}^{R,+}_j$ is a contraction on $\mathcal{X}_j^{R},$ and the derivative can be obtained using similar estimates and required conditions \eqref{eq:S_1T_1condition-resonance},\eqref{eq:S_2T_2condition-resonance} and \eqref{eq:S_3T_3condition-resonance}.  In view of definition $ \mathcal{T}_1^{R,+}$ in  \eqref{eq:def-T_1^R} and   $ \mathcal{T}_2^{R,+}$ in \eqref{eq:def-T_2^R}, we only need to prove the estimates for $\Gamma_j^{R,+}.$ This is achieved using the asymptotic behavior of $F_1^{R,+}(r)$ and $F_2^{R,+}(r)$, therefore we divide the analysis into two regions: the small-$r$ regime ($r < 1$) and the large-$r$ regime ($r \geq 1$).

\begin{claim}
\label{Claim:Gamma_1-2-3-small-r-resonance}
    For small $r<1,$ we have 
\begin{align*}
   \sup_{0 < r < 1}    | r^{-\frac{3}{2}} \Gamma_1^{R,+}(r,z, \upvarphi_1^{R}(r,z)) |
& \lesssim \xi \sup_{0 < r < 1}    | r^{-\frac{3}{2}}  \upvarphi_1^{R}(r,z)) | \\
   \sup_{0 < r < 1} | r^{-\frac{3}{2}}   \Gamma_2^{R,+}(r,z, \upvarphi_2^{R}(r,z)) |  &\lesssim \xi \sup_{0 < r < 1} | r^{-\frac{3}{2}}    \upvarphi_2^{R}(r,z) | + R_{0}^{2} \, \xi \sup_{1 \leq r \leq R_0 } |     \upvarphi_2^{R}(r,z) | \\
 \sup_{0 < r\leq  1 } | r^{\frac{1}{2}}   \Gamma_3^{R,+}(r,z,\upvarphi_3 ^{R}(r,z))  | & \lesssim   \xi  \sup_{0 < r < 1} | r^{\frac{1}{2}}  \upvarphi_3^{R}(r,z) | + \xi  R_0  \sup_{1 \leq r \leq R_0 } | e^{\frac{3}{\sqrt{2}}r}  \upvarphi_3^{R}(r,z) |   \\
   \sup_{0 < r < 1} | r^{\frac{1}{2}}   \Gamma_2^{R,+}(r,z, \upvarphi_4^{R}(r,z))  |  &\lesssim   \xi  \sup_{0 < r < 1} | r^{\frac{1}{2}}    \upvarphi_4^{R}(r,z) |+  R_0 \,  \xi    \sup_{1 \leq r \leq R_0 } |     r^{-1} \upvarphi_4^{R}(r,z) |.
\end{align*}
\end{claim}

\begin{proof}
Note that for small $r$ the behavior  of $\upvarphi_j^{R,+}$ is the same as $\upvarphi_j^{+}$ in the non-resonance case. Therefore, the proof is exactly the same as the proof of Claim \ref{Claim:Gamma_1-and-2-small-r-non-resonance} and \ref{Claim:Gamma_3-and-2-small-r-non-resonance}, and we omit the details. 
\end{proof}

\begin{claim}
\label{Claim:Gamma_1-and-2-large-r-resonance}
    For large $r>1,$ we have 
\begin{align*}
  \sup_{1 \leq r\leq R_{0}} | e^{-\frac{3}{\sqrt{2}}r}  \Gamma_1^{R,+}(r,z, \upvarphi_1^{R}(r,z)) | &\lesssim R_0 \xi \sup_{0< r \leq 1 } | r^{-\frac{3}{2}} \upvarphi_1^{R}(r,z) | + R_0 \xi \sup_{1 \leq r\leq R_{0}} |e^{-\frac{3}{\sqrt{2}}r}  \upvarphi_1^{R}(r,z)|   \\
\sup_{1 \leq r \leq R_0 }  |       \Gamma_2^{R,+}(r,z, \upvarphi_2^{R}(r,z)) | &\lesssim  \xi    \sup_{0 < r < 1} |r^{\frac{3}{2}}  \upvarphi_2^{R}(r,z)| +  R_0^2 \xi \sup_{1 \leq r \leq R_0 } | \upvarphi_2^{R}(r,z)|
\end{align*}
\end{claim}
\begin{proof}

We have 
    \begin{align*}
        \Gamma_1^{R,+}(r,z, \upvarphi_1^{R}(r,z))&=\int_0^{\infty} \mathcal{G}_1^{R,+}(r,s) \xi  \upvarphi_1^{R}(s,z) ds \\
&= i \int_0^r   \left(  d_1 \varphi_1^{R,+}(r)   (\varphi_3^{R,+}(s))^{t}  +  d_2 \varphi_2^{R,+}(r)   (\varphi_4^{R,+} (s)) ^{t} \right)  \sigma_1 \xi  \upvarphi_1^{R}(s,z) ds \\
&- i \int_0^r    \left(  d_1 \varphi_3^{R,+}(r)   (\varphi_1^{R,+}(s))^{t} +  d_2 \varphi_4^{R,+}(r)   (\varphi_2^{R,+}(s)) ^{t}  \right)  \sigma_1\xi  \upvarphi_1^{R}(s,z) ds \\
&= i \int_0^1   \left(  d_1 \varphi_1^{R,+}(r)   (\varphi_3^{R,+}(s))^{t}  +  d_2 \varphi_2^{R,+}(r)   (\varphi_4^{R,+} (s)) ^{t} \right)  \sigma_1 \xi  \upvarphi_1^{R}(s,z) ds \\
&+i \int_1^r   \left(  d_1 \varphi_1^{R,+}(r)   (\varphi_3^{R,+}(s))^{t}  +  d_2 \varphi_2^{R,+}(r)   (\varphi_4^{R,+} (s)) ^{t} \right)  \sigma_1 \xi  \upvarphi_1^{R}(s,z) ds \\
&- i \int_0^1    \left(  d_1 \varphi_3^{R,+}(r)   (\varphi_1^{R,+}(s))^{t} +  d_2 \varphi_4^{R,+}(r)   (\varphi_2^{R,+}(s)) ^{t}  \right)  \sigma_1\xi  \upvarphi_1^{R}(s,z) ds \\
&- i \int_1^r    \left(  d_1 \varphi_3^{R,+}(r)   (\varphi_1^{R,+}(s))^{t} +  d_2 \varphi_4^{R,+}(r)   (\varphi_2^{R,+}(s)) ^{t}  \right)  \sigma_1\xi  \upvarphi_1^{R}(s,z) ds \\
\end{align*}

Using the asymptotics of $\varphi_j^{R,+}(r)$ near $0,$ and for $r \geq 1$, we obtain 
 \begin{align*}
        \Gamma_1^{+}(r,z, \upvarphi_1^{R}(r,z))& \lesssim  e^{\frac{3}{\sqrt{2}}r} \left| \int_0^1  s^{-\frac{1}{2}}     \xi      \upvarphi_1^{R}(s,z) ds \right|  +   \left| \int_0^1  s^{-\frac{1}{2}}  \xi  \upvarphi_1^{R}(s,z) ds \right|   \\ 
 & + e^{-\frac{3}{\sqrt{2}}r} \left|  \int_0^1 s^{\frac{3}{2}}   \xi      \upvarphi_1^{R}(s,z) ds \right| + r \left|   \int_0^1  s^{\frac{3}{2}} \;  \xi  \upvarphi_1^{R}(s,z) ds \right|   \\
 &+e^{\frac{3}{\sqrt{2}}r} \left|  \int_1^r      e^{-\frac{3}{\sqrt{2}}s}  \xi     \upvarphi_1^{R}(s,z) ds \right| +  \left|  \int_1^r  s \, \xi  \upvarphi_1^{R}(s,z) ds \right|  \\ 
 & + e^{-\frac{3}{\sqrt{2}}r}  \left| \int_1^r  e^{\frac{3}{\sqrt{2}}s}  \xi      \upvarphi_1^{R}(s,z) ds  \right| + r \left|   \int_1^r    \xi  \upvarphi_1^{R}(s,z) ds \right| \\
 & \lesssim  e^{\frac{3}{\sqrt{2}}r} \xi \sup_{0< r \leq 1 } | r^{-\frac{3}{2}} \upvarphi_1^{R}(r,z) | +  \xi \sup_{0< r \leq 1 } | r^{-\frac{3}{2}} \upvarphi_1^{R}(r,z) | \\
& +e^{-\frac{3}{\sqrt{2}}r} \xi \sup_{0< r \leq 1 } | r^{-\frac{3}{2}} \upvarphi_1^{R}(r,z) | +  \xi r \sup_{0< r \leq 1 } | r^{-\frac{3}{2}} \upvarphi_1^{R}(r,z) | \\
& + e^{\frac{3}{\sqrt{2}}r} \xi \sup_{1 \leq r\leq R_{0}} |e^{-\frac{3}{\sqrt{2}}r}  \upvarphi_1^{R}(r,z)| + r e^{\frac{3}{\sqrt{2}}r} \xi \sup_{1 \leq r\leq R_{0}} |e^{-\frac{3}{\sqrt{2}}r}  \upvarphi_1^{R}(r,z)| 
\end{align*}

Therefore, 
\begin{align}
     \sup_{1 \leq r\leq R_{0}} | e^{-\frac{3}{\sqrt{2}}r}  \Gamma_1^{+}(r,z, \upvarphi_1^{R}(r,z)) | \lesssim R_0 \xi \sup_{0< r \leq 1 } | r^{-\frac{3}{2}} \upvarphi_1^{R}(r,z) | + R_0 \xi \sup_{1 \leq r\leq R_{0}} |e^{-\frac{3}{\sqrt{2}}r}  \upvarphi_1^{R}(r,z)| 
\end{align}

Next, we estimate $\Gamma_2^{R,+}(r,z, \upvarphi_2^{R}(r,z)).$ Similarly, we obtain 

\begin{align*}
\Gamma_2^{R,+}(r,z, \upvarphi_2^{R}(r,z))&= \int_0^{\infty} \mathcal{G}_2^{R,+}(r,s) 
\xi  \upvarphi_2^{R}(s,z) ds \\
&= - i \int_r^{R_0}   \left(  d_1 \varphi_1^{R,+}(r)   (\varphi_3^{R,+}(s))^{t}  +  d_2 \varphi_2^{R,+}(r)   (\varphi_4^{R,+} (s)) ^{t} \right)  \sigma_1 \xi  \upvarphi_2^{R}(s,z) ds \\
&- i \int_0^r    \left(  d_1 \varphi_3^{R,+}(r)   (\varphi_1^{R,+}(s))^{t} +  d_2 \varphi_4^{R,+}(r)   (\varphi_2^{R,+}(s)) ^{t}  \right)  \sigma_1\xi  \upvarphi_2^{R}(s,z) ds \\
&= - i \int_r^{R_0}   \left(  d_1 \varphi_1^{R,+}(r)   (\varphi_3^{R,+}(s))^{t}  +  d_2 \varphi_2^{R,+}(r)   (\varphi_4^{R,+} (s)) ^{t} \right)  \sigma_1 \xi  \upvarphi_2^{R}(s,z) ds \\
&- i \int_0^1    \left(  d_1 \varphi_3^{R,+}(r)   (\varphi_1^{R,+}(s))^{t} +  d_2 \varphi_4^{R,+}(r)   (\varphi_2^{R,+}(s)) ^{t}  \right)  \sigma_1\xi  \upvarphi_2^{R}(s,z) ds \\
&- i \int_1^r    \left(  d_1 \varphi_3^{R,+}(r)   (\varphi_1^{R,+}(s))^{t} +  d_2 \varphi_4^{R,+}(r)   (\varphi_2^{R,+}(s)) ^{t}  \right)  \sigma_1\xi  \upvarphi_2^{R}(s,z) ds \\
\end{align*}
Using the asymptotics of $\varphi_j^{R,+}(r)$ near $0,$ and for $r \geq 1$, we obtain 
\begin{align*}
  \Gamma_2^{R,+}(r,z, \upvarphi_2^{R}(r,z))& \lesssim    e^{-\frac{3}{\sqrt{2}}r}  
\left|  \int_0^1 s^{\frac{3}{2}}   \xi      \upvarphi_2^{R}(s,z) ds \right|  + r \left|  \int_0^1  s^{\frac{3}{2}} \;  \xi  \upvarphi_2^{R}(s,z) ds \right|  \\
 &+ e^{\frac{3}{\sqrt{2}}r} \left|   \int_r^{R_0}      e^{-\frac{3}{\sqrt{2}}s}  \xi      \upvarphi_2^{R}(s,z) ds  \right| +    \left|   \int_r^{R_0}  s \, \xi   \upvarphi_2^{R}(s,z) ds \right|  \\ 
 & +e^{-\frac{3}{\sqrt{2}}r}  \left| \int_1^r  e^{\frac{3}{\sqrt{2}}s}  \xi      \upvarphi_2^{R}(s,z) ds\right|  + r \left|   \int_1^r    \xi  \upvarphi_2^{R}(s,z) ds \right| \\
 & \lesssim  e^{-\frac{3}{\sqrt{2}}r}  \xi    \sup_{0 < r < 1} |r^{\frac{3}{2}}  \upvarphi_2^{R}(r,z)| + r \xi  \sup_{0 < r < 1} |r^{\frac{3}{2}}  \upvarphi_2^{R}(r,z)| \\ 
 & +   \xi \sup_{1 \leq r \leq R_0 } |  \upvarphi_2^{R}(r,z)|  +   R_0^2 \xi \sup_{1 \leq r \leq R_0 } |  \upvarphi_2^{R}(r,z)| \\
& +   \xi \sup_{1 \leq r \leq R_0 } |  \upvarphi_2^{R}(r,z)| + r^2 \xi \sup_{1 \leq r \leq R_0 } |  \upvarphi_2^{R}(r,z)|
\end{align*}

\begin{align}
\sup_{1 \leq r \leq R_0 }  |      \Gamma_2^{R,+}(r,z, \upvarphi_2^{R}(r,z)) | \lesssim  \xi    \sup_{0 < r < 1} |r^{\frac{3}{2}}  \upvarphi_2^{R}(r,z)| +  R_0^2 \xi \sup_{1 \leq r \leq R_0 } | \upvarphi_2^{R}(r,z)|
\end{align}
 \end{proof}

\begin{claim}
\label{Claim:Gamma_3-and-2-large-r-resonance}
    For large $r>1,$ we have  
    \begin{align*}
     \sup_{1 \leq r \leq R_0}  | e^{\frac{3}{\sqrt{2}}r}  \Gamma_3^{R,+}(r,z, \upvarphi_3^{R}(r,z)) | & \lesssim R_0   \xi  \sup_{1 \leq r \leq R_0} |e^{\frac{3}{\sqrt{2}}r}  \upvarphi_3^{R}(r,z)|   \\
  \sup_{1 \leq r \leq R_0 } |r^{-1}  \Gamma_2^{R,+}(r,z, \upvarphi_4^{R}(r,z))| & \lesssim 
   \xi  \sup_{0 < r < 1} |r^{\frac{1} {2}}  \upvarphi_4^{R}(r,z)| +  R_0^2 \xi \sup_{1 \leq r \leq R_0 } | r^{-1}  \upvarphi_4^{R}(r,z)|   
        \end{align*}
\end{claim}
\begin{proof}
We have 
\begin{align*}
    \Gamma_3^{R,+}(r,z, \upvarphi_3^{R}(r,z))&= \int_0^{\infty} \mathcal{G}_3^{R,+}(r,s) 
\, \xi  \upvarphi_3^{R}(s,z) ds \\ 
 &=  - i  \int_r^{R_0}  \left(  d_1 \varphi_1^{R,+}(r)   (\varphi_3^{R,+}(s))^{t}  +  d_2 \varphi_2^{R,+}(r)   (\varphi_4^{R,+} (s)) ^{t} \right)  \sigma_1 \xi  \upvarphi_3^{R}(s,z) ds\\
&+  i \int_r^{R_0}     \left(  d_1 \varphi_3^{R,+}(r)   (\varphi_1^{R,+}(s))^{t} +  d_2 \varphi_4^{R,+}(r)   (\varphi_2^{R,+}(s)) ^{t}  \right)  \sigma_1 \xi  \upvarphi_3^{R}(s,z) ds \\
& \lesssim e^{\frac{3}{\sqrt{2}}r}  \left|  \int_r^{R_0}      e^{-\frac{3}{\sqrt{2}}s}  \xi      \upvarphi_3^{R}(s,z) ds \right| +  \left| \int_r^{R_0} s\, \xi  \upvarphi_3^{R}(s,z) ds  \right| \\ 
 & + e^{-\frac{3}{\sqrt{2}}r}  \left| \int_r^{R_0}  e^{\frac{3}{\sqrt{2}}s}  \xi      \upvarphi_3^{R}(s,z) ds \right| + r \left|   \int_r^{R_0}   \xi  \upvarphi_3^{R}(s,z) ds \right| 
\end{align*}

Using the asymptotics of $\varphi_j^{R,+}(r)$ near $0,$ and for $r \geq 1$, we obtain
\begin{align*}
\Gamma_3^{R,+}(r,z, \upvarphi_3^{R}(r,z))   & \lesssim  e^{-\frac{3}{\sqrt{2}}r}  \xi  \sup_{1 \leq r \leq R_0} |e^{\frac{3}{\sqrt{2}}r}  \upvarphi_3^{R}(r,z)| + r e^{-\frac{3}{\sqrt{2}}r}  \xi  \sup_{1 \leq r \leq R_0} |e^{\frac{3}{\sqrt{2}}r}  \upvarphi_3^{R}(r,z)| \\
& +  R_0 e^{-\frac{3}{\sqrt{2}}r}  \xi  \sup_{1 \leq r \leq R_0} |e^{\frac{3}{\sqrt{2}}r}  \upvarphi_3^{R}(r,z)| + r e^{-\frac{3}{\sqrt{2}}r}  \xi  \sup_{1 \leq r \leq R_0} |e^{\frac{3}{\sqrt{2}}r}  \upvarphi_3^{R}(r,z)|
\end{align*}
Therefore, we have 

\begin{align*}
    \sup_{1 \leq r \leq R_0}  | e^{\frac{3}{\sqrt{2}}r}  \Gamma_3^{R,+}(r,z, \upvarphi_3^{R}(r,z)) | \lesssim R_0   \xi  \sup_{1 \leq r \leq R_0} |e^{\frac{3}{\sqrt{2}}r}  \upvarphi_3^{R}(r,z)| 
\end{align*}

Next, we estimate $\Gamma_2^{R,+}(r,z, \upvarphi_4^{R}(r,z)).$ Similarly, we have 

\begin{align*}
\Gamma_2^{R,+}(r,z, \upvarphi_4^{R}(r,z))&= \int_0^{\infty} \mathcal{G}_2(r,s) 
\xi  \upvarphi_4^{R}(s,z) ds \\
&= - i \int_r^{R_0}   \left(  d_1 \varphi_1^{R,+}(r)   (\varphi_3^{R,+}(s))^{t}  +  d_2 \varphi_2^{R,+}(r)   (\varphi_4^{R,+} (s)) ^{t} \right)  \sigma_1 \xi  \upvarphi_4^{R}(s,z) ds \\
&- i \int_0^r    \left(  d_1 \varphi_3^{R,+}(r)   (\varphi_1^{R,+}(s))^{t} +  d_2 \varphi_4^{R,+}(r)   (\varphi_2^{R,+}(s)) ^{t}  \right)  \sigma_1\xi  \upvarphi_4^{R}(s,z) ds \\
&= - i \int_r^{R_0}   \left(  d_1 \varphi_1^{R,+}(r)   (\varphi_3^{R,+}(s))^{t}  +  d_2 \varphi_2^{R,+}(r)   (\varphi_4^{R,+} (s)) ^{t} \right)  \sigma_1 \xi  \upvarphi_4^{R}(s,z) ds \\
&- i \int_0^1    \left(  d_1 \varphi_3^{R,+}(r)   (\varphi_1^{R,+}(s))^{t} +  d_2 \varphi_4^{R,+}(r)   (\varphi_2^{R,+}(s)) ^{t}  \right)  \sigma_1\xi  \upvarphi_4^{R}(s,z) ds \\
&- i \int_1^r    \left(  d_1 \varphi_3^{R,+}(r)   (\varphi_1^{R,+}(s))^{t} +  d_2 \varphi_4^{R,+}(r)   (\varphi_2^{R,+}(s)) ^{t}  \right)  \sigma_1\xi  \upvarphi_4^{R}(s,z) ds \\
\end{align*}

Using the asymptotics of $\varphi_j^{R,+}(r)$ near $0,$ and for $r \geq 1$, we obtain

\begin{align*}
 \Gamma_2^{R,+}(r,z, \upvarphi_4^{R}(r,z))& \lesssim  e^{-\frac{3}{\sqrt{2}}r}    \left|  \int_0^1 s^{\frac{3}{2}}   \xi     \upvarphi_4^{R}(s,z) ds  \right| + r \left|  \int_0^1  s^{\frac{3}{2}} \;  \xi  \upvarphi_4^{R}(s,z) ds \right|  \\
 &+e^{\frac{3}{\sqrt{2}}r}   \left|  \int_r^{R_0}      e^{-\frac{3}{\sqrt{2}}s}  \xi   \upvarphi_4^{R}(s,z) ds \right| +   \left|  \int_r^{R_0} s\, \xi   \upvarphi_4^{R}(s,z) ds \right|  \\ 
 & +e^{-\frac{3}{\sqrt{2}}r}  \left|  \int_1^r  e^{\frac{3}{\sqrt{2}}s}  \xi      \upvarphi_4^{R}(s,z) ds \right|+ r \left|  \int_1^r   \xi  \upvarphi_4^{R}(s,z) ds\right|  \\
 & \lesssim  e^{-\frac{3}{\sqrt{2}}r}   \xi  \sup_{0 < r < 1} |r^{\frac{1} {2}}  \upvarphi_4^{R}(r,z)| + r \xi  \sup_{0 < r < 1} |r^{\frac{1} {2}}  \upvarphi_4^{R}(r,z)| \\
 & + r \xi   \sup_{1 \leq r \leq R_0 } | r^{-1} \upvarphi_4^{R}(r,z)|   + R_0^3 \xi   \sup_{1 \leq r \leq R_0 } | r^{-1} \upvarphi_4^{R}(r,z)|  \\
 & + r \xi   \sup_{1 \leq r \leq R_0 } | r^{-1} \upvarphi_4^{R}(r,z)|   + r^2 \xi   \sup_{1 \leq r \leq R_0 } | r^{-1} \upvarphi_4^{R}(r,z)|
\end{align*}
Thus, 
\begin{align*}
    \sup_{1 \leq r \leq R_0 } |r^{-1}  \Gamma_2^{R,+}(r,z, \upvarphi_4^{R}(r,z))| \lesssim 
   \xi  \sup_{0 < r < 1} |r^{\frac{1} {2}}  \upvarphi_4^{R}(r,z)| +  R_0^2 \xi \sup_{1 \leq r \leq R_0 } | r^{-1}  \upvarphi_4^{R}(r,z)|    
\end{align*}
\end{proof}

Finally, using the estimates from Claim \ref{Claim:Gamma_1-2-3-small-r-resonance}, \ref{Claim:Gamma_1-and-2-large-r-resonance} and \ref{Claim:Gamma_3-and-2-large-r-resonance}, we deduce that both $\mathcal{T}_j^{R,+}$ are a contraction on $\mathcal{X}_j^{R},$ for $j=1,2.$ Similarly, we obtain the estimates for the derivatives of $\mathcal{T}_j^{R,+}$ using the conditions \eqref{eq:S_1T_1condition-resonance},\eqref{eq:S_2T_2condition-resonance} and \eqref{eq:S_3T_3condition}. This concludes the proof of Proposition \ref{prop:contraction-small-xi-resonance}.
\end{proof}

\subsection{Large $|\xi|$} Recall that by  large $|\xi|$, we refer to the regime in which the real part is large and the imaginary part is small.\\ 

As in the non-resonance case, we demonstrate that the integral equations \eqref{eq:def-T_1^R} and \eqref{eq:def-T_2^R} admit solutions via a fixed point argument, within the same function spaces $\mathcal{Y}_j$ introduced in Definition~\ref{def:Y_j-space} for the resonance case. Note that the construction of the solutions in the resonance case relies only on the asymptotic behavior of $\varphi_j^{R}$ near $0$, which is analogous to the behavior of $\varphi_j$ in the non-resonance case near the origin. \\

\begin{prop}
\label{prop:contraction-large-xi-resonance}
Let $ |\xi| > \Lambda_0>1 ,$ and $\tR_0 :=\frac{\tvarepsilon_0}{\sqrt{|\xi|}},$ where $0<\delta_0 \ll \tvarepsilon_0 \ll 1.$ Then $\mathcal{T}^{R,+}_j$ is a contraction on $\mathcal{Y}_j,$ for $j=1,2,$  i.e., $$ \|\mathcal{T}^{R,+}_j(F_j^R(\cdot,z)) \|_{\mathcal{Y}_j} \lesssim \varepsilon_0 \left\| F_j^R(\cdot,z) \right\|_{\mathcal{Y}_j} . $$
Moreover, we have $$ \| \partial_r \mathcal{T}^{R,+}_j(F_j(\cdot,z)) \|_{\mathcal{Y}_j^{\prime}} \lesssim \varepsilon_0 \left\|  \partial_r F_j^R(\cdot,z) \right\|_{\mathcal{Y}_j^{\prime}} . $$

\end{prop}
\begin{proof}

In view of the definitions of $\mathcal{T}_j^{R,+}$, we provide only the estimates for $\Gamma_j^{R,+}$, from which one can obtain the corresponding bounds for $\mathcal{T}_j^{R,+}.$

\begin{claim}
    \begin{align}
   \sup_{0 < r \leq \tR_0 } |   r^{-\frac{3}{2}} \Gamma_1^{R,+}(r,z, \upvarphi_1^{R}(r,z)) |  &\lesssim \tR_0^2 \xi \sup_{0 < r \leq \tR_0 } | r^{-\frac{3}{2}}    \upvarphi_1^{R}(r,z) | , \\
   \sup_{0 < r \leq \tR_0 } | r^{-\frac{3}{2}}   \Gamma_2^{R,+}(r,z, \upvarphi_2^{R}(r,z)) |  & \lesssim \tR_0^2 \xi \sup_{0 < r \leq \tR_0 } | r^{-\frac{3}{2}}    \upvarphi_2^{R}(r,z) | , \\
    \sup_{0 < r\leq  \tR_0 } | r^{\frac{1}{2}}   \Gamma_3^{R,+}(r,z,\upvarphi_3 ^{R}(r,z))  | & \lesssim  \tR_0^2 \xi  \sup_{0 < r \leq \tR_0 } | r^{\frac{1}{2}}  \upvarphi_3^{R}(r,z) | \\
    \sup_{0 < r \leq \tR_0 }   | r^{\frac{1}{2}}   \Gamma_2^{R,+}(r,z,\upvarphi_4 ^{R}(r,z))  | & \lesssim  \tR_0^2 \xi    \sup_{0 < r \leq \tR_0 } | r^{\frac{1}{2}}  \upvarphi_4^{R}(r,z) |
\end{align} 
\end{claim}
\begin{proof}
The proof follows exactly as in Claim~\ref{Gamma_j-beahvior-large-xi}, since the behavior of $\upvarphi_j^{R,+}$ near $r = 0$ is identical to that of $\upvarphi_j^{+}$ in the non-resonance case. We omit the details.

\end{proof}

Using the above claim, one can show that $\mathcal{T}^{+}_j$ is a contraction on $\mathcal{Y}_j$. The derivative estimates for $\mathcal{T}^{R,+}_j$ can be obtained using similar arguments, together with the conditions \eqref{eq:S_1T_1condition-resonance}, \eqref{eq:S_2T_2condition-resonance}, and \eqref{eq:S_3T_3condition-resonance}. This concludes the proof of Proposition~\ref{prop:contraction-large-xi-resonance}.

\end{proof}

\section{The $L^2$ solution near $\infty$ for small $|\xi|$}
\label{sec:nearinfty-small-xi}
In this section, we construct the Weyl-Titchmarsh matrix solutions to $
    i \mathcal{L} \Psi^{\pm}(r,z)= z \Psi^{\pm}(r,z)$  for $ z= \pm \frac{\sqrt{17}}{8}+\xi \in \Omega$ with  $|\xi|< \delta_0\ll 1.$  We denote $\Psi^{+}(r,z) \in L^2((1,\infty)),$ when $\im(z)>0,$ and $\Psi^{-}(r,z) \in L^2((1,\infty)),$ when $\im(z)<0.$ Note that, whereas in Sections~\ref{sec:Sol-near-0-non-resonance} and \ref{sec:Sol-near-0-resonance} the symbol $\pm$ was used to distinguish the solutions corresponding to $\pm \re(z)>0$ in the construction near $r=0$, here the notation $\pm$ is instead employed to distinguish between the cases $\pm \im(z)>0$. \\

We seek to construct solutions $\Psi^{\pm}(r,z)$ so that as $r \to \infty$ they have the same asymptotic behavior as the solutions $\Uppsi^{\pm}_{\infty}(r,z) \in L^2((1,\infty))$ (up to normalizations) to 
\begin{align*}
    i \mathcal{L_{\infty}} \Uppsi^{\pm}_{\infty}(r,z) = z \Uppsi^{\pm}_{\infty}(r,z)
\end{align*}
where 
\begin{align*}
    \mathcal{L_{\infty}}:= \begin{pmatrix}
        0 & L_1^{\infty}\\
        -L_2^{\infty} & 0 
    \end{pmatrix}, \qquad  L_1^{\infty}=- \frac{1}{2}\partial_r^2 + \frac{1}{8} , \quad L_2^{\infty}= ( - \frac{1}{2}\partial_r^2 + \frac{17}{8} )
\end{align*}
\subsection{Construction of $\Uppsi^{\pm}_{\infty}(r,z)$ at infinity}
Let $\Phi(r,z)=\begin{pmatrix}
    \phi(r,z) \\
    \psi(r,z)
\end{pmatrix}$ satisfying $i \mathcal{L_{\infty}}  \Phi(r,z)=z \Phi(r,z) ,$ that is,   \begin{align*}\begin{cases}
   i  \, L_1^{\infty} \psi(r,z) &= z \phi(r,z) ,\\
   -i\,  L_2^{\infty}\phi(r,z) &=  z \psi(r,z).    
\end{cases}
\end{align*}
Applying $L_1^{\infty}$ to $ -i L_2^{\infty}\phi(r,z) =  z \, \psi(r,z),$ 
and using the first equation, we obtain the fourth-order equation, 
\begin{align}
\label{eq:L1L2-infty-z-small}
    L_1^{\infty}L_2^{\infty}\phi(r,z) = z^2 \phi(r,z)
\end{align}
and the corresponding choice of $\psi$ is determined by 
\begin{align*}
   \psi(r,z) = \frac{-i }{ z}  L_2^{\infty} \phi(r,z) .
\end{align*}
Plugging in the ansatz $\phi(r,z)=e^{i k(z) r}$ into equation \eqref{eq:L1L2-infty-z-small}, or equivalently into 
\begin{align*}
    \left(  \frac{1}{4} \partial_r^4 - \frac{9}{8} \partial_r^2 +\frac{17}{64}   - z^2 \right)\phi(r,z) =0 ,
\end{align*}
leads to the equation 
\begin{align}
\label{eq:P(k,z)}
P(k,z):&=    \frac{1}{4} k^4 + \frac{9}{8} k^2 +\frac{17}{64}  - z^2 =0, \quad \text{with } \; k \equiv k(z)
\end{align}
We now determine four analytic roots in a suitable domain in the complex plane. Observe that
\begin{equation*}
   \partial_k P(k,z):= k^3+ \frac{9}{4}k= k(k^2+ \frac{9}{4}). 
\end{equation*}
Therefore, $P(k,z)=\partial_k  P(k,z)=0$ for 
\begin{align*}
    (z,k)\in \bigg \{ (\frac{\sqrt{17}}{8},0) ,(-\frac{\sqrt{17}}{8},0), ( i,\frac{3}{2}i), (i,-\frac{3}{2}i), (-i,\frac{3}{2}i), (-i,-\frac{3}{2}i)  \bigg\}
\end{align*}

It follows that by the implicit function theorem that for $z \notin (\frac{\sqrt{17}}{8},-\frac{\sqrt{17}}{8},\frac{3}{2}i, -\frac{3}{2}i )  $  the equation \eqref{eq:P(k,z)}, $P(k,z)=0,$ has four distinct roots, which we denote by $k_j(z),$ for $j=1,\cdots, 4.$  \\

For all  $x\in (-\frac{\sqrt{17}}{8},\frac{\sqrt{17}}{8})$ define
\begin{align*}
    f_1(x) &:= i\sqrt{  \frac{9}{4}- 2  \sqrt{1+x^2}    } \, ,  \\
    f_2(x) &:= i \sqrt{ \frac{9}{4}+ 2  \sqrt{1+x^2}     }\,, \\
     f_3(x) &:= -i  \sqrt{  \frac{9}{4}- 2 \sqrt{1+x^2}   }  \, , \\
      f_4(x) &:= -i \sqrt{ \frac{9}{4}+ 2  \sqrt{1+x^2}     } \,, \\
\end{align*}
where we use the convention that $\sqrt{x}>0,$ for $x>0.$ Then, we have $P(f_j(x),x)=0,$ for all $x \in (-\frac{\sqrt{17}}{8},\frac{\sqrt{17}}{8}),$ and $j=1,\cdots,4.$ \\

Define the simply-connected domain
$$ \Omega:=\C \backslash \left( [i,i \infty) \cup (-i\infty,-i] \cup (-\infty,-\frac{\sqrt{17}}{8}) \cup (\frac{\sqrt{17}}{8}, \infty) \right)  $$

For $1 \leq j \leq 4,$ we denote by $k_j(z)$ the analytic continuation of $f_j(x) $ to $\Omega, $  such that $k_j(x)=f_j(x)$ for $x \in (-\frac{\sqrt{17}}{8},\frac{\sqrt{17}}{8}).$ 
\begin{lemma}
\label{lem:k_j-behavior}
Let $z =x+iy\in \Omega,$ then 
    for $x \in (-\infty,-\frac{\sqrt{17}}{8}),$ we have 
\begin{align*}
  \lim_{y \to 0^{\pm}} k_1(x+iy)&= \mp \sqrt{2}\sqrt{\sqrt{x^2+1}-\frac{9}{8}} , \qquad  \lim_{y \to 0^{\pm}}  k_3(x+iy) = \pm \sqrt{2}\sqrt{\sqrt{x^2+1}-\frac{9}{8}}, \\
  \lim_{y \to 0^{\pm}} k_2(x+iy)&=  i \sqrt{2}{\sqrt{\sqrt{x^2+1} +\frac{9}{8}}}  , \qquad \; \quad  \lim_{y \to 0^{\pm}} k_4(x+iy)= - i \sqrt{2}{\sqrt{\sqrt{x^2+1} +\frac{9}{8}}}
\end{align*}
and for $x \in (\frac{\sqrt{17}}{8},\infty),$ we have 
\begin{align*}
  \lim_{y \to 0^{\pm}} k_1(x+iy)&= \pm \sqrt{2}\sqrt{\sqrt{x^2+1}-\frac{9}{8}} , \qquad  \lim_{y \to 0^{\pm}}  k_3(x+iy) = \mp \sqrt{2}\sqrt{\sqrt{x^2+1}-\frac{9}{8}}, \\
  \lim_{y \to 0^{\pm}} k_2(x+iy)&=  i \sqrt{2}{\sqrt{\sqrt{x^2+1} +\frac{9}{8}}}  , \qquad \; \quad  \lim_{y \to 0^{\pm}} k_4(x+iy)= - i\sqrt{2} {\sqrt{\sqrt{x^2+1} +\frac{9}{8}}}
\end{align*}
Moreover, we have
  \begin{align} \label{eq:Imk13Sign}
     \im(k_1(z)) > 0 \quad \text{and} \quad \im(k_2(z)) > 0 \quad \text{for all }\; z \in \Omega , \text{ with } \im (z)> 0  .  
  \end{align}  
Similarly, we have 
  \begin{align}  \label{eq:Imk32Sign}
      \im(k_1(z)) > 0  \quad \text{and} \quad \im(k_2(z)) > 0 \quad \text{for all }\; z \in \Omega , \text{ with } \im (z)< 0  .  
  \end{align}
\end{lemma}
\begin{proof}
Let $k_1(z):\Omega \longrightarrow \C $ be the analytic continuation of $f_1(x) $ to $\Omega, $ such that \begin{align*}
    k_1(x):= f_1(x)= i \sqrt{2}   \sqrt{  \frac{9}{8}-   \sqrt{1+x^2}     }= i \sqrt{2} \frac{\sqrt{ \left(\frac{\sqrt{17}}{8}-x \right)}  \sqrt{\left( \frac{\sqrt{17}}{8}+x \right)}   }{\sqrt{ \frac{9}{8}+   \sqrt{1+x^2}   }}  , \quad \forall x \in (-\frac{\sqrt{17}}{8},\frac{\sqrt{17}}{8}).
\end{align*}
Note that, there is no ambiguity in defining $k_1(x)$ for $x \in (-\frac{\sqrt{17}}{8},\frac{\sqrt{17}}{8}).$ We denote by $ g(z),h(z)$ and $m(z)$ the analytic continuation of $g(x),h(x)$ and $m(x)$ to $\Omega$ such that, 
\begin{align*}
   g(x):=  \sqrt{\frac{\sqrt{17}}{8}-x }  , \quad h(x):=
  \sqrt{ \frac{\sqrt{17}}{8}+x  }, \quad m(x):= \sqrt{ \frac{9}{8}+   \sqrt{1+x^2}   }, \quad  \forall x \in (-\frac{\sqrt{17}}{8},\frac{\sqrt{17}}{8}).
\end{align*}
For $x \in (-\infty,-\frac{\sqrt{17}}{8}),$ one can check that  
\begin{align*}
    \lim_{y \to 0^{\pm}} g(x+iy)= \sqrt{ \frac{\sqrt{17}}{8} -x } , \quad  \lim_{y \to 0^{\pm}} h(x+iy)= \pm i \sqrt{ \frac{-\sqrt{17}}{8} -x } 
\end{align*}
Similarly,  for $x \in (\frac{\sqrt{17}}{8},\infty),$

\begin{align*}
    \lim_{y \to 0^{\pm}} g(x+iy)= \mp i \sqrt{ x-\frac{\sqrt{17}}{8} }  , \quad  \lim_{y \to 0^{\pm}} h(x+iy)=  \sqrt{ \frac{\sqrt{17}}{8} +x } 
\end{align*}
Therefore, for $x \in (-\infty,-\frac{\sqrt{17}}{8}),$
\begin{align*}
  \lim_{y \to 0^{\pm}} g(x+iy)h(x+iy)=   \pm i \sqrt{ x^2-\frac{17}{64}  } 
\end{align*}
and for $x \in (\frac{\sqrt{17}}{8},\infty),$ we have 
\begin{align*}
  \lim_{y \to 0^{\pm}} g(x+iy)h(x+iy)=   \mp i \sqrt{ x^2-\frac{17}{64}  } 
\end{align*}

Next, observe that $m(z): \Tilde{\Omega} \to \C,$ such that $m(x)=\sqrt{ \frac{9}{8}+   \sqrt{1+x^2}   } $ for $ x \in (-\frac{\sqrt{17}}{8},\frac{\sqrt{17}}{8}) $ is well defined, where
\begin{align*}
   \Tilde{\Omega}:= \C \backslash \left([i,i \infty) \cup (-i\infty,-i]  \right) .
\end{align*} 
Indeed, if $\frac{9}{8}+ \sqrt{1+z^2} \notin (-\infty,0]$ then $m(z)$ is well defined on $\Tilde{\Omega}.$ If not, then $\sqrt{1+z^2} = -(\lambda+\frac{9}{8})<0,$ for some $\lambda>0.$ Thus, $\sqrt{1+z^2}$ is real: Let $z=x+iy $ then  $\sqrt{1+z^2}$ is real if and only if $xy=0$ and $1+x^2-y^2>0.$ If $y=0,$ then $\sqrt{1+z^2}=\sqrt{1+x^2}>0,$  and if $x=0$ then $\sqrt{1+z^2}=\sqrt{1-y^2}>0,$ which is a contradiction in both cases. Therefore $m(z)$ can be extended analytically to  $\Tilde{\Omega}.$ One can also define $\Tilde{m}(z): \Tilde{\Omega} \to \C,$ such that $\Tilde{m}(x)=\sqrt{ \frac{9}{8}+   \sqrt{1+x^2}   } $ for all $x \in \R,$ which is well defined. Since $m(z)$ and $\Tilde{m}(z)$ agree for $ x \in (-\frac{\sqrt{17}}{8},\frac{\sqrt{17}}{8}),$ they must agree on $\Tilde{\Omega}.$ Therefore, for $x \in (-\infty,-\frac{\sqrt{17}}{8}),$ we have 
\begin{align*}
  \lim_{y \to 0^{\pm}} k_1(x+iy)=  i \sqrt{2} \frac{ \pm i \sqrt{ x^2-\frac{17}{64}  }}{\sqrt{\sqrt{x^2+1} +\frac{9}{8}}} =\mp \sqrt{2}\sqrt{\sqrt{x^2+1}-\frac{9}{8}} ,
\end{align*}
and for $x \in (\frac{\sqrt{17}}{8},\infty),$ we have 
\begin{align*}
  \lim_{y \to 0^{\pm}} k_1(x+iy)=  i \sqrt{2} \frac{ \mp i \sqrt{ x^2-\frac{17}{64}  }}{\sqrt{\sqrt{x^2+1} +\frac{9}{8}}} =\pm \sqrt{2}\sqrt{\sqrt{x^2+1}-\frac{9}{8}}.
\end{align*}

Next, observe that $f_3(x)=-f_1(x),$ then similarly to $k_1(z),$ one can obtain
for $x \in (-\infty,-\frac{\sqrt{17}}{8})$  
\begin{align*}
  \lim_{y \to 0^{\pm}} k_3(x+iy) =\pm \sqrt{2}\sqrt{\sqrt{x^2+1}-\frac{9}{8}} ,
\end{align*}
and for $x \in (\frac{\sqrt{17}}{8},\infty),$ we have 
\begin{align*}
  \lim_{y \to 0^{\pm}} k_3(x+iy)= \mp \sqrt{2}\sqrt{\sqrt{x^2+1}-\frac{9}{8}}.
\end{align*}
Finally, notice that $f_2(x)=i \,  m(x),$ and $f_4(x)=-i \, m(x).$ Therefore,  for $x \in (-\infty,-\frac{\sqrt{17}}{8}) \cup (\frac{\sqrt{17}}{8},\infty),$  $k_2(z)$ and $k_4(z)$ satisfy,
\begin{align*}
  \lim_{y \to 0^{\pm}} k_2(x+iy)=  i {\sqrt{\sqrt{x^2+1} +\frac{9}{8}}} \\
   \lim_{y \to 0^{\pm}} k_4(x+iy)=  - i {\sqrt{\sqrt{x^2+1} +\frac{9}{8}}} 
\end{align*}
It remains to prove \eqref{eq:Imk32Sign} and \eqref{eq:Imk13Sign}. We begin with the case of positive imaginary part $\im(z)>0.$  Let $0<y<1,$ then we have
\begin{align*}
    k_1(iy)=i \sqrt{\frac{9}{4} - 2 \sqrt{1-y^2}}, \qquad k_2(iy)=i \sqrt{\frac{9}{4} + 2 \sqrt{1-y^2}}.
    \end{align*}
Therefore, we obtain 
\begin{align*}
   \im(k_1(iy))>0 \quad \text{and} \quad \im(k_2(iy))>0. 
\end{align*}
Assume that there exists $z_0 \in \Omega $ with $\im(z_0)>0$ such that say $\im(k_1(z_0)) =0.$ Then by \eqref{eq:P(k,z)}, we have
\begin{align*}
   z_0^2=\frac{1}{2}  k_1(z_0)^4 + \frac{9}{8} k_1(z_0)^2 +
   \frac{17}{64} >0
\end{align*}
which is impossible since $\im(z_0)>0.$ Therefore $\im(k_1(z))>0$ for all $z \in \Omega,$ with $\im(z)>0. $ Similarly, we get $\im(k_2(z))>0$ for all $z \in \Omega,$ with $\im(z)>0. $ The case of negative imaginary part $\im(z) < 0$ can be handled in the same way. This concludes the proof of Lemma \ref{lem:k_j-behavior}.    
\end{proof}

\begin{remark} \label{rem:behavior-k_j-lambda}
    Notice that $f_1(x)=f_1(-x)=-f_3(-x)=-f_3(x)$ and $f_2(x)=f_2(-x)=-f_4(x)=-f_4(x)$ for $x\in(-\frac{\sqrt{17}}{8},\frac{\sqrt{17}}{8})$. By analytic continuation, $k_1(z)=k_1(-z)=-k_3(-z)=-k_3(z)$ and $k_2(z)=k_2(-z)=-k_4(z)=-k_4(z)$ for all $z\in\Omega$. On the other hand, the lemma shows that $k^\pm_j(x):=\lim_{y\to 0^\pm}k_j(x+iy)$, $x\in\mathbb{R}\backslash (-\frac{\sqrt{17}}{8},\frac{\sqrt{17}}{8})$, $j=1,\dots,4$, satisfy $k^\pm_1(x)=-k_1^\pm(-x)=k_3^\pm(-x)=-k_3^\pm(x)$ and $k_2^\pm(x)=k_2^\pm(-x)=-k_4^\pm(x)=-k_4^\pm(-x)$. 

 Note that once the distorted Fourier transform is constructed, after passing to the real line in the spectral representation of the evolution $e^{t\calL}$ from Lemma~\ref{lem:ST1},  the resulting integral is supported over the region $\mathbb{R}\setminus(-\tfrac{\sqrt{17}}{8},\tfrac{\sqrt{17}}{8})$, which is excluded from $\Omega$. For this reason, the limits $k_j^\pm(x)$ will play a central role in the later parts of this work.
\end{remark}

Recall that, for $\phi(r,z)=e^{i k(z) r }$ to be square integrable near infinity we need $\im(z) >0.$ Then in view of view of Lemma \ref{lem:k_j-behavior}, we obtain the Weyl-Titchmarsh matrix solutions of the reference operator $\mathcal{L}_{\infty}:$  

\begin{align*}
     \Uppsi^{+}_{\infty}(r,z) :=\begin{pmatrix}
         \Uppsi^{+}_{(\infty,1)}(r,z) &  \Uppsi^{+}_{(\infty,2)}(r,z)
    \end{pmatrix} &= \begin{pmatrix}
      e^{i k_1(z) r } &     e^{i k_2(z) r } \\ 
       c_1 (z) e^{i k_1(z) r } &  c_2 (z) e^{i k_2(z) r } 
    \end{pmatrix} \qquad \im(z) > 0  ,
\end{align*}

and 
\begin{align*}
    \Uppsi^{-}_{\infty}(r,z) := \begin{pmatrix}
         \Uppsi^{-}_{(\infty,1)}(r,z) &  \Uppsi^{-}_{(\infty,2)}(r,z)
    \end{pmatrix} =\begin{pmatrix}
      e^{i k_1(z) r } &     e^{i k_2(z) r } \\ 
       c_1(z) e^{i k_1(z) r } & c_2(z)  e^{i k_2(z) r } 
    \end{pmatrix}\qquad \im(z) < 0 ,
\end{align*}
where,  \begin{align*}
   c_j(z)= c(k_j(z),z)&:=     - i     \frac{ ( \frac{1}{2} k_j(z)^2 + \frac{17}{8}) }{ z} \qquad \text{for} \; j=1,2..
\end{align*}
Note that, $c_j(-z)=-c_j(z).$ \\

We also record the two other solutions, 

\begin{align*}
   \UUppsi^{+}_{\infty}(r,z) := \begin{pmatrix}
         \Uppsi^{+}_{(\infty,3)}(r,z) &  \Uppsi^{+}_{(\infty,4)}(r,z)
    \end{pmatrix} =\begin{pmatrix}
      e^{- i k_1(z) r } &     e^{-i k_2(z) r } \\ 
      c_1(z) e^{-i k_1(z) r } & c_2(z) e^{-i k_2(z) r } 
    \end{pmatrix} \qquad \im(z) > 0 ,
\end{align*}

and 
\begin{align*}
     \UUppsi^{-}_{\infty}(r,z) :=\begin{pmatrix}
         \Uppsi^{-}_{(\infty,3)}(r,z) &  \Uppsi^{-}_{(\infty,4)}(r,z)
    \end{pmatrix} = \begin{pmatrix}
      e^{ -i k_1(z) r } &     e^{-i k_2(z) r } \\ 
       c_1(z)  e^{-i k_1(z) r } & c_2(z) e^{-i k_2(z) r } 
    \end{pmatrix}\qquad \im(z) < 0 .
\end{align*}

 \subsection{Construction of $\Psi^{+}(r,z)$ for small $|\xi|$ } 
 We start with the construction of the fundamental matrix solutions $\Psi^{+}(\cdot,z)\in L^2(1,\infty)$ and $\widetilde{\Psi}^{+}(\cdot, z) \notin L^2(1,\infty) $ to 
\begin{equation}
\label{eq:iLPsi=zPsi,Im(z)>0}
   i \mathcal{L} \Psi^{+}(\cdot, z) = z  \Psi^{+}(\cdot, z) ,  \; \text{ for } z \in \Omega \; \text{ and } \; \im(z)>0
\end{equation}
equivalently, 
\begin{equation*}
   i (\mathcal{L_{\infty}}-z ) \Psi^{+}(r, z) = V(r) \Psi^{+}(r, z)
\end{equation*}
where \begin{align*}
    V(r) = \begin{pmatrix}
        0 & \frac{3}{8 \sh(r)^2} + V_1(r) \\
        - ( \frac{3}{8 \sh(r)^2} + V_2(r)) & 0 
    \end{pmatrix}.
\end{align*}
\begin{remark}
Recall that the potentials $ V_1(r) $ and  $ V_2(r)$ decay exponentially. Consequently, for large $r $, we have  
\begin{equation}
\label{eq:decay-V}
 V_j(r) = O(e^{-2r}), \quad \text{for } j=1,2.   
\end{equation}
This specific choice of the exponential decay rate $ e^{-2r} $ is made to align with the decay behavior of the potential $ V_{GP}$, as seen in the asymptotic expansion of $ \rho(r) $ near infinity in \eqref{eq:rho_n-asymptotics-near-infty}.
\end{remark}

We define the Green's functions as 
\begin{align*}
   \mathcal{R}^{+}_1(r,s,z):&=  \Uppsi^{+}_{\infty}(r,z) S_1(s,z) \mathbb{1}_{\{ r_{\infty} \leq s \leq r \}} +  \UUppsi^{+}_{\infty}(r,z) T_1(s,z) \mathbb{1}_{\{ r \leq s \leq \infty \}}   \\
   \mathcal{R}^{+}_2(r,s,z):&=  \Uppsi^{+}_{\infty}(r,z) S_2(s,z) \mathbb{1}_{\{ r \leq s \leq \infty \}} +  \UUppsi^{+}_{\infty}(r,z) T_2(s,z)  \mathbb{1}_{\{ r \leq s \leq \infty \}} \\
   \mathcal{R}^{+}_3(r,s,z):&=  \Uppsi^{+}_{\infty}(r,z) S_3(s,z) \mathbb{1}_{\{ r_{\infty} \leq s \leq r \}} +  \UUppsi^{+}_{\infty}(r,z) T_3(s,z)  \mathbb{1}_{\{ r_{\infty} \leq s \leq r \}} \\
\end{align*}
where we require the matrices $S_1(r,z)$ and $T_1(r,z)$ for $\mathcal{R}^{+}_1$ to satisfy
\begin{align}
\label{eq:S1T1-condi-infinity-xi-small}
\begin{pmatrix}
      \Uppsi^{+}_{\infty}(r,z)    &  \UUppsi^{+}_{\infty}(r,z) \\
      \partial_r   \Uppsi^{+}_{\infty}(r,z) & \partial_r   \UUppsi^{+}_{\infty}(r,z)
\end{pmatrix}
\begin{pmatrix}
    S_1(r,z) \\
    -T_1(r,z)
\end{pmatrix}
= \begin{pmatrix}
    0 \\ \sigma_2
\end{pmatrix}
\end{align}
and we require the matrices $S_2(r,z)$ and $T_2(r,z)$ for $\mathcal{R}^{+}_2$ to satisfy
\begin{align}
\label{eq:S2T2-condi-infinity-xi-small}
\begin{pmatrix}
      \Uppsi^{+}_{\infty}(r,z)    &  \UUppsi^{+}_{\infty}(r,z) \\
      \partial_r   \Uppsi^{+}_{\infty}(r,z) & \partial_r   \UUppsi^{+}_{\infty}(r,z)
\end{pmatrix}
\begin{pmatrix}
  -  S_2(r,z) \\
    -T_2(r,z)
\end{pmatrix}
= \begin{pmatrix}
    0 \\ \sigma_2
\end{pmatrix}
\end{align}
and we require the matrices $S_3(r,z)$ and $T_3(r,z)$ for $\mathcal{R}^{+}_2$ to satisfy
\begin{align}
\label{eq:S3T3-condi-infinity-xi-small}
\begin{pmatrix}
      \Uppsi^{+}_{\infty}(r,z)    &  \UUppsi^{+}_{\infty}(r,z) \\
      \partial_r   \Uppsi^{+}_{\infty}(r,z) & \partial_r   \UUppsi^{+}_{\infty}(r,z)
\end{pmatrix}
\begin{pmatrix}
    S_3(r,z) \\
    T_3(r,z)
\end{pmatrix}
= \begin{pmatrix}
    0 \\ \sigma_2
\end{pmatrix}
\end{align}

Denote by \begin{align*}
\Psi^{\pm}(r,z)= \begin{bmatrix} \Psi^{\pm}_1 (r,z) &    \Psi^{\pm}_2(r,z)
    \end{bmatrix}
    \qquad \text{and} \qquad 
    \PPsi^{\pm}(r,z)= \begin{bmatrix}   \Psi^{\pm}_3 (r,z) & \Psi^{\pm}_4(r,z)  
    \end{bmatrix}
\end{align*} 
and  
\begin{align*}
  \Upsilon_i(r,z,\Psi_j (r,z)):= \int_0^\infty \mathcal{R}^{+}_i(r,s,z) V(s) \Psi_j(s,z)  ds  \qquad \text{for } i=1,2, 3 \text{ and } j=1,\cdots, 4, \\
\end{align*}
where $\Psi_j(r,z) $ is the $j$-th column of the matrix $\Psi(r,z).$ \\ 

Then the solutions to \begin{equation*}
   i \mathcal{L} \Psi^{+}(\cdot, z) = z  \Psi^{+}(\cdot, z) ,  \; \text{ for } z \in \Omega \; \text{ and } \; \im(z)>0
\end{equation*}
for both scenarios, $z$ near $\frac{\sqrt{17}}{8}$ (i.e., $z= \frac{\sqrt{17}}{8}+\xi$) and $z$ near $-\frac{\sqrt{17}}{8}$ (i.e., $z= -\frac{\sqrt{17}}{8}+\xi$) is given by the fixed point problems

\begin{align}
\label{eq:def-Psi+(r,z)-xi-small}
   \Psi^{+}(r,z)&:=   \Uppsi^{+}_{\infty} (r,z) + 
  T_1(r,z,\Psi^{+}(r,z)) ,
\end{align}
and 
\begin{align}
\label{eq:def-tilde-Psi+(r,z)-xi-small}
   \PPsi^{+}(r,z)&:=    \UUppsi^{+}_{\infty} (r,z) + 
  T_2(r,z,\PPsi^{+}(r,z)) ,
\end{align}
where, 
\begin{align}
\begin{split}
\label{eq:def-T_1-infty-xi-small}
  T_1(r,z,\Psi^{+}(r,z))&:= \Upsilon_1(r,z,\Psi_1^{+} (r,z))  + \Upsilon_2(r,z,\Psi_{2}^{+} (r,z)) \\
  & =\int_0^{\infty} \mathcal{R}^{+}_1(r,s,z) V(s) \Psi_{1}^{+}(s,z) ds  + \int_0^{\infty} \mathcal{R}^{+}_2(r,s,z) V(s) \Psi_{2}^{+}(s,z) ds ,     
\end{split}
\end{align}
\begin{align}
\begin{split}
\label{eq:def-T_2-infty-xi-small}
  T_2(r,z,\PPsi^{+}(r,z))&:= \Upsilon_1(r,z,\PPsi_{1}^{+} (r,z))  + \Upsilon_3(r,z,\PPsi_{2}^{+} (r,z)) \\
  & =\int_0^{\infty} \mathcal{R}^{+}_1(r,s,z) V(s) \PPsi_{1}^{+}(s,z) ds  + \int_0^{\infty} \mathcal{R}^{+}_3(r,s,z) V(s) \PPsi_{2}^{+}(s,z) ds  .
  \end{split}
\end{align}

In order to compute $S_i(r,z)$ and $T_i(r,z)$, we first need to invert the matrix defined above. To do this, we begin by computing the Wronskians between $\Uppsi^{+}_{\infty}(\cdot,z)$ and $\UUppsi^{-}_{\infty}(\cdot,z))$.

\begin{lemma}
\label{claim:D-Psi-infty-small-z}
Let $D^{\pm}(z):=W(\Uppsi^{\pm}_{\infty}(\cdot,z), \UUppsi^{\pm}_{\infty}(\cdot,z)).$ Then 
\begin{align*}
 D^{\pm}(z)=\begin{pmatrix}
 \delta(z) &0  \\ \\
0   &   \alpha(z)
   \end{pmatrix} \;  \text{ and } \; 
     (D^{\pm})^{-1}=(D^{\pm})^{-t} &= \begin{pmatrix}
\frac{1}{\delta(z)}  & 0 \\ \\
0   &   \frac{1}{\alpha(z)}
   \end{pmatrix} := \begin{pmatrix}
D_1(z) &0  \\ \\
0  &    D_2(z) 
   \end{pmatrix}
   \end{align*}   
where
\begin{align*}
   \delta(z):=  -2ik_1(z) ( 1 - c_1(z)^2), \qquad 
   \alpha(z):= - 2i k_2(z) (1 - c_2(z)^2)
\end{align*}
\begin{align*}
 c_1(z)&=     - i     \frac{ ( \frac{1}{2} k_1(z)^2 + \frac{17}{8}) }{ z} , \qquad 
  c_2(z)=    - i   \frac{ ( \frac{1}{2} k_2(z)^2 + \frac{17}{8}) }{ z}   .
\end{align*}
\end{lemma}

\begin{proof}
We focus on computing $D^{+}(z)$; the expression for $D^{-}(z)$ then follows directly from the definition of $\Uppsi^{-}$ and $\widetilde{\Uppsi}^{-}$. Let 
\begin{align*}
   D^{+}(z)&= \begin{pmatrix}
       \delta(z) & \gamma(z) \\
       \beta(z) & \alpha(z)
   \end{pmatrix}
\end{align*}
where   
  \begin{align*}
    \delta(z) & =  W(\Uppsi^{+}_{(\infty,1)}(\cdot,z), \Uppsi^{+}_{(\infty,3)}(\cdot,z)) = W( \begin{pmatrix}e^{i k_1(z) r } \\    \frac{-i (  \frac{1}{2}k_1(z)^2 + \frac{17}{8}) }{ z} e^{i k_1(z) r }  \end{pmatrix}, \begin{pmatrix}
       e^{-i k_1(z) r } \\
        \frac{-i ( \frac{1}{2} k_1(z)^2 + \frac{17}{8}) }{ z} e^{-i k_1(z) r } 
    \end{pmatrix} )\\
    & = < \begin{pmatrix}e^{i k_1(z) r } \\    \frac{-i (  \frac{1}{2}k_1(z)^2 + \frac{17}{8}) }{ z} e^{i k_1(z) r }  \end{pmatrix}, \begin{pmatrix}
   - i k_1(z)    e^{-i k_1(z) r } \\
        \frac{k_1(z) ( \frac{1}{2} k_1(z)^2 + \frac{17}{8}) }{ z} e^{-i k_1(z) r } 
    \end{pmatrix} ) >\\
   & - < \begin{pmatrix} i k_1(z) e^{i k_1(z) r } \\    \frac{
    k_1(z) (  \frac{1}{2}k_1(z)^2 + \frac{17}{8}) }{ z} e^{i k_1(z) r }  \end{pmatrix}, \begin{pmatrix}
     e^{-i k_1(z) r } \\
        \frac{i ( \frac{1}{2} k_1(z)^2 + \frac{17}{8}) }{ z} e^{-i k_1(z) r } 
    \end{pmatrix}  > \\
 &   = - 2i k_1(z) \left( 1+   \frac{ ( \frac{1}{2} k_1(z)^2 + \frac{17}{8})^2 }{ z^2}\right)
\end{align*}

 \begin{align*}
    \gamma(z) & =  W(\Uppsi^{+}_{(\infty,1)}(\cdot,z), \Uppsi^{+}_{(\infty,4)}(\cdot,z)) = W( \begin{pmatrix}e^{i k_1(z) r } \\    \frac{-i (  \frac{1}{2}k_1(z)^2 + \frac{17}{8}) }{ z} e^{i k_1(z) r }  \end{pmatrix}, \begin{pmatrix}
       e^{-i k_2(z) r } \\
        \frac{-i ( \frac{1}{2} k_2(z)^2 + \frac{17}{8}) }{ z} e^{-i k_2(z) r } 
    \end{pmatrix} )\\
   &     = - i (k_2(z)+ k_1(z))   \left( 1+   \frac{ ( \frac{1}{2} k_1(z)^2 + \frac{17}{8}) }{ z} \frac{ ( \frac{1}{2} k_2(z)^2 + \frac{17}{8}) }{ z} \right) e^{i (k_1(z)-k_2(z)) r} =0  ,
\end{align*}
which follows directly from the definitions of $k_1(z)$ and $k_2(z)$.

 \begin{align*}
    \beta(z) & =  W(\Uppsi^{+}_{(\infty,2)}(\cdot,z), \Uppsi^{+}_{(\infty,3)}(\cdot,z)) = W(  \begin{pmatrix}
       e^{i k_2(z) r } \\
        \frac{-i ( \frac{1}{2} k_2(z)^2 + \frac{17}{8}) }{ z} e^{i k_2(z) r } 
    \end{pmatrix} ,
    \begin{pmatrix}
    e^{-i k_1(z) r } \\ 
    \frac{-i (  \frac{1}{2}k_1(z)^2 + \frac{17}{8}) }{ z} e^{-i k_1(z) r }  
    \end{pmatrix})\\
   &     = - i (k_2(z)+ k_1(z))   \left( 1+   \frac{ ( \frac{1}{2} k_1(z)^2 + \frac{17}{8}) }{ z} \frac{ ( \frac{1}{2} k_2(z)^2 + \frac{17}{8}) }{ z} \right) e^{i (k_2(z)-k_1(z)) r} =0 
\end{align*}
which follows by a straightforward computation from the definitions of $k_1(z)$ and $k_2(z)$.

\begin{align*}
    \alpha(z) & =  W(\Uppsi^{+}_{(\infty,2)}(\cdot,z), \Uppsi^{+}_{(\infty,4)}(\cdot,z)) = W( \begin{pmatrix}e^{i k_2(z) r } \\    \frac{-i (  \frac{1}{2}k_2(z)^2 + \frac{17}{8}) }{ z} e^{i k_2(z) r }  \end{pmatrix}, \begin{pmatrix}
       e^{-i k_2(z) r } \\
        \frac{-i ( \frac{1}{2} k_2(z)^2 + \frac{17}{8}) }{ z} e^{-i k_2(z) r } 
    \end{pmatrix} )\\
 &= - 2i k_2(z) \left( 1+   \frac{ ( \frac{1}{2} k_2(z)^2 + \frac{17}{8})^2 }{ z^2}\right).
\end{align*}

\begin{align*}
  (D^{+}(z))^{-1} &= \frac{1}{\delta(z) \alpha(z)} \begin{pmatrix}
\alpha(z)  & 0  \\ \\
0   &  \delta(z)
   \end{pmatrix} =\begin{pmatrix}
 \frac{1}{\delta(z)} & 0 \\ \\
0   &     \frac{1}{\alpha(z)}
   \end{pmatrix} := \begin{pmatrix}
D_1(z) &0  \\ \\
0  &    D_2(z) 
\end{pmatrix}
\end{align*}

\end{proof}

\begin{remark}
\label{K_j-C_j-D_j-behavior}
 In the following two scenarios we write the asymptotics for $D_1^\pm$ and $D_2^\pm$:
\begin{itemize}
    \item Let $z= \frac{\sqrt{17}}{8}+\xi,$ for small $\xi.$ We have $| k_1(z)| \simeq  \sqrt{|\xi| }  , $ and $k_2(z)=i \frac{3}{\sqrt{2}}+O(|\xi|) $  This implies, $ c_1(z)=-i \sqrt{17}+O(|\xi|)  $ and $c_2(z)=i\frac{1}{  \sqrt{17}}+O(|\xi|).$ 
    Therefore, $|D_1^{\pm}(z)|\simeq \frac{1}{\sqrt{|\xi|} } $ and $|D_2^{\pm}(z)| \simeq 1 $
      \item Let $z= -\frac{\sqrt{17}}{8}+\xi,$ for small $\xi.$ We have $| k_1(z)| \simeq  \sqrt{|\xi| }  , $ and $ k_2(z)=i \frac{3}{\sqrt{2}}+O(|\xi|) $  This implies, $ c_1(z)=-i \sqrt{17}+O(|\xi|)  $ and $c_2(z)=i\frac{1}{  \sqrt{17}}+O(|\xi|).$ 
    Therefore, $|D_1^{\pm}(z)|\simeq \frac{1}{\sqrt{|\xi|} }$ and $|D_2^{\pm}(z)| \simeq 1 $
\end{itemize}    
\end{remark}

 Next, we compute  $S_i(r,z)$ and $T_i(r,z)$, using Lemma \ref{inver-2-matrix} and Lemma \ref{claim:D-Psi-infty-small-z}. 
 \begin{claim}
 \label{claim:def-S_iT_i-Greens-R_j}
     We have, 
    \begin{align*}
 &   \begin{cases}
    S_1(r,z)&= -  (D^{+}(z))^{-t} \, \UUppsi^{+}_{\infty}(r,z)^t \sigma_3 \sigma_2 \\
    T_1(r,z)&= -  (D^{+}(z))^{-1} \,  \Uppsi^{+}_{\infty}(r,z)^t \sigma_3 \sigma_2
    \end{cases}
    \qquad \qquad 
\begin{cases}
        S_2(r,z)&=   (D^{+}(z))^{-t} \,  \UUppsi^{+}_{\infty}(r,z)^t \sigma_3 \sigma_2 \\
    T_2(r,z)&= -  (D^{+}(z))^{-1} \,  \Uppsi^{+}_{\infty}(r,z)^t \sigma_3 \sigma_2
\end{cases} \\
& \begin{cases}
    S_3(r,z)&= -  (D^{+}(z))^{-t} \,  \UUppsi^{+}_{\infty}(r,z)^t \sigma_3 \sigma_2 \\
    T_3(r,z)&=   (D^{+}(z))^{-1} \,  \Uppsi^{+}_{\infty}(r,z)^t \sigma_3 \sigma_2
    \end{cases}
    \end{align*}
    
and \begin{align*}
\mathcal{R}^{+}_1(r,s,z):&= \begin{cases}
    i    \Uppsi^{+}_{\infty}(r,z)  (D^{+}(z))^{-t} \,  \UUppsi^{+}_{\infty}(s,z)^t \sigma_1, \qquad r_{\infty} \leq s \leq r, \\ \\
   i   \UUppsi^{+}_{\infty}(r,z)  (D^{+}(z))^{-1} \, \Uppsi^{+}_{\infty}(s,z)^t \sigma_1 , \qquad r \leq s \leq \infty.
\end{cases} \\ \\
\mathcal{R}^{+}_2(r,s,z):&= \begin{cases}
  -  i    \Uppsi^{+}_{\infty}(r,z)  (D^{+}(z))^{-t} \,  \UUppsi^{+}_{\infty}(s,z)^t \sigma_1, \qquad r \leq s \leq \infty, \\ \\
   i   \UUppsi^{+}_{\infty}(r,z)  (D^{+}(z))^{-1} \, \Uppsi^{+}_{\infty}(s,z)^t \sigma_1,  \qquad r \leq s \leq \infty.
\end{cases}  \\ \\
\mathcal{R}^{+}_3(r,s,z):&= \begin{cases}
   i    \Uppsi^{+}_{\infty}(r,z)  (D^{+}(z))^{-t} \,  \UUppsi^{+}_{\infty}(s,z)^t \sigma_1, \qquad r_{\infty} \leq s \leq r,\\ \\
 -  i   \UUppsi^{+}_{\infty}(r,z)  (D^{+}(z))^{-1} \, \Uppsi^{+}_{\infty}(s,z)^t \sigma_1,  \qquad r_{\infty} \leq s \leq r.
\end{cases}
\end{align*}
    
\end{claim}
\begin{proof}
Using \eqref{eq:S1T1-condi-infinity-xi-small} and Lemma \ref{inver-2-matrix}, we obtain

\begin{align*}
\begin{pmatrix}
    S_1(r,z) \\
    -T_1(r,z)
\end{pmatrix}
&= \begin{pmatrix}
      \Uppsi^{+}_{\infty}(r,z)    &  \UUppsi^{+}_{\infty}(r,z) \\
      \partial_r   \Uppsi^{+}_{\infty}(r,z) & \partial_r   \UUppsi^{+}_{\infty}(r,z)
\end{pmatrix}^{-1} \begin{pmatrix}
    0 \\ \sigma_2
\end{pmatrix} \\
& =  \begin{pmatrix}
        (D^{+}(z))^{-t} & 0 \\ 
        0 & (D^{+}(z))^{-1}
    \end{pmatrix} \begin{pmatrix}
        0 & -I \\
        I & 0 
    \end{pmatrix} \begin{pmatrix}
      \Uppsi^{+}_{\infty}(r,z)^t    & \partial_r   \Uppsi^{+}_{\infty}(r,z)^t  \\
     \UUppsi^{+}_{\infty}(r,z)^t & \partial_r   \UUppsi^{+}_{\infty}(r,z)^t
\end{pmatrix}      \begin{pmatrix}
       0 & \sigma_3 \\
       - \sigma_3 & 0 
   \end{pmatrix} 
   \begin{pmatrix}
    0 \\ \sigma_2
\end{pmatrix} \\
& = \begin{pmatrix}
    -  (D^{+}(z))^{-t} \, \UUppsi^{+}_{\infty}(r,z)^t  \sigma_3 \sigma _2 \\ \\
   (D^{+}(z))^{-1} \, \Uppsi^{+}_{\infty}(r,z)^t \sigma_3 \sigma_2 
\end{pmatrix}
\end{align*}
Similarly, using \eqref{eq:S2T2-condi-infinity-xi-small}, \eqref{eq:S3T3-condi-infinity-xi-small} and Lemma \ref{inver-2-matrix}, we obtain 
\begin{align*}
    \begin{cases}
        S_2(r,z)&=   (D^{+}(z))^{-t} \,  \UUppsi^{+}_{\infty}(r,z)^t \sigma_3 \sigma_2 \\
    T_2(r,z)&= -  (D^{+}(z))^{-1} \,  \Uppsi^{+}_{\infty}(r,z)^t \sigma_3 \sigma_2
\end{cases}  , \qquad \begin{cases}
    S_3(r,z)&= -  (D^{+}(z))^{-t} \,  \UUppsi^{+}_{\infty}(r,z)^t \sigma_3 \sigma_2 \\
    T_3(r,z)&=   (D^{+}(z))^{-1} \,  \Uppsi^{+}_{\infty}(r,z)^t \sigma_3 \sigma_2
    \end{cases}
\end{align*}

Using the fact that $\sigma_3 \sigma_2= -i \sigma_1,$ we obtain the desired formula for $\mathcal{R}^{+}_j(r,s),$ for $j=1,2,3.$
\end{proof}

Next, we prove the existence of the solutions to \eqref{eq:iLPsi=zPsi,Im(z)>0} as defined in \eqref{eq:def-Psi+(r,z)-xi-small} and \eqref{eq:def-tilde-Psi+(r,z)-xi-small}. This is achieved  by showing that the associated integral operators, defined in \eqref{eq:def-T_1-infty-xi-small} and \eqref{eq:def-T_2-infty-xi-small} define a contraction on the following suitable function spaces.

\begin{defi}
Let $F=\begin{bmatrix}
    F_1 & F_2
\end{bmatrix}, $ where $F_i$ are the columns of $F$ and let $r_\infty$ be a fixed number:
\begin{align*}
   \left\| F \right\|_{\mathcal{B}_1} &:= \sup_{r \geq r_{\infty}} 
   |e^{-ik_1(z)r}  F_1(r)| + \sup_{r \geq r_{\infty}} |e^{-ik_2(z)r }  F_{2}(r)| \\
   \left\| F \right\|_{\mathcal{B}_2} &:= \sup_{r \geq r_{\infty}} 
   |e^{ik_1(z) r }  F_{1}(r)| + \sup_{r \geq r_{\infty}}  |e^{ik_2(z)r}  F_{2}(r)| \\
   \left\| F \right\|_{\mathcal{B}^{\prime}_1} &:= \sup_{r \geq r_{\infty}} 
   \frac{1}{|k_1(z)|}|e^{-ik_1(z)r}  F_1(r)| + \sup_{r \geq r_{\infty}} \frac{1}{|k_2(z)|}|e^{-ik_2(z)r }  F_{2}(r)| \\
   \left\| F \right\|_{\mathcal{B}^{\prime}_2} &:= \sup_{r \geq r_{\infty}} 
   \frac{1}{|k_1(z)|}|e^{ik_1(z) r }  F_{1}(r)| + \sup_{r \geq r_{\infty}}  \frac{1}{|k_2(z)|}|e^{ik_2(z)r}  F_{2}(r)| 
\end{align*} 
\end{defi}

Note that throughout the rest of the paper, we fix small constants $\delta_0\ll \varepsilon_\infty\ll1$ such that $\frac{1}{\sqrt{\delta_0} } e^{\frac{-2\varepsilon_0}{\sqrt{\delta_0}}} \ll\varepsilon_{\infty}$.

\begin{prop}
\label{prop:contraction-T-j-small-xi-infty}
Let $0<|\xi| \leq \delta_0 < 1 ,$ and $r_{\infty} :=\frac{1}{2}\log(\frac{1}{\varepsilon_{\infty}\sqrt{|\xi|}})$. 

Then $T_j$ is a contraction on $\mathcal{B}_j,$ for $j=1,2,$  i.e., $$ \| T_j(\Psi(\cdot,z)) \|_{\mathcal{B}_j} \lesssim \varepsilon_{\infty} \left\| \Psi(\cdot,z) \right\|_{\mathcal{B}_j} . $$
Moreover,  with $\Psi^+(\cdot,z)$ and $\PPsi^+(\cdot,z)$ denoting the solutions to the fixed point problems \eqref{eq:def-Psi+(r,z)-xi-small} and \eqref{eq:def-tilde-Psi+(r,z)-xi-small}, we have $$ \|  T_1(\Psi^{+}(\cdot,z)) \|_{\mathcal{B}_1} +\|  T_2(\PPsi^{+}(\cdot,z)) \|_{\mathcal{B}_2}+\| \partial_r T_1(\Psi^{+}(\cdot,z)) \|_{\mathcal{B}_1^{\prime}} +\| \partial_r T_2(\PPsi^{+}(\cdot,z)) \|_{\mathcal{B}_2^{\prime}}\lesssim \varepsilon_{\infty}. $$
\end{prop}

\begin{proof}
We provide the proof that $T_j$ is a contraction on $\mathcal{B}_j,$ and the corresponding estimate for the derivative follows similarly by applying the same argument with the necessary conditions \eqref{eq:S1T1-condi-infinity-xi-small},\eqref{eq:S2T2-condi-infinity-xi-small} and \eqref{eq:S3T3-condi-infinity-xi-small}.

Recall that 
\begin{align*}
  T_1(r,z,\Psi^{+}(r,z))&:= \Upsilon_1(r,z,\Psi_1^{+} (r,z))  + \Upsilon_2(r,z,\Psi_{2}^{+} (r,z)) \\
  & =\int_0^{\infty} \mathcal{R}^{+}_1(r,s,z) V(s) \Psi_{1}^{+}(s,z) ds  + \int_0^{\infty} \mathcal{R}^{+}_2(r,s,z) V(s) \Psi_{2}^{+}(s,z) ds ,     \\
  T_2(r,z,\PPsi^{+}(r,z))&:= \Upsilon_1(r,z,\PPsi_{1}^{+} (r,z))  + \Upsilon_3(r,z,\PPsi_{2}^{+} (r,z)) \\
  & =\int_0^{\infty} \mathcal{R}^{+}_1(r,s,z) V(s) \PPsi_{1}^{+}(s,z) ds  + \int_0^{\infty} \mathcal{R}^{+}_3(r,s,z) V(s) \PPsi_{2}^{+}(s,z) ds  .
\end{align*}

Therefore, we limit our attention to deriving  the estimates for $\Upsilon_j^{+}$, since the corresponding bounds for $T_j$  can be deduced directly from them. 

\begin{claim}
    \label{claim:upsilon_1-2-xi-small-infty}
\begin{align}
\label{eq:Upsilon_1-estimate-small-xi}
\sup_{r \geq r_{\infty}}  |e^{- ik_1(z) r} \Upsilon_1(r,z,\Psi^{+}_{1} (r,z)) |  \lesssim \frac{1}{\sqrt{\xi}}  e^{- \, 2 r_\infty} \sup_{r \geq r_{\infty}} | e^{ - i k_1(z) r} \Psi^{+}_{1}(r,z)  |    \\
\label{eq:Upsilon_2-estimate-small-xi}
\sup_{r \geq r_{\infty}}  |e^{- ik_2(z) r} \Upsilon_2(r,z,\Psi^{+}_2(r,z))  |  \lesssim \frac{1}{\sqrt{\xi}}  e^{- 2 r_\infty} \sup_{r \geq r_{\infty}} | e^{ - i k_2(z) r}  \Psi^{+}_{2}(r,z)  | 
\end{align}
\end{claim}
\begin{proof}

First, we estimate  $ \Upsilon_1(r,z,\Psi^{+}_{1} (r,z)).$ We have
\begin{align*}
    \Upsilon_1(r,z,\Psi^{+}_{1} (r,z))&:= \int_0^{\infty} \mathcal{R}^{+}_1(r,s,z) V(s) \Psi^{+}_{1}(s,z) ds   \\
    & =\int_{r_{\infty}}^{r} i    \Uppsi^{+}_{\infty}(r,z)  (D^{+}(z))^{-t} \,  \UUppsi^{+}_{\infty}(s,z)^t \sigma_1  V(s) \Psi^{+}_{1}(s,z) ds \\
    & + \int_r^{\infty}   i  \UUppsi^{+}_{\infty}(r,z)  (D^{+}(z))^{-1} \,  \Uppsi^{+}_{\infty}(s,z)^t \sigma_1  V(s) \Psi^{+}_{1}(s,z) ds  \\
& =  i \int_{r_{\infty}}^{r}   \bigg(  D_1^{+}(z) \Uppsi^{+}_{(\infty,1)}(r,z) \Uppsi^{+}_{(\infty,3)}(s,z)^{t} 
 + D_2^{+}(z) \Uppsi^{+}_{(\infty,2)}(r,z) \Uppsi^{+}_{(\infty,4)}(s,z)^{t}  \bigg) \\
 &  \times \sigma_1   V(s) \Psi^{+}_{1}(s,z) ds +  i \int_{r}^{\infty}  \bigg(  D_1^{+}(z) \Uppsi^{+}_{(\infty,3)}(r,z) \Uppsi^{+}_{(\infty,1)}(s,z)^{t} \\
 & + D_2^{+}(z) \Uppsi^{+}_{(\infty,4)}(r,z) \Uppsi^{+}_{(\infty,2)}(s,z)^{t}  \bigg)  \times \sigma_1   V(s) \Psi^{+}_{1}(s,z) ds \\
& =  i  \mathcal{M}^{+}_{1}(z) e^{i k_1(z) r } \int_{r_\infty}^r   e^{-i k_1(z) s }    V(s) \Psi^{+}_{1}(s,z) ds  \\
&+  i  \mathcal{M}^{+}_{2} (z) e^{i k_2(z) r} \int_{r_\infty}^{r} e^{-i k_2(z) s}  V(s) \Psi^{+}_{1}(s,z) ds   \\
&+     i  \mathcal{M}^{+}_{1}(z) e^{-i k_1(z) r }\int_{r}^{\infty}  e^{i k_1(z) s }    V(s) \Psi^{+}_{1}(s,z) ds  \\
&+  i  \mathcal{M}^{+}_{2}(z)  e^{-i k_2(z) r} \int_{r}^{\infty} e^{i k_2(z) s}  V(s) \Psi^{+}_{1}(s,z) ds
\end{align*}

Where \begin{align*}
 \mathcal{M}^{+}_{1}(z):& =   \begin{pmatrix} 
 D_1^{+}(z)  c_1(z)   & 
 D_1^{+}(z)        \\
     D_1^{+}(z)  (c_1(z))^2    &   D_1^{+}(z)    c_1(z) 
\end{pmatrix} 
\qquad \text{and} \qquad 
\mathcal{M}^{+}_{2}(z)&:= \begin{pmatrix}
      D_2^{+}(z) c_2(z)  &     D_2^{+}(z)  \\ 
    D_2^{+}(z) (c_2(z))^2     &      D_2^{+}(z)  c_2(z)   
\end{pmatrix}   
\end{align*}

By remark \ref{K_j-C_j-D_j-behavior}, one can see that in both scenarios for small $\xi$ we have  $|\mathcal{M}^{+}_{1}(z) |=  O(\frac{1}{\sqrt{|\xi|}})$ and $|\mathcal{M}^{+}_{2}(z)|=O(1).$ Recall that, $V(s)=O(e^{-2s}) .$ Thus, we have 

\begin{align}
\begin{split}
\label{upsilon_1-of-psi_1}
\Upsilon_1(r,z,\Psi^{+}_{1} (r,z))&=i  \mathcal{M}^{+}_{1}(z) e^{i k_1(z) r } \int_{r_\infty}^r   e^{-i k_1(z) s }    O(e^{-2s}) \Psi^{+}_{1}(s,z) ds \\
&+     i  \mathcal{M}^{+}_{1}(z) e^{-i k_1(z) r }\int_{r}^{\infty}  e^{i k_1(z) s }     O(e^{-2s}) \Psi^{+}_{1}(s,z) ds \\
 & +  i  \mathcal{M}^{+}_{2} (z) e^{i k_2(z) r} \int_{r_\infty}^{r} e^{-i k_2(z) s}   O(e^{-2s}) \Psi^{+}_{1}(s,z) ds  \\
 &+  i  \mathcal{M}^{+}_{2}(z)  e^{-i k_2(z) r} \int_{r}^{\infty} e^{i k_2(z) s}   O(e^{-2s}) \Psi^{+}_{1}(s,z) ds 
\end{split}
\end{align}
 
 \begin{align*}
 \Upsilon_1(r,z,\Psi^{+}_{1} (r,z))     & \lesssim  \mathcal{M}^{+}_{1}(z) e^{i k_1(z) r }  e^{-2 r_{\infty}} \sup_{r\geq r_{\infty}} | e^{ -  i k_1(z) r }  \Psi^{+}_{1}(r,z) |  \\ 
& +  \mathcal{M}^{+}_{1}(z) e^{i k_1(z) r }   e^{-2 r} \sup_{r\geq r_{\infty}} | e^{ -  i k_1(z) r }  \Psi^{+}_{1}(r,z) |  \\ 
& + \mathcal{M}^{+}_{2} (z) e^{i k_1(z) r }   e^{-2 r} \sup_{r\geq r_{\infty}} | e^{ -  i k_1(z) r }  \Psi^{+}_{1}(r,z) |  \\ 
& + \mathcal{M}^{+}_{2} (z) e^{i k_1(z) r }   e^{-2 r} \sup_{r\geq r_{\infty}} | e^{ -  i k_1(z) r }  \Psi^{+}_{1}(r,z) |   
\end{align*}

which yields
\begin{align*}
\sup_{r \geq r_{\infty}}  |e^{- ik_1(z) r} \Upsilon_1(r,z,\Psi^{+}_{1} (r,z)) |  \lesssim \frac{1}{\sqrt{|\xi|}}  e^{- \, 2 r_\infty} \sup_{r \geq r_{\infty}} | e^{ - i k_1(z) r} \Psi^{+}_{1}(r,z)  | 
\end{align*}

Next, we estimate $\Upsilon_2(r,z,\Psi^{+}_2(r,z)).$ Similarly, to the estimate for $\Upsilon_1(r,z,\Psi^{+}_1(r,z))$ we have 
\begin{align*}
 \Upsilon_2(r,z,\Psi^{+}_2(r,z)) &:= \int_0^{\infty} \mathcal{R}^{+}_2(r,s,z) V(s) \Psi_{2}^{+}(s,z) ds   \\ 
 & =\int_{r}^{\infty} - i    \Uppsi^{+}_{\infty}(r,z)  (D^{+}(z))^{-t} \,  \UUppsi^{+}_{\infty}(s,z)^t \sigma_1  V(s) \Psi_{2}^{+}(s,z) ds \\
    & + \int_r^{\infty}   i  \UUppsi^{+}_{\infty}(r,z)  (D^{+}(z))^{-1} \,  \Uppsi^{+}_{\infty}(s,z)^t \sigma_1  V(s) \Psi_2^{+}(s,z) ds  \\
&=   - i  \mathcal{M}^{+}_{1}(z) e^{i k_1(z) r } \int_{r}^{\infty}    e^{-i k_1(z) s }    V(s) \Psi^{+}_{2}(s,z) ds   \\
& - i  \mathcal{M}^{+}_{2} (z) e^{i k_2(z) r} \int_{r}^{\infty}  e^{-i k_2(z) s}  V(s) \Psi^{+}_{2}(s,z) ds   \\
&+     i  \mathcal{M}^{+}_{1}(z) e^{-i k_1(z) r } \int_{r}^{\infty}  e^{i k_1(z) s }    V(s) \Psi^{+}_{2}(s,z) ds \\
& +  i  \mathcal{M}^{+}_{2}(z)  e^{-i k_2(z) r} \int_{r}^{\infty} e^{i k_2(z) s}  V(s) \Psi^{+}_{2}(s,z) ds \\
&= - i   \mathcal{M}^{+}_{1}(z) e^{i k_1(z) r }\int_{r}^{\infty}    e^{-i k_1(z) s }    O(e^{-2s}) \Psi^{+}_{2}(s,z) ds \\
&+     i  \mathcal{M}^{+}_{1}(z) e^{-i k_1(z) r }\int_{r}^{\infty}  e^{i k_1(z) s }     O(e^{-2s}) \Psi^{+}_{2}(s,z) ds \\
 & - i   \mathcal{M}^{+}_{2} (z) e^{i k_2(z) r} \int_{r}^{\infty}  e^{-i k_2(z) s}   O(e^{-2s}) \Psi^{+}_{2}(s,z) ds  \\
 &+  i  \mathcal{M}^{+}_{2}(z)  e^{-i k_2(z) r} \int_{r}^{\infty} e^{i k_2(z) s}   O(e^{-2s}) \Psi^{+}_{2}(s,z) ds \\
\end{align*}
Thus, 
\begin{align}
\label{upsilon_2-of-psi_2}
\begin{split}
|\Upsilon_2(r,z,\Psi^{+}_2(r,z))    & \lesssim  \mathcal{M}^{+}_{1}(z) e^{i k_2(z) r }  e^{-2 r} \sup_{r\geq r_{\infty}} | e^{ -  i k_2(z) r }  \Psi^{+}_{2} (r,z) |  \\ 
& +  \mathcal{M}^{+}_{1}(z)  e^{i k_2(z) r }  e^{-2 r} \sup_{r\geq r_{\infty}} | e^{ -  i k_2(z) r }  \Psi^{+}_{2}(r,z) |  \\ 
& + \mathcal{M}^{+}_{2}(z)  e^{i k_2(z) r }   e^{-2 r} \sup_{r\geq r_{\infty}} | e^{ -  i k_2(z) r } \Psi^{+}_{2}(r,z) |  \\ 
& + \mathcal{M}^{+}_{2}(z)  e^{i k_2(z) r }   e^{-2 r} \sup_{r\geq r_{\infty}} | e^{ -  i k_2(z) r }  \Psi^{+}_{2}(r,z) |
\end{split}
\end{align}

which yields\begin{align*}
\sup_{r \geq r_{\infty}}  |e^{- ik_2(z) r} \Upsilon_2(r,z,\Psi^{+}_2(r,z))  |  \lesssim \frac{1}{\sqrt{|\xi|}}  e^{- 2 r_\infty} \sup_{r \geq r_{\infty}} | e^{ - i k_2(z) r}  \Psi^{+}_{2}(r,z)  |
\end{align*}
\end{proof}
\begin{claim}
\label{claim:upsilon_1-3-xi-small-infty}
 \begin{align}
 \sup_{r \geq r_{\infty}}  |e^{ ik_1(z) r} \Upsilon_1(r,z,\Psi^{+}_{3} (r,z)) | & \lesssim \frac{1}{\sqrt{|\xi|}}  e^{- \, 2 r_\infty} \sup_{r \geq r_{\infty}} | e^{  i k_1(z) r} \Psi^{+}_{3}(r,z)  |  \\
     \sup_{r \geq r_{\infty}}  |e^{ ik_2(z) r} \Upsilon_3(r,z,\Psi^{+}_{4}(r,z))|  &\lesssim \frac{1}{\sqrt{|\xi|}}  e^{- \, 2 r_\infty} \sup_{r \geq r_{\infty}} | e^{  i k_2(z) r} \Psi^{+}_{4}(r,z)  | 
 \end{align}
\end{claim}

\begin{proof}
Next we estimate $\Upsilon_1(r,z,\Psi^{+}_3 (r,z))$ the first column of $T_2(r,z,\PPsi^{+} (r,z)).$ Similarly to the estimate for $\Upsilon_1(r,z,\PPsi^{+}_1 (r,z))$, we have
\begin{align*}
\Upsilon_1(r,z,\Psi^{+}_{3} (r,z))&=\int_0^{\infty} \mathcal{R}^{+}_1(r,s,z) V(s) \Psi^{+}_{3}(s,z) ds   \\
    &=i  \mathcal{M}^{+}_{1}(z) e^{i k_1(z) r } \int_{r_\infty}^r   e^{-i k_1(z) s }    O(e^{-2s}) \Psi^{+}_{3}(s,z) ds \\
&+     i  \mathcal{M}^{+}_{1}(z) e^{-i k_1(z) r }\int_{r}^{\infty}  e^{i k_1(z) s }     O(e^{-2s}) \Psi^{+}_{3}(s,z) ds \\
 & +  i  \mathcal{M}^{+}_{2} (z) e^{i k_2(z) r} \int_{r_\infty}^{r} e^{-i k_2(z) s}   O(e^{-2s}) \Psi^{+}_{3}(s,z) ds  \\
 &+  i  \mathcal{M}^{+}_{2}(z)  e^{-i k_2(z) r} \int_{r}^{\infty} e^{i k_2(z) s}   O(e^{-2s}) \Psi^{+}_{3}(s,z) ds \\
 & \lesssim  \mathcal{M}^{+}_{1}(z) e^{i k_1(z) r }  e^{-2 r_{\infty} } e^{-2 i k_1(z) r_{\infty} } \sup_{r\geq r_{\infty}} | e^{   i k_1(z) r }  \Psi^{+}_{3}(r,z) |  \\ 
& +  \mathcal{M}^{+}_{1}(z) e^{ - i k_1(z) r }   e^{-2 r} \sup_{r\geq r_{\infty}} | e^{   i k_1(z) r }  \Psi^{+}_{3}(r,z) |  \\ 
& + \mathcal{M}^{+}_{2} (z) e^{-i k_1(z) r }   e^{-2 r} \sup_{r\geq r_{\infty}} | e^{   i k_1(z) r }  \Psi^{+}_{3}(r,z) |  \\ 
& + \mathcal{M}^{+}_{2} (z) e^{-i k_1(z) r }   e^{-2 r} \sup_{r\geq r_{\infty}} | e^{   i k_1(z) r }  \Psi^{+}_{3}(r,z) |
\end{align*}
which yields, 
\begin{align*}
e^{ i k_1(z) r} \Upsilon_1(r,z,\Psi^{+}_{3} (r,z)) &\lesssim    \frac{1}{\sqrt{|\xi|}} e^{2 i k_1(z) (r- r_{\infty}) } e^{-2 r_{\infty}}  \sup_{r\geq r_{\infty}} | e^{   i k_1(z) r }  \Psi^{+}_{3}(r,z) |  \\ 
& +  \frac{1}{\sqrt{|\xi|}}  e^{-2 r_{\infty}} \sup_{r\geq r_{\infty}} | e^{  i k_1(z) r }  \Psi^{+}_{3}(r,z) |  \\ 
& +      e^{-2 r_{\infty}} \sup_{r\geq r_{\infty}} | e^{   i k_1(z) r }  \Psi^{+}_{3}(r,z) |  \\ 
& +   e^{-2 r_{\infty}} \sup_{r\geq r_{\infty}} | e^{   i k_1(z) r }  \Psi^{+}_{3}(r,z) |
\end{align*}

Recall that $\im(k_1(z))>0,$ for $z \in \Omega,$ with $\im(z)>0.$ Then 

\begin{align*}
\sup_{r \geq r_{\infty}}  |e^{ ik_1(z) r} \Upsilon_1(r,z,\Psi^{+}_{3} (r,z)) |  \lesssim \frac{1}{\sqrt{|\xi|}}  e^{- \, 2 r_\infty} \sup_{r \geq r_{\infty}} | e^{  i k_1(z) r} \Psi^{+}_{3}(r,z)  | 
\end{align*}

Next, we estimate $\Upsilon_3(r,z,\Psi^{+}_{4} (r,z)),$ the second column of $T_2(r,z,\PPsi^{+} (r,z)).$ We have 
\begin{align*}
\Upsilon_3(r,z,\Psi^{+}_{4} (r,z))&:= \int_0^{\infty} \mathcal{R}^{+}_3(r,s,z) V(s) \Psi^{+}_{4}(s,z) ds   \\
    & =\int_{r_{\infty}}^{r} i    \Uppsi^{+}_{\infty}(r,z)  (D^{+}(z))^{-t} \,  \UUppsi^{+}_{\infty}(s,z)^t \sigma_1  V(s) \Psi^{+}_{4}(s,z) ds \\
    & - \int_{r_{\infty}}^{r}  i  \UUppsi^{+}_{\infty}(r,z)  (D^{+}(z))^{-1} \,  \Uppsi^{+}_{\infty}(s,z)^t \sigma_1  V(s) \Psi^{+}_{4}(s,z) ds  \\
    &= i  \mathcal{M}^{+}_{1}(z) e^{i k_1(z) r } \int_{r_\infty}^r   e^{-i k_1(z) s }    O(e^{-2s}) \Psi^{+}_{4}(s,z) ds \\
&-     i  \mathcal{M}^{+}_{1}(z) e^{-i k_1(z) r } \int_{r_{\infty}}^{r}   e^{i k_1(z) s }     O(e^{-2s}) \Psi^{+}_{4}(s,z) ds \\
 & +  i  \mathcal{M}^{+}_{2} (z) e^{i k_2(z) r}  \int_{r_\infty}^{r} e^{-i k_2(z) s}   O(e^{-2s}) \Psi^{+}_{4}(s,z) ds  \\
 &-  i  \mathcal{M}^{+}_{2}(z)  e^{-i k_2(z) r} \int_{r_{\infty}}^{r} e^{i k_2(z) s}   O(e^{-2s}) \Psi^{+}_{4}(s,z) ds \\
 & \lesssim  \mathcal{M}^{+}_{1}(z) e^{ - i k_2(z) r }  e^{-2 r }  \sup_{r\geq r_{\infty}} | e^{   i k_2(z) r }  \Psi^{+}_{4}(r,z) |  \\ 
& +  \mathcal{M}^{+}_{1}(z) e^{ - i k_2(z) r }   e^{-2 r} \sup_{r\geq r_{\infty}} | e^{   i k_2(z) r }  \Psi^{+}_{4}(r,z) |  \\ 
& + \mathcal{M}^{+}_{2} (z) e^{-i k_2(z) r }   e^{-2 r} \sup_{r\geq r_{\infty}} | e^{   i k_2(z) r }  \Psi^{+}_{4}(r,z) |  \\ 
& + \mathcal{M}^{+}_{2} (z) e^{-i k_2(z) r }   e^{-2 r} \sup_{r\geq r_{\infty}} | e^{   i k_2(z) r }  \Psi^{+}_{4}(r,z) |
\end{align*}

which yields, 

\begin{align*}
\sup_{r \geq r_{\infty}}  |e^{ ik_2(z) r} \Upsilon_3(r,z,\Psi^{+}_{4}(r,z))|  \lesssim \frac{1}{\sqrt{|\xi|}}  e^{- \, 2 r_\infty} \sup_{r \geq r_{\infty}} | e^{  i k_2(z) r} \Psi^{+}_{4}(r,z)  | 
\end{align*}

\end{proof}

 Using Claim \ref{claim:upsilon_1-2-xi-small-infty}, we deduce that 
 \begin{align*}
  \left\| T_1(\cdot, z, \Psi^{+}(\cdot,z))  \right\|_{\mathcal{B}_1} \lesssim \frac{1}{\sqrt{|\xi|}} e^{- 2 r_\infty}   \left\|\Psi^{+} (\cdot, z)  \right\|_{\mathcal{B}_1}  . 
\end{align*}
Let $\varepsilon_{\infty}:= \frac{1}{\sqrt{|\xi|}} e^{- 2 r_\infty}. $
 Therefore, $T_1$ is a contraction on $\mathcal{B}_1.$
 Similarly, by Claim  \ref{claim:upsilon_1-3-xi-small-infty}, we deduce that 
 \begin{align*}
  \left\| T_2(\cdot, z, \PPsi^{+}(\cdot,z))  \right\|_{\mathcal{B}_2} \lesssim \varepsilon_{\infty}  \left\|\PPsi^{+} (\cdot, z)  \right\|_{\mathcal{B}_2}  . 
\end{align*}
 and $T_2$ is a contraction on $\mathcal{B}_2.$ The estimates for the derivatives can be obtained using the similar argument with conditions \eqref{eq:S1T1-condi-infinity-xi-small},\eqref{eq:S2T2-condi-infinity-xi-small} and \eqref{eq:S3T3-condi-infinity-xi-small}.
 This concludes the proof of Proposition \ref{prop:contraction-T-j-small-xi-infty}.
\end{proof}

\subsection{Construction of $\Psi^{-}(\cdot,z)$ for small $|\xi|$} The construction of $\Psi^{-}(r,z) \in L^2((1,\infty))$ when $\im(z)<0,$ is very similar to one for $\Psi^{+}(r,z)$ and can be obtained directly from $\Psi^{+}(r,z)$ by symmetry.  Recall that, for $z \in \Omega$ with $\im(z)<0,$ we have 
\begin{align*}
    \Uppsi^{-}_{\infty}(r,z)  =\begin{pmatrix}
      e^{i k_1(z) r } &     e^{i k_2(z) r } \\ \\
       c_1(z) e^{i k_1(z) r } & c_2(z)  e^{i k_2(z) r } 
    \end{pmatrix} ,  \quad 
   \UUppsi^{-}_{\infty}(r,z) = \begin{pmatrix}
      e^{ -i k_1(z) r } &     e^{-i k_2(z) r } \\ \\
       c_1(z)  e^{-i k_1(z) r } & c_2(z) e^{-i k_2(z) r } 
    \end{pmatrix}  
\end{align*}

We seek to construct a fundamental matrix solutions $\Psi^{-}(\cdot,z)\in L^2(1,\infty)$ and $\widetilde{\Psi}^{-}(\cdot, z) \notin L^2(1,\infty) $ to 
\begin{equation}
\label{eq:iLPsi=zPsi,Im(z)<0}
   i \mathcal{L} \Psi^{-}(\cdot, z) = z  \Psi^{-}(\cdot, z) ,  \; \text{ for } z \in \Omega \; \text{ and } \; \im(z)<0
\end{equation}

By Lemma \ref{lem:k_j-behavior} we have $k_1(z)=k_1(-z),$ $k_2(z)=k_2(-z),$ for $\im(z)<0,$ and $c_1(-z)=-c_1(z)$. Thus, we define 
\begin{align}
\label{eq:def-upPsi-minus-small-xi}
    \Uppsi^{-}_{\infty}(r,z)=\sigma_3  \Uppsi^{+}_{\infty}(r,-z) \qquad \text{and} \qquad  \UUppsi^{-}_{\infty}(r,z)=\sigma_3  \UUppsi^{+}_{\infty}(r,-z)  .
\end{align}

Note that $ \Uppsi^{-}_{\infty}$ and $\UUppsi^{-}_{\infty}$ satisfy the equation  $i \mathcal{L}_{\infty}   \Uppsi  (r,z)  = z \Uppsi(r,z) .$  Therefore, we define two solutions to $i \mathcal{L} \Psi(r,z)  = z \Psi(r,z) $ for $z \in \Omega$ with $\im(z)<0$ by,
\begin{align}
\label{eq:def-Psi-minus-small-xi}
    \Psi^{-}(r,z):=\sigma_3  \Psi^{+}(r,-z) \qquad \text{and} \qquad \PPsi^{-}(r,z):=\sigma_3  \PPsi^{+}(r,-z)
\end{align}

Thus it follows from Proposition~\ref{prop:contraction-T-j-small-xi-infty} that $$ \|  T_1(\Psi^{-}(\cdot,z)) \|_{\mathcal{B}_1} +\|  T_2(\PPsi^{-}(\cdot,z)) \|_{\mathcal{B}_2}+\| \partial_r T_1(\Psi^{-}(\cdot,z)) \|_{\mathcal{B}_1^{\prime}} +\| \partial_r T_2(\PPsi^{-}(\cdot,z)) \|_{\mathcal{B}_2^{\prime}}\lesssim \varepsilon_{\infty}. $$

\section{The $L^2$ solution near $\infty$ for large $|\xi|$}
\label{sec:nearinfty-large-xi}
In this section, we establish the existence of the Weyl-Titchmarsh matrix solutions for large $|\xi|,$ (i.e., large real part and small imaginary part) to $$
    i \mathcal{L} \Psi^{\pm}(r,z)= z \Psi^{\pm}(r,z),$$
    for $ z= \pm \frac{\sqrt{17}}{8}+\xi \in \Omega$ with  $1<\Lambda_0<|\xi|$, and $\Lambda_0\gg1$, equivalently, 
    
\begin{equation}
\label{eq:iL_E= V_E}
   i (\mathcal{L}_{E}-z ) \Psi(r, z) = V_{E}(r) \Psi(r, z)
\end{equation}
where $\Psi^{+}(r,z) \in L^2((1,\infty)),$ when $\im(z)>0,$ and $\Psi^{-}(r,z) \in L^2((1,\infty)),$ when $\im(z)<0,$ 
\begin{align*}
   \mathcal{L}_{E}:= \begin{pmatrix}
        0 &  L_1^{E}\\
        -L_2^{E} & 0 
    \end{pmatrix}, \quad L_1^{E}:=\frac{1}{2} L_0^E + \frac{1}{8},
    \quad
    L_2^E:=  \frac{1}{2}  L_0^E + \frac{17}{8 }, \quad 
    L_0^E:= - \partial_r^2+ \frac{3}{4r^2} ,
\end{align*} 
and
\begin{align}
\label{eq:defV_E}
    V_{E}(r) := \begin{pmatrix}
        0 & \frac{3}{8} ( \frac{1}{r^2} - \frac{1}{ \sh(r)^2}) + V_1(r) \\
        - ( \frac{3}{8} ( \frac{1}{r^2} - \frac{1}{ \sh(r)^2}) + V_2(r)) & 0 
    \end{pmatrix}.
\end{align}
The key observation is that our analysis relies on the fact that, in the large energy regime, the Euclidean operator $\mathcal{L}_E$ provides a good approximation to the operator $\mathcal{L}$ on $\mathbb{H}^2$.\\

We seek to construct solutions $\Psi^{\pm}(r,z)$ so that it behaves asymptotically as $r \to \infty$ like the solutions $\Uppsi^{\pm}_{E}(r,z) \in L^2((1,\infty))$ to 
\begin{align*}
    i \mathcal{L}_{E} \Uppsi^{\pm}_{E}(r,z) = z \Uppsi^{\pm}_{E}(r,z).
\end{align*}

\subsection{Bessel functions and Stokes phenomenon}
The aim of this subsection is to recall key properties of the Jost solutions associated with the scalar Bessel operator  $L_0^E := -\partial_r^2 + \frac{3}{4r^2}.$  We consider the differential equation 
\begin{align}
    \label{eq:ode-bessel}
   - h^{\prime \prime}(z) + \frac{3}{4 z^2} h(z) = h(z)
\end{align}
whose fundamental system consists of the modified Hankel functions of the second kind: \\ $\{ \sqrt{z} H^{(1)}_{1}(z),  \sqrt{z} H^{(1)}_{2}(z) \}$ as a fundamental system. Our objective is to establish uniform asymptotic estimates for these functions in the closed upper half-plane.

\begin{lemma} 
\label{lem:h-properties} 
    The  equation 
   \begin{equation}
   \label{eq:ODE-Bessel-2}
  - h^{\prime \prime}(z) + \frac{3}{4 z^2} h(z) = h(z)
   \end{equation}
    admits a holomorphic  fundamental system $\{ h_+(z), h_-(z)\}$ in the upper half-plane $\im (z)>0$. The recessive branch $h_+(z)$ satisfies
    \begin{equation}
        \begin{aligned}
        \label{eq:H0rec}
    h_+ (z) &\sim e^{ iz} \text{\ as\ }|z|\to\infty\\
    \big| e^{-iz} h_+ (z)\big| &\lesssim 1, \quad \big| {z^2} (e^{- iz} h_+ (z))'\big| &\lesssim 1, \quad \big| {z^3} (e^{- iz} h_+ (z))''\big| \lesssim 1,
    \end{aligned}
    \end{equation}
    uniformly in $z\in\Omega_+:=\{ z\in\C \::\: \im z\ge0,\; |z|\ge1\}$.
   The dominant branch satisfies  
 \begin{equation}
        \begin{aligned}
        \label{eq:H0dom}
    h_- (z) &\sim e^{ -iz} \text{\ as\ }|z|\to\infty\\
    \big| e^{iz} h_- (z)\big| &\le C,\quad |e^{iz}h_{-}'(z)|\le C, \quad \big| {z^2} (e^{ iz} h_- (z))'\big| &\le C_\delta
    \end{aligned}
    \end{equation}
    uniformly in 
    \[ z\in\Omega_{+,\delta}:=\{ z \in \C \::\: \im (z)\ge0,\; |z|\ge1,\; \im( z) \geq -\delta \re (z) \}\]
    where $\delta\in(0,1)$ is arbitrary but fixed. Here $C$ is an absolute constant, while $C_\delta\to\infty$ as $\delta\to 0$. In fact, 
   \begin{equation}
       \label{eq:stokes}
       h_-(x) \sim e^{-ix} + 2i e^{ix}\text{\ \ as\ \ } x\to-\infty
   \end{equation}
     An analogous statement holds in the lower half-plane.  
    \end{lemma}
\begin{proof}
First, we consider a solution in the form $h_{-}(z):=e^{-i z } z^{\frac{3}{2}} g(z).$ Plugging this onto the equation \eqref{eq:ode-bessel}, we have 
\begin{align}
\label{eq:f}
 g^{\prime \prime}(z)+ (\frac{3}{z}-2i) g^{\prime} - \frac{3i}{z} g  =0  
\end{align}
To solve this ODE in the right half-plane $ \re( z) > 0$, we employ the Laplace transform \begin{align*}
    g(z)= \int_0^{\infty} F(p) e^{- p z} dp.
\end{align*}
Substituting this into \eqref{eq:f} and using integration by parts, one can see that $F$ satisfies 
\begin{align*}
  p(p+2i) F(p)  -3 \int_0^p F(q) q dq - 3i \int_0^p F(q) dq=0  
\end{align*}
Differentiating, we obtain a first-order ODE for $F$, whose solution is given by 
\begin{align*}
    F(p)=C_0 \sqrt{p(p+2i)},
\end{align*}
for some complex $C_0 \neq 0.$ Therefore, we obtain a dominant branch $h_{-}(z)$ in $\re(z)>0$ of the form 
\begin{align*}
    h_{-}(z)= C_0 e^{- i z} z^{\frac{3}{2}} \int_0^{\infty} e^{-pz}  \sqrt{p(p+2i)} dp. 
\end{align*}
The remainder of the proof is very similar to \cite[Lemma 4.2]{LSS25} and we will be brief.  Matching the asymptotics of the integral above with the desired $h_{-}(z) \sim e^{-iz}$ we get $C_0:=e^{-
i \frac{\pi}{4}} \frac{\sqrt{2}}{\sqrt{\pi}}.$
Using the Residue theorem, we can shift the $p$-contour to $e^{-i \theta} (0,\infty)$ for $0 \leq \theta < \frac{\pi}{2}.$ Note that the circular arcs vanishes. Therefore, we analytically extend $h_{-}(z)$ to  $\re(e^{-i \theta} z)>0.$ Moreover, using change of variables we obtain for $z \in \Omega_{+,\delta}$
\begin{align*}
h_{-}(z)=C_0 e^{-iz} \int_{e^{i\beta}(0,\infty)} e^{-t} \sqrt{t(2i+ \frac{t}{z})} dt,
\end{align*}
provided $0<\beta < \frac{\pi}{2}$ is sufficiently close to $\frac{\pi}{2}.$  In order to compute $\lim_{y \to 0^{+}} h_{-}(x+iy)$ for $x<-1,$ we rotate the $p$-contour further to the negative half-axis in the clockwise orientation, i.e., to $\gamma:=e^{-i (\pi-\delta)}(0,\infty),$ with $\delta>0$ small. Since in this process we cross the branch cut $[-2i,-i\infty),$ then we shift the contour around this branch cut. We denote by $\Gamma$ the Hankel contour going $-i \infty$ circling around $-2i$ counter-clockwise, and then returning to $-i\infty.$ See figure below for contour. Note that, in view of the exponentially decaying factor the contribution of the arcs contour vanishes as $R \to \infty.$
Then by Cauchy's theorem we have for all $z \in \Omega_{+},$ 
\begin{align*}
   h_{-}(z)=C_0 e^{-iz} z^{\frac{3}{2}} \int_{\gamma} e^{-pz}  f(p) dp - C_0 e^{-iz} z^{\frac{3}{2}} \int_{\Gamma} e^{-pz}  f(p) dp .
\end{align*}
The desired assertions now follow from the asymptotics of these integrals, which can be obtained as in \cite[Lemma 4.2]{LSS25}. The proof for $h_{+}$ is similar.
\end{proof}

Note that we can identity $h_{\pm}(z)$ with the modified Hankel functions 
\begin{align}
\label{eq:hpm-Hankel}
        h_+(z)= \sqrt{\frac{\pi}{2}}e^{i\frac{\pi}{4}}  \sqrt{z}\, H_1^{(1)}(z),\qquad h_-(z)= \sqrt{\frac{\pi}{2}}e^{-i\frac{\pi}{4}}  \sqrt{z}\, H_1^{(2)}(z) \quad \text{for all} \quad \im (z) > 0.
\end{align}

\subsection{Construction of solution at infinity}
We aim to construct solutions $\Uppsi^{\pm}_{E}(r,z)$ to  $ ( i \mathcal{L}_{E}-z ) \Uppsi^{\pm}_{E}(r,z) =0$ for $\pm \im(z) >0,$ with $\Uppsi^{\pm}_{E}(\cdot,z) \in L^2(1,\infty).$ Let $\Phi(r,z)=\begin{pmatrix}
    \phi(r,z) \\
    \psi(r,z)
\end{pmatrix}$ satisfying $i \mathcal{L}_{E}  \Phi(r,z)=z \Phi(r,z) ,$ that is,   \begin{align*}
\begin{cases}
   i  \, L_1^{E} \psi(r,z) = z \phi(r,z) , \\
   -i\,  L_2^{E} \phi(r,z)  =  z \psi(r,z).    
\end{cases}
\end{align*}
Thus $\phi(r,z)$ must satisfy the fourth-order equation, 
\begin{align}
\label{L1L2-infty-large-xi}
    L_1^{E}L_2^{E}\phi(r,z) = z^2 \phi(r,z)
\end{align}
and the corresponding choice of $\psi(r,z)$ is determined by 
\begin{align*}
   \psi(r,z) = \frac{-i }{ z}  L_2^{E} \phi(r,z) .
\end{align*}
Imposing the ansatz $\phi(r,z)=h_{+}(k(z)r)$ for the solution to the equation \eqref{L1L2-infty-large-xi}, leads to same the equation in \eqref{eq:P(k,z)}, i.e,  
\begin{align*}
P(k,z):&=    \frac{1}{4} k^4 + \frac{9}{8} k^2 +\frac{17}{64}  - z^2 =0, \quad \text{with } \; k \equiv k(z)  .
\end{align*}
Therefore, we obtain the existence of the four distinct roots $k_j(z)$ for $j=1,\cdots,4$ whose properties are described in Lemma \ref{lem:k_j-behavior}. Then, we obtain the Weyl-Titchmarsh matrix solutions of the operator $\mathcal{L}_E,$ for large $\xi$

\begin{align*}
     \Uppsi^{+}_{E}(r,z) :=\begin{pmatrix}
         \Uppsi^{+}_{(E,1)}(r,z) &  \Uppsi^{+}_{(E,2)}(r,z)
    \end{pmatrix} &= \begin{pmatrix}
      h_{+}(k_1(z)r) &      h_{+}(k_2(z)r) \\ \\
       c_1 (z)  h_{+}(k_1(z)r) &  c_2 (z)  h_{+}(k_2(z)r)
    \end{pmatrix} \qquad \im(z) > 0  ,
\end{align*}

and 
\begin{align*}
    \Uppsi^{-}_{E}(r,z) := \begin{pmatrix}
         \Uppsi^{-}_{(E,1)}(r,z) &  \Uppsi^{-}_{(E,2)}(r,z)
    \end{pmatrix} =\begin{pmatrix}
      h_{+}(k_1(z)r) &      h_{+}(k_2(z)r) \\ \\
       c_1(z)  h_{+}(k_1(z)r) & c_2(z)   h_{+}(k_2(z)r)
    \end{pmatrix}\qquad \im(z) < 0 ,
\end{align*}
where,  \begin{align*}
   c_j(z)= c(k_j(z),z)&:=     - i     \frac{ ( \frac{1}{2} k_j(z)^2 + \frac{17}{8}) }{ z} \qquad \text{for} \; j=1,2.
\end{align*}

Recall that, $c_j(-z)=-c_j(z).$ In addition, we provide expressions for the other two solutions 

\begin{align*}
   \UUppsi^{+}_{E}(r,z) := \begin{pmatrix}
         \Uppsi^{+}_{(E,3)}(r,z) &  \Uppsi^{+}_{(E,4)}(r,z)
    \end{pmatrix} =\begin{pmatrix}
      h_{-}(k_1(z)r) &     h_{-}(k_2(z)r) \\ \\
      c_1(z) h_{-}(k_1(z)r) & c_2(z) h_{-}(k_2(z)r) 
    \end{pmatrix}, \qquad \im(z) > 0 
\end{align*}

and 
\begin{align*}
     \UUppsi^{-}_{E}(r,z) :=\begin{pmatrix}
         \Uppsi^{-}_{(E,3)}(r,z) &  \Uppsi^{-}_{(E,4)}(r,z)
    \end{pmatrix} = \begin{pmatrix}
      h_{-}(k_1(z)r) &     h_{-}(k_2(z)r) \\ \\
       c_1(z)  h_{-}(k_1(z)r) & c_2(z) h_{-}(k_2(z)r)
    \end{pmatrix}\qquad \im(z) < 0 .
\end{align*}

\subsection{Construction of $\Psi^{+}(r,z)$ for large $|\xi|$.} We start with the construction of the fundamental matrix solutions $\Psi^{+}(\cdot,z)\in L^2(1,\infty)$ and $\widetilde{\Psi}^{+}(\cdot, z) \notin L^2(1,\infty) $ to the equation \eqref{eq:iL_E= V_E}, 
$$
   i (\mathcal{L}_{E}-z ) \Psi(r, z) = V_{E}(r) \Psi(r, z),  \; \text{ for } z \in \Omega \; \text{ and } \; \im(z)>0,
$$
where in view of the definition of $V_E(r)$ in \eqref{eq:defV_E} and the decay of $V_1(r)$ and $V_2(r),$ we have
\begin{align}
\label{eq:decay-V_E}
|V(r)| \lesssim \langle r \rangle^{-2}
\end{align}

We define the Green's functions as 
\begin{align*}
   \mathcal{R}^{+}_{E,1}(r,s,z):&=  \Uppsi^{+}_{E}(r,z) S_1(s,z) \mathbb{1}_{\{ \tr_{\infty} \leq s \leq r \}} +  \UUppsi^{+}_{E}(r,z) T_1(s,z) \mathbb{1}_{\{ r \leq s \leq \infty \}}   \\
  \mathcal{R}^{+}_{E,2}(r,s,z):&=  \Uppsi^{+}_{E}(r,z) S_2(s,z) \mathbb{1}_{\{ r \leq s \leq \infty \}} +  \UUppsi^{+}_{E}(r,z) T_2(s,z)  \mathbb{1}_{\{ r \leq s \leq \infty \}} \\
 \mathcal{R}^{+}_{E,3}(r,s,z):&=  \Uppsi^{+}_{E}(r,z) S_3(s,z) \mathbb{1}_{\{ \tr_{\infty} \leq s \leq r \}} +  \UUppsi^{+}_{E}(r,z) T_3(s,z)  \mathbb{1}_{\{ \tr_{\infty} \leq s \leq r \}} \\
\end{align*}
where we require the matrices $S_1(r,z)$ and $T_1(r,z)$ for $\mathcal{R}^{+}_{E,1}$ to satisfy
\begin{align}
\label{eq:S1T1-condi-infinity-xi-large}
\begin{pmatrix}
      \Uppsi^{+}_{E}(r,z)    &  \UUppsi^{+}_{E}(r,z) \\
      \partial_r   \Uppsi^{+}_{E}(r,z) & \partial_r   \UUppsi^{+}_{E}(r,z)
\end{pmatrix}
\begin{pmatrix}
    S_1(r,z) \\
    -T_1(r,z)
\end{pmatrix}
= \begin{pmatrix}
    0 \\ \sigma_2
\end{pmatrix}
\end{align}
and we require the matrices $S_2(r,z)$ and $T_2(r,z)$ for $\mathcal{R}^{+}_{E,2}$ to satisfy
\begin{align}
\label{eq:S2T2-condi-infinity-xi-large}
\begin{pmatrix}
      \Uppsi^{+}_{E}(r,z)    &  \UUppsi^{+}_{E}(r,z) \\
      \partial_r   \Uppsi^{+}_{E}(r,z) & \partial_r   \UUppsi^{+}_{E}(r,z)
\end{pmatrix}
\begin{pmatrix}
  -  S_2(r,z) \\
    -T_2(r,z)
\end{pmatrix}
= \begin{pmatrix}
    0 \\ \sigma_2
\end{pmatrix}
\end{align}

and we require the matrices $S_3(r,z)$ and $T_3(r,z)$ for $\mathcal{R}^{+}_{E,3}$ to satisfy
\begin{align}
\label{eq:S3T3-condi-infinity-xi-large}
\begin{pmatrix}
      \Uppsi^{+}_{E}(r,z)    &  \UUppsi^{+}_{E}(r,z) \\ 
      \partial_r   \Uppsi^{+}_{E}(r,z) & \partial_r   \UUppsi^{+}_{E}(r,z)
\end{pmatrix}
\begin{pmatrix}
    S_3(r,z) \\
    T_3(r,z)
\end{pmatrix}
= \begin{pmatrix}
    0 \\ \sigma_2
\end{pmatrix}
\end{align}

Denote by \begin{align*}
\Psi^{\pm}(r,z)= \begin{bmatrix} \Psi^{\pm}_1 (r,z) &    \Psi^{\pm}_2(r,z)
    \end{bmatrix}
    \qquad \text{and} \qquad 
    \widetilde{\Psi}^{\pm}(r,z)= \begin{bmatrix}   \Psi^{\pm}_3 (r,z) & \PPsi^{\pm}_4(r,z)  
    \end{bmatrix}
\end{align*} 
and  
\begin{align*}
  \Upsilon_{i,E}(r,z,\Psi_j (r,z)):= \int_0^\infty \mathcal{R}^{+}_{E,i}(r,s,z) V(s) \Psi_j(s,z)  ds  \qquad \text{for } i=1,2, 3 \text{ and } j=1,\cdots, 4. 
\end{align*}
where $\Psi_j(r,z) $ is the $j$-th column of the matrix $\Psi(r,z).$ \\ 

In both scenarios, for large $|\xi|$, namely $\re(z) > 0$ (i.e., $z = \frac{\sqrt{17}}{8} + \xi$) and $\re(z) < 0$ (i.e., $z = -\frac{\sqrt{17}}{8} + \xi$), the solution is given as a solution to the fixed-point problem:

\begin{align}
\label{eq:def-Psi+(r,z)-xi-large}
   \Psi^{+}(r,z)&:=   \Uppsi^{+}_{E} (r,z) + 
  T_{E,1}(r,z,\Psi^{+}(r,z))  
  \end{align}
and 
\begin{align}
\label{eq:def-tilde-Psi+(r,z)-xi-large}
   \PPsi^{+}(r,z)&:=    \UUppsi^{+}_{E} (r,z) + 
  T_{E,2}(r,z,\PPsi^{+}(r,z)) 
  \end{align}

\begin{align}
\begin{split}
\label{eq:def-T_1-infty-xi-large}
  T_{E,1}(r,z,\Psi^{+}(r,z))&:= \Upsilon_{E,1}(r,z,\Psi_1^{+} (r,z))  + \Upsilon_{E,2}(r,z,\Psi_{2}^{+} (r,z)) \\
  & =\int_0^{\infty} \mathcal{R}^{+}_{E,1}(r,s,z) V_E(s) \Psi_{1}^{+}(s,z) ds  + \int_0^{\infty} \mathcal{R}^{+}_{E,2}(r,s,z) V_E(s) \Psi_{2}^{+}(s,z) ds  
\end{split}
\end{align}
\begin{align}
\label{eq:def-T_2-infty-xi-large}
  T_{E,2}(r,z,\PPsi^{+}(r,z))&:= \Upsilon_{E,1}(r,z,\Psi_{3}^{+} (r,z))  + \Upsilon_{E,3}(r,z,\PPsi_{4}^{+} (r,z)) \\ \nonumber
  & =\int_0^{\infty} \mathcal{R}^{+}_{E,1}(r,s,z) V_E(s) \Psi_{3}^{+}(s,z) ds  + \int_0^{\infty} \mathcal{R}^{+}_{E,3}(r,s,z) V_E(s) \PPsi_{4}^{+}(s,z) ds .
\end{align}
To compute \( S_i(r,z) \) and \( T_i(r,z) \), we start by inverting the matrix introduced above. This requires computing the Wronskians between \( \Uppsi^{\pm}_{E}(\cdot, z) \) and \( \UUppsi^{\pm}_{E}(\cdot, z) \).
\begin{claim}
\label{claim:D-Psi-E} 
Let $D_{E}^{\pm}(z):=W(\Uppsi^{\pm}_{E}(\cdot,z), \UUppsi^{\pm}_{E}(\cdot,z)).$ Then 
\begin{align*}
 D_{E}^{\pm}(z)=\begin{pmatrix}
 \pm \delta(z) &0  \\ 
0   &   \alpha(z)
   \end{pmatrix} \qquad \text{and} \qquad 
   (D_{E}^{\pm})^{-1}=(D_{E}^{\pm})^{-t} = \begin{pmatrix}
\frac{\pm1}{\delta(z)}  & 0 \\
0   &   \frac{1}{\alpha(z)}
   \end{pmatrix} := \begin{pmatrix}
D_{E,1}(z) &0  \\ 
0  &    D_{E,2}(z) 
   \end{pmatrix}.
\end{align*}
where 
\begin{align*}
   \delta(z):&= -2ik_1(z) ( 1 - c_1(z)^2), \quad  c_1(z)=     - i     \frac{ ( \frac{1}{2} k_1(z)^2 + \frac{17}{8}) }{ z}  \\
   \alpha(z):&= - 2i k_2(z) (1 - c_2(z)^2), \quad  c_2(z)=    - i   \frac{ ( \frac{1}{2} k_2(z)^2 + \frac{17}{8}) }{ z}
\end{align*}
\end{claim}
\begin{proof}

We only compute $D_{E}^{+}(z),$ and to obtain $D_{E}^{-}(z) $ one can use the fact that $k_3(z)=-k_1(z).$ 
Let 
\begin{align*}
   D_{E}^{+}(z)&= \begin{pmatrix}
       \delta(z) & \gamma(z) \\
       \beta(z) & \alpha(z)
   \end{pmatrix}.
\end{align*}
Recall that $h_{\pm}(k_j(z)) $ are two linearly independent solutions to the following ODE 
\begin{align*}
    -u^{\prime \prime} + \frac{3}{4 r^2} u = k_j(z)^2 u.
\end{align*}
Therefore, their Wronskians must be independent of $r,$ and can be evaluated using their asymptotic behavior near $0,$ or infinity 
We have \begin{align*}
    \delta(z) & =  W(\Uppsi^{+}_{(E,1)}(\cdot,z), \Uppsi^{+}_{(E,3)}(\cdot,z)) = W( \begin{pmatrix} h_{+}(k_1(z) r)  \\    c_1(z) h_{+}(k_1(z) r)  \end{pmatrix}, \begin{pmatrix}
      h_{-}(k_1(z) r) \\
        c_1(z) h_{-}(k_1(z) r) 
    \end{pmatrix} )\\
    & = k_1(z) h_{+}(k_1(z) r)h_{-}^{\prime}(k_1(z) r) -   k_1(z) h_{+}^{\prime}(k_1(z) r) h_{-}(k_1(z) r)  \\
    &- k_1(z) c_1(z)^2 h_{+}(k_1(z) r) h_{-}^{\prime}(k_1(z) r) + - k_1(z) c_1(z)^2 h_{+}^{\prime}(k_1(z) r)h_{-}(k_1(z) r)   \\
 &   = -2 i k_1(z) (1 + c_1(z)^2)  .
\end{align*}

 \begin{align*}
    \gamma(z) & =  W(\Uppsi^{+}_{(E,1)}(\cdot,z), \Uppsi^{+}_{(E,4)}(\cdot,z)) = W( \begin{pmatrix} h_{+}(k_1(z) r)  \\    c_1(z) h_{+}(k_1(z) r)  \end{pmatrix}, \begin{pmatrix}
      h_{-}(k_2(z) r) \\
        c_2(z) h_{-}(k_2(z) r) 
    \end{pmatrix} )\\
    & = k_2(z) h_{+}(k_1(z) r)h_{-}^{\prime}(k_2(z) r) -   k_1(z) h_{+}^{\prime}(k_1(z) r) h_{-}(k_2(z) r)  \\
    &- k_2(z) c_1(z)c_2(z) h_{+}(k_1(z) r) h_{-}^{\prime}(k_2(z) r) + - k_1(z) c_1(z) c_2(z)  h_{+}^{\prime}(k_1(z) r)h_{-}(k_2(z) r)   \\
    &= -i (k_2(z) + k_1(z)) (1 + c_1(z)c_2(z)) =0,
\end{align*}
where we use the fact $ (1 + c_1(z)c_2(z)) = 0$ by definitions of $k_1(z)$ and $k_2(z).$ 
Similarly, we have 
 \begin{align*}
    \beta(z) & =   W(\Uppsi^{+}_{(E,2)}(\cdot,z), \Uppsi^{+}_{(E,3)}(\cdot,z)) = W( \begin{pmatrix} h_{+}(k_2(z) r)  \\    c_2(z) h_{+}(k_2(z) r)  \end{pmatrix}, \begin{pmatrix}
      h_{-}(k_1(z) r) \\
        c_1(z) h_{-}(k_1(z) r) 
    \end{pmatrix} )\\
&=     -i (k_2(z) + k_1(z)) (1 + c_1(z)c_2(z)) =0
\end{align*}
\begin{align*}
    \alpha(z) & =   W(\Uppsi^{+}_{(E,2)}(\cdot,z), \Uppsi^{+}_{(E,4)}(\cdot,z)) = W( \begin{pmatrix} h_{+}(k_2(z) r)  \\    c_2(z) h_{+}(k_2(z) r)  \end{pmatrix}, \begin{pmatrix}
      h_{-}(k_2(z) r) \\
        c_2(z) h_{-}(k_2(z) r) 
    \end{pmatrix} )\\
    & = k_2(z) h_{+}(k_2(z) r)h_{-}^{\prime}(k_2(z) r) -   k_2(z) h_{+}^{\prime}(k_2(z) r) h_{-}(k_2(z) r)  \\
    &- k_2(z) c_2(z)^2 h_{+}(k_2(z) r) h_{-}^{\prime}(k_2(z) r) + - k_1(z) c_2(z)^2 h_{+}^{\prime}(k_2(z) r)h_{-}(k_2(z) r)   \\
 &   = -2 i k_2(z) (1 + c_2(z)^2)  
\end{align*}
Then 
\begin{align*}
  (D_{E}^{+}(z))^{-1} &= \frac{1}{\delta(z) \alpha(z)} \begin{pmatrix}
\alpha(z)  & 0  \\ 
0   &  \delta(z)
   \end{pmatrix} =\begin{pmatrix}
 \frac{1}{\delta(z)} & 0 \\ 
0   &     \frac{1}{\alpha(z)}
   \end{pmatrix} := \begin{pmatrix}
D_{E,1}(z) &0  \\ 
0  &    D_{E,2}(z) 
\end{pmatrix}
\end{align*}
\end{proof}

\begin{remark}
Note that the Wronskians $D_E^{\pm}(z)$, defined for large $\xi$ as the Wronskians between $\Uppsi^{\pm}_{E}(\cdot,z)$ and $\UUppsi^{\pm}_{E}(\cdot,z)$, coincide with $D^{\pm}(z)$, the corresponding Wronskians for small $\xi$ between $\Uppsi^{\pm}_{\infty}(\cdot,z)$ and $\UUppsi^{\pm}_{\infty}(\cdot,z)$. This equivalence follows from the fact that $\Uppsi^{\pm}_{E}$ and $\Uppsi^{\pm}_{\infty}$ share the same asymptotic behavior. Accordingly, we use the notation $D^{\pm}(z)$ throughout the paper to refer to both cases. 
\end{remark}

\begin{remark}
\label{K_j-C_j-D_j-behavior-large-xi}
Note that , in the following two scenarios: 
\begin{itemize}
  \item  Let $z= \frac{\sqrt{17}}{8}+\xi,$ $\xi=x+iy$ with $x$ large and $y$ small.  Then we have $k_1(z)^2 =  2|x|(1+O(|\xi|^{-1}))$ and $k_2(z)^2= -2|x|(1+O(|\xi|^{-1})) $  This implies, $ c_1(z)=-i +O(|\xi|^{-1})  $ and $c_2(z)=i+O(|\xi|^{-1}).$ 
    Therefore, $|D_1^{\pm}(z)|\simeq \frac{1}{\sqrt{|\xi|}} $ and $|D_2^{\pm}(z)| \simeq \frac{1}{\sqrt{|\xi|}}.$
 \item Let $z= -\frac{\sqrt{17}}{8}+\xi,$ $\xi=x+iy$ with $|x|$ large and $y$ small. Then we also have $|D_1^{\pm}(z)|\simeq \frac{1}{\sqrt{|\xi|}} $ and $|D_2^{\pm}(z)| \simeq \frac{1}{\sqrt{|\xi|}} $.
\end{itemize}    
\end{remark}
 Next, we compute  $S_i(r,z)$ and $T_i(r,z)$, using Lemma \ref{inver-2-matrix} and Claim \ref{claim:D-Psi-E}. 

\begin{claim}
     We have, 
    \begin{align*}
 &   \begin{cases}  
    S_1(r,z)&= -  (D^{+}(z))^{-t} \, \UUppsi^{+}_{E}(r,z)^t \sigma_3 \sigma_2 \\
    T_1(r,z)&= -  (D^{+}(z))^{-1} \,  \Uppsi^{+}_{E}(r,z)^t \sigma_3 \sigma_2
    \end{cases}
    \qquad \qquad 
\begin{cases}
        S_2(r,z)&=   (D^{+}(z))^{-t} \,  \UUppsi^{+}_{E}(r,z)^t \sigma_3 \sigma_2 \\
    T_2(r,z)&= -  (D^{+}(z))^{-1} \,  \Uppsi^{+}_{E}(r,z)^t \sigma_3 \sigma_2
\end{cases} \\
& \begin{cases}
    S_3(r,z)&= -  (D^{+}(z))^{-t} \,  \UUppsi^{+}_{E}(r,z)^t \sigma_3 \sigma_2 \\
    T_3(r,z)&=   (D^{+}(z))^{-1} \,  \Uppsi^{+}_{E}(r,z)^t \sigma_3 \sigma_2
    \end{cases}
    \end{align*}
    
and \begin{align*}
\mathcal{R}^{+}_{E,1}(r,s,z):&= \begin{cases}
    i    \Uppsi^{+}_{E}(r,z)  (D^{+}(z))^{-t} \,  \UUppsi^{+}_{E}(s,z)^t \sigma_1, \qquad \tr_{\infty} \leq s \leq r, \\ \\
   i   \UUppsi^{+}_{E}(r,z)  (D^{+}(z))^{-1} \, \Uppsi^{+}_{E}(s,z)^t \sigma_1 , \qquad r \leq s \leq \infty.
\end{cases} \\ \\
\mathcal{R}^{+}_{E,2}(r,s,z):&= \begin{cases}
  -  i    \Uppsi^{+}_{E}(r,z)  (D^{+}(z))^{-t} \,  \UUppsi^{+}_{E}(s,z)^t \sigma_1, \qquad r \leq s \leq \infty, \\ \\
   i   \UUppsi^{+}_{E}(r,z)  (D^{+}(z))^{-1} \, \Uppsi^{+}_{E}(s,z)^t \sigma_1,  \qquad r \leq s \leq \infty.
\end{cases}  \\ \\
\mathcal{R}^{+}_{E,3}(r,s,z):&= \begin{cases}
   i    \Uppsi^{+}_{E}(r,z)  (D^{+}(z))^{-t} \,  \UUppsi^{+}_{E}(s,z)^t \sigma_1, \qquad \tr_{\infty} \leq s \leq r,\\ \\
 -  i   \UUppsi^{+}_{E}(r,z)  (D^{+}(z))^{-1} \, \Uppsi^{+}_{E}(s,z)^t \sigma_1,  \qquad \tr_{\infty} \leq s \leq r.
\end{cases}
\end{align*}
    
\end{claim}
\begin{proof}
 Since the proof is very similar to that of Claim \ref{claim:def-S_iT_i-Greens-R_j} for the small $\xi$ regime, we omit the details. 
\end{proof}

\begin{defi}
Let $F=\begin{bmatrix}
    F_1 & F_2
\end{bmatrix}, $ where $F_i$ are the columns of $F$ and let $\tr_\infty$ be a fixed number:   
\begin{align*}
   \left\| F \right\|_{ \mathcal{\tilde{B}}_1} &:= \sup_{r \geq \tr_{\infty}} 
   |e^{-ik_1(z)r}  F_1(r)| + \sup_{r \geq \tr_{\infty}} |e^{-ik_2(z)r }  F_{2}(r)| \\
   \left\| F \right\|_{\mathcal{\tilde{B}}_2} &:= \sup_{r \geq \tr_{\infty}} 
   |e^{ik_1(z) r }  F_{1}(r)| + \sup_{r \geq \tr_{\infty}}  |e^{ik_2(z)r}  F_{2}(r)| \\
   \left\| F \right\|_{\mathcal{\tilde{B}}^{\prime}_1} &:= \sup_{r \geq \tr_{\infty}} 
   \frac{1}{|k_1(z)|}|e^{-ik_1(z)r}  F_1(r)| + \sup_{r \geq \tr_{\infty}} \frac{1}{|k_2(z)|}|e^{-ik_2(z)r }  F_{2}(r)| \\
   \left\| F \right\|_{\mathcal{\tilde{B}}^{\prime}_2} &:= \sup_{r \geq \tr_{\infty}} 
   \frac{1}{|k_1(z)|}|e^{ik_1(z) r }  F_{1}(r)| + \sup_{r \geq \tr_{\infty}}  \frac{1}{|k_2(z)|}|e^{ik_2(z)r}  F_{2}(r)| 
\end{align*} 
\end{defi}

\begin{prop}
\label{prop:contraction-T-j-large-xi-infty}
Let $1<\Lambda_0<|\xi|  ,$ and $ \tr_{\infty} :=\frac{ \tvarepsilon_{\infty} }{\sqrt{|\xi|}}$ where $ 0<\tvarepsilon_{0} <\tvarepsilon_{\infty}\ll1$ and $\Lambda_0\gg \tvarepsilon_{\infty}^{-1}$.

Then $T_{E,j}$ is a contraction on $\mathcal{\tilde{B}}_j$ for $j=1,2,$  i.e., $$ \| T_{E,j}(\Psi(\cdot,z)) \|_{\mathcal{\tilde{B}}_j} \lesssim \tvarepsilon_{\infty} \left\| \Psi(\cdot,z) \right\|_{\mathcal{\tilde{B}}_j} . $$
Moreover,  with $\Uppsi^+(\cdot,z)$ and ${\tilde{\Uppsi}}^+(\cdot,z)$ denoting the solutions to the fixed point problems \eqref{eq:def-Psi+(r,z)-xi-large} and \eqref{eq:def-tilde-Psi+(r,z)-xi-large}, we have $$ \|  T_{E,1}(\Uppsi^{+}(\cdot,z)) \|_{\mathcal{\tilde{B}}_1} +\|  T_{E,2}({\tilde{\Uppsi}}^{+}(\cdot,z)) \|_{\mathcal{\tilde{B}}_2}+\| \partial_r T_{E,1}(\Uppsi^{+}(\cdot,z)) \|_{\mathcal{\tilde{B}}_1^{\prime}} +\| \partial_r T_{E,2}({\tilde{\Uppsi}}^{+}(\cdot,z)) \|_{\mathcal{\tilde{B}}_2^{\prime}}\lesssim \varepsilon_{\infty}. $$
\end{prop}

\begin{proof}
We provide the proof that $T_{E,j}$ is a contraction on $\mathcal{\tilde{B}}_j,$ and the corresponding estimate for the derivative follows similarly by applying the same argument with the necessary conditions \eqref{eq:S1T1-condi-infinity-xi-large},\eqref{eq:S2T2-condi-infinity-xi-large} and \eqref{eq:S3T3-condi-infinity-xi-large}.

\begin{claim}
\label{claim:contraction-Upsilon-large-xi}
Let $\tr_{\infty}:=\frac{\tvarepsilon_{\infty}}{\sqrt{|\xi|}},$ $|\xi|>\Lambda_0.$ Then we have 
\begin{align*}
\sup_{r \geq \tr_{\infty}}  |e^{- ik_1(z) r} \Upsilon_{E,1}(r,z,\Psi^{+}_{1} (r,z)) |  & \lesssim_{\tvarepsilon_{\infty}} \frac{1}{\sqrt{|\xi|}}   \sup_{r \geq \tr_{\infty}} | e^{ - i k_1(z) r} \Psi^{+}_{1}(r,z)  | \\
\sup_{r \geq \tr_{\infty}}  |e^{- ik_2(z) r} \Upsilon_{E,2}(r,z,\Psi^{+}_{2} (r,z)) |  & \lesssim_{\tvarepsilon_{\infty}} \frac{1}{\sqrt{|\xi|}}   \sup_{r \geq \tr_{\infty}} | e^{ - i k_2(z) r} \Psi^{+}_{2}(r,z)  |  \\
\sup_{r \geq \tr_{\infty}}  |e^{ik_1(z) r} \Upsilon_{E,1}(r,z,\Psi^{+}_{3} (r,z)) |  & \lesssim_{\tvarepsilon_{\infty}} \frac{1}{\sqrt{|\xi|}}   \sup_{r \geq \tr_{\infty}} | e^{  i k_1(z) r} \Psi^{+}_{3}(r,z)  | \\
\sup_{r \geq \tr_{\infty}}  |e^{ik_2(z) r} \Upsilon_{E,3}(r,z, \PPsi^{+}_{4} (r,z)) |  & \lesssim_{\tvarepsilon_{\infty}} \frac{1}{\sqrt{|\xi|}}   \sup_{r \geq \tr_{\infty}} | e^{  i k_2(z) r} \PPsi^{+}_{4}(r,z)  |
\end{align*}
\end{claim}
\begin{proof}
 Recall that, we have 
\begin{align} \label{eq:upsilon_1-of-psi_1-large-xi-h_+}
\begin{split}
    \Upsilon_{E,1}(r,z,\Psi^{+}_{1} (r,z))&:= \int_0^{\infty} \mathcal{R}^{+}_{E,1}(r,s,z) V_E(s) \Psi^{+}_{1}(s,z) ds   \\ 
    & =\int_{\tr_{\infty}}^{r} i    \Uppsi^{+}_{E}(r,z)  (D^{+}(z))^{-t} \,  \UUppsi^{+}_{E}(s,z)^t \sigma_1  V(s) \Psi^{+}_{1}(s,z) ds \\
    & + \int_r^{\infty}   i  \UUppsi^{+}_{E}(r,z)  (D^{+}(z))^{-1} \,  \Uppsi^{+}_{E}(s,z)^t \sigma_1  V_E(s) \Psi^{+}_{1}(s,z) ds  \\
& =  i  \mathcal{M}^{+}_{1}(z) h_{+}( k_1(z) r ) \int_{\tr_\infty}^r  h_{-}( k_1(z) s )    V_E(s) \Psi^{+}_{1}(s,z) ds   \\
&+  i  \mathcal{M}^{+}_{2} (z) h_{+}( k_2 (z) r )\int_{\tr_\infty}^{r}  h_{-}( k_2(z) s )  V_E(s) \Psi^{+}_{1}(s,z) ds   \\
&+     i  \mathcal{M}^{+}_{1}(z)  h_{-}( k_1(z) r ) \int_{r}^{\infty}   h_{+}( k_1(z) s )    V_E(s) \Psi^{+}_{1}(s,z) ds \\
 & +  i  \mathcal{M}^{+}_{2}(z)   h_{-}( k_2(z) r ) \int_{r}^{\infty}  h_{+}( k_2(z) s )  V_E(s) \Psi^{+}_{1}(s,z) ds,
 \end{split}
\end{align}
where \begin{align*}
 \mathcal{M}^{+}_{1}(z):& =   \begin{pmatrix} 
 D_1^{+}(z)  c_1(z)   & 
 D_1^{+}(z)        \\
     D_1^{+}(z)  (c_1(z))^2    &   D_1^{+}(z)    c_1(z) 
\end{pmatrix} 
\qquad \text{and} \qquad 
\mathcal{M}^{+}_{2}(z)&:= \begin{pmatrix}
      D_2^{+}(z) c_2(z)  &     D_2^{+}(z)  \\ 
    D_2^{+}(z) (c_2(z))^2     &      D_2^{+}(z)  c_2(z)   
\end{pmatrix}   
\end{align*}
 
By remark \ref{K_j-C_j-D_j-behavior-large-xi}, one can see that in both scenarios for large $|\xi|$ we have  $|\mathcal{M}^{+}_{1}(z)| =O( \frac{1}{\sqrt{|\xi|}})$ and $|\mathcal{M}^{+}_{2}(z)|=O( \frac{1}{\sqrt{|\xi|}}).$ Combining this with the decay estimate \eqref{eq:decay-V_E}, namely, $V_{E}(r) \lesssim \langle r \rangle^{-2} ,$ we have

\begin{align}
\label{upsilon_1-of-psi_1-large-xi}
\begin{split}
    \Upsilon_{E,1}(r,z,\Psi^{+}_{1} (r,z))& \lesssim_{\tvarepsilon_{\infty}} \frac{1}{\sqrt{|\xi|}} |e^{i k_1(z) r }|   \int_{\tr_\infty}^r   |e^{-i k_1(z) s }|   \langle s \rangle^{-2}  | \Psi^{+}_{1}(s,z) |ds   \\
&+  \frac{1}{\sqrt{|\xi|}} |e^{i k_2(z) r }|  \int_{\tr_\infty}^{r}  |e^{-i k_2(z) s }|    \langle s \rangle^{-2} |  \Psi^{+}_{1}(s,z)  | ds   \\
&+  \frac{1}{\sqrt{|\xi|}} |e^{-i k_1(z) r }| \int_{r}^{\infty}    |e^{i k_1(z) s }|    \langle s \rangle^{-2} | \Psi^{+}_{1}(s,z) |  ds \\
 & +   \frac{1}{\sqrt{|\xi|}} |e^{-i k_2(z) r }|  \int_{r}^{\infty}  |e^{i k_2(z) s }|  \langle s \rangle^{-2} | \Psi^{+}_{1}(s,z) | ds.  
 \end{split}
\end{align}
Therefore, we have
 \begin{align*}
| e^{ -  i k_1(z) r } \Upsilon_{E,1}(r,z,\Psi^{+}_{1} (r,z)) |    & \lesssim_{\tvarepsilon_{\infty}}  \frac{1}{\sqrt{|\xi|}}   \int_{\tr_\infty}^r \langle s \rangle^{-2} ds    \sup_{ \tr_{\infty} \geq r } | e^{ -  i k_1(z) r }  \Psi^{+}_{1}(r,z) |  \\ 
& + \frac{1}{\sqrt{|\xi|}}   \int_{\tr_\infty}^r  |e^{i (k_2(z)-k_1(z)) (r-s) } |  \langle s \rangle^{-2} ds \sup_{r\geq \tr_{\infty}} | e^{ -  i k_1(z) r }  \Psi^{+}_{1}(r,z) |  \\ 
& + \frac{1}{\sqrt{|\xi|}}    \int_{r}^{\infty} |e^{2ik_1(s-r)}|  \langle s \rangle^{-2} ds  \sup_{r\geq \tr_{\infty}} | e^{ -  i k_1(z) r }  \Psi^{+}_{1}(r,z) |  \\ 
& +  \frac{1}{\sqrt{|\xi|}}   \int_r^{\infty} |   e^{i (k_2(z)+k_1(z) ) (s-r) }|  \langle s \rangle^{-2} ds \sup_{r\geq \tr_{\infty}} | e^{ -  i k_1(z) r }  \Psi^{+}_{1}(r,z) |
\end{align*}

using the behavior of $k_1(z)$ and $k_2(z)$ for large $\xi,$ we obtain
\begin{align*}
\sup_{r \geq \tr_{\infty}}  |e^{- ik_1(z) r} \Upsilon_{E,1}(r,z,\Psi^{+}_{1} (r,z)) |  \lesssim_{\tvarepsilon_{\infty}} \frac{1}{\sqrt{|\xi|}}   \sup_{r \geq \tr_{\infty}} | e^{ - i k_1(z) r} \Psi^{+}_{1}(r,z)  | .
\end{align*}

Next, we estimate $\Upsilon_{E,2}(r,z,\Psi^{+}_{2} (r,z)).$ We have 
\begin{align*}
    \Upsilon_{E,2}(r,z,\Psi^{+}_{2} (r,z))&:= \int_0^{\infty} \mathcal{R}^{+}_{E,2}(r,s,z) V_E(s) \Psi^{+}_{2}(s,z) ds   \\ 
    & =\int_{r}^{\infty} i    \Uppsi^{+}_{E}(r,z)  (D^{+}(z))^{-t} \,  \UUppsi^{+}_{E}(s,z)^t \sigma_1  V(s) \Psi^{+}_{2}(s,z) ds \\
    & + \int_r^{\infty}   i  \UUppsi^{+}_{E}(r,z)  (D^{+}(z))^{-1} \,  \Uppsi^{+}_{E}(s,z)^t \sigma_1  V_E(s) \Psi^{+}_{2}(s,z) ds  \\
& =  i  \mathcal{M}^{+}_{1}(z) h_{+}( k_1(z) r ) \int_{r}^{\infty}  h_{-}( k_1(z) s )    V_E(s) \Psi^{+}_{2}(s,z) ds   \\
&+  i  \mathcal{M}^{+}_{2} (z) h_{+}( k_2 (z) r ) \int_{r}^{\infty}  h_{-}( k_2(z) s )  V_E(s) \Psi^{+}_{2}(s,z) ds   \\
&+     i  \mathcal{M}^{+}_{1}(z)  h_{-}( k_1(z) r ) \int_{r}^{\infty}   h_{+}( k_1(z) s )    V_E(s) \Psi^{+}_{2}(s,z) ds \\
 & +  i  \mathcal{M}^{+}_{2}(z)   h_{-}( k_2(z) r ) \int_{r}^{\infty}  h_{+}( k_2(z) s )  V_E(s) \Psi^{+}_{2}(s,z) ds,
\end{align*}
Similarly, we have 
\begin{align}
\label{upsilon_2-of-psi_2-large-xi}
\begin{split}
 \Upsilon_{E,2}(r,z,\Psi^{+}_{2} (r,z))     & \lesssim_{\tvarepsilon_{\infty}}  \frac{1}{\sqrt{|\xi|}}  |e^{i k_1(z) r} | \int_{r}^{\infty} |e^{i (k_2(z)-k_1(z)) s } |  \langle s \rangle^{-2} ds    \sup_{ r \geq \tr_{\infty}  } | e^{ -  i k_2(z) r }  \Psi^{+}_{2}(r,z) |  \\ 
& + \frac{1}{\sqrt{|\xi|}}  | e^{i k_2(z) r } | \int_{r}^{\infty}    \langle s \rangle^{-2} ds \sup_{r\geq \tr_{\infty}} | e^{ -  i k_2(z) r }  \Psi^{+}_{2}(r,z) |  \\ 
& + \frac{1}{\sqrt{|\xi|}}   | e^{-i k_1(z) r } |  \int_{r}^{\infty} |   e^{i (k_2(z)+k_1(z) ) s }|  \langle s \rangle^{-2} ds  \sup_{r\geq \tr_{\infty}} | e^{ -  i k_2(z) r }  \Psi^{+}_{2}(r,z) |  \\ 
& +  \frac{1}{\sqrt{|\xi|}}  |e^{-ik_2(z) r}|  \int_r^{\infty} |e^{2ik_2(z) s}|   \langle s \rangle^{-2} ds \sup_{r\geq \tr_{\infty}} | e^{ -  i k_2(z) r }  \Psi^{+}_{2}(r,z) |
 \end{split}
\end{align} 
Thus, we get
 \begin{align*}
| e^{ -  i k_2(z) r } \Upsilon_{E,2}(r,z,\Psi^{+}_{2} (r,z)) |    & \lesssim_{\tvarepsilon_{\infty}}  \frac{1}{\sqrt{|\xi|}}  \int_{r}^{\infty} |e^{i (k_1(z)-k_2(z)) (r-s) } |  \langle s \rangle^{-2} ds    \sup_{ r \geq \tr_{\infty}  } | e^{ -  i k_2(z) r }  \Psi^{+}_{2}(r,z) |  \\ 
& + \frac{1}{\sqrt{|\xi|}}   \int_{r}^{\infty}    \langle s \rangle^{-2} ds \sup_{r\geq \tr_{\infty}} | e^{ -  i k_2(z) r }  \Psi^{+}_{2}(r,z) |  \\ 
& + \frac{1}{\sqrt{|\xi|}}    \int_{r}^{\infty} |   e^{i (k_2(z)+k_1(z) ) (s-r) }|  \langle s \rangle^{-2} ds  \sup_{r\geq \tr_{\infty}} | e^{ -  i k_2(z) r }  \Psi^{+}_{2}(r,z) |  \\ 
& +  \frac{1}{\sqrt{|\xi|}}   \int_r^{\infty} |e^{2ik_2(z) (s-r)}|   \langle s \rangle^{-2} ds \sup_{r\geq \tr_{\infty}} | e^{ -  i k_2(z) r }  \Psi^{+}_{2}(r,z) |.
\end{align*} 
Using the behavior of $k_1(z)$ and $k_2(z)$ for large $\xi,$ we obtain
\begin{align*}
\sup_{r \geq \tr_{\infty}}  |e^{- ik_2(z) r} \Upsilon_{E,2}(r,z,\Psi^{+}_{2} (r,z)) |  \lesssim_{\tvarepsilon_{\infty}} \frac{1}{\sqrt{|\xi|}}   \sup_{r \geq \tr_{\infty}} | e^{ - i k_2(z) r} \Psi^{+}_{2}(r,z)  | .
\end{align*}
Similarly, to the estimate for $\Upsilon_{E,1}(r,z,\Psi^{+}_{1} (r,z))$ and $\Upsilon_{E,2}(r,z,\Psi^{+}_{2} (r,z)),$ we obtain 
\begin{align*}
\sup_{r \geq \tr_{\infty}}  |e^{ik_1(z) r} \Upsilon_{E,1}(r,z,\Psi^{+}_{3} (r,z)) |  & \lesssim_{\tvarepsilon_{\infty}} \frac{1}{\sqrt{|\xi|}}   \sup_{r \geq \tr_{\infty}} | e^{  i k_1(z) r} \Psi^{+}_{3}(r,z)  | \\
\sup_{r \geq \tr_{\infty}}  |e^{ik_2(z) r} \Upsilon_{E,2}(r,z,\PPsi^{+}_{4} (r,z)) |  & \lesssim_{\tvarepsilon_{\infty}} \frac{1}{\sqrt{|\xi|}}   \sup_{r \geq \tr_{\infty}} | e^{  i k_2(z) r} \PPsi^{+}_{4}(r,z)  | .
\end{align*}
\end{proof}

Using claim \ref{claim:contraction-Upsilon-large-xi}, we obtain that $T_{E,j}$ is a contraction on $\mathcal{B}_j$, for $j=1,2.$ The derivative estimates for $T_{E,j}$ can be obtained using similar arguments, together with the conditions \eqref{eq:S1T1-condi-infinity-xi-large}, \eqref{eq:S2T2-condi-infinity-xi-large}, and \eqref{eq:S3T3-condi-infinity-xi-large}. This concludes the proof of Proposition \ref{prop:contraction-T-j-large-xi-infty}.

\end{proof}
Throughout the rest of the paper (for the non-resonance case), we define $\Psi_4^{+}(r,z),$ as a linear combination of $\Psi_1^{+},\Psi_2^{+},\PPsi_4^{+},$ i.e., 
\begin{equation}
\label{def-Psi-4}
    \Psi^{+}_4(r,z)= \PPsi^{+}_4(r,z) +\beta \Psi^{+}_3(r,z) + \gamma \Psi^{+}_1(r,z),
\end{equation}
where $\alpha=-\frac{W(\Psi^{+}_1(\cdot,z),\PPsi^{+}_4(\cdot,z))}{W(\Psi^{+}_1(\cdot,z),\Psi^{+}_3(\cdot,z))}$ and $\beta=-\frac{W(\Psi^{+}_3(\cdot,z),\PPsi^{+}_4(\cdot,z))}{W(\Psi^{+}_3(\cdot,z),\Psi^{+}_1(\cdot,z))}.$ This renormalization is chosen such that $\Psi_4^+$ has vanishing Wronskians with $\Psi_1^+$ and $\Psi_3^+$. \\

Note that $  \Psi^{+}_4$ satisfies the same asymptotics as $\PPsi^{+}_4$ and $\{\Psi^{+}_1,\Psi^{+}_2,\Psi^{+}_3,\Psi^{+}_4\}$ are linearly independent solutions to $i \mathcal{L} \Psi^{+}(r,z)=z \Psi^{+}(r,z)$ for large $|z|. $ Therefore, we define
\begin{align*}
\Psi^{+}(r,z)= \begin{bmatrix} \Psi^{+}_1 (r,z) &    \Psi^{+}_2(r,z)
    \end{bmatrix}
    \qquad \text{and} \qquad 
    \PPsi^{+}(r,z)= \begin{bmatrix}   \Psi^{+}_3 (r,z) & \Psi^{+}_4(r,z)  
    \end{bmatrix}
\end{align*}

\subsection{Construction of $\Psi^{-}(\cdot,z)$ for large $|\xi|$} Similarly to the case for small $\xi$, the construction of $\Psi^{-}(r,z) \in L^2((1,\infty))$ when $\im(z)<0,$ is obtained directly from $\Psi^{+}(r,z)$ by symmetry.  Recall that, for $z \in \Omega$ with $\im(z)<0,$ we have 
\begin{align*}
    \Uppsi^{-}_{E}(r,z) =\begin{pmatrix}
      h_{+}(k_1(z)r) &      h_{+}(k_2(z)r) \\ \\
       c_1(z)  h_{+}(k_1(z)r) & c_2(z)   h_{+}(k_2(z)r)
    \end{pmatrix} , \quad    \UUppsi^{-}_{E}(r,z) = \begin{pmatrix}
      h_{-}(k_1(z)r) &     h_{-}(k_2(z)r) \\ \\
       c_1(z)  h_{-}(k_1(z)r) & c_2(z) h_{-}(k_2(z)r)
    \end{pmatrix}
\end{align*}

We seek to construct the fundamental matrix solutions $\Psi^{-}(\cdot,z)\in L^2(1,\infty)$ and $\widetilde{\Psi}^{-}(\cdot, z) \notin L^2(1,\infty) $ to $   i \mathcal{L} \Psi^{-}(\cdot, z) = z  \Psi^{-}(\cdot, z) .$  for $ z \in \Omega$ with $\im(z)<0$ and  $z=\pm \frac{\sqrt{17}}{8} +\xi,$ for large $\xi.$ \\

Since by Lemma \ref{lem:k_j-behavior}, we have $k_1(z)=k_1(-z),$ $k_2(z)=k_2(-z),$ for $\im(z)<0,$ and $c_1(-z)=-c_1(z)$. Thus, we define 
\begin{align}
\label{eq:def-upPsi-minus-large-xi}
    \Uppsi^{-}_{E}(r,z)=\sigma_3  \Uppsi^{+}_{E}(r,-z) \qquad \text{and} \qquad  \UUppsi^{-}_{E}(r,z)=\sigma_3  \UUppsi^{+}_{E}(r,-z)  .
\end{align}

Note that $ \Uppsi^{-}_{E}$ and $\UUppsi^{-}_{E}$ satisfy the equation  $i \mathcal{L}_{E}   \Uppsi  (r,z)  = z \Uppsi(r,z) .$  Therefore, we define $ \Psi^{-}(r,z)$ as 
\begin{align}
\label{eq:def-Psi-minus-large-xi}
    \Psi^{-}(r,z):=\sigma_3  \Psi^{+}(r,-z) \qquad \text{and} \qquad \PPsi^{-}(r,z):=\sigma_3  \PPsi^{+}(r,-z)
\end{align}

Thus it follows from Proposition~\ref{prop:contraction-T-j-large-xi-infty} that $$ \|  T_{E,1}(\Psi^{+}(\cdot,z)) \|_{\mathcal{ \tilde{B}}_1} +\|  T_{E,2}(\PPsi^{+}(\cdot,z)) \|_{\mathcal{ \tilde{B}}_2}+\| \partial_r T_{E,1}(\Psi^{+}(\cdot,z)) \|_{\mathcal{ \tilde{B}}_1^{\prime}} +\| \partial_r T_{E,2}(\PPsi^{+}(\cdot,z)) \|_{\mathcal{ \tilde{B}}_2^{\prime}}\lesssim \varepsilon_{\infty}. $$

\section{Green's kernel of $i\mathcal{L}-z$ for the non-resonance case}
\label{sec:GreenKernel-non-resonance}
In this section, we compute the distorted Fourier transform for $\mathcal{L}.$ In \S \ref{subsec:ConCoef}, we first compute the connection coefficients of the solutions constructed near $0$ and near $\infty.$ Then in \S \ref{subsec:GreenKernel} we study the Green's kernel of $(i \calL -z).$  Next, we justify that the limit as $b \to 0^{+}$ can be taken inside the integral representation of the evolution in \eqref{eq:ST1}. This is proved in Subsection \ref{subsec:JumpResol} and relies on the analysis of the resolvent kernel $(i \mathcal{L} - z)^{-1}$ for $\pm \im(z)>0.$

\subsection{Connecting the solutions from $0$ and $\infty$}
\label{subsec:ConCoef}
The purpose of this subsection is to link the local and asymptotics behaviors of solutions to $(i\mathcal{L} - z)\Psi = 0$ from $r = 0$ and at $r = \infty$ for small and large $|\xi|$. To achieve this, we evaluate or estimate the Wronskians between the solutions constructed from $r=0$ and those constructed at infinity. In addition, we evaluate the Wronskians between the columns of solutions $\varphi_j$ defined near the origin and the solutions $\Psi_j$ defined near infinity. 
These computations will later be used to express one set of solutions in terms of the other. \\

Recall that throughout this paper (for non-resonance case), we choose $\varepsilon_0$ and $\varepsilon_{\infty}$ such that $r_{\infty}< r_0 . $ For $0<|\xi|< \delta_0,$ we fix  $ r_\varepsilon\equiv r_{\varepsilon}(z) \in (r_{\infty},r_0)$ such that $(\frac{r_{\varepsilon}}{20}, 6 r_{\varepsilon})\subseteq (r_{\infty},r_0)$, and $r_\varepsilon\simeq \frac{\varepsilon}{\sqrt{|\xi|}}$ for some small constant $\varepsilon$.

Denote by: 
\begin{align*}
    \Lambda_{ij}(z)&:=W(\upvarphi_i(\cdot,z),  \upvarphi_j(\cdot,z)) ,\\
     \mu^{\pm}_{ij}(z)&:= W ( \Psi^{\pm}_i(\cdot, z), \Psi^{\pm}_j(\cdot, z) ), \\
      \omega_{ij}^{\pm}(z)&:= W( \Psi^{\pm}_{i}(\cdot,z), \upvarphi_j(\cdot,z)).
\end{align*} 

\begin{lemma}
\label{Lambda_j-behavior}
Let $z \in \Omega,$ then we have
 \begin{align*}
 \Lambda_{12}(z)=\Lambda_{14}(z)=0 , \quad \Lambda_{13}(z) =2c_1^2 c_3^2 , \quad 
 \Lambda_{24}(z) =2c_2^1 c_4^1, \quad  |\Lambda_{23}(z)|  \lesssim  e^{-\frac{3}{\sqrt{2}}r_{\varepsilon}}, \quad  |\Lambda_{34}(z)| \lesssim e^{-\frac{3}{\sqrt{2}}r_{\varepsilon}}. 
\end{align*}
\end{lemma}
\begin{proof}
We evaluate the first four Wronskians at $r = 0$ and the remaining two at $r = 2r_{\varepsilon}$, using the asymptotic behavior of $\upvarphi_j(r,z)$ for both small and large values of $r$. The first four Wronskians follow as in Lemma~\ref{D_0}. For the last two, since $\upvarphi_3$ decays exponentially while $\upvarphi_1$ exhibits polynomial growth as $r \to \infty$, the required bounds can be established by evaluating the Wronskians at $r = 2r_{\varepsilon}$.
\end{proof}

\begin{lemma}
\label{lemma:omega_j-behavior}
Let $0<|\xi|<\delta_0,$ and for $z=\pm \frac{\sqrt{17}}{8}+ \xi \in \Omega$, with $\im(z)>0$ we have  
    \begin{align*}
  |\omega_{11}^{+}(z)|   &\lesssim  e^{\frac{3}{\sqrt{2}} \frac{r_{\varepsilon}}{5}}  , \quad |\omega_{21}^{+} (z)| \simeq 1 , \quad |\omega_{31}^{+}  (z)| \lesssim  e^{\frac{3}{\sqrt{2}} \frac{r_{\varepsilon}}{5}}, \quad |\omega_{41}^{+}  (z)| \lesssim   e^{\frac{6}{\sqrt{2}} \frac{r_{\varepsilon}}{10}} \\
   |\omega_{12}^{+}(z)|  & \simeq 1  , \quad |\omega_{22}^{+} (z)| \lesssim  e^{-\frac{3}{\sqrt{2}} r_{\varepsilon}} , \quad |\omega_{32}^{+}  (z)| \simeq 1, \quad |\omega_{42}^{+}  (z)| \lesssim  e^{\frac{3}{\sqrt{2}} \frac{r_{\varepsilon}}{5}} \\
   |\omega_{13}^{+}(z)| &  \lesssim e^{-\frac{3}{\sqrt{2}} r_{\varepsilon}}  , \quad 
   |\omega_{23}^{+} (z)| \lesssim e^{-\frac{6}{\sqrt{2}} r_{\varepsilon}} , \quad |\omega_{33}^{+}  (z)| \lesssim e^{-\frac{3}{\sqrt{2}} r_{\varepsilon}},
   \quad |\omega_{43}^{+}  (z)| \simeq 1 \\
   \omega_{14}^{+}(z) &=c \, k_1(z) + O(|\xi|^{\frac{3}{2}})  , \quad 
   |\omega_{24}^{+} (z)| \lesssim e^{-\frac{3}{\sqrt{2}} r_{\varepsilon}} , \quad \omega_{34}^{+}  (z) = c \,  k_1(z)   + O(|\xi|^{\frac{3}{2}}) ,
   \quad |\omega_{44}^{+}  (z)| \lesssim e^{\frac{3}{\sqrt{2}} \frac{r_{\varepsilon}}{5}}.
    \end{align*}
    where $c>0$ is constant. Moreover, $d^{+}(z):=\omega_{22}^{+}(z) \omega_{11}^{+}(z)-\omega_{21}^{+}(z) \omega_{12}^{+}(z) \neq 0 $ and $|d^{+}(z) | \simeq 1.$
\end{lemma}

\begin{proof}
Recall that for $j = 1,3$ and $l = 2,4$, the coefficients satisfy $ \tc_j^2 = \pm \frac{i}{\sqrt{17}} \, \tc_j^1,$ and $\tc_l^2 = \mp i \sqrt{17} \, \tc_l^1. $ In addition, for $j = 1,2$, we have
$$
c_j(z) = -i \, \frac{\frac{1}{2} k_j(z)^2 + \frac{17}{8}}{\xi + \frac{\sqrt{17}}{8}}.
$$

The asymptotic behavior of the functions $k_j(z)$ and $c_j(z)$ as $z \to \pm \frac{\sqrt{17}}{8},$ with $\im(z) >0$ is given by:
$$
k_1(z) = \pm \frac{\sqrt{2\sqrt{17}}}{3} \sqrt{z} + O(z), \quad
k_2(z) = i \frac{3}{\sqrt{2}} + O(z),
$$ $$
c_1(z) = -i \sqrt{17} + O(z), \quad
c_2(z) = i \frac{1}{\sqrt{17}} + O(z).
$$

We begin by estimating $\omega_{l1}^{+} := W\big( \Psi^{+}_{l}(r,z), \upvarphi_1(r,z) \big)$.  
To achieve this, we use the observation above, the definitions of $\Psi^{+}_{l}(r,z)$ from \eqref{eq:def-Psi+(r,z)-xi-small} and \eqref{eq:def-tilde-Psi+(r,z)-xi-small}, together with the definition of $\upvarphi_1(r,z)$ in \eqref{eq:defF_1(r,z)-scenarioI} and \eqref{eq:asumptotic-phi_j(r)-non-resonance}.  In addition, we use of the asymptotic estimates established in the proof of Propositions \ref{Prop:contraction-T_j(F_j)-xi-small} and \ref{prop:contraction-T-j-small-xi-infty}, particularly Claim \ref{claim:upsilon_1-2-xi-small-infty}, \ref{claim:upsilon_1-3-xi-small-infty}, and \ref{Claim:Gamma_1-and-2-large-r-non-resonance}. It follows that, 
\begin{align*}
      \omega_{11}^{+}(z) &:= W( \Psi^{+}_{1}(r,z), \upvarphi_1(r,z))\\
      &= W \bigg( \begin{pmatrix}
           e^{i k_1(z) r } \\
            c_1 (z) e^{i k_1(z) r } 
      \end{pmatrix} (1+O(\frac{1}{\sqrt{|\xi|}}e^{-2r_{\infty}}) ),
      \begin{pmatrix}
\tc_1^1 e^{\frac{3}{\sqrt{2}}r} \\
\tc_1^2 e^{\frac{3}{\sqrt{2}}r}
      \end{pmatrix} (1+O(r^{-1} )) \bigg) 
\end{align*} 
evaluating at $ r=\frac{r_{\varepsilon}}{5},$ by a naive bound we get $
  | \omega_{11}^{+}(z)| \lesssim e^{\frac{3}{\sqrt{2}} \frac{r_{\varepsilon}}{5}} . $ Next, we consider $l=2,$
\begin{align*}
  \omega_{21}^{+} (z)&:=W( \Psi^{+}_{2}(r,z), \upvarphi_1(r,z))\\
  &= W\bigg( \begin{pmatrix}
 e^{i k_2(z) r } \\
 c_2(z)  e^{ i k_2(z) r } 
   \end{pmatrix} (1+O(\frac{1}{\sqrt{|\xi|}}e^{-2r_{\infty}}) ),  \begin{pmatrix}
\tc_1^1 e^{\frac{3}{\sqrt{2}}r} \\
\tc_1^2 e^{\frac{3}{\sqrt{2}}r}
      \end{pmatrix} (1+O(r^{-1} ))  \bigg) \simeq 1
\end{align*}
Indeed, evaluating at $r=r_\varepsilon$,
\begin{align*}
W(\begin{pmatrix}
           e^{i k_2(z) r } \\
            c_2 (z) e^{i k_2(z) r } 
      \end{pmatrix} ,
      \begin{pmatrix}
\tc_1^1 e^{\frac{3}{\sqrt{2}}r} \\
\tc_1^2 e^{\frac{3}{\sqrt{2}}r}
      \end{pmatrix} ) & = e^{ik_2(z) r } e^{\frac{3}{\sqrt{2}}r} ( \frac{3}{\sqrt{2}} - ik_2(z)) ( \tc_1^1 - c_2(z) \tc_1^2)  \\
      & = e^{O(z) r} (\frac{6}{\sqrt{2}}+O(z)) ( \tc_1^1 - (\frac{i}{ \sqrt{17}}+O(z)) \frac{i}{\sqrt{17}} \tc_1^1 ) \\ 
      &\lesssim e^{O(z) r} (1+O(z))  \simeq 1
\end{align*}
By a similar argument as for $ \omega_{11}^{+}(z),$ we get $
 | \omega_{31}^{+}(z)|   \lesssim  e^{\frac{3}{\sqrt{2}} \frac{r_{\varepsilon}}{5}} .$ Next, for $l=4$ we have 
\begin{align*}
 \omega_{41}^{+}(z) &:= W( \Psi^{+}_{4}(r,z), \upvarphi_1(r,z))\\
 &= W\bigg(
   \begin{pmatrix}
 e^{- i k_2(z) r } \\
 c_2(z)  e^{- i k_2(z) r } 
   \end{pmatrix} (1+O(\frac{1}{\sqrt{|\xi|}}e^{-2r_{\infty}}) ),  \begin{pmatrix}
\tc_1^1 e^{\frac{3}{\sqrt{2}}r} \\
\tc_1^2 e^{\frac{3}{\sqrt{2}}r} 
      \end{pmatrix} (1+O(r^{-1} )) \bigg)
\end{align*}
therefore,  by evaluating at $r=\frac{r_{\varepsilon}}{10}  ,$ we obtain  $|\omega_{41}^{+}(z) | \lesssim e^{\frac{6}{\sqrt{2}} \frac{r_{\varepsilon}}{10}}.$ Indeed, 
\begin{align*}
 W(\begin{pmatrix}
 e^{- i k_2(z) r } \\
 c_2(z)  e^{- i k_2(z) r } 
   \end{pmatrix} ,  \begin{pmatrix}
\tc_1^1 e^{\frac{3}{\sqrt{2}}r} \\
\tc_1^2 e^{\frac{3}{\sqrt{2}}r} 
      \end{pmatrix} ) =  e^{- i k_2(z) r }   e^{\frac{3}{\sqrt{2}}r} ( \frac{3}{\sqrt{2}} + i k_2(z))( \tc_1^1 - c_2(z) \tc_1^2 ) .
\end{align*}

Next, we estimate $\omega_{l2}^{+} := W\big( \Psi^{+}_{l}(r,z), \upvarphi_2(r,z) \big)$, using a similar argument with the definition of $\upvarphi_2(r,z)$ in \eqref{eq:defF_1(r,z)-scenarioI}, \eqref{eq:asumptotic-phi_j(r)-non-resonance} and Claim \ref{Claim:Gamma_1-and-2-large-r-non-resonance}, we obtain 
\begin{align*}
     | \omega_{12}^{+}(z) | &:= |  W( \Psi^{+}_{1}(r,z), \upvarphi_2(r,z))| \\
      &= \left|W( \begin{pmatrix}
           e^{i k_1(z) r } \\
            c_1 (z) e^{i k_1(z) r } 
      \end{pmatrix} (1+O(\frac{1}{\sqrt{|\xi|}}e^{-2r_{\infty}}) ) ,
      \begin{pmatrix}
\tc_2^1 r \\
\tc_2^2 r
      \end{pmatrix} (1+O(r^{-2}) ) ) \right| \simeq 1 
\end{align*}

\begin{align*}
      \omega_{22}^{+}(z) &:= W( \Psi^{+}_{2}(r,z), \upvarphi_2(r,z)) \\
      &= W( \begin{pmatrix}
           e^{i k_2(z) r } \\
            c_2 (z) e^{i k_2(z) r } 
      \end{pmatrix}(1+O(\frac{1}{\sqrt{|\xi|}}e^{-2r_{\infty}}) )  ,
      \begin{pmatrix}
\tc_2^1 r \\
\tc_2^2 r
      \end{pmatrix}(1+O(r^{-2}) ) ) .
\end{align*} 
Thus, by evaluating at $r=2r_{\varepsilon},$ we obtain the bound $| \omega_{22}^{+}(z)| \lesssim  e^{-\frac{3}{\sqrt{2}}r_{\varepsilon}}  .$ Similarly to $\omega_{12}^{+}(z), $ we find $|  \omega_{32}^{+}(z)| \simeq 1.$ Using the exponential growth in the asymptotics of $\Psi^{+}_4$ and evaluating the Wronskian at $r=\frac{r_{\varepsilon}}{5},$ we obtain the rough estimate $|\omega_{42}^{+}(z)| \lesssim  e^{\frac{3}{\sqrt{2}} \frac{r_{\varepsilon}}{5}  } .$ \\

Next, we estimate $\omega_{l3}^{+}(z) := W\big( \Psi^{+}_{l}(r,z), \upvarphi_3(r,z) \big)$, using a similar argument with the definition of $\upvarphi_3(r,z)$ in \eqref{eq:defF_2(r,z)-scenarioI}, \eqref{eq:asumptotic-phi_j(r)-non-resonance} and Claim \ref{Claim:Gamma_3-and-2-large-r-non-resonance}. Using the exponential decay in the asymptotics of $\varphi^{+}_3$  and evaluating the Wronskian at $r=r_{\varepsilon},$ we obtain the rough estimate 
\begin{align*}
 | \omega_{13}^{+}(z)|    \lesssim  e^{-\frac{3}{\sqrt{2}}r_{\varepsilon}}  , \quad |\omega_{23}^{+}(z)| \lesssim  e^{-\frac{6}{\sqrt{2}}r_{\varepsilon}} , \quad |    \omega_{33}^{+}(z)|  \lesssim  e^{-\frac{3}{\sqrt{2}}r_{\varepsilon}}  
\end{align*}
Using a similar argument, and noting that $\Psi^{+}_{4}$ exhibits exponential growth at the same rate as $\varphi^{+}_3$, we get $ |  \omega_{43}^{+}(z)| \simeq 1 .  $

Using similar argument with the definition of $\upvarphi_4(r,z)$ in \eqref{eq:defF_2(r,z)-scenarioI}, \eqref{eq:asumptotic-phi_j(r)-non-resonance} and Claim \ref{Claim:Gamma_3-and-2-large-r-non-resonance}, we estimate $ \omega_{l4}^{+}:= W( \Psi^{+}_{l}(r,z), \upvarphi_4(r,z)).$  

\begin{align*}
      \omega_{14}^{+}(z) &:= W( \Psi^{+}_{1}(r,z),  \upvarphi_4(r,z))  \\
      & = W( \begin{pmatrix}
           e^{i k_1(z) r } \\
            c_1 (z) e^{i k_1(z) r } 
      \end{pmatrix} (1+O(\frac{1}{\sqrt{|\xi|}}e^{-2r_{\infty}}) ) ,
      \begin{pmatrix}
\tc_4^1  \\
\tc_4^2 
      \end{pmatrix}(1+O(r^{-1} ))  )  = c \, k_1(z)+ + O(|\xi|^{\frac{3}{2}})
\end{align*} 
Indeed,  
\begin{align*}
    W( \Psi^{+}_{1}(r,z), \upvarphi_4(r,z)) = W( \begin{pmatrix}
           e^{i k_1(z) r } \\
            c_1 (z) e^{i k_1(z) r } 
      \end{pmatrix} ,
      \begin{pmatrix}
\tc_4^1  \\
\tc_4^2 
      \end{pmatrix} ) &= e^{ik_1(z) r }  i k_1(z) ( - \tc_4^1 +  c_1(z) \tc_4^2) \\ &=e^{ik_1(z) r }  i k_1(z) ( - \tc_4^1 + (-i \sqrt{17} + O(z)) (-i) \sqrt{17}  \tc_4^1 ) \\
      & = c \, k_1(z) + + O(|\xi|^{\frac{3}{2}}).
\end{align*}
Using the exponential decay of $\Psi^{+}_{2}(r,z), $ and by evaluating the Wronskian at $r=r_{\varepsilon},$ we get $|\omega_{24}^{+}(z)|  \lesssim      e^{ - \frac{3}{\sqrt{2}} r_{\varepsilon} }.$ Similarly to the estimate for $\omega_{14}^{+}(z)$ we have  $|\omega_{34}^{+}(z)| =c \, k_1(z) + + O(|\xi|^{\frac{3}{2}}).$  Finally, due to the exponential growth in the asymptotics of  $\Psi^{+}_{2}(r,z), $ and evaluating the at$r=\frac{r_{\varepsilon}}{5}$ we obtain $|\omega_{44}^{+}(z)| \lesssim   e^{\frac{3}{\sqrt{2}} \frac{r_{\varepsilon}}{5}}.$ The last assertion of the lemma follows from the preceding estimate on $\omega_{ij}$, which concludes the proof of Lemma \ref{lemma:omega_j-behavior}.

\end{proof}

Next, we state the relation between the connection coefficients for $\im(z)>0$ and  $\im(z)<0,$  that is, the identities relating $ \omega_{ij}^{-}(z)= \omega_{ij}^{+}(z). $ The following identities hold for both large and small $|\xi|.$

\begin{lemma}
\label{lem:parity-of-omega-j}
   Let $z\in \Omega,$ then we have 
\begin{align*}
    \omega_{i1}^{-}(z)=- \omega^{+}_{i1}(-z), \quad  \omega_{i2}^{-}(z)= \omega^{+}_{i2}(-z) , \quad  
    \omega_{i3}^{-}(z)= -\omega^{+}_{i3}(-z), \quad
    \omega_{i4}^{-}(z)= \omega^{+}_{i4}(-z)
\end{align*}
    Moreover, \begin{align}
    \label{eq:parity-w21-w22}
        \omega^{+}_{21}(-z)= -\omega^{+}_{21}(z), \quad \omega^{+}_{22}(-z)=  \omega^{+}_{22}(z)
    \end{align}
\end{lemma}
\begin{proof}
The relations between $\omega_{ij}^{-}(z)$ and $\omega_{ij}^{+}(z)$ follow immediately from the  identities \eqref{eq:sym-varphi_j}, \eqref{eq:def-Psi-minus-small-xi} and \eqref{eq:def-Psi-minus-large-xi}
\begin{align*}
 \Psi^{-}(r,z)&=\sigma_3  \Psi^{+}(r,-z) \qquad  \upvarphi_1(r,z)=-\sigma_3 \upvarphi_1(r,-z), \; \;
\qquad \upvarphi_3(r,z)=-\sigma_3 \upvarphi_3(r,-z), \\
\PPsi^{-}(r,z)& =\sigma_3  \PPsi^{+}(r,-z) \qquad \upvarphi_2 (r,z)=\sigma_3 \upvarphi_2(r,-z), \qquad  \qquad \upvarphi_4(r,z)=\sigma_3 \varphi_4(r,-z).  
\end{align*}
For the parity claim \eqref{eq:parity-w21-w22} of $\omega^{+}_{21}(z)$ and $\omega^{+}_{22}(z),$ follows the fact that $\Psi^{+}_{2}(r, -z)= \sigma_3 \Psi^{+}_{2}(r, z). $
\end{proof}
 Next, we evaluate the Wronskians between the solutions constructed at infinity for  small values of $|\xi|$. Note that
\begin{equation}
    \label{eq:Mu-symm}
    \mu^{+}_{ij}(-z) = W\left( \sigma_3 \Psi^{-}_i(\cdot, z), \sigma_3 \Psi^{-}_j(\cdot, z) \right) = \mu^{-}_{ij}(z).
\end{equation}
Therefore, it suffices to evaluate the Wronskians $\mu^{+}_{ij}$ for $\im(z) > 0$ and the corresponding values of $\mu^{-}_{ij}$ then follow by symmetry.

\begin{lemma}
\label{lem:mu-behavior-small-xi}
Let $z= \xi \pm \frac{\sqrt{17}}{8} \in \Omega$ with $\pm \im(z)>0$ and 
 $0<|\xi|<\delta_0.$ Then for some non-zero constant $c$ we have  
\begin{align*}
 \mu^{\pm}_{11}(z)&=\mu^{\pm}_{12}(z)=0,   \quad \mu_{13}^{\pm}(z)=c \, k_1(z) \pm O(|\xi|^{\frac{3}{2}}), \quad  |\mu_{14}^{\pm}(z)| \lesssim e^{\frac{3}{\sqrt{2}} \frac{r_{\epsilon}}{5}} \\
 \mu^{\pm}_{22}(z)&=\mu^{\pm}_{23}(z)=0, \quad \mu^{\pm}_{24}(z) \simeq 1, \quad |\mu^{\pm}_{34}(z) | \lesssim  e^{\frac{3}{\sqrt{2}} \frac{r_{\epsilon}}{5}}. 
\end{align*}
\end{lemma}

\begin{proof}
Note that, in view of the large $r$ asymptotics of $\Psi^{+}_j(r,z)$ for small $|\xi|$ and using similar computation to those in the proof of Lemma \ref{lemma:omega_j-behavior} for estimating $\omega^{+}_{ij}$ we obtain the desired results.
\end{proof}

Next,  we evaluate the Wronskians between the solutions constructed from $r=0$ and those constructed at infinity for large $|\xi|.$ Recall that throughout this paper (for non-resonance case), we choose $\tvarepsilon_0$ and $\tvarepsilon_{\infty}$ such that $\tr_{\infty}=\frac{\tvarepsilon_{\infty}}{\sqrt{|\xi|}}< \tr_0=\frac{\tvarepsilon_0}{\sqrt{|\xi|}}. $ For $1<\Lambda_0<|\xi|,$ we fix $\tr_\varepsilon \equiv \tr_{\varepsilon}(z) \in (\tr_{\infty},\tr_0)$  such that $(\frac{\tr_{\varepsilon}}{10}, 4 \tr_{\varepsilon})\subseteq (\tr_{\infty},\tr_0),$ and $\tr_{\varepsilon}\simeq \frac{\varepsilon}{\sqrt{|\xi|}}$ for some small constant $\varepsilon.$  We will evaluate the Wronskians for some $r \in (\frac{\tr_{\varepsilon}}{10}, 4 \tr_{\varepsilon}).$  \\

Note that, in view of our choices of $\tr_{\infty}$ and $\tr_0$, and for sufficiently large $|\xi|$, the evaluation of the Wronskians relies on the small $r$ behavior of $\Psi^{+}_l$ and $\varphi_j$.

\begin{lemma}
\label{lem:omega_j-behavior-large-xi}
Let $z=  \xi \pm \frac{\sqrt{17}}{8}  \in \Omega,$ with $\im(z)>0$ and $1<\Lambda_0 < |\xi|.$  Then we have 
    \begin{align*}
  |\omega_{11}^{+}(z)|  &= -2 c_{+} k_1(z)^{-\frac{ 1}{2}} c_1(z) c_1^2 + O(|\xi|^{-1}) \\
  |\omega_{21}^{+}(z)| & = -2 c_{+} k_2(z)^{- \frac{ 1}{2}} c_2(z) c_1^2 + O(|\xi|^{-1})  \\
 |\omega_{12}^{+}(z)| & = 2 c_{+} k_1(z)^{-\frac{1}{2}}(c_2^1 -c_1(z) c_2^2) + O(|\xi|^{-1}) \\ 
  |\omega_{22}^{+}(z)| & = 2 c_{+} k_2(z)^{-\frac{1}{2}}(c_2^1 -c_2(z) c_2^2) + O(|\xi|^{-1}) \\
   |\omega_{31}^{+}(z)| & =-2 c_{-} k_1(z)^{-\frac{ 1}{2}} c_1(z) c_1^2 + O(|\xi|^{-1})   \\
   |\omega_{32}^{+}(z)| & = 2 c_{-} k_1(z)^{-\frac{1}{2}}(c_2^1 -c_1(z) c_2^2) + O(|\xi|^{-1})  \\
    |\omega_{41}^{+}(z)| & = 2 c_{-} k_2(z)^{-\frac{1}{2}} c_2(z)c_1^2 + O(|\xi|^{-1}) \\
     |\omega_{42}^{+}(z)| & =2 c_{-} k_2(z)^{-\frac{1}{2}}(c_2^1 -c_2(z) c_2^2) + O(|\xi|^{-1}) 
    \end{align*} 
    where $c_{\pm}\neq 0$ constants, $c_j(z)$ are complex, $c_1^2\neq 0 $ and $c_2^2$ are real constants. Moreover, 
  \begin{align*}
     d^{+}(z):=\omega_{22}^{+}(z) \omega_{11}^{+}(z) -\omega_{12}^{+}(z)  \omega_{21}^{+}(z)= c \, k_1(z)^{-\frac{1}{2}}  k_2(z)^{-\frac{1}{2}} + O(|\xi|^{-1}) \neq 0. 
  \end{align*}
  where $c>0$ constant. 
\end{lemma}
\begin{proof}
The proof closely follows the argument used in Lemma~\ref{lemma:omega_j-behavior}. For $\Psi^{+}_l$, we employ its definition in \eqref{eq:def-Psi+(r,z)-xi-large} and \eqref{eq:def-tilde-Psi+(r,z)-xi-large} for large $\xi$, together with the asymptotic behavior of $h_{\pm}(z)$ given in \eqref{eq:hpm-Hankel} for small $r$, and the estimates derived in the proof of Proposition~\ref{prop:contraction-T-j-large-xi-infty}, particularly Claim~\ref{claim:contraction-Upsilon-large-xi}. For $\varphi_j$, we use the definitions in \eqref{eq:defF_1(r,z)-scenarioI} and \eqref{eq:defF_2(r,z)-scenarioI}, along with the estimates from the proof of Proposition~\ref{prop:contraction-large-xi}, especially Claim~\ref{Gamma_j-beahvior-large-xi}.
We omit the details. 
\end{proof}

 Next, we evaluate the Wronskians between the solutions constructed at infinity for  large values of $|\xi|$ for $\im(z)>0.$ The corresponding values of $\mu^{-}_{ij}$ follow by symmetry see \eqref{eq:Mu-symm}.

\begin{lemma}
\label{lem:mu-behavior-large-xi}
Let $z=  \xi \pm \frac{\sqrt{17}}{8}  \in \Omega,$ with $\im(z)>0$ and $1<\Lambda_0 < |\xi|,$ and  Then for some non-zero constant $c$ we have
\begin{align*}
 \mu^{\pm}_{11}(z)&=\mu^{\pm}_{12}(z)=0,   \quad \mu_{13}^{\pm}(z)=c \, k_1(z) \pm O(|\xi|^{\frac{3}{2}}), \quad  \mu_{14}^{\pm}(z)=0,\\
 \mu^{\pm}_{22}(z)&=\mu^{\pm}_{23}(z)=0, \quad |\mu^{\pm}_{24}(z)| \simeq \sqrt{|\xi|} ,  \quad \mu^{\pm}_{34}(z) =0.
\end{align*}
\end{lemma}

\begin{proof}
 Using \eqref{def-Psi-4}, the definition of $\Psi_4^{+}(r,z)$ for large $|\xi|,$  we obtain $\mu^{+}_{41}(z)=\mu^{+}_{43}(z)=0.$ All the other terms can be obtained similarly to the proof of Lemma \ref{lem:omega_j-behavior-large-xi}.  
\end{proof}

\subsection{Solving for the Green's Kernel} 
\label{subsec:GreenKernel}
The goal of this section is to compute the resolvent kernel $(i\mathcal{L}-z)^{-1},$ for $\pm \im(z)>0.$ Therefore, to solve 
\begin{equation}
\label{eq:iL-z-resol}
  (  i \mathcal{L} - z ) \Psi = \Phi, \quad \text{ for } z\in \Omega, \; \text{ with }\pm \im(z)>0,
\end{equation}
 we define the Green's functions as
\begin{align*}
    \mathcal{K}^{\pm}(r,s,z):= \Psi^{\pm}(r,z) S(s,z) \mathbb{1}_{\{ 0<s \leq r  \}} + 
    F_1(r,z) T(s,z) \mathbb{1}_{ \{r \leq s < \infty  \} }, 
\end{align*}
where we require the matrices $S(r,z)$ and $T(r,z)$ to satisfy 
\begin{align*}
    \begin{pmatrix}
        \Psi^{\pm}(r,z) & F_1(r,z) \\
        \partial_r \Psi^{\pm}(r,z) & \partial_r F_1(r,z) 
    \end{pmatrix}
    \begin{pmatrix}
        S(r,z) \\ - T(r,z)
    \end{pmatrix}
    = \begin{pmatrix}
        0 \\ \sigma_2  
    \end{pmatrix}.
\end{align*}
 Then a solution to \eqref{eq:iL-z-resol} for $z \in \Omega, $ with $\pm \im(z)>0,$ is given by 
 \begin{equation}
     \Psi(r) := \int_0^{\infty}   \mathcal{K}^{\pm}(r,s,z) \Phi(s) ds. 
 \end{equation}

The computation of $S(r,z)$ and $T(r,z)$ requires inverting the matrix defined above. We compute the Wronskians between $\Psi^{\pm}(\cdot,z)$ and $F_1(\cdot,z)$ and therefore we obtain the Green's kernel functions.

\begin{lemma}
\label{lem:kernel-K}
 Let $\calD^{\pm}(z):=W(\Psi^{\pm}(\cdot,z),F_1(\cdot,z)).$ Then we have 
 \begin{align*}
\calD^{\pm}(z)=   \begin{pmatrix}
       \omega_{11}^{\pm}(z) & \omega_{12}^{\pm}(z) \\
       \omega_{21}^{\pm}(z) & \omega_{22}^{\pm}(z)
   \end{pmatrix} , \qquad 
 ( \calD^{\pm})^{-1}(z)  : = \frac{1}{d^{\pm}(z) } \begin{pmatrix}
       \omega_{22}^{\pm}(z) & - \omega_{12}^{\pm}(z) \\  
       -\omega_{21}^{\pm}(z) & \omega_{11}^{\pm}(z)
   \end{pmatrix} , 
 \end{align*}
 where $d^{\pm}(z):=\omega_{22}^{\pm}(z) \omega_{11}^{\pm}(z) -\omega_{12}^{\pm}(z)  \omega_{21}^{\pm}(z) \neq 0 .$ \\
 
 Moreover, for $z \in \Omega$ with $\pm \im(z)>0,$ we have 
    \begin{align*}
        \begin{cases}
            S(r,z) = - (\calD^{\pm}(z))^{-t} F_1(r,z)^t \sigma_3 \sigma_2, \\ 
            T(r,z) = - (\calD^{\pm}(z))^{-1} \Psi^{\pm}(r,z)^t \sigma_3 \sigma_2.  
        \end{cases}
    \end{align*}
and
 \begin{align*}
     \mathcal{K}^{\pm}(r,s,z):= \begin{cases}
         i \Psi^{\pm}(r,z) (\calD^{\pm}(z))^{-t} F_1(s,z)^t \sigma_1, \qquad 0<s\leq r , \\ 
         i F_1(r,z) (\calD^{\pm}(z))^{-1} \Psi^{\pm}(s,z)^t \sigma_1, \qquad r \leq s < \infty,
     \end{cases}
 \end{align*}
\end{lemma}
\begin{proof}
Using Lemma \ref{lemma:omega_j-behavior} and \ref{lem:omega_j-behavior-large-xi}, we have $d^{\pm}(z)\neq 0,$ for small and large $z.$ For the middle frequencies, we will show later in Lemma \ref{lem:EV1} that $d^{\pm}(z)=0$ is equivalent to the existence of an eigenvalue of $i \mathcal{L}$ at $z$. However, by Assumptions \ref{asmp:2} and \ref{asmp:3}, the operator $i \mathcal{L}$ has a real spectrum and no threshold or embedded eigenvalues (see also Lemma \ref{lem:GP-thres-emb-EV}, \ref{lem:GP-spec} for $i\calL_{GP}$). Therefore, $\calD^{\pm}(z)$ is invertible. Similarly to the proof of Lemma \ref{lem:def_S-T-near 0}, we obtain the desired results and we omit the details. 
\end{proof}

Next, we seek to express the Green's functions only in terms of $F_j(\cdot,z),$ i.e., in terms of  $\upvarphi_j.$ Since $\{ \upvarphi_1,\upvarphi_2,\upvarphi_3,\upvarphi_4  \}$ is a fundamental system for $i \mathcal{L} \Psi(r,z)= z \Psi(r,z),$ then we can express $\Psi^{\pm}_{l}(r,z)$ in terms of $\upvarphi_j(r,z),$ i.e., there is exist  $a^{\pm}_{lj} (z)$ such that 
\begin{align}
\label{eq:Psi-interms-phi_j}
    \Psi^{\pm}_{l} (r,z) = \sum_{j=1}^4 a^{\pm}_{lj} (z) \upvarphi_j (r,z), \qquad l=1,\cdots,4.
\end{align}

Therefore, by Lemma \ref{lem:kernel-K}, we have 
\begin{align}
\label{eq:GKermel-phi_j-plus}
\begin{split}
\mathcal{K}^{+}(r,s,z) &=\frac{i}{d^{+}} \sum_{j=1}^4 (   \omega_{22}^{+} a_{1j}^{+} - \omega_{12}^{+}(z) a_{2j}^{+} ) \left(
\upvarphi_j(r,z) \upvarphi_1^t(s,z) \sigma_1 \mathbb{1}_{\{ 0<s \leq r  \}} + \upvarphi_1(r,z) \upvarphi_j^t(s,z) \sigma_1  \mathbb{1}_{\{ r \leq s < \infty  \}}
\right) \\
&+ \frac{i}{d^{+}} \sum_{j=1}^4 (     -\omega_{21}^{+} a_{1j}^{+} + \omega_{11}^{+} a_{2j}^{+} ) \left(
\upvarphi_j(r,z) \upvarphi_2^t(s,z) \sigma_1 \mathbb{1}_{\{ 0<s \leq r  \}} + \upvarphi_2(r,z) \upvarphi_j^t(s,z) \sigma_1  \mathbb{1}_{\{ r \leq s < \infty  \}}
\right) 
\end{split}
\end{align}
\begin{align}
\label{eq:GKermel-phi_j-minus}
\begin{split}
\mathcal{K}^{-}(r,s,z) &=\frac{i}{d^{-}} \sum_{j=1}^4 (   \omega_{22}^{-} a_{1j}^{-} - \omega_{12}^{-}(z) a_{2j}^{-} ) \left(
\upvarphi_j(r,z) \upvarphi_1^t(s,z) \sigma_1 \mathbb{1}_{\{ 0<s \leq r  \}} + \upvarphi_1(r,z) \upvarphi_j^t(s,z) \sigma_1  \mathbb{1}_{\{ r \leq s < \infty  \}}
\right) \\
&+ \frac{i}{d^{-}} \sum_{j=1}^4 (     -\omega_{21}^{-} a_{1j}^{-} + \omega_{11}^{-} a_{2j}^{-} ) \left(
\upvarphi_j(r,z) \upvarphi_2^t(s,z) \sigma_1 \mathbb{1}_{\{ 0<s \leq r  \}} + \upvarphi_2(r,z) \upvarphi_j^t(s,z) \sigma_1  \mathbb{1}_{\{ r \leq s < \infty  \}}
\right) 
\end{split}
\end{align}

Similarly, we will express the columns of $F_1$ in terms of $\Psi_j.$ Since $\{\Psi^{\pm}_1 (\cdot,z),\Psi^{\pm}_2 (\cdot,z) ,\Psi^{\pm}_3 (\cdot,z) ,\Psi^{\pm}_4  (\cdot,z)    \}$ form a fundamental system for $i \mathcal{L} F_1(r,z)=z F_1(r,z),$ then we can express $\upvarphi_l(r,z)$ in terms of   $\Psi^{\pm}_j(r,z),$ i.e., there exist $b^{\pm}_{lj}(z)$ such that 
\begin{align}
\label{exp-phi-terms-psi}
    \upvarphi_l(r,z)= \sum_{j=1}^{4} b^{\pm}_{lj} \Psi^{\pm}_j(r,z), \quad l=1,\cdots, 4.
\end{align}

\begin{lemma} \label{lem:EV1}
Let $z \in \Omega$ such that $z=\lambda + i y ,$ where $\lambda=\xi \pm \frac{\sqrt{17}}{8}$ and $\delta_0<|\xi|<R.$ Then 
    $d^{\pm}(\lambda)=0$ if and only if $i \calL$ has an eigenvalue at $z.$ 
\end{lemma}
\begin{proof}
We will prove only the case $\operatorname{Im}(z) \geq 0$ and the case $\operatorname{Im}(z) \leq 0$ can be treated similarly. We first assume that, there exists an eigenvalue at $z.$ Then there exist $(\alpha_1,\alpha_2) \neq 0$ such that $ \alpha_1 \varphi_1 + \alpha_2 \varphi_2 = \beta_1 \Psi_1^{+} + \beta_2 \Psi_2^{+},$ for some $\beta_1, \beta_2.$ Notice that in view of \eqref{exp-phi-terms-psi}, we must have $\alpha_1 b_{13}^{+} + \alpha_2 b_{23}^{+}=0$ and $\alpha_1 b_{14}^{+} + \alpha_2 b_{24}^{+}=0,$ i.e, $\det \begin{pmatrix}
    b_{13}^{+} & b_{23}^{+} \\ b_{14} & b_{24}^{+}
\end{pmatrix} =0.$ Recall that $d^{+}= \omega_{11}^{+} \omega_{22}^{+} - \omega_{12}^{+} \omega_{21}^{+}.$ Using \eqref{exp-phi-terms-psi} and the fact that $\mu_{12}^{+}=\mu_{24}^{+}=\mu_{23}^{+}=\mu_{21}^{+}=\mu_{14}^{+}=0,$  we have $\omega_{11}^{+}=b_{13}^{+} \mu_{13}^{+},$ $\omega_{22}^{+}=b_{24}^{+} \mu_{24}^{+},$ $\omega_{12}^{+}=b_{23}^{+} \mu_{13}^{+}$ and $\omega_{21}^{+}=b_{14}^{+} \mu_{24}^{+}.$ Therefore, $d^{+}=\mu_{13}^{+} \mu_{24}^{+} \det \begin{pmatrix}
    b_{13}^{+} & b_{23}^{+} \\ b_{14}^{+} & b_{24}^{+}
\end{pmatrix}=0.$  Next, we assume that $d^{+}=0.$ For that, we express $F_1$ in terms of $\Psi^{+}$ and $\PPsi^{+},$ i.e., $F_1=\Psi^{+} A + \PPsi^{+} B,$ for some $A$ and $B.$ Since $W(\Psi^{+},F_1)=W(\Psi^{+},\PPsi^{+})B$ and $\mu_{14}=\mu_{23}=0,$ we have $d^{+}=\mu_{13} \mu_{24} \det(B)=0.$ This implies that $\det(B)=\det(\bar{B})=0.$ Thus there exists $v=\begin{pmatrix}
    \gamma \\ \delta
\end{pmatrix}\neq 0$ such that $\Bar{B}v=0$. Let $\bar{A} v =\begin{pmatrix}
    \alpha \\ \beta
\end{pmatrix} .$ Then, we have $\overline{F}_1 v = \overline{\Psi}^{+} \bar{A} v + \overline{\PPsi}^{+} \bar{B} v=\overline{\Psi}^{+} \bar{A} v ,$ which yields $\gamma \bar{\varphi}_1 + \delta \bar{\varphi}_2= \alpha  \overline{\Psi}^{+}_1  + \beta \overline{\Psi}^{+}_2. $ Recall that for $\im(z)>0,$ we have $\im(k_1(z))>0$ and $\im(k_2(z))>0.$ Then the asymptotics of $\varphi_1$ and $\varphi_2$ near the origin, together with those of $\Psi_1$ and $\Psi_2$ near infinity, show that the relation $\bar{\gamma} \varphi_1 + \bar{\delta} \varphi_2= \bar{\alpha}  \Psi^{+}_1  + \bar{\beta} \Psi^{+}_2$ 
guarantees the existence of eigenvalues for large $z,$ with $\im(z)>0.$ Next, we consider the case where $z=\lambda+i0.$  Notice that, in view of the relation $\gamma \bar{\varphi}_1 + \delta \bar{\varphi}_2= \alpha  \overline{\Psi}^{+}_1  + \beta \overline{\Psi}^{+}_2 $ and the behavior of $k_1(z)$ and $k_2(z)$ as $\im(z) \to 0$ (see Lemma \ref{lem:k_j-behavior}), if $\alpha=0$ then we have an eigenfunction. Next, we show that $\alpha=0.$ Note that for $\im (z)=0$, if $\Phi$ satisfies $i\mathcal{L}\Phi=z\Phi$, then $\sigma_3\overline \Phi$ is a solution of $i\mathcal{L}(\sigma_3\overline{\Phi})=z(\sigma_3\overline{\Phi})$. This shows that the Wronskians below are independent of $r$: 
\begin{align*}
    W(F_1, \sigma_3 \overline{F_1}) &=A^t W( \Psi^{+}, \sigma_3 \overline{\Psi}^{+} ) \bar{A} + A^t  W( \Psi^{+}, \sigma_3 \overline{\PPsi}^{+} ) \bar{B }\\
   & + B^t  W( \PPsi^{+}, \sigma_3 \overline{\Psi}^{+} )  \bar{A} + B^t W( \PPsi^{+}, \sigma_3 \overline{\PPsi}^{+} ) \bar{B}.
\end{align*}
Evaluating at $r=0$, the left-hand side is zero. Using the fact that $\bar{B}v=0,$ we have $     0=\langle W(F_1, \sigma_3 F_1) v, v \rangle= \langle W(\Psi^{+}, \sigma_3 \overline{\Psi}^{+}) \bar{A}v, \bar{A}v \rangle.$ In view of the asymptotics of $\Psi^{+},$ we have $W(\Psi^{+}, \sigma_3 \overline{\Psi}^{+})= \begin{pmatrix}
    k(\lambda) & 0 \\ 0& 0
\end{pmatrix}.$ Therefore, we have $0=\langle W(\Psi^{+}, \sigma_3 \overline{\Psi}^{+}) \begin{pmatrix}
    \alpha \\ \beta
\end{pmatrix}, \begin{pmatrix}
    \alpha \\ \beta
\end{pmatrix} \rangle= k(\lambda) |\alpha|^2,$ which yields $\alpha=0.$ This concludes the proof of Lemma~\ref{lem:EV1}.
\end{proof}

Next, in view of the explicit formula for the Green's kernel and the behavior of the coefficients $\omega_{ij}^{\pm}$ established in Lemmas~\ref{lemma:omega_j-behavior} and~\ref{lem:omega_j-behavior-large-xi}, all terms involved are sufficiently regular to justify taking the limit $b \to 0^+$ under the integral in Lemma~\ref{lem:ST1}.  We also observe below that the coefficients of the terms $\varphi_3(r,z)\varphi_2^t(s,z)$ and $\varphi_2(r,z) \varphi_3^t(s,z)$ appearing in equations~\eqref{eq:GKermel-phi_j-plus} and~\eqref{eq:GKermel-phi_j-minus} drop out. Moreover, as we will see later in Proposition~\ref{jump-resol}, the coefficients of $\varphi_l(r,z) \varphi_1(s,z)$ and $\varphi_1(r,z) \varphi_l(s,z)$ for $l = 3,4$, as well as those of $\varphi_4(r,z) \varphi_2(s,z)$ and $\varphi_2(r,z) \varphi_4(s,z)$, are independent of $z$, and their differences vanishes in the integral kernel of the resolvent jump.

\begin{coro}
\label{kernel-no-sigularity-non-resonance}
Let $z=\xi \pm \frac{\sqrt{17}}{8} \in \Omega,$ where $\pm \im(z)>0$  and  $0<|\xi|<\delta_0.$  Then in both scenarios, for some $\alpha_j^\pm$ and $\beta_j^\pm$, we have 
\begin{align}
\label{eq:Resol-kernel-small-xi-2}
\begin{split}
\mathcal{K}^{\pm}(r,s,z) &=\frac{i}{d^{\pm}} \sum_{j=1}^4 \alpha^{\pm}_j(z)\left(
\upvarphi_j(r,z) \upvarphi_1^t(s,z) \sigma_1 \mathbb{1}_{\{ 0<s \leq r  \}} + \upvarphi_1(r,z) \upvarphi_j^t(s,z) \sigma_1  \mathbb{1}_{\{ r \leq s < \infty  \}}
\right) \\
&+ \frac{i}{d^{\pm}} \sum_{\substack{j=1 \\ j \neq 3}}^4 \beta^{\pm}_j(z) \left(
\upvarphi_j(r,z) \upvarphi_2^t(s,z) \sigma_1 \mathbb{1}_{\{ 0<s \leq r  \}} + \upvarphi_2(r,z) \upvarphi_j^t(s,z) \sigma_1  \mathbb{1}_{\{ r \leq s < \infty  \}}
\right)  
\end{split}
\end{align}
where
\begin{align}
\begin{split}
\label{eq:coef-resol-small-xi}
    |\alpha^{\pm}_1(z)| &\lesssim  e^{-\frac{6}{\sqrt{2}} r_{\varepsilon}}, \quad  |\alpha^{\pm}_2(z)| \lesssim e^{-\frac{3}{\sqrt{2}} r_{\varepsilon}}, \quad  |\alpha^{\pm}_3(z)| \lesssim 1, \quad  |\alpha^{\pm}_4(z)| \lesssim 1,\\
     |\beta^{\pm}_1(z)| &\lesssim e^{-\frac{3}{\sqrt{2}} r_{\varepsilon}}, \quad |\beta^{\pm}_2(z)| \lesssim \sqrt{|\xi|}, \quad
     |\beta^{\pm}_4(z)| \lesssim 1.
\end{split}
\end{align}
\end{coro}

\begin{proof}
We first note that the coefficient of the $\upvarphi_3(r,z) \upvarphi_2^t(s,z) \sigma_1 \mathbb{1}_{\{ 0<s \leq r  \}} + \upvarphi_2(r,z) \upvarphi_3^t(s,z) \sigma_1  \mathbb{1}_{\{ r \leq s < \infty  \}}$ vanishes, i.e., 
    \begin{align*}
 - \omega_{21}^{\pm}(z) a^{\pm}_{13}+ \omega_{11}^{\pm}(z) a^{\pm}_{23}(z) =0     
    \end{align*}

Indeed, we have $a^{\pm}_{13}= \frac{-1}{  \Lambda_{13}} \omega_{11}^{\pm} $
and $a^{\pm}_{23}= \frac{-1}{  \Lambda_{13}} \omega_{21}^{\pm}.  $ Then 
\begin{align*}
     - \omega_{21}^{\pm}(z) a^{\pm}_{13}+ \omega_{11}^{\pm}(z) a^{\pm}_{23}(z) =  - \omega_{21}^{\pm}(z) \frac{-1}{  \Lambda_{13}} \omega_{11}^{\pm} +   \omega_{11}^{\pm}(z)  \frac{-1}{  \Lambda_{13}} \omega_{21}^{\pm} = 0 
\end{align*} 
By Lemma \ref{Lambda_j-behavior}, it follows that 
\begin{align*}
  \omega_{l1}^{\pm} &= -  \Lambda_{13} a^{\pm}_{l3}, \qquad \qquad  \qquad  \qquad \longrightarrow | \omega_{l1}^{\pm}| \simeq |a^{\pm}_{l3} |  \\
  \omega_{l2}^{\pm} &= -  \Lambda_{23} a^{\pm}_{l3}-  \Lambda_{24} a^{\pm}_{l4}   \; \qquad \qquad \longrightarrow | \omega_{l2}^{\pm}| \simeq |a^{\pm}_{l4} | \\
  \omega_{l3}^{\pm} &=\Lambda_{13} a^{\pm}_{l1} + \Lambda_{23} a^{\pm}_{l2} -\Lambda_{34} a^{\pm}_{l4}  \quad \longrightarrow | \omega_{l3}^{\pm}| \simeq |a^{\pm}_{l1} | \\
    \omega_{l4}^{\pm} &=\Lambda_{24} a^{\pm}_{l2} + \Lambda_{34} a^{\pm}_{l3}   \qquad  \qquad \quad \longrightarrow | \omega_{l4}^{\pm}| \simeq |a^{\pm}_{l2} |
\end{align*}
Then all other estimates follows from Lemma \ref{lemma:omega_j-behavior} and \ref{lem:parity-of-omega-j} together with the fact that for both scenarios $|k_1(z)| \simeq \sqrt{|\xi|}.$ 
\end{proof}

\subsection{Computing the jump of the resolvent}
\label{subsec:JumpResol}
We are now in a position to derive a representation of the flow $e^{t \mathcal{L}}$ in terms of the resolvent. Recall that from Lemma~\ref{lem:ST1} that we define $e^{t \mathcal{L}}$ as limit in the weak sense, that is, for any compactly supported smooth functions $\Phi, \Psi \in L^2_r$, we have
\begin{align*}
    \langle e^{t \mathcal{L}} \Phi, \Psi \rangle= \lim_{R \to \infty} \lim_{b \to 0^{+}}  \frac{1}{2\pi i} \int_{b-iR}^{b+iR} e^{it \lambda} \langle \left( e^{-bt}(i\mathcal{L}-(\lambda + i b))^{-1} - e^{bt} (i\mathcal{L}-(\lambda - i b))^{-1}  \right) , \Phi, \Psi \rangle d \lambda.
\end{align*}
For any compactly supported smooth functions $\Phi$ and $\Psi$ in $L^2_r$ define
\begin{align}
\label{eq:I_R-def}
    \mathcal{I}_R(\Phi,\Psi):=\lim_{b \to 0^{+}}  \frac{1}{2\pi i} \int_{-R+ib}^{R+ib} e^{it \lambda} \langle \left( e^{-bt}(i\mathcal{L}-(\lambda + i b))^{-1} - e^{bt}(i\mathcal{L}-(\lambda - i b))^{-1}  \right)  \Phi, \Psi \rangle d \lambda.
\end{align}
Note that $\lim_{R\to\infty}\mathcal{I}_R(\Phi,\Psi)$  is what we expect to be the main part of the contour integral representing $\langle e^{t \mathcal{L}} \Phi, \Psi \rangle$. Below we will also use the notation
$$ I_R =[-R,R] \backslash (-\frac{\sqrt{17}}{8},\frac{\sqrt{17}}{8}),\qquad I =\mathbb{R} \backslash (-\frac{\sqrt{17}}{8},\frac{\sqrt{17}}{8}).$$

\begin{prop}
\label{jump-resol}
    Let $I_R =[-R,R] \backslash (-\frac{\sqrt{17}}{8},\frac{\sqrt{17}}{8}) $ and $I=\R \backslash (-\frac{\sqrt{17}}{8},\frac{\sqrt{17}}{8}). $ For any compactly supported functions $\Phi, \Psi \in L^2_r(0,\infty), $ we have 
    \begin{align}
    \label{eq:Evol-tensor-form}
  \langle e^{t \mathcal{L}} \Phi, \Psi \rangle =    \lim_{R \to \infty}  \mathcal{I}_R(\Phi,\Psi)=   \frac{1}{2 \pi i } \int_{I} e^{it \lambda} \left< \int_0^{\infty} F_1(\cdot,\lambda) \calW(\lambda)F_1(s,\lambda)^{t} \sigma_1 \Phi(s)ds, \Psi(\cdot) \right>_{L^2_r} d \lambda ,
      \end{align}
   where \begin{align*}
       \calW(\lambda):=\kappa(\lambda) \calD^{-}(\lambda)^{-1} \underbar{e}_{11} \calD^{+}(\lambda)^{-t}, \quad \underbar{e}_{11}:=\begin{pmatrix}
      1 & 0 \\ 0 & 0
  \end{pmatrix}, \quad \calD^{\pm}(\lambda):=\calD^{\pm}(\lambda\pm i0),
   \end{align*}   
and $\kappa(\lambda)=W (\sigma_3 \Psi_1^{+}(\cdot,-\lambda), \Psi_1^{+}(\cdot,\lambda) )$ is an odd function in $\lambda.$ Moreover, for $\lambda=\xi\pm \frac{\sqrt{17}}{8} ,$ we have $|\kappa(\lambda)| \simeq \sqrt{|\xi|}$ for small and large $|\xi|.$ 
  \end{prop}

\begin{proof} 
We first justify that the limit as $b \to 0^{+}$ can be taken inside the $\lambda$-integral in \eqref{eq:I_R-def} for $\mathcal{I}_R(\Phi,\Psi).$ 
For this, we will use the resolvent kernel expression given in Corollary \ref{kernel-no-sigularity-non-resonance}. First, notice that for $z=\xi\pm \frac{\sqrt{17}}{8} \in \Omega$ with $0<|\xi|<\delta_0,$ we have $F_1(\cdot,z)$ and $F_2(\cdot,z)$ are continuous function as $z$ crosses the real line for $r\leq r_0=\frac{\varepsilon_0}{\sqrt{|\xi|}}.$ The same holds for all $z$ as a consequence of the theory of existence, uniqueness and continuous dependence on parameters for ordinary differential equations. It follows from the bounds \eqref{eq:coef-resol-small-xi} from Corollary \ref{kernel-no-sigularity-non-resonance} that the integrand or resolvent kernels in \eqref{eq:I_R-def} are non-singular for $0<|\xi|<\delta_0.$ Therefore, in view of Green's kernel expression in \eqref{eq:Resol-kernel-small-xi-2}, 
we may pass the limit inside the integral in \eqref{eq:I_R-def} provided $d^{\pm} \neq 0$. Since, by Assumption \ref{asmp:3} $ i \calL$ has no embedded eigenvalues  (see also Lemma \ref{lem:GP-thres-emb-EV} for $i\calL_{GP}),$  by Lemma \ref{lem:EV1} we have $d^{\pm}\neq0 .$ Moreover, $ i \calL$ has no gap eigenvalues in $(-\frac{\sqrt{17}}{8},\frac{\sqrt{17}}{8}).$ Therefore the limit as $b \to 0^{+}$ can be taken inside the integral and we obtain \begin{align*}
    \mathcal{I}_{R}(\Phi,\Psi)= \frac{1}{2\pi i} \int_{I_R} e^{it \lambda} \langle \left( (i\mathcal{L}-(\lambda + i 0))^{-1} - (i\mathcal{L}-(\lambda - i 0))^{-1}  \right)  \Phi, \Psi \rangle d \lambda.
\end{align*}
The rest of the proof is to use this expression to extract the distorted Fourier transform representation of the evolution in tensorial form as \eqref{eq:Evol-tensor-form}. We closely follow the argument in \cite[Proposition 5.9]{LSS25}. We denote the integral kernel of the jump of the resolvent across the essential spectrum by 
\begin{align*}
    \mathcal{Q}(r,s,\lambda)&:= \left( i \mathcal{L}- (\lambda+i0) \right)^{-1 } (r,s) -\left( i \mathcal{L}- (\lambda - i 0) \right)^{-1 } (r,s)\\
    &=\mathcal{K}^{+}(r,s,\lambda+i0) - \mathcal{K}^{-}(r,s,\lambda-i0)
\end{align*}

By Lemma \ref{lem:kernel-K}, we have 
\begin{align}
\label{eq:Q-rep2}
\begin{split}
    \mathcal{Q}(r,s,\lambda) 
    &=  i \left( \Psi^{+}(r,\lambda) (\calD^{+}(\lambda))^{-t} F_1(s,\lambda)^t - \Psi^{-}(r,\lambda) (\calD^{-}(\lambda))^{-t} F_1(s,\lambda)^t \right) \sigma_1 \mathbb{1}_{\{ 0<s\leq r \}} \\
&+  i \left(    F_1(r,\lambda) (\calD^{+}(\lambda))^{-1} \Psi^{+}(s,\lambda)^t -   F_1(r,\lambda) (\calD^{-}(\lambda))^{-1} \Psi^{-}(s,\lambda)^t \right) \sigma_1 \mathbb{1}_{\{ r \leq s\leq \infty \}}
\end{split}
\end{align}
where $\Psi^{\pm}(r,\lambda):=\Psi^{\pm}(r,\lambda \pm i 0)$ and $D^{\pm}(\lambda):=D^{\pm}(\lambda \pm i 0).$ Note that, from above the Green's kernel $\mathcal{K}$ satisfies 
\begin{equation}
\label{eq:symQ}
       \mathcal{Q}(r,s,\lambda)^t = \sigma_1    \mathcal{Q}(r,s,\lambda) \sigma_1
\end{equation}

Next, we want to express the kernel $\mathcal{K}$ in terms of $F_j(r,\lambda)$ only. Thus, one can write $\Psi^{\pm}$ in terms of $F_j,$ i.e., there exist two matrices $N_{\pm}(z)$ and $M_{\pm}(z)$ such that 
\begin{equation}
\label{eq:Psi-F1F2}
\Psi^{\pm}(r,z) :=   F_2(r,z) N^\pm(z) + F_1(r,z) M^\pm(z),
\end{equation}
where the entries of $N_\pm$ can be determined in terms of the coefficients $a_{lj}$ from \eqref{eq:Psi-interms-phi_j}. Using the behavior of $F_1(r,z)$ as $r \to 0,$ we have 
$\mathcal{W}(F_1(\cdot,z)M^\pm(z), F_1(\cdot,z))=0. $ Therefore, 
\begin{align*}
    \calD^{\pm}(z)=W( \Psi^{\pm}(\cdot,z),F_1(\cdot,z))= N^\pm(z)^t W( F_2(\cdot,z),F_1(\cdot,z)).
\end{align*}
Let $\Uplambda(z):=W(F_2(\cdot,z),F_1(\cdot,z))).$ Note that the entries of this matrix are precisely the coefficients $\Lambda_{ij}$ given by Lemma \ref{Lambda_j-behavior} and therefore $\mathrm{det} (\mathrm{\Uplambda(z)}) \neq 0.$ Since $\calD^{\pm}(z)$ and $\Uplambda(z)$ are invertible then $N^\pm(z)^t$  is invertible with 
\begin{equation*}
    \calD^{\pm}(z)^{-t}=N^\pm(z)^{-1} \Uplambda(z)^{-t}
\end{equation*}
Inserting \eqref{eq:Psi-F1F2} into \eqref{eq:Q-rep2}, we find that the terms with one factor $F_1$ and one factor $F_2$ cancel out. Therefore, $ \mathcal{Q}(r,s,\lambda)$ must be of the form 
\begin{align*}
   \mathcal{Q}(r,s,\lambda)= i F_1(r,\lambda) \mathcal{M}^{+}(s,\lambda) \sigma_1 \mathbb{1}_{\{ 0<s\leq r \}} + i F_1(r,\lambda)  \mathcal{M}^{-}(s,\lambda) \sigma_1 \mathbb{1}_{\{ 0<r\leq s \}}. 
\end{align*}

By continuity, we must $F_1(r,\lambda)( \mathcal{M}^{+}(r,\lambda)-\mathcal{M}^{-}(s,\lambda))=0.$ However, $\det(F_1) \neq 0$ because ifnot we can find a $c=c(r,\lambda)$ such that $c\varphi_{11}=\varphi_{21}$ and $c\varphi_{12}=\varphi_{22},$ which impossible in view of their asymptotics, see Proposition \ref{Prop:contraction-T_j(F_j)-xi-small}. Therefore, we must have $\mathcal{M}^{+}(r,\lambda)=\mathcal{M}^{-}(s,\lambda),$ whence 
\begin{align*}
 \mathcal{Q}(r,s,\lambda)= iF_1(r,\lambda) \mathcal{M}(s,\lambda) \sigma_1.
\end{align*}
Using \eqref{eq:symQ}, we have 
\begin{align*}
    i \sigma_1 \mathcal{M}(s,\lambda)^t F_1(r,\lambda)^t= i \sigma_1 F_1(s,\lambda) \mathcal{M}(r,\lambda) \sigma_1^2 , 
\end{align*}
which yields
\begin{align*}
 F_1(s,\lambda)^{-1} \mathcal{M}(s,\lambda)^t= \mathcal{M}(r,\lambda)F_1(r,\lambda)^{-t}=:\calW(\lambda)    
\end{align*}
Therefore, we have 
\begin{align*}
    \mathcal{Q}(r,s,\lambda)= i F_1(r,\lambda) \calW(\lambda) F_1(s,\lambda)^t \sigma_1 ,
\end{align*}
and by \eqref{eq:symQ}, we have \begin{equation} \label{eq:W-sym}
  \calW(\lambda)^t = \calW(\lambda) .
\end{equation} 
Next, we compute $\calW(\lambda).$ Firs notice that in view of the asymptotics of the columns $\Psi^{+},$ we have 
\begin{align} \label{eq:zeroWrons}
    W( \Psi^{+}(\cdot,\lambda) (\calD^{+}(\lambda))^{-t},  \Psi^{+}(\cdot, \lambda) ) = (\calD^{+}(\lambda))^{-1}  W( \Psi^{+}(\cdot,\lambda) ,  \Psi^{+}(\cdot, \lambda) )=0.
\end{align}
From \eqref{eq:Q-rep2} we have for $r \geq s,$ 
\begin{align*}
   \mathcal{Q}(r,s,\lambda)= i \left(\Psi^{+}(r,\lambda) (\calD^{+}(\lambda))^{-t} - \Psi^{-}(r,\lambda) (\calD^{-}(\lambda))^{-t}  \right)F_1(s,\lambda) \sigma_1.
\end{align*}
Hence 
\begin{align*}
    F_1(r,\lambda) \calW(\lambda)=\Psi^{+}(r,\lambda)  (\calD^{+}(\lambda))^{-t} - \Psi^{-}(r,\lambda) (\calD^{-}(\lambda))^{-t}
\end{align*}

On the one hand, using \eqref{eq:zeroWrons}, we obtain
\begin{align*}
  W(  F_1(\cdot,\lambda) \calW(\lambda), \Psi^{+}(\cdot,\lambda)) &= - W( \Psi^{-}(\cdot,\lambda) (\calD^{-}(\lambda))^{-t},  \Psi^{+}(\cdot,\lambda) ) \\
  &= - (\calD^{-}(\lambda))^{-t} W( \Psi^{-}(\cdot,\lambda),  \Psi^{+}(\cdot,\lambda)),
\end{align*}
on the other hand, by the definition of $\calD^{+}(\lambda)$ together with \eqref{eq:W-sym} we have 
\begin{align*}
      W(  F_1(\cdot,\lambda) \calW(\lambda), \Psi^{+}(\cdot,\lambda))= -\calW(\lambda)^t \calD^{+}(\lambda)^t = -\calW(\lambda) \calD^{+}(\lambda)^t.
\end{align*}
This implies that 
\begin{align*}
   \calW(\lambda)= \calD^{-}(\lambda)^{-1} W( \Psi^{-}(\cdot,\lambda),  \Psi^{+}(\cdot,\lambda))   \calD^{+}(\lambda)^{-t}.
\end{align*}

Finally, in view of the behavior of $k_j^{\pm}(\lambda):=k_j(\lambda \pm i0)$ from Lemma \ref{lem:k_j-behavior}, we have
\begin{align*}
    W( \Psi^{-}(\cdot,\lambda),  \Psi^{+}(\cdot,\lambda)) = \begin{pmatrix} \kappa(\lambda) & 0 \\ 0 & 0 
    \end{pmatrix},
\end{align*}
where \begin{align*}
   \kappa(\lambda)=W (\Psi_1^{-}(\cdot,\lambda), \Psi_1^{+}(\cdot,\lambda) )=W (\sigma_3 \Psi_1^{+}(\cdot,-\lambda), \Psi_1^{+}(\cdot,\lambda) ).
\end{align*}
It is clear that $\kappa(-\lambda)=-\kappa(\lambda).$  Recall that, by Lemma \ref{lem:k_j-behavior} and Remark \ref{rem:behavior-k_j-lambda}, we have  $k_1^{\pm}(\lambda)=-k_1^{\pm}(-\lambda),$ Therefore, using the asymptotic of $\Psi^{+}(r,\lambda)$ for small and large $|\xi|,$
we have $| \kappa(\lambda)| \simeq \sqrt{|\xi|} .$
\end{proof}

Next, we write a more concrete expression of the tensorial structure of $ F_1(\cdot,\lambda) \calW(\lambda)F_1(s,\lambda)^{t} \sigma_1.$

\begin{lemma}
\label{Resol-interms-phi}
Let $z=\lambda\pm i 0 $ such that  $\lambda=\xi\pm \frac{\sqrt{17}}{8}.$ Then we have, 
    \begin{align}
    \label{eq:Tesor-struct}
      F_1(\cdot,\lambda) \calW(\lambda)F_1(s,\lambda)^{t} = \frac{\kappa(\lambda)}{d^{+}(\lambda) d^{-}(\lambda)} \Theta(r,\lambda) \Theta(s,\lambda)^{t} ,
    \end{align}
where $d^{\pm}(\lambda):=\det(\calD^{\pm}(\lambda \pm i 0)),\; \omega_{ij}^\pm(\lambda)=\omega_{ij}^\pm(\lambda\pm i0)$ and $\Theta(r,\lambda) :=\begin{pmatrix}
    \Theta_1(r,\lambda) \\ 
    \Theta_2(r,\lambda)
\end{pmatrix}$
satisfies 
\begin{align}
\label{def-theta}
    \Theta(r,\lambda) := \omega_{22}^{+}(\lambda) \upvarphi_1(r,\lambda) - w_{21}^{+}(\lambda) \upvarphi_2(r,\lambda).
\end{align}
Moreover, $ \frac{\kappa(\lambda)}{d^{+}(\lambda) d^{-}(\lambda)} $ is an odd function,
$\Theta_1(r,\lambda) $ is an odd  function in $\lambda$ and $\Theta_2(r,\lambda) $ is an even function in $\lambda.$  
\end{lemma}
\begin{proof}
Recall that $\calW(\lambda)=\kappa(\lambda) \calD^{-}(\lambda)^{-1} \underbar{e}_{11} \calD^{+}(\lambda)^{-t},$ and $
 ( \calD^{\pm})^{-1}(z)  : = \frac{1}{d^{\pm}(z) } \begin{pmatrix}
       \omega_{22}^{\pm}(z) & - \omega_{12}^{\pm}(z)   \\
       -\omega_{21}^{\pm}(z) & \omega_{11}^{\pm}(z)
   \end{pmatrix} .$
Using the relation between $\omega_{ij}^{+}$ and $\omega_{ij}^{-}$ from Lemma \ref{lem:parity-of-omega-j}, we obtain 
\begin{align*}
    \calW(\lambda)=\frac{\kappa(\lambda)}{d^{+}(\lambda) d^{-}(\lambda)} \begin{pmatrix}
        (\omega_{22}^{+}(\lambda))^2 & -\omega_{21}^{+}(\lambda) \omega_{22}^{+}(\lambda) \\
   - \omega_{22}^{+}(\lambda)  \omega_{21}^{+}(\lambda)  &  (\omega_{21}^{+}(\lambda))^2
    \end{pmatrix}.
\end{align*}
Next, writing  $F_1=\varphi_1 \begin{bmatrix}
    1 & 0
\end{bmatrix} + \varphi_2 \begin{bmatrix}
   0 & 1
\end{bmatrix}  $ and $F_1^t=\begin{bmatrix}
    1 \\  0
\end{bmatrix}\varphi_1^t +\begin{bmatrix}
   0 \\ 1
\end{bmatrix}   \varphi_2^t ,$ and using the representation of $ \calW(\lambda)$ above, we obtain \eqref{eq:Tesor-struct} by direct computation. Next, we justify the parity of this tensorial structure. By Lemma \ref{lem:parity-of-omega-j}, we have  $  d^{-}(\lambda)=-d^{+}(-\lambda).$  Therefore, $d^{-}(\lambda) d_{+}(\lambda) =   -d_{+}(-\lambda)d_{+}(\lambda), $ i.e., even function in $\lambda.$ From Proposition \ref{jump-resol} we have $\kappa(\lambda)$ is an odd function. Thus $ \frac{\kappa(\lambda)}{d^{+}(\lambda) d^{-}(\lambda)} $ is an odd function. Denote by $\varphi_j= \begin{pmatrix}
    \varphi_{j1} \\ \varphi_{j1}
\end{pmatrix}   .$ Using the identities \eqref{eq:sym-varphi_j}, we have 
$ \upvarphi_{11}(r,\lambda)=- \upvarphi_{11}(r,-\lambda), $  $\upvarphi_{12}(r,\lambda)= \upvarphi_{12}(r,-\lambda)$ and $ \upvarphi_{21}(r,\lambda)= \upvarphi_{21}(r,-\lambda), $  $\upvarphi_{22}(r,\lambda)= - \upvarphi_{22}(r,-\lambda). $ Therefore, $\Theta_1(r,-\lambda)=-\Theta_1(r,\lambda)$ and $\Theta_2(r,-\lambda)=\Theta_2(r,\lambda).$
\end{proof}

In the remainder of this section, we write the expression $\Theta(r,\lambda) $ as well as the tensorial structure of $ F_1(\cdot,\lambda) \calW(\lambda)F_1(s,\lambda)^{t} \sigma_1,$  in terms $\Psi_j.$ It turns out that $\Psi_4$ does not contributions to the final expression. Recall that from \eqref{exp-phi-terms-psi} we have 
\begin{align}
    \upvarphi_l(r,z)= \sum_{j=1}^{4} b^{\pm}_{lj} \Psi^{\pm}_j(r,z), \quad l=1,\cdots, 4.
\end{align}

In order, to write $\Theta(r,\lambda) $ in terms of  $\Psi_j(r,z)$ we only need $l=1,2.$\\

\begin{lemma}
\label{behav-omega-b} 
Let $z=\lambda\pm i 0 $ such that  $\lambda=\xi\pm \frac{\sqrt{17}}{8}.$  We have 
\begin{align}
\label{drop-out-omega-b}
     \omega^{+}_{22} b^{\pm}_{14}  =  \omega_{12}^{\pm}  b^{\pm}_{24}=\frac{ \omega_{12}^{\pm}  \omega_{22}^{\pm}}{\mu^{\pm}_{42}}
\end{align}
and for small $|\xi|,$ we have 
  \begin{align*}
 |\omega^{+}_{22} b^{\pm}_{11}| & \lesssim  \frac{1}{\sqrt{|\xi|}} e^{-\frac{3}{\sqrt{2}} \frac{4}{5}r_{\varepsilon  }  } , \quad  |\omega^{+}_{22} b^{\pm}_{12}| \lesssim \frac{1}{\sqrt{|\xi|}}  e^{-\frac{3}{\sqrt{2}} \frac{4}{5}r_{\varepsilon  }  }, \\
 |\omega^{+}_{22} b^{\pm}_{13}| & \lesssim  \frac{1}{\sqrt{|\xi|}}
 e^{-\frac{3}{\sqrt{2}} \frac{4}{5}r_{\varepsilon  }  },  \quad   | \omega_{12}^{\pm} b^{\pm}_{21}|   \lesssim  \frac{1}{\sqrt{|\xi|}}  , \\
 |  \omega_{12}^{\pm} b^{\pm}_{22} | &\lesssim  \frac{1}{\sqrt{|\xi|}}     e^{\frac{3}{\sqrt{2}} \frac{r_{\varepsilon}}{5}  } , \quad  |  \omega_{12}^{\pm} b^{\pm}_{23} | \lesssim  \frac{1}{\sqrt{|\xi|}}
  \end{align*}
and for large $|\xi|,$ we have 
\begin{align*}
 |\omega^{+}_{22} b^{\pm}_{11}| & \lesssim  \frac{1}{|\xi| }  , \quad  |\omega^{+}_{22} b^{\pm}_{12}| \lesssim \frac{1}{|\xi| } , \quad |\omega^{+}_{22} b^{\pm}_{13}| \lesssim \frac{1}{|\xi| }
,  \\
   | \omega_{12}^{\pm} b^{\pm}_{21}|   &\lesssim \frac{1}{|\xi| }   , \quad  |  \omega_{12}^{\pm} b^{\pm}_{22} | \lesssim \frac{1}{|\xi| }  , \quad  |  \omega_{12}^{\pm} b^{\pm}_{23} | \lesssim \frac{1}{|\xi| } . 
  \end{align*}
\end{lemma}

\begin{proof}
   
By Lemma \ref{lem:mu-behavior-small-xi}, it follows that 
\begin{align*}
  \omega_{1l}^{\pm} &=  b^{\pm}_{l3}  \mu^{\pm}_{13} +  b^{\pm}_{l4} \mu^{\pm}_{14}   , \quad 
  \omega_{2l}^{\pm} =   b^{\pm}_{l4} \mu^{\pm}_{24} ,  \quad 
  \omega_{3l}^{\pm} = b^{\pm}_{l1} \mu^{\pm}_{31} + b^{\pm}_{l4} \mu^{\pm}_{34} ,  \\
    \omega_{4l}^{\pm} &= b^{\pm}_{l1} \mu^{\pm}_{41} + b^{\pm}_{l2} \mu^{\pm}_{42} + b^{\pm}_{l3} \mu^{\pm}_{43}
\end{align*}
Therefore, 
\begin{align*}
   b^{\pm}_{l4}   &=  \frac{ \omega_{2l}^{\pm}}{\mu^{\pm}_{24}}  ,  \qquad 
  b^{\pm}_{l3}   = \frac{1}{ \mu^{\pm}_{13}} \left(   \omega_{1l}^{\pm} -  \frac{ \omega_{2l}^{\pm}}{\mu^{\pm}_{24}}  \mu^{\pm}_{14}   \right) , \quad 
   b^{\pm}_{l1}   =  \frac{1}{ \mu^{\pm}_{31}} \left(   \omega_{3l}^{\pm} -   \frac{ \omega_{2l}^{\pm}}{\mu^{\pm}_{24}}  \mu^{\pm}_{34}  \right) ,  \\
 b^{\pm}_{l2}  &= \frac{1}{ \mu^{\pm}_{42}}    \left(  \omega_{4l}^{\pm} -  b^{\pm}_{l1} \mu^{\pm}_{41}  -b^{\pm}_{l3} \mu^{\pm}_{43} \right). 
\end{align*}    
For small $|\xi|, $ by Lemma \ref{lemma:omega_j-behavior}, \ref{lem:parity-of-omega-j} and \ref{lem:mu-behavior-small-xi} we have 
\begin{align*}
  b^{\pm}_{14}&= \frac{ \omega_{21}^{\pm}}{\mu^{\pm}_{24}} , \quad |
  b^{\pm}_{11}| = | \frac{ 1}{\mu^{\pm}_{31}} \left(  \omega_{31}^{\pm} -  \frac{ \omega_{21}^{\pm} \mu^{\pm}_{34}}{\mu^{\pm}_{24}}
 \right) | \lesssim \frac{1}{c \, k_1(z) + O(|\xi|^{\frac{3}{2}}) }   e^{\frac{3}{\sqrt{2}} \frac{r_{\varepsilon}}{5} } , \\
 |  b^{\pm}_{13} | &=  | \frac{1}{ \mu^{\pm}_{13}} \left(   \omega_{11}^{\pm} -  \frac{ \omega_{21}^{\pm}}{\mu^{\pm}_{24}}  \mu^{\pm}_{14}   \right) | \lesssim  \frac{1}{c \, k_1(z) + O(|\xi|^{\frac{3}{2}}) } e^{\frac{3}{\sqrt{2}} \frac{r_{\varepsilon}}{5} }    \\
 |  b^{\pm}_{12} | & = | \frac{1}{ \mu^{\pm}_{42}}    \left(  \omega_{41}^{\pm} -  b^{\pm}_{11} \mu^{\pm}_{41}  -b^{\pm}_{13} \mu^{\pm}_{43} \right)  | \lesssim \frac{1}{c \, k_1(z) + O(|\xi|^{\frac{3}{2}}) }   e^{\frac{3}{\sqrt{2}} \frac{2 r_{\varepsilon}}{5} }    
\end{align*}

and 
\begin{align*}
  b^{\pm}_{24}&= \frac{ \omega_{22}^{\pm}}{\mu^{\pm}_{24}} , \quad 
| b^{\pm}_{21 } | = | \frac{ 1}{\mu^{\pm}_{31}} \left(  \omega_{32}^{\pm} -  \frac{ \omega_{22}^{\pm} \mu^{\pm}_{34}}{\mu^{\pm}_{24}}   \right)  | \lesssim \frac{1}{c \, k_1(z) + O(|\xi|^{\frac{3}{2}}) } \\
 |  b^{\pm}_{23}| &=  | \frac{1}{ \mu^{\pm}_{13}} \left(   \omega_{12}^{\pm} -  \frac{ \omega_{22}^{\pm}}{\mu^{\pm}_{24}}  \mu^{\pm}_{14}   \right) |     \lesssim \frac{1}{c \, k_1(z) + O(|\xi|^{\frac{3}{2}}) }  , \\
|  b^{\pm}_{22} | &= | \frac{1}{ \mu^{\pm}_{42}}    \left(  \omega_{42}^{\pm} -  b^{\pm}_{21} \mu^{\pm}_{41}  -b^{\pm}_{23} \mu^{\pm}_{43} \right)  |  \lesssim    \frac{1}{c \, k_1(z) + O(|\xi|^{\frac{3}{2}}) }      e^{\frac{3}{\sqrt{2}} \frac{r_{\varepsilon}}{5}  }  
\end{align*} 

Similarly, for large $|\xi|, $ using Lemma \ref{lem:omega_j-behavior-large-xi}, \ref{lem:parity-of-omega-j}, \ref{lem:mu-behavior-large-xi} we have 
\begin{align*}
  b^{\pm}_{14}&= \frac{ \omega_{21}^{\pm}}{\mu^{\pm}_{24}}, \quad 
 |  b^{\pm}_{11} | =  | \frac{ 1}{\mu^{\pm}_{31}}  \omega_{31}^{\pm} | \lesssim \frac{1}{|\xi|^{\frac{3}{4}}}, \qquad 
   | b^{\pm}_{13} | =   | \frac{1}{ \mu^{\pm}_{13}}    \omega_{11}^{\pm}  | \lesssim \frac{1}{|\xi|^{\frac{3}{4}}}     , \quad 
 |  b^{\pm}_{12} | = | \frac{1}{ \mu^{\pm}_{42}}     \omega_{41}^{\pm}  |  \lesssim \frac{1}{|\xi|^{\frac{3}{4}}}    
\end{align*}
and 
\begin{align*}
  b^{\pm}_{24}&= \frac{ \omega_{22}^{\pm}}{\mu^{\pm}_{24}}, \quad 
 | b^{\pm}_{21} | = | \frac{ 1}{\mu^{\pm}_{31}}   \omega_{32}^{\pm}  | \lesssim  \frac{1}{|\xi|^{\frac{3}{4}}}  , \quad  
 |  b^{\pm}_{23} | =  | \frac{1}{ \mu^{\pm}_{13}}  \omega_{12}^{\pm}  | \lesssim   \frac{1}{|\xi|^{\frac{3}{4}}}  , \quad  
 | b^{\pm}_{22} | = | \frac{1}{ \mu^{\pm}_{42}}      \omega_{42}^{\pm}   | \lesssim  \frac{1}{|\xi|^{\frac{3}{4}}} .
\end{align*}

\end{proof}

\begin{lemma} Let $z=\lambda\pm i 0 $ such that  $\lambda=\xi\pm \frac{\sqrt{17}}{8}.$ Then we have
    \begin{align*}
        \Theta(r,\lambda) = \sum_{j=1}^{3} \theta^{+}_j(\lambda) \Psi^{+}_j(r,\lambda), \quad \text{ where }   \theta^{+}_j(\lambda) :=\omega_{22}^{+}(\lambda) b_{1j}^{+} - \omega_{21}^{+}(\lambda)  b_{2j}^{+}, 
    \end{align*}
where for small $|\xi|$, we have
\begin{align*}
    \theta^{+}_1(\lambda) = O(\frac{1}{\sqrt{|\xi|}})  \quad      \theta^{+}_2(\lambda) =O(\frac{1}{\sqrt{|\xi|}}     e^{\frac{3}{\sqrt{2}} \frac{r_{\varepsilon}}{5}  })     , \quad 
     \theta^{+}_3(\lambda) =  O(\frac{1}{\sqrt{|\xi|}})
\end{align*}
and for large $|\xi|$, we have
\begin{align*}
    \theta^{+}_1(\lambda) = O(\frac{1}{|\xi|})  \quad      \theta^{+}_2(\lambda) =O(\frac{1}{|\xi|}    )     , \quad 
     \theta^{+}_3(\lambda) = O(\frac{1}{|\xi|})
\end{align*}

\end{lemma}
\begin{proof}
    By definition of $\Theta(r,\lambda)$ in \eqref{def-theta} with \eqref{exp-phi-terms-psi} and \eqref{drop-out-omega-b}, we obtain 
    \begin{align*}
        \Theta(r,\lambda) = \sum_{j=1}^{3} \theta^{+}_j(\lambda) \Psi^{+}_j(r,\lambda), \quad \text{ where }   \theta^{+}_j(\lambda) :=\omega_{22}^{+}(\lambda) b_{1j}^{+} - \omega_{21}^{+}(\lambda)  b_{2j}^{+}. 
    \end{align*}
Therefore, 
\begin{align*}
     \theta^{+}_1 :&=\omega_{22}^{+} b_{11}^{+} - \omega_{21}^{+}  b_{21}^{+} \\
    &=\omega_{22}^{+}  \frac{1}{ \mu^{+}_{31}} \left(   \omega_{31}^{+} -   \frac{ \omega_{21}^{+} }{\mu^{+}_{24}}  \mu^{+}_{34}  \right) - \omega_{21}^{+} \frac{1}{ \mu^{+}_{31}} \left(   \omega_{32}^{+} -   \frac{ \omega_{22}^{+}}{\mu^{+}_{24}}  \mu^{+}_{34}  \right)     \\
    &= \frac{1}{ \mu^{+}_{31}} \left( \omega_{22}^{+}    \omega_{31}^{+} -   \omega_{21}^{+}   \omega_{32}^{+}  \right) \\ 
     \theta^{+}_2 :&=\omega_{22}^{+} b_{12}^{+} - \omega_{21}^{+}  b_{22}^{+} \\
     &=\omega_{22}^{+}   \frac{1}{ \mu^{+}_{42}}    \left(  \omega_{41}^{+} -  b^{+}_{11} \mu^{+}_{41}  -b^{+}_{13} \mu^{+}_{43} \right) - \omega_{21}^{+}  \frac{1}{ \mu^{+}_{42}}    \left(  \omega_{42}^{+} -  b^{+}_{21} \mu^{+}_{41}  -b^{+}_{23} \mu^{+}_{43}   \right) \\
     &= \frac{1}{ \mu^{+}_{42}} ( \omega_{22}^{+}     \omega_{41}^{+}- \omega_{21}^{+}  \omega_{42}^{+}  )  - \omega_{22}^{+}   \frac{1}{ \mu^{+}_{42}}  \left(    b^{+}_{11} \mu^{+}_{41}  + b^{+}_{13} \mu^{+}_{43} \right) - \omega_{21}^{+}  \frac{1}{ \mu^{+}_{42}}    \left(    b^{+}_{21} \mu^{+}_{41} +b^{+}_{23} \mu^{+}_{43}   \right) 
\end{align*}
Note that, for large $\lambda, $ $ \theta^{+}_2=\frac{1}{ \mu^{+}_{42}} ( \omega_{22}^{+}     \omega_{41}^{+}- \omega_{21}^{+}  \omega_{42}^{+}  ).$
\begin{align*}
    \theta^{+}_3 :&=  \omega_{22}^{+} b_{13}^{+} - \omega_{21}^{+}  b_{23}^{+}  \\
    & = \omega_{22}^{+} \frac{1}{ \mu^{+}_{13}} \left(   \omega_{11}^{+} -  \frac{ \omega_{21}^{+}}{\mu^{+}_{24}}  \mu^{+}_{14}   \right) - \omega_{21}^{+}  \frac{1}{ \mu^{+}_{13}} \left(   \omega_{12}^{+} -  \frac{ \omega_{22}^{+}}{\mu^{+}_{24}}  \mu^{+}_{14}   \right) \\
    & = \frac{1}{ \mu^{+}_{13}}  \left(  \omega_{22}^{+}  \omega_{11}^{+} -\omega_{21}^{+}  \omega_{12}^{+} \right) 
\end{align*}
and by Lemma \ref{behav-omega-b}, we obtain the desired estimates for $\theta_j.$ 
\end{proof}

\begin{lemma}
\label{Resol-interms-psi}
Let $z=\lambda\pm i 0 $ such that  $\lambda=\xi\pm \frac{\sqrt{17}}{8}.$ Then we have
    \begin{align}
    \label{theta_in_terms-Psi} 
    \begin{split}
     \Theta(r,\lambda)& =\gamma_1(\lambda) \left( \omega^{+}_{11}(\lambda) \sigma_3  \Psi_1^{+}(r,-\lambda)  +   \omega^{+}_{11}(-\lambda)  \Psi^{+}_1(r,\lambda) \right)
  \\
 & + \gamma_2(\lambda)   (\sigma_3) \Psi_1^{+}(r,-\lambda)  + \gamma_3(\lambda)  \Psi^{+}_1(r,\lambda) + \gamma_4(\lambda)  \Psi^{+}_2(r,\lambda)
\end{split}
\end{align}
where for small $|\xi|, $ we have 
\begin{align*}
 \gamma_1(\lambda)   = O(\frac{1}{\sqrt{|\xi|} }e^{-\frac{3}{\sqrt{2}} r_{\varepsilon}} ), \quad \gamma_2(\lambda) = O(\frac{1}{\sqrt{|\xi|}}), \quad \gamma_3(\lambda)=  O(\frac{1}{\sqrt{|\xi|}} ) , \quad |\gamma_4(\lambda)| \lesssim e^{\frac{3}{\sqrt{2}} \frac{2 r_{\epsilon}}{5}}.
\end{align*}
and for large $|\xi|, $ we have 
\begin{align*}
    \gamma_1(\lambda)   = O(\frac{1}{|\xi|^{\frac{3}{4}}} ), \quad \gamma_2(\lambda) = O(\frac{1}{|\xi|}), \quad \gamma_3(\xi)=  O(\frac{1}{|\xi|}) ,  \quad  \gamma_4(\lambda) = O(\frac{1}{|\xi|}).
\end{align*}
\end{lemma}

\begin{proof}

Since $\mu^{+}_{13}=-\mu^{+}_{31}, $ we have 
\begin{align*}
     \Theta(r,\lambda)& = - \frac{1}{ \mu^{+}_{13}(\lambda)} \left( \omega_{22}^{+}(\lambda)    \omega_{31}^{+}(\lambda) -   \omega_{21}^{+}(\lambda)   \omega_{32}^{+} (\lambda) \right)  \Psi^{+}_1(r,\lambda)\\
     &+ \frac{1}{ \mu^{+}_{13}(\lambda)} \left(  \omega_{22}^{+}(\lambda)  \omega_{11}^{+} (\lambda)-\omega_{21}^{+}(\lambda)  \omega_{12}^{+}(\lambda) \right) \Psi^{+}_3(r,\lambda)  + \theta^{+}_2(\lambda) \Psi^{+}_2(r,\lambda) \\
     &= \frac{1}{ \mu^{+}_{13}(\lambda)}  \omega_{22}^{+}(\lambda)  \left(   \omega_{11}^{+} (\lambda)  \Psi^{+}_3(r,\lambda)  -  \omega_{31}^{+}(\lambda) 
     \Psi^{+}_1(r,\lambda) \right) \\
    & - \frac{1}{ \mu^{+}_{13}(\lambda)}   \omega_{21}^{+}(\lambda) \left(  \omega_{12}^{+}(\lambda)\Psi^{+}_3(r,\lambda) -  \omega_{32}^{+} (\lambda)\Psi^{+}_1(r,\lambda)
      \right)+ \theta^{+}_2(\lambda) \Psi^{+}_2(r,\lambda) 
\end{align*}
Denote by:
\begin{align*}
  \mathcal{P}_1(r,\lambda):&=\omega_{11}^{+} (\lambda)  \Psi^{+}_3(r,\lambda)  -  \omega_{31}^{+}(\lambda) \Psi^{+}_1(r,\lambda)    ,\\
   \mathcal{P}_2(r,\lambda):&= \omega_{12}^{+}(\lambda)\Psi^{+}_3(r,\lambda) -  \omega_{32}^{+} (\lambda)\Psi^{+}_1(r,\lambda).
\end{align*}

Thus, 

\begin{align}
\label{theta-interms-P_1-2}
\begin{split}
     \Theta(r,\lambda)& = \frac{1}{ \mu^{+}_{13}(\lambda)}  \omega_{22}^{+}(\lambda) \mathcal{P}_1(r,\lambda)     - \frac{1}{ \mu^{+}_{13}(\lambda)}   \omega_{21}^{+}(\lambda) \mathcal{P}_2(r,\lambda) + \theta^{+}_2(\lambda) \Psi^{+}_2(r,\lambda)      
\end{split}
\end{align}

Since $\{ \Psi_1^{+}(\cdot, \lambda), - \sigma_3 \Psi_1^{+}(\cdot, -\lambda),  \Psi_2^{+}(\cdot, \lambda),  \Psi_4^{+}(\cdot, \lambda)    \}$ is a fundamental system for the equation $i \mathcal{L} \Psi=z\Psi,$ then we can write $ \Psi_3^{+}(r,\lambda)$ as
\begin{align*}
    \Psi_3^{+}(r,\lambda)= \alpha_1(\lambda) \Psi_1^{+}(r,\lambda) + \alpha_2(\lambda) (-\sigma_3) \Psi_1^{+}(r,-\lambda) , + \alpha_3(\lambda) \Psi_2^{+}(r,\lambda) + \alpha_4(\lambda) \Psi_4^{+}(r,\lambda).
\end{align*}

Taking the Wronskians on the left and right-hand side with $\Psi_2(r,\lambda),$ and using the asymptotics of $\Psi_j$ one can see that  $\alpha_4(\lambda) W(\Psi^{+}_4(r,\lambda),\Psi^{+}_2(r,\lambda)=0.$ Thus,   $\alpha_4(\lambda)=0,$ and 
\begin{align}
    \Psi_3^{+}(r,\lambda)= \alpha_1(\lambda) \Psi_1^{+}(r,\lambda) + \alpha_2(\lambda) (-\sigma_3) \Psi_1^{+}(r,-\lambda) + \alpha_3(\lambda) \Psi_2^{+}(r,\lambda) .
\end{align}

Taking the Wronskians on the left and right-hand side with $\sigma_3 \Psi_1(r,-\lambda),$ and using the asymptotics of $\Psi_j$ one can see that  $\alpha_1(\lambda) W(\Psi^{+}_1(r,\lambda),\sigma_3 \Psi^{+}_1(r,-\lambda)=0.$ Therefore, $\alpha_1(\lambda)=0$ and now  we have
\begin{align}
    \label{Psi_3-interms-psi_j}
    \Psi_3^{+}(r,\lambda)=  \alpha_2(\lambda) (-\sigma_3) \Psi_1^{+}(r,-\lambda) + \alpha_3(\lambda) \Psi_2^{+}(r,\lambda) .
\end{align}

Using the fact $\upvarphi_1(r,\lambda):=-\sigma_3 \upvarphi_1(r,-\lambda) ,$ we have
$W(-\sigma_3 \Psi_1^{+}(\cdot, -\lambda),\upvarphi_1(\cdot, \lambda))=\omega_{11}(-\lambda).  $ Now, taking the  Wronskians on the left and right-hand side with $\upvarphi_1(r,\lambda)$ leads to 
\begin{align*}
    \omega^{+}_{31}(\lambda)=  
  \alpha_2(\lambda)  \omega^{+}_{11}(-\lambda)+ \alpha_3(\lambda) \omega^{+}_{21}(\lambda).
\end{align*}
Therefore, 
\begin{align*}
     \mathcal{P}_1(r,\lambda)&=\omega_{11}^{+} (\lambda)  \left(  \alpha_2(\lambda) (-\sigma_3) \Psi_1^{+}(r,-\lambda) + \alpha_3(\lambda) \Psi_2^{+}(r,\lambda) \right) \\
     &-  \left( 
  \alpha_2(\lambda)  \omega^{+}_{11}(-\lambda)+ \alpha_3(\lambda) \omega^{+}_{21}(\lambda) \right) \Psi^{+}_1(r,\lambda)   \\
  &=  \alpha_2(\lambda)  \left( \omega_{11}^{+} (\lambda)   (-\sigma_3) \Psi_1^{+}(r,-\lambda) -  \omega^{+}_{11}(-\lambda)  \Psi^{+}_1(r,\lambda) \right) \\
  &+ \alpha_3(\lambda) \left( \omega_{11}^{+} (\lambda)  \Psi_2^{+}(r,\lambda)-  \omega_{21}^{+} (\lambda)  \Psi_1^{+}(r,\lambda)   \right)
\end{align*}

\begin{align}
\label{P_1}
    \begin{split}
 \mathcal{P}_1(r,\lambda)&=   \alpha_2(\lambda)  \left( \omega_{11}^{+} (\lambda)   (-\sigma_3) \Psi_1^{+}(r,-\lambda) -  \omega^{+}_{11}(-\lambda)  \Psi^{+}_1(r,\lambda) \right) \\
  &+ \alpha_3(\lambda) \left( \omega_{11}^{+} (\lambda)  \Psi_2^{+}(r,\lambda)-  \omega_{21}^{+} (\lambda)  \Psi_1^{+}(r,\lambda)   \right)       
    \end{split}
\end{align}

Using the fact $\upvarphi_2(r,\lambda):=\sigma_3 \upvarphi_2(r,-\lambda) ,$ we have
$W(-\sigma_3 \Psi_1^{+}(\cdot, -\lambda),\upvarphi_2(\cdot, \lambda))=-\omega_{12}(-\lambda).  $ Now, taking the  Wronskians on the left and right-hand side of \eqref{Psi_3-interms-psi_j} with $\upvarphi_2(r,\lambda)$ leads to 
\begin{align*}
    \omega^{+}_{32}(\lambda)=  -  \alpha_2(\lambda)  \omega^{+}_{12}(-\lambda)+ \alpha_3(\lambda) \omega^{+}_{22}(\lambda).
\end{align*}
Therefore,
\begin{align*}
     \mathcal{P}_2(r,\lambda):&= \omega_{12}^{+}(\lambda)\Psi^{+}_3(r,\lambda) -  \omega_{32}^{+} (\lambda) \Psi^{+}_1(r,\lambda) \\
     &=  \omega_{12}^{+}(\lambda) \left(  \alpha_2(\lambda) (-\sigma_3) \Psi_1^{+}(r,-\lambda) + \alpha_3(\lambda) \Psi_2^{+}(r,\lambda)  \right) \\
&     - \left(  -
  \alpha_2(\lambda)  \omega^{+}_{12}(-\lambda)+ \alpha_3(\lambda) \omega^{+}_{22}(\lambda) \right)  \Psi^{+}_1(r,\lambda)\\
  &= \alpha_2(\lambda) \left( \omega_{12}^{+}(\lambda)   (-\sigma_3) \Psi_1^{+}(r,-\lambda) +   \omega^{+}_{12}(-\lambda) \Psi^{+}_1(r,\lambda)    \right)  \\
 & + \alpha_3(\lambda) \left(    \omega_{12}^{+}(\lambda)   \Psi_2^{+}(r,\lambda) -   \omega^{+}_{22}(\lambda) \Psi^{+}_1(r,\lambda) \right) 
\end{align*}

\begin{align}
\label{P_2}
\begin{split}
  \mathcal{P}_2(r,\lambda)& =   \alpha_2(\lambda) \left( \omega_{12}^{+}(\lambda)   (-\sigma_3) \Psi_1^{+}(r,-\lambda) +   \omega^{+}_{12}(-\lambda) \Psi^{+}_1(r,\lambda)    \right)  \\
 & + \alpha_3(\lambda) \left(    \omega_{12}^{+}(\lambda)   \Psi_2^{+}(r,\lambda) -   \omega^{+}_{22}(\lambda) \Psi^{+}_1(r,\lambda) \right) 
\end{split}
\end{align}

Substitute \eqref{P_1} and \eqref{P_2} in the equation \eqref{theta-interms-P_1-2}, we obtain 
\begin{align*}
     \Theta(r,\lambda)& =\gamma_1(\lambda) \left( \omega^{+}_{11}(\lambda) \sigma_3  \Psi_1^{+}(r,-\lambda)  +   \omega^{+}_{11}(-\lambda)  \Psi^{+}_1(r,\lambda) \right)
  \\
 & + \gamma_2(\lambda)   (\sigma_3) \Psi_1^{+}(r,-\lambda)  + \gamma_3(\lambda)  \Psi^{+}_1(r,\lambda) + \gamma_4(\lambda)  \Psi^{+}_2(r,\lambda),
\end{align*}
where
\begin{align*}
    \gamma_1(\lambda)  &:= -\frac{1}{ \mu^{+}_{13}(\lambda)}  \omega_{22}^{+}(\lambda) \alpha_2(\lambda) \\
     \gamma_2(\lambda) &:=   \frac{1}{ \mu^{+}_{13}(\lambda)}   \omega_{21}^{+}(\lambda)  \alpha_2(\lambda) \omega_{12}^{+}(\lambda)  \\
     \gamma_3(\lambda) &:= - \frac{1}{ \mu^{+}_{13}(\lambda)}   \omega_{21}^{+}(\lambda)  \alpha_2(\lambda) \omega^{+}_{12}(-\lambda)\\
     \gamma_4(\lambda) &:= \theta^{+}_2(\lambda)+ \frac{1}{ \mu^{+}_{13}(\lambda)}  \omega_{22}^{+}(\lambda)  \alpha_3(\lambda)  \omega_{11}^{+} (\lambda)   - \frac{1}{ \mu^{+}_{13}(\lambda)}   \omega_{21}^{+}(\lambda)  \alpha_3(\lambda)   \omega_{12}^{+}(\lambda)  .
\end{align*}

Taking the Wronskians on the left and right-hand sides of \eqref{Psi_3-interms-psi_j} with $\Psi^{+}_1(\cdot,\lambda),$ we obtain 
\begin{align*}
    \mu^{+}_{13}(\lambda)= \alpha_2(\lambda) W( \sigma_3 \Psi_1^{+}(\cdot,-\lambda) , \Psi^{+}_1(\cdot,\lambda) )= \alpha_2(\lambda) c k_1^{+}(\lambda) ( 1- c_1(\lambda)^2).
\end{align*}
Thus, we have 
\begin{align*}
\frac{\alpha_2(\lambda)}{\mu^{+}_{13}(\lambda)} = \frac{1}{ c k_1^{+}(\lambda) ( 1- c_1(\lambda)^2)} 
\end{align*}
Therefore, by Lemma \ref{lemma:omega_j-behavior}  for small $|\xi|$ we have 
\begin{align*}
  \gamma_1(\lambda)   = O(\frac{1}{\sqrt{|\xi|} }e^{-\frac{3}{\sqrt{2}} r_{\varepsilon}} ), \quad \gamma_2(\lambda) = O(\frac{1}{\sqrt{|\xi|}}), \quad \gamma_3(\lambda)=  O(\frac{1}{\sqrt{|\xi|}} ) 
\end{align*}
and by Lemma \ref{lem:omega_j-behavior-large-xi} for large $|\xi|$ we have 
\begin{align*}
  \gamma_1(\lambda)   = O\left(\frac{1}{|\xi|^{\frac{3}{4}}} \right), \quad \gamma_2(\lambda) = O(\frac{1}{|\xi|}), \quad \gamma_3(\lambda)=  O(\frac{1}{|\xi|}) .
\end{align*}
Similarly, taking the Wronskians on the left and right-hand sides of \eqref{Psi_3-interms-psi_j} with $\Psi^{+}_4(\cdot,\lambda),$ we obtain 
$\mu^{+}_{34}(\lambda)= \alpha_3(\lambda) W(  \Psi_2^{+}(\cdot,\lambda) , \Psi^{+}_4(\cdot,\lambda) )= \alpha_3(\lambda) c k_2(\lambda) ( 1- c_2(\lambda)^2).$ Thus
\begin{align*}
 \frac{ \alpha_3(\lambda) }{ \mu^{+}_{13}(\lambda) } = \frac{ \mu^{+}_{34}(\lambda)}{c k_2^{+}(\lambda) ( 1- c_2(\lambda)^2)}
\end{align*}
Recall that for small $|\xi|$ we have $\theta^{+}_2(\lambda) =O(\frac{1}{\sqrt{|\xi|}}     e^{\frac{3}{\sqrt{2}} \frac{r_{\varepsilon}}{5}  })  ,  |\omega_{11}^{+}|   \lesssim  e^{\frac{3}{\sqrt{2}} \frac{r_{\varepsilon}}{5}} $ and $|\mu^{\pm}_{34} | \lesssim  e^{\frac{3}{\sqrt{2}} \frac{r_{\epsilon}}{5}},$ then $  |\gamma_4(\lambda)| \lesssim e^{\frac{3}{\sqrt{2}} \frac{2 r_{\epsilon}}{5}}.$
For large $|\xi|,$ we have $\theta^{+}_2(\lambda)=O(\frac{1}{|\xi|}  ),$ $\omega_{11}(\lambda)=c \, (k_1^{+}(\lambda))^{-\frac{1}{2}}+O(|\lambda|^{-1})$ and $\mu^{\pm}_{34}(z) =0,$ then $
   |\gamma_4(\lambda)| \lesssim  \frac{1}{|\xi|}.$

\end{proof}

\section{Green's kernel of $i\mathcal{L}-z$ for the resonance case} 
\label{sec:GreenKernel-resonance}
In this section, we compute the distorted Fourier transform of $i\mathcal{L}$ in the resonance case. The argument is similar to the non-resonance case in Section~\ref{sec:GreenKernel-non-resonance}. Hence, we only provide the necessary details and proofs that differ from the non-resonance case. We refer the reader to section \ref{sec:GreenKernel-non-resonance} for more details, specially for justifying that the limit as $b \to 0^{+}$ can be taken inside the integral representation of the evolution in \eqref{eq:ST1}.

\subsection{Connecting the solutions from $0$ and $\infty$}
The purpose of this subsection is to connect the local behavior of solutions to $(i\mathcal{L} - z)\Psi = 0$ near $r = 0$ with their asymptotic behavior as $r \to \infty$, for both small and large $|\xi|$. Our approach is based on analyzing the Wronskians that connect the solutions obtained near $r=0$ with those constructed at infinity. In particular, we examine the Wronskians between the solutions $\varphi_j^R$ defined at the origin and the solutions $\Psi_j^R$ defined at infinity. This will allow us to represent one family of solutions in terms of the other. \\

Recall that throughout this paper (for the resonance case), we choose $\varepsilon_0$ and $\varepsilon_{\infty}$ such that $r_{\infty}< R_0 . $ For $0<|\xi|< \delta_0,$ we fix  $ r_\varepsilon\equiv r_{\varepsilon}(z) \in (r_{\infty},R_0)$ such that $(\frac{r_{\varepsilon}}{20}, 6 r_{\varepsilon})\subseteq (r_{\infty},R_0)$, and $r_\varepsilon\simeq \frac{\varepsilon}{\sqrt{|\xi|}}$ for some small constant $\varepsilon$.  We will evaluate the Wronskians for some $r \in (\frac{r_{\varepsilon}}{10}, 4 r_{\varepsilon}).$

Denote by: 
\begin{align*}
    \Lambda_{ij}^{R}(z)&:=W(\upvarphi_i^{R}(\cdot,z),  \upvarphi_j^{R}(\cdot,z)) \\
    \omega_{ij}^{R,\pm}(z)&:= W( \Psi^{R,\pm}_{i}(\cdot,z), \upvarphi_j^{R}(\cdot,z))
\end{align*} 

\begin{lemma}
\label{Lambda_j-behavior-resonance-case}
Let $z \in \Omega,$ then we have
 \begin{align*}
 \Lambda_{12}^{R}(z)=\Lambda_{14}^{R}(z)=0 , \quad \Lambda_{13}^{R}(z) =2c_1^2 c_3^2 , \quad 
 \Lambda_{24}^{R}(z) =2c_2^1 c_4^1, \quad  |\Lambda_{23}^{R}(z)|  \lesssim  e^{-\frac{3}{\sqrt{2}}r_{\varepsilon}}, \quad  |\Lambda_{34}^{R}(z)| \lesssim r_{\varepsilon} e^{-\frac{3}{\sqrt{2}} r_{\varepsilon}}. 
\end{align*}
\end{lemma}

\begin{proof}
We determine the values of the first four Wronskians at $r = 0$, while the last two are computed at $r = 2 r_{\varepsilon}$. This choice is guided by the asymptotic properties of $\upvarphi_j(r,z)$ in the regimes of small and large $r$. 
The first four cases are obtained directly from Lemma \ref{D_0-resonance}. 
For the remaining two, note that $\upvarphi_3$ decays exponentially, $\upvarphi_2$ stays bounded, and $\upvarphi_4^{R}$ grows polynomially as $r \to \infty$. These behaviors justify evaluating the Wronskians at $r = 2 r_{\varepsilon}$ to obtain the necessary estimates.
\end{proof}

\begin{lemma}
\label{omega_j-behavior-resonance}
Let $z=\xi \pm \frac{\sqrt{17}}{8}  \in \Omega$ with $\im(z)>0 $ and $0<|\xi|<\delta_0.$ Then, we have   
    \begin{align*}
  |\omega_{11}^{R,+}(z)|   &\lesssim  e^{\frac{3}{\sqrt{2}} \frac{r_{\varepsilon}}{5}}  , \quad |\omega_{21}^{R,+} (z)| \simeq 1 , \quad |\omega_{31}^{R,+}  (z)| \lesssim  e^{\frac{3}{\sqrt{2}} \frac{r_{\varepsilon}}{5}}, \quad |\omega_{41}^{R,+}  (z)| \lesssim   e^{\frac{6}{\sqrt{2}} \frac{r_{\varepsilon}}{10}} \\
   \omega_{12}^{R,+}(z)  &= c \, k_1(z) + O(|\xi|^{\frac{3}{2}})   , \quad |\omega_{22}^{R,+} (z)| \lesssim   e^{-\frac{3}{\sqrt{2}} r_{\varepsilon}} , \quad \omega_{32}^{R,+}  (z) =c \, k_1(z) + O(|\xi|^{\frac{3}{2}}), \quad |\omega_{42}^{R,+}  (z)| \lesssim  e^{\frac{3}{\sqrt{2}} \frac{r_{\varepsilon}}{5}} \\
   |\omega_{13}^{R,+}(z)| &  \lesssim e^{-\frac{3}{\sqrt{2}} r_{\varepsilon}}  , \quad 
   |\omega_{23}^{R,+} (z)| \lesssim e^{-\frac{6}{\sqrt{2}} r_{\varepsilon}} , \quad |\omega_{33}^{R,+}  (z)| \lesssim e^{-\frac{3}{\sqrt{2}} r_{\varepsilon}},
   \quad |\omega_{43}^{R,+}  (z)| \simeq 1 \\
  | \omega_{14}^{R,+}(z) | & \simeq 1   , \quad 
   |\omega_{24}^{R,+} (z)| \lesssim   e^{-\frac{3}{\sqrt{2}} r_{\varepsilon}} , \quad |\omega_{34}^{R,+}  (z)| \simeq 1 ,
   \quad |\omega_{44}^{R,+}  (z)| \lesssim  e^{\frac{3}{\sqrt{2}} \frac{r_{\varepsilon}}{5}}.
    \end{align*}
    where $c>0$ is constant. Moreover, $d^{R,+}(z):=\omega_{22}^{R,+}(z) \omega_{11}^{R,+}(z)-\omega_{21}^{R,+}(z) \omega_{12}^{R,+}(z) \neq 0 $ and $|d^{R,+}(z) | \simeq \sqrt{|\xi|}.$
\end{lemma}

\begin{proof}
Recall that, the coefficients satisfy $\tc_j^2= \pm \frac{i}{\sqrt{17}} \tc_j^1$ for $j=1,3$ and $\tc_l^2= \mp i \sqrt{17} \, \tc_l^1$ for $l=2,4$ and  $c_j(z)=     - i     \frac{ ( \frac{1}{2} k_j(z)^2 + \frac{17}{8}) }{ (\xi + \frac{\sqrt{17}}{8}) }$  for  $j=1,2.$ Moreover, the asymptotic of the functions $k_j(z)$ and $c_j(z)$ as $z \to \pm \frac{\sqrt{17}}{8}$ are given by $k_1(z)=\pm \frac{\sqrt{2\sqrt{17}}}{3} \sqrt{z} + O(z),$ $k_2(z)= i \frac{3}{\sqrt{2}} + O(z) ,$ $c_1(z)=-i \sqrt{17} +O(z)$ and $c_2(z)= i \frac{1}{ \sqrt{17}}+O(z).$  Observe that $\upvarphi_1^{R}(r,z)$ and $\upvarphi_3^{R}(r,z)$ have the same asymptotic behavior as $\upvarphi_1(r,z)$ and $\upvarphi_3(r,z)$, respectively. Consequently, the estimates for $\omega_{l1}^{R,+}$ and $\omega_{l3}^{R,+}$ can be derived in the same way as in the proof of Lemma~\ref{lemma:omega_j-behavior} for the non-resonance case. Therefore, it remains to compute $\omega_{l2}^{R,+}$ and $\omega_{l4}^{R,+}$. Using the asymptotic of $\Psi^{R,+}_{l}(r,z)$ and $\upvarphi_2^{R}(r,z)$ for large $r,$ we obtain
\begin{align*}
      \omega_{12}^{R,+}(z) &:= W( \Psi^{R,+}_{1}(r,z), \upvarphi_2^{R}(r,z)) \\
      &= W( \begin{pmatrix}
           e^{i k_1(z) r } \\
            c_1 (z) e^{i k_1(z) r } 
      \end{pmatrix}  (1+O(\frac{1}{\sqrt{|\xi|}}e^{-2r_{\infty}}) ),
      \begin{pmatrix}
\tc_2^1  \\
\tc_2^2 
      \end{pmatrix}(1+O(r^{-1} ))  ) = ck_1(z) + O(|\xi|^{\frac{3}{2}}). 
\end{align*} 
Similarly, we get $ \omega_{32}^{R,+}(z) = c k_1(z) + O(|\xi|^{\frac{3}{2}}).$ Using the exponential decay of $ \Psi^{R,+}_{2}(r,z),$ we get 

\begin{align*}
      \omega_{22}^{R,+}(z) &:= W( \Psi^{R,+}_{2}(r,z), \upvarphi_2^{R}(r,z)) \\
      &= W( \begin{pmatrix}
           e^{i k_2(z) r } \\
            c_2 (z) e^{i k_2(z) r } 
      \end{pmatrix}  (1+O(\frac{1}{\sqrt{|\xi|}}e^{-2r_{\infty}}) ) ,
      \begin{pmatrix}
\tc_2^1  \\
\tc_2^2 
      \end{pmatrix}(1+O(r^{-1} ))  ) ,
      \end{align*} 
      which yields, $   | \omega_{22}^{R,+}(z)|    \lesssim  e^{-\frac{3}{\sqrt{2}}r_{\varepsilon}}  .$ Analogously, we get $|  \omega_{42}^{R,+}(z)| \lesssim e^{\frac{3}{\sqrt{2}} \frac{r_{\varepsilon}}{5}  }. $ Next, we estimate $ \omega_{l4}^{R,+}:= W( \Psi^{R,+}_{l}(r,z), \upvarphi_4^{R}(r,z)).$

\begin{align*}
      \omega_{14}^{R,+}(z) &:= W( \Psi^{R,+}_{1}(r,z), \upvarphi_4^{R}(r,z)) \\
      &= W( \begin{pmatrix}
           e^{i k_1(z) r } \\
            c_1 (z) e^{i k_1(z) r } 
      \end{pmatrix}  (1+O(\frac{1}{\sqrt{|\xi|}}e^{-2r_{\infty}}) ) ,
      \begin{pmatrix}
\tc_4^1 r \\
\tc_4^2 r 
      \end{pmatrix} (1+O(r^{-2} )))  \simeq 1
\end{align*}
By the same argument, we have $|  \omega_{34}^{R,+}(z)|\simeq 1.  $  Using again the asymptotics behavior of $ \Psi^{R,+}_{2}(r,z),$ we get  
\begin{align*}
      \omega_{24}^{R,+}(z) &:= W( \Psi^{R,+}_{2}(r,z), \upvarphi_4^{R}(r,z)) \\
      &= W( \begin{pmatrix}
           e^{i k_2(z) r } \\
            c_2 (z) e^{i k_2(z) r } 
      \end{pmatrix} (1+O(\frac{1}{\sqrt{|\xi|}}e^{-2r_{\infty}}) ) ,
      \begin{pmatrix}
\tc_4^1  r\\
\tc_4^2 r
      \end{pmatrix}(1+O(r^{-2} )) ) 
\end{align*} 
which yields, $ |  \omega_{24}^{R,+}(z)| \lesssim  r_{\varepsilon}    e^{ - \frac{3}{\sqrt{2}} r_{\varepsilon} } .$ Similarly, we have $|  \omega_{44}^{R,+}(z)|  \lesssim   e^{\frac{3}{\sqrt{2}} \frac{r_{\varepsilon}}{5}}.$
\end{proof}

Similarly to the non-resonance case, $\omega_{ij}^{R,-}(z)$ satisfy the same parity and
the following identities hold for $z \in \Omega$

\begin{lemma}
\label{lem:parity-of-omega-j-resonance}
   Let $z\in \Omega,$ then we have 
\begin{align*}
    \omega_{i1}^{R,-}(z)=- \omega^{R,+}_{i1}(-z), \quad  \omega_{i2}^{R,-}(z)= \omega^{+}_{i2}(-z) , \quad  
    \omega_{i3}^{R,-}(z)= -\omega^{R,+}_{i3}(-z), \quad
    \omega_{i4}^{R,-}(z)= \omega^{R,+}_{i4}(-z)
\end{align*}
    Moreover, \begin{align}
    \label{eq:parity-w21-w22-resonance}
        \omega^{R,+}_{21}(-z)= -\omega^{R,+}_{21}(z), \quad \omega^{R,+}_{22}(-z)=  \omega^{R,+}_{22}(z).
    \end{align}
\end{lemma}
\begin{proof}
The proof is similar to that of Lemma \ref{lem:parity-of-omega-j} and follow immediately from the identities \eqref{eq:sys-varphi_j-resonance}, \eqref{eq:def-Psi-minus-small-xi} and \eqref{eq:def-Psi-minus-large-xi}. We omit the details.
\end{proof}

Next,  we evaluate the Wronskians $\omega_{ij}$ for large $|\xi|.$ Recall that throughout this paper (for the resonance case), we choose $\tvarepsilon_0$ and $\tvarepsilon_{\infty}$ such that $\tr_{\infty}=\frac{\tvarepsilon_{\infty}}{\sqrt{|\xi|}}< \tR_0=\frac{\tvarepsilon_0}{\sqrt{|\xi|}}. $ For $1<\Lambda_0<|\xi|,$ we fix $\tr_\varepsilon \equiv \tr_{\varepsilon}(z) \in (\tr_{\infty},\tR_0)$  such that $(\frac{\tr_{\varepsilon}}{10}, 4 \tr_{\varepsilon})\subseteq (\tr_{\infty},\tR_0),$ and $\tr_{\varepsilon}\simeq \frac{\varepsilon}{\sqrt{|\xi|}}$ for some small constant $\varepsilon.$  We will evaluate the Wronskians for some $r \in (\frac{\tr_{\varepsilon}}{10}, 4 \tr_{\varepsilon}).$  \\

Note that, in view of our choices of $\tr_{\infty}$ and $\tR_0$, and for sufficiently large $|\xi|$, the evaluation of the Wronskians relies on the small $r$ behavior of $\Psi^{+}_l$ and $\varphi_j$.

\begin{lemma}
\label{omega_j-behavior-large-xi-resonance} 
Let $z= \xi \pm \frac{\sqrt{17}}{8}   \in \Omega,$ with $\im(z)>0$ and $1<\Lambda_0 < |\xi|.$ Then we have   
    \begin{align*}
  |\omega_{11}^{R,+}(z)|  &= -2 c_{+} k_1(z)^{-\frac{ 1}{2}} c_1(z) c_1^2 + O(|\xi|^{-1}) \\
  |\omega_{21}^{R,+}(z)| & = -2 c_{+} k_2(z)^{- \frac{ 1}{2}} c_2(z) c_1^2 + O(|\xi|^{-1})  \\
 |\omega_{12}^{R,+}(z)| & = 2 c_{+} k_1(z)^{-\frac{1}{2}}(c_2^1 -c_1(z) c_2^2) + O(|\xi|^{-1}) \\ 
  |\omega_{22}^{R,+}(z)| & = 2 c_{+} k_2(z)^{-\frac{1}{2}}(c_2^1 -c_2(z) c_2^2) + O(|\xi|^{-1}) \\
   |\omega_{31}^{R,+}(z)| & =-2 c_{-} k_1(z)^{-\frac{ 1}{2}} c_1(z) c_1^2 + O(|\xi|^{-1})   \\
   |\omega_{32}^{R,+}(z)| & = 2 c_{-} k_1(z)^{-\frac{1}{2}}(c_2^1 -c_1(z) c_2^2) + O(|\xi|^{-1})  \\
    |\omega_{41}^{R,+}(z)| & = 2 c_{-} k_2(z)^{-\frac{1}{2}} c_2(z) c_1^2 + O(|\xi|^{-1}) \\
     |\omega_{42}^{R,+}(z)| & =2 c_{-} k_2(z)^{-\frac{1}{2}}(c_2^1 -c_2(z) c_2^2) + O(|\xi|^{-1}) 
    \end{align*} 
    where $c_{\pm}\neq 0$ constants, $c_j(z)$ are complex, $c_1^2\neq 0 $ and $c_2^2$ are real constants. Moreover, 
  \begin{align*}
     d^{R,\pm}(z):=\omega_{22}^{R,+}(z) \omega_{11}^{R,+}(z) -\omega_{12}^{R,+}(z)  \omega_{21}^{R,+}(z)= c k_1(z)^{-\frac{1}{2}}  k_2(z)^{-\frac{1}{2}} + O(|\xi|^{-1}) \neq 0. 
  \end{align*}
  where $c>0$ constant.   
\end{lemma}
\begin{proof}
The proof is very similar to the one for Lemma \ref{lem:omega_j-behavior-large-xi} for the non-resonance case using the behavior $\upvarphi_j^{R}(r,z)$ for small $r.$
\end{proof}

\subsection{Solving for the Green's Kernel} \label{subsec:GreenKernel-resonance}
Let $z\in \Omega,$ with $\pm \im(z)>0,$ then we define the Green's kernel for the resonance case for $i \mathcal{L}-z$ as
\begin{align*}
    \mathcal{K}^{R,\pm}(r,s,z):= \Psi^{R,\pm}(r,z) S(s,z) \mathbb{1}_{\{ 0<s \leq r  \}} + 
    F_1^{R}(r,z) T(s,z) \mathbb{1}_{ \{r \leq s < \infty  \} }, 
\end{align*}
where we require the matrices $S(r,z)$ and $T(r,z)$ to satisfy 
\begin{align*}
    \begin{pmatrix}
        \Psi^{R,\pm}(r,z) & F_1^{R}(r,z) \\
        \partial_r \Psi^{R,\pm}(r,z) & \partial_r F_1^{R}(r,z) 
    \end{pmatrix}
    \begin{pmatrix}
        S(r,z) \\ - T(r,z)
    \end{pmatrix}
    = \begin{pmatrix}
        0 \\ \sigma_2 . 
    \end{pmatrix}
\end{align*}
Then a solution to $(  i \mathcal{L} - z ) \Psi = \Phi$ for $z \in \Omega, $ with $\pm \im(z)>0,$ is given by 
 \begin{equation}
     \Psi(r) := \int_0^{\infty}   \mathcal{K}^{R,\pm}(r,s,z) \Phi(s) ds. 
 \end{equation}

\begin{lemma}
\label{kernel-K-resonance}
 Let $\calD^{R,\pm}(z):=W(\Psi^{R,\pm}(\cdot,z),F_1^{R}(\cdot,z)).$ Then we have 
 \begin{align*}
\calD^{R,\pm}(z)=   \begin{pmatrix}
       \omega_{11}^{R,\pm}(z) & \omega_{12}^{R,\pm}(z) \\
       \omega_{21}^{R,\pm}(z) & \omega_{22}^{R,\pm}(z)
   \end{pmatrix} , \qquad 
 ( \calD^{R,\pm})^{-1}(z)  : = \frac{1}{d^{R,\pm}(z) } \begin{pmatrix}
       \omega_{22}^{R,\pm}(z) & - \omega_{12}^{R,\pm}(z) \\  
       -\omega_{21}^{R,\pm}(z) & \omega_{11}^{R,\pm}(z)
   \end{pmatrix} , 
 \end{align*}
 where $d^{R,\pm}(z):=\omega_{22}^{R,\pm}(z) \omega_{11}^{R,\pm}(z) -\omega_{12}^{R,\pm}(z)  \omega_{21}^{R,\pm}(z)\neq 0 .$ \\

Moreover, for $z \in \Omega$ with $\pm \im(z)>0,$ we have  
\begin{align*}
        \begin{cases}
            S(r,z) = - (\calD^{R,\pm}(z))^{-t} (F_1^{R}(r,z))^t \sigma_3 \sigma_2, \\ 
            T(r,z) = - (\calD^{R,\pm}(z))^{-1} \Psi^{R,\pm}(r,z)^t \sigma_3 \sigma_2.  
        \end{cases}
    \end{align*}
and
 \begin{align*}
     \mathcal{K}^{R,\pm}(r,s,z):= \begin{cases}
         i \Psi^{R,\pm}(r,z) (\calD^{R,\pm}(z))^{-t} (F_1^R(s,z))^t \sigma_1, \qquad 0<s\leq r , \\ 
         i F_1^{R}(r,z) (\calD^{R,\pm}(z))^{-1} \Psi^{R,\pm}(s,z)^t \sigma_1, \qquad r \leq s < \infty,
     \end{cases}
 \end{align*} 
\end{lemma}
\begin{proof}
  Note that by Lemma \ref{omega_j-behavior-resonance}, \ref{lem:parity-of-omega-j-resonance} and \ref{omega_j-behavior-large-xi-resonance}, we have $d^{R,\pm}(z) \neq 0 ,$ for large and small $\xi.$ Using Lemma \ref{lem:EV1} and the fact that $i\calL$ has a real spectrum and no threshold or embedded eigenvalues, we obtain $d^{\pm}\neq 0.$ Therefore, $\calD^{R,\pm}(z)$ is invertible. The remainder of the proof is analogous to that of Lemma~\ref{lem:kernel-K} in the non-resonant case, and we omit the details. 
\end{proof}
Next, we express $\Psi^{\pm}_{l}(r,z)$ in terms of $\upvarphi_j^{R}(r,z),$ i.e., there is exist  $a^{R,\pm}_{lj} (z)$ such that 
\begin{align*}
    \Psi^{\pm}_{l} (r,z) = \sum_{j=1}^4 a^{R,\pm}_{lj} (z) \upvarphi_j^R (r,z), \qquad l=1,\cdots,4.
\end{align*}
Therefore, by Lemma \ref{kernel-K-resonance}, we have 
\begin{align*}
\mathcal{K}^{R,\pm}(r,s,z) &=\frac{i}{d^{R,\pm}} \sum_{j=1}^4 (   \omega_{22}^{R,\pm} a_{1j}^{R,\pm} - \omega_{12}^{R,R,\pm}(z) a_{2j}^{R,\pm} ) \bigg(
\upvarphi_j^{R}(r,z) \upvarphi_1^R(s,z)^t \sigma_1 \mathbb{1}_{\{ 0<s \leq r  \}}  \\
& + \upvarphi_1^R(r,z) \upvarphi_j^R(s,z)^t \sigma_1  \mathbb{1}_{\{ r \leq s < \infty  \}}
\bigg) + \frac{i}{d^{R,\pm}} \sum_{j=1}^4 (     -\omega_{21}^{R,\pm}(z) a_{1j}^{R,\pm} + \omega_{11}^{R,\pm}(z) a_{2j}^{R,\pm} ) \\
&\times \bigg(
\upvarphi_j^{R}(r,z) \upvarphi_2^R(s,z)^t \sigma_1 \mathbb{1}_{\{ 0<s \leq r  \}} 
+ \upvarphi_2^R(r,z) \upvarphi_j^R(s,z)^t \sigma_1  \mathbb{1}_{\{ r \leq s < \infty  \}}
\bigg) 
\end{align*}

\begin{coro}
\label{kernel-no-sigularity-resonance}
Let $z=\xi \pm \frac{\sqrt{17}}{8} \in \Omega,$ where $\pm \im(z)>0$  and  $0<|\xi|<\delta_0.$ Then in both scenarios, we have 
\begin{align*}
\mathcal{K}^{R,\pm}(r,s,z) &=\frac{i}{d^{R,\pm}} \sum_{j=1}^4 \alpha^{R,\pm}_j(z)\left(
\upvarphi_j^R(r,z) \upvarphi_1^R(s,z)^t \sigma_1 \mathbb{1}_{\{ 0<s \leq r  \}} + \upvarphi_1^R(r,z) \upvarphi_j^R(s,z)^t \sigma_1  \mathbb{1}_{\{ r \leq s < \infty  \}}
\right) \\
&+ \frac{i}{d^{R,\pm}} \sum_{\substack{j=1 \\ j \neq 3}}^4 \beta^{R,\pm}_j(z) \left(
\upvarphi_j^R(r,z) \upvarphi_2^R(s,z)^t \sigma_1 \mathbb{1}_{\{ 0<s \leq r  \}} + \upvarphi_2^R(r,z) \upvarphi_j^R(s,z)^t \sigma_1  \mathbb{1}_{\{ r \leq s < \infty  \}}
\right) 
\end{align*}
where 
\begin{align*}
    |\alpha^{R,\pm}_1(z)| &\lesssim e^{-\frac{6}{\sqrt{2}} r_{\varepsilon}}, \quad  |\alpha^{R,\pm}_2(z)| \lesssim e^{-\frac{3}{\sqrt{2}} r_{\varepsilon}}, \quad  |\alpha^{R,\pm}_3(z)| \lesssim \sqrt{|\xi|} + e^{-\frac{3}{\sqrt{2}} \frac{4}{5} r_{\varepsilon}} , \quad  |\alpha^{\pm}_4(z)| \lesssim 1,\\
     |\beta^{R,\pm}_1(z)| &\lesssim e^{-\frac{3}{\sqrt{2}} r_{\varepsilon}}, \quad |\beta^{R,\pm}_2(z)| \lesssim 1, \quad
     |\beta^{R,\pm}_4(z)| \lesssim 1.
\end{align*}
\end{coro}
\begin{proof}
Similarly to the proof of Lemma \ref{kernel-no-sigularity-non-resonance} the coefficient of the $\upvarphi_3^R(r,z) \upvarphi_2^R(s,z)^t \sigma_1 \mathbb{1}_{\{ 0<s \leq r  \}} + \upvarphi_2^R(r,z) \upvarphi_3^R(s,z)^t \sigma_1  \mathbb{1}_{\{ r \leq s < \infty  \}}$ vanishes, i.e., 
$- \omega_{21}^{R,\pm}(z) a^{R,\pm}_{13}+ \omega_{11}^{R,\pm}(z) a^{R,\pm}_{23}(z) =0 .$ By Lemma \ref{Lambda_j-behavior-resonance-case}, it follows that 
\begin{align*}
  \omega_{l1}^{R,\pm} &= -  \Lambda_{13}^R a^{R,\pm}_{l3}, \qquad \qquad  \qquad  \qquad \longrightarrow | \omega_{l1}^{R,\pm}| \simeq |a^{R,\pm}_{l3} |  \\
  \omega_{l2}^{R,\pm} &= -  \Lambda_{23}^R a^{R,\pm}_{l3}-  \Lambda_{24}^R a^{R,\pm}_{l4}   \; \qquad \qquad \longrightarrow | \omega_{l2}^{R,\pm}| \simeq |a^{R,\pm}_{l4} | \\
  \omega_{l3}^{R,\pm} &=\Lambda_{13}^R a^{R,\pm}_{l1} + \Lambda_{23}^R a^{R,\pm}_{l2} -\Lambda_{34}^R a^{R,\pm}_{l4}  \quad \longrightarrow | \omega_{l3}^{R,\pm}| \simeq |a^{R,\pm}_{l1} | \\
    \omega_{l4}^{R,\pm} &=\Lambda_{24}^R a^{R,\pm}_{l2} + \Lambda_{34}^R a^{R,\pm}_{l3}   \qquad  \qquad \quad \longrightarrow | \omega_{l4}^{R,\pm}| \simeq |a^{R,\pm}_{l2} |
\end{align*}

Then all other estimates follows from Lemma \ref{omega_j-behavior-resonance}, \ref{lem:parity-of-omega-j-resonance} and the fact that for both scenarios $|k_1(z)| \simeq \sqrt{|\xi|}.$ 
\end{proof}

\subsection{Computing the jump of the resolvent}

\begin{prop}
\label{jump-resol-resonance}
    Let $I =\R \backslash (-\frac{\sqrt{17}}{8},\frac{\sqrt{17}}{8}) .$ Then for any compactly supported functions $\Phi, \Psi \in L^2_r(0,\infty), $ we have 
    \begin{align*}
      <  e^{t \mathcal{L}}  \Phi, \Psi>= \frac{1}{2 \pi i } \int_{I} e^{it \lambda} \left< \int_0^{\infty} F_1^R(\cdot,\lambda) \calW^R(\lambda)F_1^R(s,\lambda)^{t} \sigma_1 \Phi(s)ds, \Psi(\cdot) \right>_{L^2_r} d \lambda ,
      \end{align*}
   where \begin{align*}
       \calW^R(\lambda):=\kappa^R(\lambda) \calD^{R,-}(\lambda)^{-1} \underbar{e}_{11} \calD^{R,+}(\lambda)^{-t}, \quad \underbar{e}_{11}:=\begin{pmatrix}
      1 & 0 \\ 0 & 0
  \end{pmatrix}, \quad \calD^{R,\pm}(\lambda):=D^{R,\pm}(\lambda\pm i0),
   \end{align*}   
  and $\kappa^R(\lambda)=W(\sigma_3 \Psi^{+}_1(\cdot,-\lambda),\Psi^{+}_1(\cdot,\lambda))$ is an odd function in $\lambda.$ Moreover, for $\lambda=\xi \pm \frac{\sqrt{17}}{8}$ we have $|\kappa^R(\lambda)|\simeq \sqrt{|\xi|}$ for small and large $|\xi|.$ 
\end{prop}

\begin{proof}
One can use the same argument as in the proof Proposition \ref{jump-resol} to obtain the desired results. 
\end{proof}

Next, we write more concrete expression of the tensorial structure of $ F_1^R(\cdot,\lambda) W^R(\lambda)F_1^R(s,\lambda)^{t} \sigma_1.$

\begin{lemma}
\label{Resol-interms-phi-resonance}
Let $z=\lambda \pm i 0 $ such that $\lambda=\xi \pm \frac{\sqrt{17}}{8}.$ Then we have 
    \begin{align*}
      F_1^R(\cdot,\lambda) \calW^R(\lambda)F_1^R(s,\lambda)^{t} = \frac{\kappa^R(\lambda)}{d^{R,+}(\lambda) d^{R,-}(\lambda)} \Theta^R(r,\lambda) \Theta^R(s,\lambda)^{t} ,
    \end{align*}
where $d^{R,\pm}(\lambda):=\det(\calD^{R,\pm}(\lambda \pm i 0)),\;  \omega_{ij}^{R,\pm}(\lambda)=\omega_{ij}^{R,\pm}(\lambda\pm i0)$ and $\Theta^R(r,\lambda) :=\begin{pmatrix}
    \Theta_1^R(r,\lambda) \\ 
    \Theta_2^R(r,\lambda)
\end{pmatrix}$
satisfies 
\begin{align}
\label{def-theta-resonance}
    \Theta^R(r,\lambda) := \omega_{22}^{R,+}(\lambda) \upvarphi_1^R(r,\lambda) - w_{21}^{R,+}(\lambda) \upvarphi_2^R(r,\lambda).
\end{align}
Moreover, $ \frac{\kappa^R(\lambda)}{d^{R,+}(\lambda) d^{R,-}(\lambda)} $ is an odd function,
$\Theta_1^R(r,\lambda) $ is an odd function in $\lambda$ and $\Theta_2^R(r,\lambda) $ is an even function in $\lambda$  
\end{lemma}
\begin{proof}
The proof follows in the same way as for Lemma \ref{Resol-interms-phi} and is therefore omitted.
\end{proof}

Next we express $\upvarphi_l^R(r,z)$ in terms of   $\Psi^{\pm}_j(r,z),$ i.e., there exist $b^{R,\pm}_{lj}(z)$ such that 
\begin{align}
\label{exp-phi-terms-psi-resonance}
    \upvarphi_l^R(r,z)= \sum_{j=1}^{4} b^{R,\pm}_{lj} \Psi^{
    \pm}_j(r,z), \quad l=1,\cdots, 4.
\end{align}

In order to write $\Theta^R(r,\lambda) $ in terms of  $\Psi_j(r,z)$ we only need $l=1,2.$\\

\begin{lemma}
\label{behav-omega-b-resonance}
Let $z=\lambda \pm i0$ such that $\lambda=\xi\pm \frac{\sqrt{17}}{8}.$ Then we have 
\begin{align}
\label{drop-out-omega-b-resonance}
     \omega^{R,+}_{22} b^{R,\pm}_{14}  =  \omega_{12}^{R,\pm}  b^{R,\pm}_{24}=\frac{ \omega_{12}^{R,\pm}  \omega_{22}^{R,\pm}}{\mu^{\pm}_{42}}
\end{align}
and for small $|\xi|,$ we have 
  \begin{align*}
 |\omega^{R,+}_{22} b^{R,\pm}_{11}| & \lesssim  \frac{1}{\sqrt{|\xi|}} e^{-\frac{3}{\sqrt{2}} \frac{4}{5}r_{\varepsilon  }  } , \quad  |\omega^{R,+}_{22} b^{R,\pm}_{12}| \lesssim \frac{1}{\sqrt{|\xi|}}  e^{-\frac{3}{\sqrt{2}} \frac{3}{5}r_{\varepsilon  }  }, \quad 
 |\omega^{R,+}_{22} b^{R,\pm}_{13}|   \lesssim \frac{1}{\sqrt{|\xi|}} 
 e^{-\frac{3}{\sqrt{2}} \frac{4}{5}r_{\varepsilon  }  },  \\
   | \omega_{12}^{R,\pm} b^{R,\pm}_{21}|   & \lesssim \sqrt{|\xi|} +   e^{-\frac{3}{\sqrt{2} } \frac{4}{5} r_{\varepsilon}}  , \quad
   |  \omega_{12}^{R,\pm} b^{R,\pm}_{22} |  \lesssim \sqrt{|\xi|}   e^{\frac{3}{\sqrt{2}} \frac{r_{\varepsilon}}{5}  } , \quad  |  \omega_{12}^{R,\pm} b^{R,\pm}_{23} | \lesssim \sqrt{|\xi|}  +   e^{-\frac{3}{\sqrt{2} } \frac{4}{5} r_{\varepsilon}}  ,
  \end{align*}
and for large $|\xi|,$ we have 
\begin{align*}
 |\omega^{R,+}_{22} b^{R,\pm}_{11}| & \lesssim  \frac{1}{|\xi| }  , \quad  |\omega^{R,+}_{22} b^{R,\pm}_{12}| \lesssim \frac{1}{|\xi| } , \quad |\omega^{R,+}_{22} b^{R,\pm}_{13}| \lesssim \frac{1}{|\xi| }
,  \\
   | \omega_{12}^{R,\pm} b^{R,\pm}_{21}|   &\lesssim \frac{1}{|\xi| }   , \quad  |  \omega_{12}^{R,\pm} b^{R,\pm}_{22} | \lesssim \frac{1}{|\xi| }  , \quad  |  \omega_{12}^{R,\pm} b^{R,\pm}_{23} | \lesssim \frac{1}{|\xi| } . 
  \end{align*}
\end{lemma}

\begin{proof} 
\ref{behav-omega-b}
By Lemma \ref{lem:mu-behavior-small-xi}, it follows that 
\begin{align*}
  \omega_{1l}^{R,\pm} &=  b^{R,\pm}_{l3}  \mu^{\pm}_{13} +  b^{R,\pm}_{l4} \mu^{\pm}_{14}   , \quad 
  \omega_{2l}^{R,\pm} =   b^{R,\pm}_{l4} \mu^{\pm}_{24} ,  \quad 
  \omega_{3l}^{R,\pm} = b^{R,\pm}_{l1} \mu^{\pm}_{31} + b^{R,\pm}_{l4} \mu^{\pm}_{34}   \\
    \omega_{4l}^{R,\pm} &= b^{R,\pm}_{l1} \mu^{\pm}_{41} + b^{R,\pm}_{l2} \mu^{\pm}_{42} + b^{R,\pm}_{l3} \mu^{\pm}_{43}.
\end{align*}

Therefore, 

\begin{align*}
  b^{R,\pm}_{l4}   &=  \frac{ \omega_{2l}^{R,\pm}}{\mu^{\pm}_{24}}  , \quad 
  b^{R,\pm}_{l3}   =  \frac{1}{ \mu^{\pm}_{13}} \left(   \omega_{1l}^{R,\pm} -  \frac{ \omega_{2l}^{R,\pm}}{\mu^{\pm}_{24}}  \mu^{\pm}_{14}   \right) , \quad 
   b^{R,\pm}_{l1}   = \frac{1}{ \mu^{\pm}_{31}} \left(   \omega_{3l}^{R,\pm} -   \frac{ \omega_{2l}^{R,\pm}}{\mu^{\pm}_{24}}  \mu^{\pm}_{34}  \right),  \\
 b^{R,\pm}_{l2} & = \frac{1}{ \mu^{\pm}_{42}}    \left(  \omega_{4l}^{R,\pm} -  b^{R,\pm}_{l1} \mu^{\pm}_{41}  -b^{R,\pm}_{l3} \mu^{\pm}_{43} \right) .
\end{align*}    

For small $|\xi|, $ by Lemma \ref{omega_j-behavior-resonance} \ref{lem:parity-of-omega-j-resonance} and \ref{lem:mu-behavior-small-xi}, we have 
\begin{align*}
  b^{R,\pm}_{14}&= \frac{ \omega_{21}^{R,\pm}}{\mu^{\pm}_{24}}  , \qquad
|   b^{R,\pm}_{11}|=| \frac{ 1}{\mu^{\pm}_{31}} \left(  \omega_{31}^{R,\pm} -  \frac{ \omega_{21}^{R,\pm} \mu^{\pm}_{34}}{\mu^{\pm}_{24}}
 \right)|  \lesssim \frac{1}{c \, k_1(z) + O(|\xi|^{\frac{3}{2}}) }   e^{\frac{3}{\sqrt{2}} \frac{r_{\varepsilon}}{5} }  \\
   |b^{R,\pm}_{13}|&=  |\frac{1}{ \mu^{\pm}_{13}} \left(   \omega_{11}^{R,\pm} -  \frac{ \omega_{21}^{R,\pm}}{\mu^{\pm}_{24}}  \mu^{\pm}_{14}   \right)  | \lesssim \frac{1}{c \, k_1(z) + O(|\xi|^{\frac{3}{2}}) } e^{\frac{3}{\sqrt{2}} \frac{r_{\varepsilon}}{5} }    \\
 |  b^{R,\pm}_{12} | &= | \frac{1}{ \mu^{\pm}_{42}}    \left(  \omega_{41}^{R,\pm} -  b^{R,\pm}_{11} \mu^{\pm}_{41}  -b^{R,\pm}_{13} \mu^{\pm}_{43} \right)  |  \lesssim \frac{1}{c \, k_1(z) + O(|\xi|^{\frac{3}{2}}) }   e^{\frac{3}{\sqrt{2}} \frac{2 r_{\varepsilon}}{5} }    
\end{align*}

and 
\begin{align*} 
  b^{R,\pm}_{24}&= \frac{ \omega_{22}^{R,\pm}}{\mu^{\pm}_{24}}  , \qquad
|b^{R,\pm}_{21}| = |\frac{ 1}{\mu^{\pm}_{31}} \left(  \omega_{32}^{R,\pm} -  \frac{ \omega_{22}^{R,\pm} \mu^{\pm}_{34}}{\mu^{\pm}_{24}}   \right) |  \lesssim 1 + \frac{1}{c \, k_1(z) + O(|\xi|^{\frac{3}{2}}) } e^{-\frac{3}{\sqrt{2}} \frac{4 }{5}r_{\varepsilon}}\\
  | b^{R,\pm}_{23} | &=   | \frac{1}{ \mu^{\pm}_{13}} \left(   \omega_{12}^{R,\pm} -  \frac{ \omega_{22}^{R,\pm}}{\mu^{\pm}_{24}}  \mu^{\pm}_{14}   \right)  |\lesssim 1 + \frac{1}{c \, k_1(z) + O(|\xi|^{\frac{3}{2}}) } e^{-\frac{3}{\sqrt{2}} \frac{4 }{5}r_{\varepsilon}} \\
  |  b^{R,\pm}_{22}  | &= | \frac{1}{ \mu^{\pm}_{42}}    \left(  \omega_{42}^{R,\pm} -  b^{R,\pm}_{21} \mu^{\pm}_{41}  -b^{R,\pm}_{23} \mu^{\pm}_{43} \right)  | \lesssim    \frac{1}{c \, k_1(z) + O(|\xi|^{\frac{3}{2}}) } e^{-\frac{3}{\sqrt{2}} \frac{3}{5} r_{\varepsilon} }    +   e^{\frac{3}{\sqrt{2}} \frac{r_{\varepsilon}}{5}  }  
\end{align*} 

Similarly, for large $|\xi|, $ using Lemma \ref{omega_j-behavior-large-xi-resonance}, \ref{lem:parity-of-omega-j-resonance} and \ref{lem:mu-behavior-small-xi}, we have
\begin{align*}
  b^{R,\pm}_{14}&= \frac{ \omega_{21}^{R,\pm}}{\mu^{\pm}_{24}}, \quad 
 | b^{R,\pm}_{11}|= |\frac{ 1}{\mu^{\pm}_{31}}  \omega_{31}^{R,\pm} | \lesssim \frac{1}{|\xi|^{\frac{3}{4}}}, \qquad 
  |  b^{R,\pm}_{13} | =  | \frac{1}{ \mu^{\pm}_{13}}    \omega_{11}^{R,\pm} |  \lesssim \frac{1}{|\xi|^{\frac{3}{4}}}     , \quad 
 |  b^{R,\pm}_{12} | = | \frac{1}{ \mu^{\pm}_{42}}     \omega_{41}^{R,\pm}   | \lesssim \frac{1}{|\xi|^{\frac{3}{4}}}    
\end{align*}
and 
\begin{align*}
  b^{R,\pm}_{24}&= \frac{ \omega_{22}^{R,\pm}}{\mu^{\pm}_{24}}, \quad 
| b^{R,\pm}_{21} |  = | \frac{ 1}{\mu^{\pm}_{31}}   \omega_{32}^{R,\pm} | \lesssim  \frac{1}{|\xi|^{\frac{3}{4}}}  , \quad  
|  b^{R,\pm}_{23}| =   | \frac{1}{ \mu^{\pm}_{13}}  \omega_{12}^{R,\pm}  | \lesssim   \frac{1}{|\xi|^{\frac{3}{4}}}  , \quad  
|   b^{R,\pm}_{22} | = | \frac{1}{ \mu^{\pm}_{42}}      \omega_{42}^{R,\pm}  |  \lesssim  \frac{1}{|\xi|^{\frac{3}{4}}} .
\end{align*}

\end{proof}

\begin{lemma}
\label{theta_j-behavior-resonance}
Let $z=\lambda \pm i0,$ such that $\lambda=\xi \pm \frac{\sqrt{17}}{8}.$ Then we have
    \begin{align*}
        \Theta^{R}(r,\lambda) = \sum_{j=1}^{3} \theta^{R,+}_j(\lambda) \Psi^{+}_j(r,\lambda), \quad \text{ where }   \theta^{R,+}_j(\lambda) :=\omega_{22}^{R,+}(\lambda) b_{1j}^{R,+} - \omega_{21}^{R,+}(\lambda)  b_{2j}^{R,+}, 
    \end{align*}
where for small $|\xi|$, we have
\begin{align*}
    \theta^{R,+}_1(\lambda) = O(\frac{1}{\sqrt{|\xi|}})  \quad    
    | \theta^{R,+}_2(\lambda) | \lesssim    e^{\frac{3}{\sqrt{2}} \frac{3}{5} r_{\varepsilon}  }     , \quad 
     | \theta^{R,+}_3(\lambda) | \lesssim 1
\end{align*}
and for large $|\xi|$, we have
\begin{align*}
    \theta^{R,+}_1(\lambda) = O(\frac{1}{|\xi|}) , \quad      \theta^{R,+}_2(\lambda) =O(\frac{1}{|\xi|}    )     , \quad 
     \theta^{R,+}_3(\lambda) = O(\frac{1}{|\xi|})
\end{align*}

\end{lemma}
\begin{proof}
    By definition of $\Theta^{R}(r,\lambda)$ in \eqref{def-theta-resonance} with \eqref{exp-phi-terms-psi-resonance} and \eqref{drop-out-omega-b-resonance}, we have
    \begin{align*}
        \Theta^{R}(r,\lambda) = \sum_{j=1}^{3} \theta^{R,+}_j(\lambda) \Psi^{R,+}_j(r,\lambda), \quad \text{ where }   \theta^{R,+}_j(\lambda) :=\omega_{22}^{R,+}(\lambda) b_{1j}^{R,+} - \omega_{21}^{+}(\lambda)  b_{2j}^{R,+}.
    \end{align*}
Therefore, we have
\begin{align*}
     \theta^{R,+}_1 &=    \frac{1}{ \mu^{+}_{31}} \left( \omega_{22}^{R,+}    \omega_{31}^{R,+} -   \omega_{21}^{R,+}   \omega_{32}^{R,+}  \right) , \\
     \theta^{R,+}_2 &=   \frac{1}{ \mu^{+}_{42}} ( \omega_{22}^{R,+}     \omega_{41}^{R,+}- \omega_{21}^{R,+}  \omega_{42}^{R,+}  )  - \omega_{22}^{R,+}   \frac{1}{ \mu^{+}_{42}}  \left(    b^{R,+}_{11} \mu^{+}_{41}  + b^{R,+}_{13} \mu^{+}_{43} \right) - \omega_{21}^{R,+}  \frac{1}{ \mu^{+}_{42}}    \left(    b^{R,+}_{21} \mu^{+}_{41} +b^{R,+}_{23} \mu^{+}_{43}   \right) 
\end{align*}
Note that, for large $|\xi|, $ we have $ \theta^{R,+}_2=\frac{1}{ \mu^{+}_{42}} ( \omega_{22}^{R,+}     \omega_{41}^{R,+}- \omega_{21}^{R,+}  \omega_{42}^{R,+}  ).$ Moreover, 
\begin{align*}
    \theta^{R,+}_3 &=    \frac{1}{ \mu^{+}_{13}}  \left(  \omega_{22}^{R,+}  \omega_{11}^{R,+} -\omega_{21}^{R,+}  \omega_{12}^{R,+} \right) .
\end{align*}
Using Lemma \ref{behav-omega-b-resonance}, we obtain the desired estimates for $\theta_j^{R,+}.$ 
\end{proof}

\begin{lemma}
\label{Theta-interms-psi-resonance}
Let $z=\lambda \pm i0,$ such that $\lambda=\xi \pm \frac{\sqrt{17}}{8}.$ Then we have
    \begin{align}
    \label{theta_in_terms-Psi-resonance}
    \begin{split}
     \Theta^{R}(r,\lambda)& =\gamma_1^R(\lambda) \left( \omega^{R,+}_{11}(\lambda) \sigma_3  \Psi_1^{R,+}(r,-\lambda)  +   \omega^{R,+}_{11}(-\lambda)  \Psi^{R,+}_1(r,\lambda) \right)
  \\
 & + \gamma_2^R(\lambda)   (\sigma_3) \Psi_1^{R,+}(r,-\lambda)  + \gamma_3^R(\lambda)  \Psi^{R,+}_1(r,\lambda) + \gamma_4(\lambda)  \Psi^{R,+}_2(r,\lambda)
\end{split}
  \end{align}
where for small $|\xi|, $ we have 
\begin{align*}
 \gamma_1^R(\lambda)   = O(\frac{1}{\sqrt{|\xi|} }e^{-\frac{3}{\sqrt{2}} r_{\varepsilon}} ), \quad |\gamma_2^R(\lambda)| \lesssim 1 , \quad |\gamma_3^R(\lambda)|\lesssim 1 , \quad |\gamma_4^R(\lambda)| \lesssim e^{\frac{3}{\sqrt{2}} \frac{2 r_{\epsilon}}{5}}.
\end{align*}
and for large $|\xi|, $ we have 
\begin{align*}
    \gamma_1^R(\lambda)   = O(\frac{1}{|\xi|^{\frac{3}{4}}} ), \quad \gamma_2^R(\lambda) = O(\frac{1}{|\xi|}), \quad \gamma_3^R(\lambda)=  O(\frac{1}{|\xi|}) ,  \quad  \gamma_4^R(\xi) = O(\frac{1}{|\xi|}) .
\end{align*}
\end{lemma}

\begin{proof}
Following computations analogous to one in the proof of Lemma \ref{Resol-interms-psi} leads to 
\begin{align*} 
    \gamma_1^{R}(\lambda)  &:= -\frac{1}{ \mu^{+}_{13}(\lambda)}  \omega_{22}^{R,+}(\lambda) \alpha_2^{R}(\lambda) \\
     \gamma_2^{R}(\lambda) &:=   \frac{1}{ \mu^{+}_{13}(\lambda)}   \omega_{21}^{R,+}(\lambda)  \alpha_2^{R}(\lambda) \omega_{12}^{R,+}(\lambda)  \\
     \gamma_3^{R}(\lambda) &:= - \frac{1}{ \mu^{+}_{13}(\lambda)}   \omega_{21}^{R,+}(\lambda)  \alpha_2^{R}(\lambda) \omega^{R,+}_{12}(-\lambda)\\
     \gamma_4^{R}(\lambda) &:= \theta^{R,+}_2(\lambda)+ \frac{1}{ \mu^{R,+}_{13}(\lambda)}  \omega_{22}^{R,+}(\lambda)  \alpha_3^{R}(\lambda)  \omega_{11}^{R,+} (\lambda)   - \frac{1}{ \mu^{+}_{13}(\lambda)}   \omega_{21}^{R,+}(\lambda)  \alpha_3^{R}(\lambda)   \omega_{12}^{R,+}(\lambda)  
\end{align*}
where
\begin{align*}
\frac{\alpha_2^{R}(\lambda)}{\mu^{+}_{13}(\lambda)} &= \frac{1}{ c k_1^{+}(\lambda) ( 1- c_1(\lambda)^2)} , \qquad 
\frac{ \alpha_3^{R}(\lambda) }{ \mu^{+}_{13}(\lambda) } = \frac{\mu^{+}_{34}(\lambda)}{\mu^{+}_{13}(\lambda) }\frac{1}{c k_2^{+}(\lambda) ( 1- c_2(\lambda)^2)} 
\end{align*}
Therefore, using Lemma \ref{lem:mu-behavior-small-xi} with Lemma \ref{omega_j-behavior-resonance}  for small $|\xi|$ and Lemma \ref{omega_j-behavior-large-xi-resonance} for large $|\xi|,$ we obtain the desired estimate for $\gamma_1^R,\gamma_2^R,\gamma_3^R.$ For $\gamma_4^R$, we use the same argument together with Lemma \ref{theta_j-behavior-resonance}, which provides the estimate for $\theta_2^R$.
\end{proof}

\section{Proof of the $L^2$ bounds for $e^{t\calL}$} \label{sec:L2bounds}
In this proof, we establish the $L^2$-boundedness of the operator $\mathcal{L}$ in both the resonance and non-resonance cases. Recall that, throughout the paper in the non-resonance case, for $\lambda= \xi \pm \frac{\sqrt{17}}{8} $ with $0 < |\xi| < \delta_0$, we choose $\varepsilon_0$ and $\varepsilon_{\infty}$ such that $r_{\infty} < r_0.$ Moreover, we fix $r_{\varepsilon} \in (r_{\infty},r_0)$ such that $r_{\varepsilon}\simeq \frac{\varepsilon}{\sqrt{|\xi|}},$ for some $\varepsilon>0.$ For $1 < \Lambda_0 < |\xi|$, we choose $\tvarepsilon_0$ and $\tvarepsilon_{\infty}$ so that 
$\widetilde{r}_{\infty} = \frac{\widetilde{\varepsilon}_{\infty}}{\sqrt{|\xi|}} < \widetilde{r}_0 = \frac{\widetilde{\varepsilon}_0}{\sqrt{|\xi|}}$ and we fix $\tr_\varepsilon \in (\tr_\infty, \tr_0)$ such that $\tr_\varepsilon \simeq \frac{\varepsilon}{\sqrt{|\xi|}},$ for some $\varepsilon>0.$ Note that, for simplicity we use the same symbol $\varepsilon$ in the small and large $\xi$-regimes, although they represent two different constants. In the resonance case, we similarly fix $R_0$ and $\widetilde{R}_0$ for small and large values of $\xi$. \\

In order to obtain $L^2$-estimates in small and large $\xi$-regimes, 
we define a smooth, even cut-off function $\chi_0(x)$ such that $\chi_0(x)=1$ for $|x| \leq \frac{9\tvarepsilon}{10}$ and $\chi_0(x)=0$ for $|x| \geq \tvarepsilon ,$ with $ \varepsilon_{\infty} <6 \tvarepsilon < \varepsilon,$ for $0<|\xi|<\delta_0$ and $ \tvarepsilon_{\infty} <6 \tvarepsilon < \varepsilon ,$ for $|\xi|>\Lambda_0>1,$ where $\varepsilon$ is as defined above. We denote by $\chi_1(x):= 1 - \chi_0(x).$ \\

Note that throughout this section we work with $k_j(\lambda) = k_j^{+}(\lambda) = \lim_{y \to 0^{+}} k_j(\lambda + i y)$ for $j=1,2$. In addition, we have $k_1(-\lambda) = -k_1(\lambda)$ and $k_2(-\lambda) = k_2(\lambda)$. For the precise limiting behavior of $k_j$, see Lemma~\ref{lem:k_j-behavior} and Remark~\ref{rem:behavior-k_j-lambda}. \\

\subsection{$L^2$-estimate of the distorted Fourier basis}
We start by establishing the $L^2$-estimate for the distorted Fourier basis. In particular, we decompose $\Psi_1^{+}(r,\lambda)$ into a leading order term and a remainder term in $L^2_r$, and we also provide the $L^2$-estimate for $\Psi_2^{+}(r,\lambda)$ for both large and small values of $\lambda$.

\begin{lemma} 
\label{L^2-bounds-of-Psi}
Let $\lambda= \xi + \frac{\sqrt{17}}{8}$, such that $0<\xi <\delta_0,$  then we have 
\begin{align*}
\left\|  \Psi_2^{+}(r,\lambda) \chi_1( \sqrt{\xi} \cdot)   \right\|_{L^2_r(0,\infty)}   \lesssim e^{- \frac{3}{\sqrt{2}} \frac{9 \tvarepsilon}{ 10 \sqrt{|\xi|}}}    
\end{align*}

and the following decomposition for  $\Psi_1(r,\lambda)$
 \begin{align*}
    \Psi_1^{+}(r,\lambda)&=    \Uppsi^{+}_{(\infty,1)} (r,\lambda) (1 + O( \frac{1}{\sqrt{|\xi|}}   e^{-2 r_{\infty}} )) +  \Upupsilon_1(r,\lambda,\Psi^{+}_{1} (r,\lambda))
\end{align*}
where,
\begin{align}
\left\|  \Upupsilon_1(\cdot,\lambda,\Psi^{+}_{1} (\cdot,\lambda))  \chi_1( \sqrt{\xi} \cdot) \right\|_{L^2_r(0,\infty)}  \lesssim \frac{1}{\sqrt{|\xi|}}   e^{- 2 \frac{9\tvarepsilon}{10\sqrt{|\xi|}}} .
\end{align}

\end{lemma}
\begin{proof}
First, we prove the estimate for $\Psi_1.$ Recall that, \begin{align*}
    \Psi_1^{+}(r,\lambda)&=  \Uppsi^{+}_{(\infty,1)} (r,\lambda) +\Upsilon_1(r,\lambda,\Psi_1^{+} (r,\lambda)) \\
    &=   \Uppsi^{+}_{(\infty,1)} (r,\lambda)  + \int_0^{\infty} \mathcal{R}^{+}_1(r,s,\lambda) V(s) \Psi_{1}^{+}(s,\lambda) ds 
\end{align*}
where, $
    \Uppsi^{+}_{(\infty,1)} (r,\lambda) = \begin{pmatrix}
      e^{i k_1(\lambda)r} \\ c_1(\lambda)  e^{i k_1(\lambda)r}
  \end{pmatrix}  $ and by \eqref{upsilon_1-of-psi_1}, we have 
\begin{align*}
   \Upsilon_1(r,\lambda,\Psi^{+}_{1} (r,\lambda))&=i  \mathcal{M}^{+}_{1}(\lambda) e^{i k_1(\lambda) r } \int_{r_\infty}^r   e^{-i k_1(\lambda) s }    O(e^{-2s}) \Psi^{+}_{1}(s,\lambda) ds \\
&+     i  \mathcal{M}^{+}_{1}(\lambda) e^{-i k_1(\lambda) r }\int_{r}^{\infty}  e^{i k_1(\lambda) s }     O(e^{-2s}) \Psi^{+}_{1}(s,\lambda) ds \\
 & +  i  \mathcal{M}^{+}_{2} (\lambda) e^{i k_2(\lambda) r} \int_{r_\infty}^{r} e^{-i k_2(\lambda) s}   O(e^{-2s}) \Psi^{+}_{1}(s,\lambda) ds  \\
 &+  i  \mathcal{M}^{+}_{2}(\lambda)  e^{-i k_2(\lambda) r} \int_{r}^{\infty} e^{i k_2(\lambda) s}   O(e^{-2s}) \Psi^{+}_{1}(s,\lambda) ds 
\end{align*}
Recall that  $|\mathcal{M}^{+}_{1}(\lambda) |=  O(\frac{1}{\sqrt{|\xi|}})$ and $|\mathcal{M}^{+}_{2}(\lambda)|=O(1).$ Then the first integral satisfies 
\begin{align*}
 e^{i k_1(\lambda) r } \left| i  \mathcal{M}^{+}_{1}(\lambda)  \int_{r_\infty}^r   e^{-i k_1(\lambda) s }    O(e^{-2s}) \Psi^{+}_{1}(s,\lambda) ds  \right| &\lesssim  e^{i k_1(\lambda) r } \frac{1}{\sqrt{|\xi|}}  e^{-2 r_{\infty}} \sup_{r\geq r_{\infty}} | e^{ -  i k_1(\lambda) r }  \Psi^{+}_{1}(r,\lambda) |  \\
  & \lesssim e^{i k_1(\lambda) r }
  \frac{1}{\sqrt{|\xi|}}   e^{-2 r_{\infty}}
\end{align*}
Denote by $\Upupsilon_1(r,\lambda,\Psi^{+}_{1} (r,\lambda)) $ the remaining part of $\Upsilon_1(r,\lambda,\Psi^{+}_{1} (r,\lambda))$, i.e., 
\begin{align*}
  \Upupsilon_1(r,\lambda,\Psi^{+}_{1} (r,\lambda)) &:=  i  \mathcal{M}^{+}_{1}(\lambda) e^{-i k_1(\lambda) r }\int_{r}^{\infty}  e^{i k_1(\lambda) s }     O(e^{-2s}) \Psi^{+}_{1}(s,\lambda) ds \\
 & +  i  \mathcal{M}^{+}_{2} (\lambda) e^{i k_2(\lambda) r} \int_{r_\infty}^{r} e^{-i k_2(\lambda) s}   O(e^{-2s}) \Psi^{+}_{1}(s,\lambda) ds  \\
 &+  i  \mathcal{M}^{+}_{2}(\lambda)  e^{-i k_2(\lambda) r} \int_{r}^{\infty} e^{i k_2(\lambda) s}   O(e^{-2s}) \Psi^{+}_{1}(s,\lambda) ds 
\end{align*}
Thus, we have 
 \begin{align*}
    \Psi_1^{+}(r,\lambda)&=    \Uppsi^{+}_{(\infty,1)} (r,\lambda) (1 + O( \frac{1}{\sqrt{|\xi|}}   e^{-2 r_{\infty}} )) +  \Upupsilon_1(r,\lambda,\Psi^{+}_{1} (r,\lambda))
\end{align*}
where, 
\begin{align*}
 | \Upupsilon_1(r,\lambda,\Psi^{+}_{1} (r,\lambda))   |& \lesssim    \frac{1}{\sqrt{|\xi|}}    e^{-2 r} \sup_{r\geq r_{\infty}} | e^{ -  i k_1(\lambda) r }  \Psi^{+}_{1}(r,\lambda) |   +   \frac{1}{\sqrt{|\xi|}}   e^{-2 r} \sup_{r\geq r_{\infty}} | e^{ -  i k_1(\lambda) r }  \Psi^{+}_{1}(r,\lambda) |\\
 &+   \frac{1}{\sqrt{|\xi|}}   e^{-2 r} \sup_{r\geq r_{\infty}} | e^{ -  i k_1(\lambda) r }  \Psi^{+}_{1}(r,\lambda) | \\
 & \lesssim  \frac{1}{\sqrt{|\xi|}}     e^{-2 r}   ,
\end{align*}
which yields, 
\begin{align*}
\left\|  \Upupsilon_1(\cdot,\lambda,\Psi^{+}_{1} (\cdot,\lambda))  \chi_1( \sqrt{\xi} \cdot) \right\|_{L^2_r(0,\infty)} \lesssim \left\|  \Upupsilon_1(\cdot,\lambda,\Psi^{+}_{1} (\cdot,\lambda))  \right\|_{L^2_r(\frac{9\tvarepsilon}{10 \sqrt{|\xi|}},\infty)} \lesssim \frac{1}{\sqrt{|\xi|}}   e^{- 2 \frac{9\tvarepsilon}{10\sqrt{|\xi|}}}
\end{align*}

Next, we estimate $\Psi_2^{+}(\cdot,\lambda).$ Recall that 
 \begin{align*}
 \Psi_2^{+}(r,\lambda)&=  \Uppsi^{+}_{(\infty,2)} (r,\lambda) +\Upsilon_2(r,\lambda,\Psi_2^{+} (r,\lambda)) \\
    &=   \Uppsi^{+}_{(\infty,2)} (r,\lambda)  + \int_0^{\infty} \mathcal{R}^{+}_1(r,s,\lambda) V(s) \Psi_{2}^{+}(s,\lambda) ds 
\end{align*}
where $
    \Uppsi^{+}_{(\infty,2)} (r,\lambda) = \begin{pmatrix}
      e^{i k_2(\lambda)r} \\ c_2(\lambda)  e^{i k_2(\lambda)r}
  \end{pmatrix}  .$ Using Remark \ref{rem:behavior-k_j-lambda} and \eqref{upsilon_2-of-psi_2}, we have 
\begin{align*}
| \Upsilon_2(r,\lambda,\Psi^{+}_2(r,\lambda))  |  
&\lesssim  \frac{1}{\sqrt{|\xi|}} |e^{i k_2(\lambda) r }|   e^{-2 r}, 
\end{align*}
Then 
\begin{align*}
\left\|  \Psi_2(r,\lambda) \chi_1( \sqrt{\xi} \cdot)   \right\|_{L^2_r(0,\infty)} \lesssim \left\|  \Psi_2(r,\lambda)    \right\|_{L^2_r(\frac{9\tvarepsilon}{10\sqrt{|\xi|}},\infty)}   \lesssim         e^{- \frac{3}{\sqrt{2}} \frac{9\tvarepsilon}{10 \sqrt{|\xi|}}}  .
\end{align*} 
\end{proof}
\begin{lemma} 
\label{L^2-bounds-of-Psi-large-lambda}
Let $\lambda=\xi + \frac{\sqrt{17}}{8}$ such that $ \xi> \Lambda_0>1,$ then we have 
\begin{align*}
\left\|  \Psi_2^{+}(r,\lambda) \chi_1( \sqrt{\xi} \cdot)   \right\|_{L^2_r(0,\infty)}   \lesssim_{\tvarepsilon}  \frac{1}{\sqrt{|\xi|}}      
\end{align*}
and the following decomposition for  $\Psi_1(r,\lambda)$
 \begin{align*}
    \Psi_1^{+}(r,\lambda)&=    \Uppsi^{+}_{(E,1)} (r,\lambda) (1+O(\frac{1}{\sqrt{|\xi|}}) )+  \Upupsilon_1(r,\lambda,\Psi^{+}_{1} (r,\lambda)),
\end{align*}
where,
\begin{align*}
\left\|  \Upupsilon_1(\cdot,\lambda,\Psi^{+}_{1} (\cdot,\lambda))  \chi_1( \sqrt{\xi} \cdot) \right\|_{L^2_r(0,\infty)}  \lesssim_{\tvarepsilon} \frac{1}{\sqrt{|\xi|} }.
\end{align*}

\end{lemma}
\begin{proof}
First, we prove the estimate for $\Psi_1.$ Recall that, \begin{align*}
    \Psi_1^{+}(r,\lambda)&=  \Uppsi^{+}_{(E,1)} (r,\lambda) +\Upsilon_{E,1}(r,\lambda,\Psi_1^{+} (r,\lambda)) \\
    &=   \Uppsi^{+}_{(E,1)} (r,\lambda)  + \int_0^{\infty} \mathcal{R}^{+}_{E,1}(r,s,\lambda) V_E(s) \Psi_{1}^{+}(s,\lambda) ds ,
\end{align*}
where, $V_{E}(r) =O( \langle r \rangle^{-2} )$ and 
$\Uppsi^{+}_{(E,1)} (r,\lambda) = \begin{pmatrix}
      h_{+}(k_1(\lambda)r) \\ c_1(\lambda)  h_{+}(k_1(\lambda)r)
  \end{pmatrix}  $ and by \eqref{eq:upsilon_1-of-psi_1-large-xi-h_+}, we have 
  \begin{align*}
\Upsilon_{E,1}(r,\lambda,\Psi_1^{+} (r,\lambda))    & =  i  \mathcal{M}^{+}_{1}(\lambda) h_{+}( k_1(\lambda) r ) \int_{\tr_\infty}^r  h_{-}( k_1(\lambda) s )    V_E(s) \Psi^{+}_{1}(s,\lambda) ds   \\
&+  i  \mathcal{M}^{+}_{2} (\lambda) h_{+}( k_2 (\lambda) r )\int_{\tr_\infty}^{r}  h_{-}( k_2(\lambda) s )  V_E(s) \Psi^{+}_{1}(s,\lambda) ds   \\
&+     i  \mathcal{M}^{+}_{1}(\lambda)  h_{-}( k_1(\lambda) r ) \int_{r}^{\infty}   h_{+}( k_1(\lambda) s )    V_E(s) \Psi^{+}_{1}(s,\lambda) ds \\
 & +  i  \mathcal{M}^{+}_{2}(\lambda)   h_{-}( k_2(\lambda) r ) \int_{r}^{\infty}  h_{+}( k_2(\lambda) s )  V_E(s) \Psi^{+}_{1}(s,\lambda) ds,  
  \end{align*}
Note that the first integral satisfies 
\begin{align*}
 \left| i  \mathcal{M}^{+}_{1}(\lambda)\int_{\tr_\infty}^r  h_{-}( k_1(\lambda) s )    V_E(s) \Psi^{+}_{1}(s,\lambda) ds  \right|   & \lesssim_{\tvarepsilon} 
\frac{1}{\sqrt{|\xi|}} \langle \tr_{\infty} \rangle^{-2} \sup_{\tr_{\infty} \geq r } |e^{-i k_1(\lambda) r} \Psi_1^{+}(r,\lambda)|\\
    & \lesssim_{\tvarepsilon}  \frac{1}{\sqrt{|\xi|}}
\end{align*}
and denote by $\Upupsilon_{E,1}(r,\lambda,\Psi^{+}(r,\lambda)$ the remaining part of $\Upsilon_{E,1}(r,\lambda,\Psi^{+}_{1} (r,\lambda)).$ Therefore, by \eqref{upsilon_1-of-psi_1-large-xi}, we have
\begin{align*}
  \Upupsilon_{E,1}(r,\lambda,\Psi^{+}(r,\lambda) & \lesssim_{\tvarepsilon} 
  \frac{1}{\sqrt{|\xi|}} |e^{i k_2(\lambda) r }|  \int_{\tr_\infty}^{r}  |e^{-i k_2(\lambda) s }|    \langle s \rangle^{-2} |  \Psi^{+}_{1}(s,\lambda)  | ds   \\
&+  \frac{1}{\sqrt{|\xi|}} |e^{-i k_1(\lambda) r }| \int_{r}^{\infty}    |e^{i k_1(\lambda) s }|    \langle s \rangle^{-2} | \Psi^{+}_{1}(s,\lambda) |  ds \\
 & +   \frac{1}{\sqrt{|\xi|}} |e^{-i k_2(\lambda) r }|  \int_{r}^{\infty}  |e^{i k_2(\lambda) s }|  \langle s \rangle^{-2} | \Psi^{+}_{1}(s,\lambda) | ds.  
\end{align*}
which yields, $
  |  \Upupsilon_{E,1}(r,\lambda,\Psi^{+}_{1} (r,\lambda)) |  \lesssim_{\tvarepsilon} \frac{1}{\sqrt{|\xi|}} \langle r \rangle^{-1} .$ Hence, we have
\begin{align*}
\left\|   \Upupsilon_{E,1}(\cdot,\lambda,\Psi^{+}_{1} (\cdot,\lambda))  \chi_1( \sqrt{\xi} \cdot) \right\|_{L^2_r(0,\infty)} \lesssim \left\|  \Upupsilon_{E,1}(\cdot,\lambda,\Psi^{+}_{1} (\cdot,\lambda))  \right\|_{L^2_r(\frac{9\tvarepsilon}{10 \sqrt{\xi}},\infty)} \lesssim_{\tvarepsilon} \frac{1}{\sqrt{|\xi|}}  .
\end{align*}
Next, we estimate $\Psi_2^{+}(\cdot,\lambda).$ Recall that 
 \begin{align*}
 \Psi_2^{+}(r,\lambda)&=  \Uppsi^{+}_{(E,2)} (r,\lambda) +\Upsilon_{E,2}(r,\lambda,\Psi_2^{+} (r,\lambda)) \\
    &=   \Uppsi^{+}_{(E,2)} (r,\lambda)  + \int_0^{\infty} \mathcal{R}^{+}_{E,2}(r,s,\lambda) V(s) \Psi_{2}^{+}(s,\lambda) ds 
\end{align*}
where, $
    \Uppsi^{+}_{(E,2)} (r,\lambda) = \begin{pmatrix}
      h_{+}(k_2(\lambda)r) \\ c_2(\lambda)  h_{+}(k_2(\lambda)r)
  \end{pmatrix}  $. Therefore by \eqref{upsilon_2-of-psi_2-large-xi}, we have 
\begin{align*}
 \Upsilon_{E,2}(r,\lambda,\Psi^{+}_{2} (r,\lambda))     & \lesssim_{\tvarepsilon}  \frac{1}{\sqrt{|\xi|}}  |e^{i k_1(\lambda) r} | \int_{r}^{\infty} |e^{i (k_2(\lambda)-k_1(\lambda)) s } |  \langle s \rangle^{-2} ds    \sup_{ r \geq r_{\infty}  } | e^{ -  i k_2(\lambda) r }  \Psi^{+}_{2}(r,\lambda) |  \\ 
& + \frac{1}{\sqrt{|\xi|}}  | e^{i k_2(\lambda) r } | \int_{r}^{\infty}    \langle s \rangle^{-2} ds \sup_{r\geq r_{\infty}} | e^{ -  i k_2(\lambda) r }  \Psi^{+}_{2}(r,\lambda) |  \\ 
& + \frac{1}{\sqrt{|\xi|}}   | e^{-i k_1(\lambda) r } |  \int_{r}^{\infty} |   e^{i (k_2(\lambda)+k_1(\lambda) ) s }|  \langle s \rangle^{-2} ds  \sup_{r\geq r_{\infty}} | e^{ -  i k_2(\lambda) r }  \Psi^{+}_{2}(r,\lambda) |  \\ 
& +  \frac{1}{\sqrt{|\xi|}}  |e^{-ik_2(\lambda) r}|  \int_r^{\infty} |e^{2ik_2(\lambda) s}|   \langle s \rangle^{-2} ds \sup_{r\geq r_{\infty}} | e^{ -  i k_2(\lambda) r }  \Psi^{+}_{2}(r,\lambda) |.
\end{align*} 
Hence, $
  |  \Upsilon_{E,2}(r,\lambda,\Psi^{+}_{2} (r,\lambda)) | \lesssim_{\tvarepsilon} \frac{1}{\sqrt{|\xi|}} e^{-\sqrt{2}\sqrt{|\xi|} r}  .$ Then 
\begin{align*}
\left\|  \Psi_2^{+}(r,\lambda) \chi_1( \sqrt{\xi} \cdot)   \right\|_{L^2_r(0,\infty)} \lesssim \left\|  \Psi_2(r,\lambda)    \right\|_{L^2_r(\frac{9\tvarepsilon}{10\sqrt{|\xi|}},\infty)}   \lesssim_{\tvarepsilon}  \frac{1}{\sqrt{|
\xi|}}      .
\end{align*} 
\end{proof}

\begin{lemma}
\label{L^2-L^2-boud-Psi-1}
    Let $\lambda=\xi + \frac{\sqrt{17}}{8}$ such that $\xi \in (0,\delta_0) ,$ and  $\Phi \in L^2_r $ then we have 
\begin{align}
\label{eq:L2boundsPsi1-samll-lambda}
\left( \int_0^{\delta_0} \frac{1}{\sqrt{|\xi|}} \left| \int_0^{\infty} \Uppsi^{+}_{(\infty,1)} (r,\pm \lambda) \chi_1( \sqrt{\xi} r ) \Phi(r) dr \right|^2 d \xi \right)^{\frac{1}{2}} \lesssim \left\| \Phi \right\|_{L^2_r}.
\end{align}
It follows that, 
\begin{align} \label{eq:L2boundsPsi1-full-samll-lambda}
  \left( \int_0^{\delta_0}  \frac{1}{\sqrt{|\xi|}}  \left| \int_0^{\infty} \Psi_1^{+}(r, \pm\lambda) \chi_1( \sqrt{\xi} r ) \Phi(r) dr \right|^2 d \xi \right)^{\frac{1}{2}} \lesssim \left\| \Phi \right\|_{L^2_r}  
\end{align}
\end{lemma}
\begin{proof}
   We only prove the case $\Uppsi^{+}_{(\infty,1)} (r,+\lambda),$ and the other case can be obtained using a similar argument. Using change of variable $\tau=\sqrt{\xi},$ we have 
   \begin{align*}
      \int_0^{\delta_0} \frac{1}{\sqrt{|\xi|}} \left| \int_0^{\infty} \Uppsi^{+}_{(\infty,1)} (r, \lambda) \chi_1( \sqrt{\xi} r ) \Phi(r) dr \right|^2 d \xi  = \int_0^{\sqrt{\delta_0}}  \left| \int_0^{\infty} \Uppsi^{+}_{(\infty,1)} (r, \lambda(\tau^2)) \chi_1( \tau r ) \Phi(r) dr \right|^2 d\tau.
   \end{align*}
Define the operator 
   \begin{align*}
       (T v ) (\tau):=  \int_0^{\infty}   e^{i k_1(\lambda(\tau^2)) r}  \chi_1( \tau r ) v(r) dr , \quad 0< \tau < \sqrt{\delta_0}.
   \end{align*}
For the estimate \eqref{eq:L2boundsPsi1-samll-lambda} stated in the lemma we need to show that the operator $T$ is bounded from $L^2_r([0,\infty))$ into $L^2_\tau([0,\sqrt{\delta_0}))$. To this end, we will apply the Cotlar-Stein lemma. Let $\eta(x)=\chi_0(\frac{x}{2})-\chi_0(x),$ we decompose the operator $T=\sum_{j=1}^{\infty} T_j$ dyadically,
\begin{align*}
   (T_j v) (\tau):= \int_0^{\infty}  e^{i k_1(\lambda(\tau^2)) r}  \chi_1( \tau r ) \eta( 2^{-j} \tau r ) v(r) dr,
\end{align*}
and 
\begin{align*}
   (T_k^{\ast} u) (r):= \int_0^{\sqrt{\delta_0}}  e^{-i k_1(\lambda(\mu^2)) r}  \chi_1( \mu r ) \eta( 2^{-j} \mu r ) u(\mu) d\mu.
\end{align*}
Next, we prove  the off-diagonal decay of the operator norms $\|T_j T_k^\ast\|_{L^2_\tau \to L^2_\tau}$ and $\|T_j^\ast T_k\|_{L^2_r \to L^2_r}$, for $1 \leq j,k \leq \infty$. The integral kernels of these operators are given by
 \begin{align*}
            \bigl( T_j T_k^\ast u \bigr)(\tau) &= \int_0^{\sqrt{\delta_0}} \mathcal{H}_{jk}(\tau, \mu) u(\mu) d \mu, \\ 
            \bigl( T_j^\ast T_k v \bigr)(r) &= \int_0^\infty \mathcal{G}_{jk}(r,s) v(s) d s,
\end{align*}
    with 
 \begin{align*}
            \mathcal{H}_{jk}(\tau,\mu) &:= \int_0^\infty e^{ik_1(\lambda(\tau^2)) -k_1(\lambda(\mu^2))r}  
           \eta (2^{-j} \tau r )\eta (2^{-k} \mu r ) dr, \\
           \mathcal{G}_{jk}(r,s) &:= \int_0^{\sqrt{\delta_0}} e^{-ik_1(\lambda(\tau^2))(r-s)}  \eta (2^{-j} \tau r ) \eta (2^{-k} \tau s ) d \tau.
\end{align*}
Next, we derive operator norm estimates for $T_j T_k^{\ast}$ with $1 \leq j,k < \infty.$ The support of the cutoff functions implies that the kernel $\mathcal{H}_{jk}(\tau,\mu)$ vanishes unless
\begin{align*}
            2^{-j} \tau r \simeq 1 \simeq 2^{-k} \mu r \quad \Rightarrow \quad \frac{\tau}{2^j} \simeq \frac{\mu}{2^k}.   
\end{align*}
The following kernel estimates follow from two integrations by parts, or from a straightforward estimate of the integral, 
\begin{align*}
            \bigl|\mathcal{H}_{j,k}(\lambda,\mu)\bigr| \lesssim \frac{\frac{2^j}{\tau}}{\bigl( 1 + \frac{2^j}{\tau}\bigl| k_1(\lambda^2)-k_1(\mu^2) \bigr|\bigr)^2} \simeq \frac{\frac{2^k}{\mu}}{\bigl( 1 + \frac{2^k}{\mu}\bigl| k_1(\lambda^2)-k_1(\mu^2) \bigr|\bigr)^2}.   
\end{align*}
Using the fact that $k_1(\lambda(\tau^2))-k_1(\lambda(\mu^2) )\simeq \tau - \mu$ for $0 \leq \tau,\mu \leq \sqrt{\delta_0}$, together with Schur's test we obtain $\|T_j T_k^\ast\|_{L^2_\tau \to L^2_\tau} \lesssim 2^{-|j-k|}$. Similarly, one obtain the off-diagonal decay estimates $\|T_j^\ast T_k\|_{L^2_r \to L^2_r} \lesssim 2^{-|j-k|}$. Applying the Cotlar-Stein lemma yields the $L^2$ bound \eqref{eq:L2boundsPsi1-samll-lambda}, which, together with the decomposition of $\Psi^{+}_1$ in Lemma~\ref{L^2-bounds-of-Psi}, implies \eqref{eq:L2boundsPsi1-full-samll-lambda}.
\end{proof}

\begin{lemma}
\label{L^2-L^2-boud-Psi-1-large-lambda}
    Let $\lambda=\xi + \frac{\sqrt{17}}{8}$ such that $\xi \in (\Lambda_0,\infty) ,$ and  $\Phi \in L^2_r $ then we have 
\begin{align}
\left( \int_{\Lambda_0}^{\infty} \frac{1}{\sqrt{|\xi|}} \left| \int_0^{\infty} \Uppsi^{+}_{(E,1)} (r,\pm \lambda) \chi_1( \sqrt{\xi} r ) \Phi(r) dr \right|^2 d \xi \right)^{\frac{1}{2}} \lesssim \left\| \Phi \right\|_{L^2_r}
\end{align}
It follows that, 
\begin{align}
  \left( \int_{\Lambda_0}^{\infty}  \frac{1}{\sqrt{|\xi|}}  \left| \int_0^{\infty} \Psi_1^{+}(r, \pm\lambda) \chi_1( \sqrt{\xi} r ) \Phi(r) dr \right|^2 d \xi \right)^{\frac{1}{2}} \lesssim \left\| \Phi \right\|_{L^2_r}  
\end{align}
\end{lemma}
\begin{proof}
The proof closely follows that of Lemma~\ref{L^2-L^2-boud-Psi-1} and is therefore omitted.
\end{proof}
\subsection{$L^2$-bound of the operator $\mathcal{L}$}

Recall that by Proposition \ref{jump-resol} and Lemma \ref{Resol-interms-phi}, for any compactly supported functions  $\Phi,\Psi \in L^2_r(0,\infty)$  such that $\Phi=\begin{pmatrix}
    \phi_1 \\
    \phi_2
\end{pmatrix}$ and $\Psi=\begin{pmatrix}
    \psi_1 \\
    \psi_2
\end{pmatrix} , $
we have 
    \begin{align*}
      <  e^{t \mathcal{L}} \Phi, \Psi> &= \frac{1}{2 \pi i } \int_{I} e^{it \lambda}\frac{\kappa(\lambda)}{d^{+}(\lambda) d^{-}(\lambda)}  \left< \int_0^{\infty}  \Theta(\cdot,\lambda) \Theta(s,\lambda)^{t} \sigma_1 \Phi(s)ds, \Psi(\cdot) \right>_{L^2_r} d \lambda , \\
      &= \frac{1}{2 \pi i } \int_{I} e^{it \lambda}\frac{\kappa(\lambda)}{d^{+}(\lambda) d^{-}(\lambda)}  \left<  \Theta(\cdot,\lambda)  ,\sigma_1\Phi(\cdot) \right>_{L^2_r} \left< \Theta(\cdot,\lambda) , \Psi(\cdot) \right>_{L^2_r} d \lambda
\end{align*}
where $ \frac{\kappa(\lambda)}{d^{+}(\lambda) d^{-}(\lambda)}$  is an odd function and $\Theta_1(r,\lambda)$ is an odd function in $\lambda$ and $\Theta_2(r,\lambda)$ is an even function. Therefore, we obtain 
\begin{align*}
     <  e^{t \mathcal{L}} \Phi, \Psi> =  \frac{1}{  \pi i  } \int_{\frac{\sqrt{17}}{8}}^{\infty}  \cos(t \lambda) \frac{\kappa(\lambda)}{d^{+}(\lambda) d^{-}(\lambda)} \bigg(& <\Theta_1(\cdot, \lambda), \phi_2(\cdot) ><\Theta_2(\cdot, \lambda), \psi_2(\cdot) >  \\
     &+ <\Theta_2(\cdot, \lambda), \phi_1(\cdot) > <\Theta_1(\cdot, \lambda), \psi_1(\cdot) >\bigg)d \lambda  \\
    +  \frac{1}{  \pi   } \int_{\frac{\sqrt{17}}{8}}^{\infty}  \sin(t \lambda) \frac{\kappa(\lambda)}{d^{+}(\lambda) d^{-}(\lambda)} \bigg(&  <\Theta_1(\cdot, \lambda), \phi_2(\cdot) > <\Theta_1(\cdot, \lambda), \psi_1(\cdot) > \\
  & + <\Theta_2(\cdot, \lambda), \phi_1(\cdot) ><\Theta_2(\cdot, \lambda), \psi_2(\cdot) > \bigg) d \lambda 
\end{align*}

We write $\lambda = \xi + \tfrac{\sqrt{17}}{8}$ with $\xi \in (0,\infty)$, and regard $\lambda$ as a function of $\xi$ and vice versa. Therefore, we divide the above integrals into integration region $r \lesssim \frac{\varepsilon}{\sqrt{\xi}}$ and $r \gtrsim  \frac{\varepsilon}{\sqrt{\xi}}.$ For small $r,$ and depending on the value of $\lambda,$ we use the representation in Lemma \ref{Resol-interms-phi} with \eqref{def-theta}. For large values of $r,$ we use the representation \eqref{theta_in_terms-Psi} from Lemma \ref{Resol-interms-psi}. 

We denote by:
\begin{align*}
    \mathcal{I}^1_j(\lambda)&:= <\Theta_1(\cdot, \lambda) \chi_j( \sqrt{\xi} \cdot), \phi_2(\cdot) >, \quad 
\mathcal{J}^1_k(\lambda):=  <\Theta_2(\cdot, \lambda) \chi_k( \sqrt{\xi} \cdot), \psi_2(\cdot) >, \\
 \mathcal{I}^2_j(\lambda)&:=<\Theta_2(\cdot, \lambda) \chi_j( \sqrt{\xi} \cdot) , \phi_1(\cdot) > ,\quad 
\mathcal{J}^2_k(\lambda):=  <\Theta_1(\cdot, \lambda)\chi_k( \sqrt{\xi} \cdot), \psi_1(\cdot) >.
\end{align*}

Then we have 
\begin{align*}
     <  e^{t \mathcal{L}} \Phi, \Psi> &=  \sum_{0 \leq j,k \leq 1} \frac{1}{  \pi i  } \int_{0}^{\infty}  \cos(t \lambda) \frac{\kappa(\lambda)}{d^{+}(\lambda) d^{-}(\lambda)}  \big( \mathcal{I}^1_j(\lambda)   \mathcal{J}^1_k(\lambda)  +  \mathcal{I}^2_j(\lambda) \mathcal{J}^2_k(\lambda) \big)d \xi   \\
  &  +  \sum_{0 \leq j,k \leq 1}  \frac{1}{  \pi   } \int_{0}^{\infty}  \sin(t \lambda) \frac{\kappa(\lambda)}{d^{+}(\lambda) d^{-}(\lambda)}  \big( \mathcal{I}^1_j(\lambda)  \mathcal{J}^2_k(\lambda) 
   + \mathcal{I}^2_j(\lambda)    \mathcal{J}^1_k(\lambda)  \big) d\xi
\end{align*}
Next we estimate the contribution of $\mathcal{I}_j^l(\lambda)\mathcal{J}_k^l(\lambda),$ for the non-resonance and resonance case, in the Stone-type formula above, for $0 \leq j , k \leq 1$ and $l=1,2.$ 
\subsubsection{Proof of the non-resonance case.} We divide the proof into two steps for small and large $\xi.$ \\

\textbf{Step 1: Small $\xi \in (0,\delta_0).$}\\
\textbf{Case 1: } contribution of $\mathcal{I}_0(\lambda)\mathcal{J}_0(\lambda).$
\begin{claim} We have 
\label{case1-contribution-I_0J_0}
\begin{align*}
   & \int_{0}^{\delta_0}   \left|  \frac{\kappa(\lambda)}{d^{+}(\lambda) d^{-}(\lambda)} \mathcal{I}_0^1(\lambda) \mathcal{J}_0^k(\lambda)  \right| d \xi \lesssim   \left\| \phi_2 \right\|_{L^2_r}  \left\| \psi_2 \right\|_{L^2_r}, \quad \text{ for } \;  k=1,2. \\
   & \int_{0}^{ \delta_0}    \left|   \frac{\kappa(\lambda)}{d^{+}(\lambda) d^{-}(\lambda)} \mathcal{I}_0^2(\lambda) \mathcal{J}_0^k(\lambda) \right| d \xi \lesssim   \left\| \phi_2 \right\|_{L^2_r}  \left\| \psi_2 \right\|_{L^2_r} , \quad \text{ for } \;  k=1,2.
\end{align*}

\end{claim}
\begin{proof}
We restrict our attention to proving the estimate for $\mathcal{I}_0^1(\lambda)\,\mathcal{J}_0^1(\lambda)$, since the remaining cases follow by analogous arguments.  We denote by  $\upvarphi_1(r,\lambda)=\begin{pmatrix}
    \upvarphi_{11}(r,\lambda) \\
    \upvarphi_{12}(r,\lambda)
\end{pmatrix}$
and $\upvarphi_2(r,\lambda)=\begin{pmatrix}
    \upvarphi_{21}(r,\lambda) \\
    \upvarphi_{22}(r,\lambda)
\end{pmatrix}. $ Recall that by \eqref{def-theta}, we have $
        \Theta(r,\lambda) := \omega_{22}^{+}(\lambda) \upvarphi_1(r,\lambda) - \omega_{21}^{+}(\lambda) \upvarphi_2(r,\lambda).$ Hence, we have
    \begin{align*}
    & \int_{0}^{\delta_0}    \frac{\kappa(\lambda)}{d^{+}(\lambda) d^{-}(\lambda)} \mathcal{I}_0^1(\lambda) \mathcal{J}_0^1(\lambda) d\xi\\
&= \int_{0}^{\delta_0}   \frac{\kappa(\lambda)}{d^{+}(\lambda) d^{-}(\lambda)}  \bigg( < (\omega_{22}^{+} (\lambda)\upvarphi_{11}(\cdot,\lambda) -w_{21}^{+}(\lambda) \upvarphi_{21}(\cdot,\lambda) )  \chi_0( \sqrt{\xi} \cdot), \phi_2(\cdot) >  \bigg) \\
& \qquad \qquad \qquad \qquad  \quad  \times   \bigg( <( \omega_{22}^{+}(\lambda) \upvarphi_{12}(\cdot,\lambda)  - w_{21}^{+}(\lambda) \upvarphi_{22}(\cdot,\lambda)) \chi_0( \sqrt{\xi} \cdot), \psi_2(\cdot) > \bigg)  d\xi \\
 &= \int_{0}^{\delta_0}   \frac{\kappa(\lambda)}{d^{+}(\lambda) d^{-}(\lambda)} < \omega_{22}^{+} (\lambda)\upvarphi_{11}(\cdot,\lambda)   \chi_0( \sqrt{\xi} \cdot), \phi_2(\cdot) > < \omega_{22}^{+}(\lambda) \upvarphi_{12}(\cdot,\lambda)   \chi_0( \sqrt{\xi} \cdot), \psi_2(\cdot) >  d\xi  \\
& + \int_{0}^{\delta_0}   \frac{\kappa(\lambda)}{d^{+}(\lambda) d^{-}(\lambda)} < \omega_{22}^{+} (\lambda)\upvarphi_{11}(\cdot,\lambda)   \chi_0( \sqrt{\xi} \cdot), \phi_2(\cdot) > <   w_{21}^{+}(\lambda) \upvarphi_{22}(\cdot,\lambda)  \chi_0( \sqrt{\xi} \cdot), \psi_2(\cdot) >   d\xi \\
& + \int_{0}^{\delta_0}   \frac{\kappa(\lambda)}{d^{+}(\lambda) d^{-}(\lambda)} <   w_{21}^{+}(\lambda) \upvarphi_{21}(\cdot,\lambda)  \chi_0( \sqrt{\xi} \cdot), \phi_2(\cdot) >  < \omega_{22}^{+}(\lambda) \upvarphi_{12}(\cdot,\lambda)   \chi_0( \sqrt{\xi} \cdot), \psi_2(\cdot) >  d\xi \\
&+  \int_{0}^{\delta_0}   \frac{\kappa(\lambda)}{d^{+}(\lambda) d^{-}(\lambda)}  <   w_{21}^{+}(\lambda) \upvarphi_{21}(\cdot,\lambda)  \chi_0( \sqrt{\xi} \cdot), \phi_2(\cdot) >  <   w_{21}^{+}(\lambda) \upvarphi_{22}(\cdot,\lambda)  \chi_0( \sqrt{\xi} \cdot), \psi_2(\cdot) >   d\xi \\
&:=\mathrm{I}+ \mathrm{II}+ \mathrm{III} + \mathrm{IV}.
\end{align*}
By Lemma \ref{lemma:omega_j-behavior}, \ref{lem:parity-of-omega-j} and the fact that $|\kappa(\lambda)|\simeq \sqrt{|\xi|}$  we have 
\begin{align*}
 \left| \frac{\kappa(\lambda) (\omega_{22}^{+}(\lambda))^2 }{d^{+}(\lambda) d^{-}(\lambda)}  \right|  \lesssim \sqrt{ |\xi|}  e^{-\frac{6}{\sqrt{2}} r_{\varepsilon}} , \quad
 \left| \frac{\kappa(\lambda)\omega_{22}^{+}(\lambda) \omega_{21}^{+}(\lambda) }{d^{+}(\lambda) d^{-}(\lambda)}  \right|   \lesssim \sqrt{|\xi|}e^{-\frac{3}{\sqrt{2}} r_{\varepsilon}}, \quad  \left| \frac{\kappa(\lambda)( \omega_{21}^{+}(\lambda))^2 }{d^{+}(\lambda) d^{-}(\lambda)}  \right|   \lesssim \sqrt{|\xi|}
\end{align*} 
Recall that $\upvarphi_1(r,\lambda)$ and $\upvarphi_1(r,\lambda)$ satisfies, 
\begin{align*}
    | \upvarphi_1(r,\lambda) | &\lesssim r^{\frac{3}{2}}, \qquad  \quad   | \upvarphi_2(r,\lambda) | \lesssim r^{\frac{3}{2}} \quad \text{ for small } r , \\ 
   | \upvarphi_1(r,\lambda) | &\lesssim e^{\frac{3}{\sqrt{2}}r}, \qquad   | \upvarphi_2(r,\lambda) | \lesssim r  \quad \text{ for large  } r \leq r_{\varepsilon} . 
\end{align*}

We first estimate the most delicate term which is the last one $\mathrm{IV}. $ We have 
\begin{align*}
| \mathrm{IV}|   &\lesssim  \int_{0}^{\delta_0}  \sqrt{\xi}  <  \upvarphi_{21}(\cdot,\lambda)  \chi_0( \sqrt{\xi} \cdot),  \phi_2(\cdot) >  <   \upvarphi_{22}(\cdot,\lambda)  \chi_0( \sqrt{\xi} \cdot), \psi_2(\cdot) > d\xi\\
&\lesssim  \int_{0}^{\delta_0}  \sqrt{\xi}  \left\|\upvarphi_{21}(\cdot,\lambda)  \chi_0( \sqrt{\xi} \cdot) \right\|_{L^2_r} 
\left\|\upvarphi_{22}(\cdot,\lambda)  \chi_0( \sqrt{\xi} \cdot) \right\|_{L^2_r}  d\xi \left\| \phi_2  \right\|_{L^2_r} 
\left\| \psi_2 \right\|_{L^2_r} \\
& \lesssim  \left\| \xi^{\frac{1}{4}}   <  \upvarphi_{21}(\cdot,\lambda)  \chi_0( \sqrt{\xi} \cdot),  \phi_2(\cdot) >    \right\|_{L^2_\xi(0,\delta_0)} 
\left\| \xi^{\frac{1}{4}}    <   \upvarphi_{22}(\cdot,\lambda)  \chi_0( \sqrt{\xi} \cdot), \psi_2(\cdot) >   \right\|_{L^2_\xi(0,\delta_0)}  
\end{align*}

Next, we estimate the first term separately $\left\| \xi^{\frac{1}{4}}   <  \upvarphi_{21}(\cdot,\lambda)  \chi_0( \sqrt{\xi} \cdot),  \phi_2(\cdot) >    \right\|_{L^2_\xi(0,\delta_0)} .$ Using a change of variable $\zeta=\sqrt{\xi}$, we obtain 
\begin{align*}
    \int_0^{\delta_0}  \sqrt{\xi} \left( \int_0^{\infty} \upvarphi_{21}(r,\lambda)  \chi_0( \sqrt{\xi} r)  \phi_2(r) dr \right)^2 d\xi = 2 \int_0^{\sqrt{\delta_0}}  \zeta^2 \left( \int_0^{\infty} \upvarphi_{21}(r,\lambda(\zeta^2))  \chi_0( \zeta r)  \phi_2(r) dr \right)^2 d\zeta 
\end{align*}
We denote by 
\begin{align*}
T v(\zeta)  :=  \int_0^{\infty}  \zeta  \upvarphi_{21}(r,\lambda(\zeta^2))  \chi_0( \zeta r) v(r) dr   , \quad 
T^{\ast} w (r):= \int_0^{\sqrt{\delta_0}}  \sigma   \upvarphi_{21}(r,\lambda(\sigma^2))   \chi_0( \sigma r)  w(\sigma) d\sigma.
\end{align*}
Therefore
\begin{align*}
TT^{\ast} w( \zeta  )&= \int_0^{\infty}  \zeta   \upvarphi_{21}(r,\lambda(\zeta^2))  \chi_0( \zeta r)  \int_0^{\sqrt{\delta_0}}  \sigma   \upvarphi_{21}(r,\lambda(\sigma^2))   \chi_0( \sigma r)  w(\sigma) d\sigma  dr     \\
&= \int_0^{\sqrt{\delta_0}} \int_0^{\infty}  \zeta \sigma  \upvarphi_{21}(r,\lambda(\zeta^2)) \upvarphi_{21}(r,\lambda(\sigma^2))    \chi_0( \zeta r)  \chi_0( \sigma r)  dr w(\sigma) d\sigma \\
&= \int_0^{\sqrt{\delta_0}}  K(\zeta,\sigma) w(\sigma) d\sigma .
\end{align*}
where, 
 \begin{align*}
     K(\zeta,\sigma)= \int_0^{\infty}  \zeta \sigma  \upvarphi_{21}(r,\lambda(\zeta^2)) \upvarphi_{21}(r,\lambda(\sigma^2))    \chi_0( \zeta r)  \chi_0( \sigma r)  dr
 \end{align*} 
Moreover, we have 
\begin{align*}
    \int_0^{\sqrt{\delta_0}} |K(\zeta,\sigma) | d\sigma &\lesssim  \int_0^{\sqrt{\delta_0}} 
    \int_0^1  \zeta \sigma  r^{3} \chi_0( \zeta r)  \chi_0( \sigma r)    dr d\sigma  + \int_0^{\sqrt{\delta_0}}    \int_1^\infty    \zeta \sigma  r^2 \chi_0( \zeta r)  \chi_0( \sigma r)      dr d\sigma \\
    &\lesssim 1+ \int_1^{\frac{1}{\zeta}}  \zeta r^2 \int_0^{\frac{1}{r}} \sigma d\sigma dr
    \lesssim 1.
\end{align*}
By symmetry, one can obtain  $\int_0^{\sqrt{\delta_0}} |K(\zeta,\sigma) | d\zeta \lesssim 1.$ Therefore, $\left\|K(\zeta,\sigma)\right\|_{L^1(d\sigma)}$  and $\left\|K(\zeta,\sigma)\right\|_{L^1(d\zeta)}$ are bounded uniformly, then by Schur's test we obtain $\left\| TT^{\ast} \phi_2 \right\|_{L^2_r} \lesssim \left\|\phi_2 \right\|_{L^2_r} . $ Thus, by duality we obtain 
\begin{align}
\label{L^2-boundsPhi21chi0}
  \left\| \xi^{\frac{1}{4}}   <  \upvarphi_{21}(\cdot,\lambda)  \chi_0( \sqrt{\xi} \cdot),  \phi_2(\cdot) >    \right\|_{L^2_\xi(0,\delta_0)} \lesssim   \left\|\phi_2 \right\|_{L^2_r}
\end{align}
Similarly, one can check that 
\begin{align*}
    \left\| \xi^{\frac{1}{4}}    <   \upvarphi_{22}(\cdot,\lambda)  \chi_0( \sqrt{\xi} \cdot), \psi_2(\cdot) >   \right\|_{L^2_\xi(0,\delta_0)}  \lesssim \left\|\psi_2 \right\|_{L^2_r}  
\end{align*}
which yields, 
\begin{align*}
    | \mathrm{IV}|   & \lesssim \left\|\phi_2 \right\|_{L^2_r}   \left\|\psi_2 \right\|_{L^2_r}  .
\end{align*}
Next, we estimate $\mathrm{I},$ and the other two terms $\mathrm{II} $ and $\mathrm{III}$ can be estimated analogously.  Using the Cauchy-Schwarz inequality and the asymptotics of $\upvarphi_1(r,\lambda)$ and $\upvarphi_1(r,\lambda)$ we obtain, 

\begin{align*}
&  \left| \int_{0}^{\delta_0}    \frac{\kappa(\lambda) (\omega_{22}^{+}(\lambda))^2}{d^{+}(\lambda) d^{-}(\lambda)}    <  \upvarphi_{11}(\cdot,\lambda)   \chi_0( \sqrt{\xi} \cdot), \phi_2(\cdot) > <  \upvarphi_{12}(\cdot,\lambda)   \chi_0( \sqrt{\xi} \cdot), \psi_2(\cdot) > d\xi
\right| \\& \lesssim \int_{0}^{\delta_0}  \sqrt{\xi} e^{-\frac{3}{\sqrt{2}} r_{\varepsilon}} \left\| \upvarphi_{11}(\cdot,\lambda) \chi_0( \sqrt{\xi} \cdot)  \right\|_{L^2_r} \left\| \phi_2 \right\|_{L^2_r} \left\| \upvarphi_{12}(\cdot,\lambda) \chi_0( \sqrt{\xi} \cdot)  \right\|_{L^2_r} \left\| \psi_2 \right\|_{L^2_r} d\xi\\
& \lesssim \int_{0}^{\delta_0}  \sqrt{\xi}  e^{-\frac{3}{\sqrt{2}} r_{\varepsilon}} \int_0^{\frac{\tvarepsilon}{ \sqrt{\xi}}} e^{2 \frac{3}{\sqrt{2}}r} dr d\xi \left\| \phi_2 \right\|_{L^2_r} \left\| \psi_2 \right\|_{L^2_r}
\\
& \lesssim \left\| \phi_2 \right\|_{L^2_r}  \left\| \psi_2 \right\|_{L^2_r}
\end{align*}
where we have used the asymptotics of $\upvarphi_1$ and the fact that $r_\varepsilon=\frac{\varepsilon}{\sqrt{\xi}},$ with $6 \tvarepsilon < \varepsilon.$  
Similarly, one can estimate all the other terms for $\mathcal{I}_0^2(\lambda) \mathcal{J}_0^2(\lambda).$ 
\end{proof}

\textbf{Case 2: } Contribution of $\mathcal{I}_1(\lambda)\mathcal{J}_1(\lambda).$
 \begin{claim} We have
 \label{Claim-step1-case2}
\begin{align*}
  & \int_{0}^{\delta_0}   \left|  \frac{\kappa(\lambda)}{d^{+}(\lambda) d^{-}(\lambda)} \mathcal{I}_1^1(\lambda) \mathcal{J}_1^k(\lambda) \right|  d\xi \lesssim   \left\| \phi_2 \right\|_{L^2_r}  \left\| \psi_2 \right\|_{L^2_r} , \quad \text{ for } \;  k=1,2. \\
   & \int_{0}^{ \delta_0}  \left|   \frac{\kappa(\lambda)}{d^{+}(\lambda) d^{-}(\lambda)} \mathcal{I}_1^2(\lambda) \mathcal{J}_1^k(\lambda) \right| d\xi \lesssim   \left\| \phi_2 \right\|_{L^2_r}  \left\| \psi_2 \right\|_{L^2_r}, \quad \text{ for } \;  k=1,2.
\end{align*}

\end{claim}
\begin{proof}
It suffices to establish the estimate for $\mathcal{I}_0^1(\lambda)\,\mathcal{J}_0^1(\lambda)$, as the proof of the other cases can be obtained in the same way.
Recall that, by Lemma \ref{Resol-interms-psi}, we have 
    \begin{align*}
     \Theta(r,\lambda)& =\gamma_1(\lambda) \left( \omega^{+}_{11}(\lambda) \sigma_3  \Psi_1^{+}(r,-\lambda)  +   \omega^{+}_{11}(-\lambda)  \Psi^{+}_1(r,\lambda) \right)
  \\
 & + \gamma_2(\lambda)   \sigma_3 \Psi_1^{+}(r,-\lambda)  + \gamma_3(\lambda)  \Psi^{+}_1(r,\lambda) + \gamma_4(\lambda)  \Psi^{+}_2(r,\lambda),
\end{align*}
where for small $|\xi|,$ \begin{align}
\label{gamma_j-behavior}
 \gamma_1(\lambda)   = O(\frac{1}{\sqrt{|\xi|} }e^{-\frac{3}{\sqrt{2}} r_{\varepsilon}} ), \quad \gamma_2(\lambda) = O(\frac{1}{\sqrt{|\xi|}}), \quad \gamma_3(\lambda)=  O(\frac{1}{\sqrt{|\xi|}} ) , \quad |\gamma_4(\lambda)| \lesssim e^{\frac{3}{\sqrt{2}} \frac{2 r_{\epsilon}}{5}}.
\end{align}
Hence, we have
\begin{align*}
       & \int_{0}^{\delta_0}    \frac{\kappa(\lambda)}{d^{+}(\lambda) d^{-}(\lambda)} \mathcal{I}_1^1(\lambda) \mathcal{J}_1^1(\lambda) d\xi \\
  & =  \int_{0}^{\delta_0}    \frac{\kappa(\lambda)}{d^{+}(\lambda) d^{-}(\lambda)} <\Theta_1(\cdot, \lambda) \chi_1( \sqrt{\xi} \cdot), \phi_2(\cdot) >  <\Theta_2(\cdot, \lambda) \chi_1( \sqrt{\xi} \cdot), \psi_2(\cdot) > d\xi   \\
&  =\int_{0}^{\delta_0}    \frac{\kappa(\lambda)}{d^{+}(\lambda) d^{-} (\lambda)}   < \gamma_1(\lambda) \left( \omega^{+}_{11}(\lambda)   \Psi_{11}^{+}(r,-\lambda)  +   \omega^{+}_{11}(-\lambda)  \Psi^{+}_{11}(r,\lambda) \right)  \chi_1( \sqrt{\xi} \cdot), \phi_2(\cdot) >  \\
& \qquad \qquad \qquad \qquad  \quad  \cdot < \gamma_1(\lambda) \left( -\omega^{+}_{11}(\lambda)   \Psi_{12}^{+}(r,-\lambda)  +   \omega^{+}_{11}(-\lambda)  \Psi^{+}_{12}(r,\lambda) \right)  \chi_1( \sqrt{\xi} \cdot), \psi_2(\cdot) > d\xi  \\
&+ \int_{0}^{\delta_0}    \frac{\kappa(\lambda)}{d^{+}(\lambda) d^{-} (\lambda)}  < \gamma_1(\lambda) \left( \omega^{+}_{11}(\lambda)  \Psi_{11}^{+}(r,-\lambda)  +   \omega^{+}_{11}(-\lambda)  \Psi^{+}_{11}(r,\lambda) \right)  \chi_1( \sqrt{\xi} \cdot), \phi_2(\cdot) >  \\
& \qquad  \qquad \qquad  \qquad \quad  \cdot < \left( - \gamma_2(\lambda)   \Psi_{12}^{+}(r,-\lambda)  + \gamma_3(\lambda)  \Psi^{+}_{12}(r,\lambda) + \gamma_4(\lambda)  \Psi^{+}_{22}(r,\lambda) \right) \chi_1( \sqrt{\xi} \cdot) ,  \psi_2(\cdot) > d\xi \\
&+ \int_{0}^{\delta_0}    \frac{\kappa(\lambda)}{d^{+}(\lambda) d^{-} (\lambda)}    < \left(  \gamma_2(\lambda)   \Psi_{11}^{+}(r,-\lambda)  + \gamma_3(\lambda)  \Psi^{+}_{11}(r,\lambda) + \gamma_4(\lambda)  \Psi^{+}_{21}(r,\lambda) \right) \chi_1( \sqrt{\xi} \cdot) ,  \phi_2(\cdot) >  \\
&  \qquad \qquad \qquad \qquad  \quad  \cdot  < \gamma_1(\lambda) \left( - \omega^{+}_{11}(\lambda)   \Psi_{12}^{+}(r,-\lambda)  +   \omega^{+}_{11}(-\lambda)  \Psi^{+}_{12}(r,\lambda) \right)  \chi_1( \sqrt{\xi} \cdot), \psi_2(\cdot) > d\xi   \\
& + \int_{0}^{\delta_0}   \frac{\kappa(\lambda)}{d^{+}(\lambda) d^{-} (\lambda)}  < \left(  \gamma_2(\lambda)    \Psi_{11}^{+}(r,-\lambda)  + \gamma_3(\lambda)  \Psi^{+}_{11}(r,\lambda) + \gamma_4(\lambda)  \Psi^{+}_{21}(r,\lambda) \right) \chi_1( \sqrt{\xi} \cdot) ,  \phi_2(\cdot) >  \\
& \qquad \qquad \qquad \qquad  \quad \cdot < \left( - \gamma_2(\lambda)  \Psi_{12}^{+}(r,-\lambda)  + \gamma_3(\lambda)  \Psi^{+}_{12}(r,\lambda) + \gamma_4(\lambda)  \Psi^{+}_{22}(r,\lambda) \right) \chi_1( \sqrt{\xi} \cdot) ,  \psi_2(\cdot) >  d\xi \\
&:= \mathrm{I}+ \mathrm{II}+ \mathrm{III} + \mathrm{IV}.
\end{align*}
By Lemma \ref{lemma:omega_j-behavior}, Proposition \ref{jump-resol} and  \eqref{gamma_j-behavior}, we have 
\begin{align*}
     \left| \frac{\kappa(\lambda) \gamma_1(\lambda)^2 }{d^{+}(\lambda) d^{-}(\lambda)}  \right|  & \lesssim \frac{1}{\sqrt{| \xi|}}e^{-\frac{6}{\sqrt{2}} r_{\varepsilon}}  , \quad  \left| \frac{\kappa(\lambda) \gamma_2(\lambda)^2 }{d^{+}(\lambda) d^{-}(\lambda)}  \right|  \lesssim \frac{1}{\sqrt{|\xi|}} ,  \quad  \left| \frac{\kappa(\lambda) \gamma_3(\lambda)^2 }{d^{+}(\lambda) d^{-}(\lambda)}  \right|  \lesssim \frac{1}{\sqrt{|\xi|}}, \quad \\  \left| \frac{\kappa(\lambda) \gamma_4(\lambda)^2 }{d^{+}(\lambda) d^{-}(\lambda)}  \right|  &\lesssim \sqrt{|\xi|}  e^{\frac{3}{\sqrt{2}} \frac{4 r_{\epsilon}}{5}}  , \quad  
         \left| \frac{\kappa(\lambda) \gamma_1(\lambda) \gamma_2(\lambda) }{d^{+}(\lambda) d^{-}(\lambda)}  \right|  \lesssim \frac{1}{\sqrt{|\xi|}}e^{-\frac{3}{\sqrt{2}} r_{\varepsilon}}  , \quad   \left| \frac{\kappa(\lambda) \gamma_1(\lambda) \gamma_3(\lambda) }{d^{+}(\lambda) d^{-}(\lambda)}  \right|  \lesssim \frac{1}{\sqrt{|\xi|}}e^{-\frac{3}{\sqrt{2}} r_{\varepsilon}}  , 
         \quad  \\
         \left| \frac{\kappa(\lambda) \gamma_1(\lambda) \gamma_4(\lambda) }{d^{+}(\lambda) d^{-}(\lambda)}  \right|  &\lesssim e^{-\frac{3}{\sqrt{2}} \frac{3}{5} r_{\varepsilon}} , \quad 
      \left| \frac{\kappa(\lambda) \gamma_2(\lambda) \gamma_3(\lambda) }{d^{+}(\lambda) d^{-}(\lambda)}  \right|  \lesssim \frac{1}{\sqrt{|\xi|}}, \quad   \left| \frac{\kappa(\lambda) \gamma_2(\lambda) \gamma_4(\lambda) }{d^{+}(\lambda) d^{-}(\lambda)}  \right|  \lesssim  e^{\frac{3}{\sqrt{2}} \frac{2r_{\varepsilon}}{5}}  , \quad  \\
      \left| \frac{\kappa(\lambda) \gamma_3(\lambda) \gamma_4(\lambda) }{d^{+}(\lambda) d^{-}(\lambda)}  \right|  & \lesssim  e^{\frac{3}{\sqrt{2}} \frac{2r_{\varepsilon}}{5}} .
\end{align*}

Notice that the main terms are the one that has $\left| \frac{\kappa(\lambda) \gamma_i(\lambda) \gamma_j(\lambda) }{d^{+}(\lambda) d^{-}(\lambda)}  \right| \lesssim \frac{1}{\sqrt{|\xi|}}, $ and $\left| \frac{\kappa(\lambda) \gamma_i(\lambda) \gamma_4(\lambda) }{d^{+}(\lambda) d^{-}(\lambda)}  \right| \lesssim e^{\frac{3}{\sqrt{2}} \frac{4r_{\varepsilon}}{5}},   $ for some $i,j \in \{1,2,3,4 \}.$ The second term involves $\gamma_4,$ and therefore it includes an exponential decay factor given by $\Psi^{+}_{2}(r,\lambda).$  We will mainly focus on estimating the first term with no decaying factor. However, we will first estimate $ \mathrm{I}$ and we omit the proof for all similar terms that include a decay factor.  \\

First, we estimate for first term $\mathrm{I}.$ By Lemma \ref{L^2-bounds-of-Psi}, we have the following decomposition for  $\Psi_1(r,\lambda)$
 \begin{align}
 \label{decom-of-Psi_1}
    \Psi_1(r,\lambda)&=    \Uppsi^{+}_{(\infty,1)} (r,\lambda) (1 + O( \frac{1}{\sqrt{|\xi|}}   e^{-2 r_{\infty}} )) +  \Upupsilon_1(r,\lambda,\Psi^{+}_{1} (r,\lambda))
\end{align}
where,
\begin{align}
\left\|  \Upupsilon_1(\cdot,\lambda,\Psi^{+}_{1} (\cdot,\lambda))  \chi_1( \sqrt{\xi} \cdot) \right\|_{L^2_r(0,\infty)}  \lesssim \frac{1}{\sqrt{|\xi|}}   e^{- 2\frac{9\tvarepsilon}{10\sqrt{\xi}}} 
\end{align}
Thus, we only focus on the contribution of $    \Uppsi^{+}_{(\infty,1)} (r,\lambda) = \begin{pmatrix}
      e^{i k_1(\lambda)r} \\ c_1(\lambda)  e^{i k_1(\lambda)r}
  \end{pmatrix}  .$
By Lemma \ref{lemma:omega_j-behavior}, we have $|\omega_{11}^{+}(\pm \lambda)|  \lesssim  e^{\frac{3}{\sqrt{2}} \frac{r_{\varepsilon}}{5}}.$ Hence, by Cauchy-Schwarz inequality and by Lemma \ref{L^2-L^2-boud-Psi-1}, we obtain 
\begin{align*}
&\bigg|  \int_{0}^{\delta_0}   \frac{1}{\sqrt{|\xi|}}e^{-\frac{6}{\sqrt{2}} r_{\varepsilon}}   <  \left( \omega^{+}_{11}(\lambda)    e^{-i k_1(\lambda)r}  +   \omega^{+}_{11}(-\lambda)  e^{i k_1(\lambda)r} \right)  \chi_1( \sqrt{\xi} \cdot), \phi_2(\cdot) >  \\
& \qquad \qquad \qquad   \cdot  <  \left(- \omega^{+}_{11}(\lambda)   c_1(-\lambda)  e^{-i k_1(\lambda)r}  +   \omega^{+}_{11}(-\lambda)  c_1(\lambda)  e^{i k_1(\lambda)r}  \right)  \chi_1( \sqrt{\xi} \cdot), \psi_2(\cdot) > \bigg| d\xi \\
& \qquad \qquad \qquad  \lesssim  \left\| \phi_2 \right\|_{L^2_r}  \left\| \psi_2 \right\|_{L^2_r} 
\end{align*}
The estimate for the $ \Upupsilon_1$ can be obtained similarly using Cauchy-Schwarz inequality and Lemma \ref{L^2-bounds-of-Psi}. Therefore, we obtain the desired estimate for $ \mathrm{I}. $ Note that, $ \mathrm{II}$ and $ \mathrm{III}$ mainly contain an exponential decaying factor similar 
$ \mathrm{I}, $ and using the the exponential decay of $\Psi_2^{+}$ for one that contains $\gamma_1 \gamma_4,$ one can obtain the desired estimate and we omit the details.  \\

Next, we estimate $ \mathrm{IV}$ and will focus on the first term and all the other one can be estimate similarly. Using Cauchy-Schwarz inequality and the fact that $\left| \frac{\kappa(\lambda) \gamma_2(\lambda)^2 }{d^{+}(\lambda) d^{-}(\lambda)}  \right|  \lesssim \frac{1}{\sqrt{|\xi|}} $ together with Lemma \ref{L^2-L^2-boud-Psi-1}, we have
\begin{align*}
& \left|  \int_{0}^{\delta_0}   \frac{\kappa(\lambda)\gamma_2(\lambda)^2}{d^{+}(\lambda) d^{-} (\lambda)}  <  \Psi_{11}^{+}(r,-\lambda)  \chi_1( \sqrt{\xi} \cdot) ,  \phi_2(\cdot) >   <  \Psi_{12}^{+}(r,-\lambda)  \chi_1( \sqrt{\xi} \cdot) ,  \psi_2(\cdot) >  d\lambda \right| \\
& \lesssim   \int_{0}^{\delta_0}    \frac{1}{\sqrt{|\xi|}}  \left| <  \Psi_{11}^{+}(r,-\lambda)  \chi_1( \sqrt{\xi} \cdot) ,  \phi_2(\cdot) >   <  \Psi_{12}^{+}(r,-\lambda)  \chi_1( \sqrt{\xi} \cdot) ,  \psi_2(\cdot) >  \right| d\xi \\ 
& \lesssim  \left( \int_0^{\delta_0}  \frac{1}{\sqrt{|\xi|}}  \left| \int_0^{\infty} \Psi_{11}^{+}(r, -\lambda) \chi_1( \sqrt{\xi} r ) \phi_2(r) dr \right|^2 d \xi \right)^{\frac{1}{2}} \\
& \times \left( \int_0^{\delta_0}  \frac{1}{\sqrt{|\xi|}}  \left| \int_0^{\infty} \Psi_{12}^{+}(r, -\lambda) \chi_1( \sqrt{\xi} r ) \psi_2(r) dr \right|^2 d \xi \right)^{\frac{1}{2}} \\
& \lesssim \left\| \phi_2 \right\|_{L^2_r}  \left\| \psi_2 \right\|_{L^2_r} 
\end{align*}
 Using similar argument, one can estimate the other terms in  $ \mathrm{IV}, $ and this concludes the estimates for 
$\mathcal{I}_1^1(\lambda) \mathcal{J}_1^1(\lambda)$ and similarly for $\mathcal{I}_1^2(\lambda) \mathcal{J}_1^2(\lambda).$
\end{proof}

\textbf{Case 3: } Contribution of $\mathcal{I}_0(\lambda)\mathcal{J}_1(\lambda).$
 \begin{claim}  \label{claim:case3-nonresonance-small-lambda}
We have
\begin{align*}
  & \int_{0}^{\delta_0} \left|   \frac{\kappa(\lambda)}{d^{+}(\lambda) d^{-}(\lambda)} \mathcal{I}_0^1(\lambda) \mathcal{J}_1^k(\lambda)  \right| 
 d\xi \lesssim   \left\| \phi_2 \right\|_{L^2_r}  \left\| \psi_2 \right\|_{L^2_r} , \quad \text{ for } \;  k=1,2.    \\
   & \int_{0}^{ \delta_0}  \left|   \frac{\kappa(\lambda)}{d^{+}(\lambda) d^{-}(\lambda)} \mathcal{I}_0^2(\lambda) \mathcal{J}_1^k(\lambda) \right| 
 d\xi \lesssim   \left\| \phi_2 \right\|_{L^2_r}  \left\| \psi_2 \right\|_{L^2_r}, \quad \text{ for } \;  k=1,2. 
\end{align*} 
\end{claim}

\begin{proof} By Lemma \ref{Resol-interms-phi} and \ref{Resol-interms-psi}, we have 
    \begin{align*}
    & \int_{0}^{\delta_0}    \frac{\kappa(\lambda)}{d^{+}(\lambda) d^{-}(\lambda)} \mathcal{I}_0^1(\lambda) \mathcal{J}_1^1(\lambda) d\xi \\
   & =  \int_{0}^{\delta_0}    \frac{\kappa(\lambda)}{d^{+}(\lambda) d^{-}(\lambda)} <\Theta_1(\cdot, \lambda) \chi_0( \sqrt{\xi} \cdot), \phi_2(\cdot) >  <\Theta_2(\cdot, \lambda) \chi_1( \sqrt{\xi} \cdot), \psi_2(\cdot) >d\xi  \\
&= \int_{0}^{\delta_0}   \frac{\kappa(\lambda)}{d^{+}(\lambda) d^{-}(\lambda)}  \bigg( < (\omega_{22}^{+} (\lambda)\upvarphi_{11}(\cdot,\lambda) - \omega_{21}^{+}(\lambda) \upvarphi_{21}(\cdot,\lambda) )  \chi_0( \sqrt{\xi} \cdot), \phi_2(\cdot) >  \bigg) \\
& \qquad \qquad \qquad  \times \bigg( < \gamma_1(\lambda) \left(- \omega^{+}_{11}(\lambda)   \Psi_{12}^{+}(r,-\lambda)  +   \omega^{+}_{11}(-\lambda)  \Psi^{+}_{12}(r,\lambda) \right)  \chi_1( \sqrt{\xi} \cdot), \psi_2(\cdot) > \\
   & \qquad \qquad \qquad \quad     + < \left(  - \gamma_2(\lambda)    \Psi_{12}^{+}(r,-\lambda)  + \gamma_3(\lambda)  \Psi^{+}_{12}(r,\lambda) + \gamma_4(\lambda)  \Psi^{+}_{22}(r,\lambda) \right) \chi_1( \sqrt{\xi} \cdot) ,  \psi_2(\cdot) > 
   \bigg) d\xi
\end{align*}
By Lemma \ref{lemma:omega_j-behavior},\ref{lem:parity-of-omega-j}, \ref{Resol-interms-psi} and the fact that $|\kappa(\lambda)|\simeq \sqrt{|\xi|}$ (see Proposition \ref{jump-resol}), we have
\begin{align*}
     \left| \frac{\kappa(\lambda)\omega_{22}^{+}(\lambda)  \omega_{11}^{+}(\lambda)  \gamma_1(\lambda) }{d^{+}(\lambda) d^{-}(\lambda)}  \right|  & \lesssim e^{-\frac{3}{\sqrt{2}} \frac{9}{5} r_{\varepsilon}}  , \quad   \left| \frac{\kappa(\lambda)\omega_{21}^{+}(\lambda)  \omega_{11}^{+}(\lambda)  \gamma_1(\lambda) }{d^{+}(\lambda) d^{-}(\lambda)}  \right|   \lesssim e^{-\frac{3}{\sqrt{2}} \frac{4}{5} r_{\varepsilon}} , \\
       \left| \frac{\kappa(\lambda)\omega_{22}^{+}(\lambda)    \gamma_k(\lambda) }{d^{+}(\lambda) d^{-}(\lambda)}  \right|  & \lesssim e^{-\frac{3}{\sqrt{2}}  r_{\varepsilon}}  , \qquad 
        \left| \frac{\kappa(\lambda)\omega_{21}^{+}(\lambda)  \gamma_k(\lambda) }{d^{+}(\lambda) d^{-}(\lambda)}  \right|   \lesssim 1 , \quad \text{ for } k=2,3, \\    
         \left| \frac{\kappa(\lambda)\omega_{22}^{+}(\lambda)    \gamma_4(\lambda) }{d^{+}(\lambda) d^{-}(\lambda)}  \right|  & \lesssim \sqrt{|\xi|}  e^{-\frac{3}{\sqrt{2}} \frac{3}{5}  r_{\varepsilon}}  , \quad  \left| \frac{\kappa(\lambda)\omega_{21}^{+}(\lambda)    \gamma_4(\lambda) }{d^{+}(\lambda) d^{-}(\lambda)}  \right|  \lesssim \sqrt{| \xi|} e^{\frac{3}{\sqrt{2}} \frac{2 r_{\epsilon}}{5}}.
\end{align*}
Notice that the main terms are the one that has $\left| \frac{\kappa(\lambda) \omega_{21}^{+}(\lambda)  \gamma_k(\lambda)}{d^{+}(\lambda) d^{-}(\lambda)}  \right| \lesssim 1, $ for $k=2,3$ and $\left| \frac{\kappa(\lambda)\omega_{21}^{+}(\lambda)    \gamma_4(\lambda) }{d^{+}(\lambda) d^{-}(\lambda)}  \right|  \lesssim \sqrt{| \xi |} e^{\frac{3}{\sqrt{2}} \frac{2 r_{\epsilon}}{5}}   .$ The second term involves $\gamma_4,$ and therefore it includes an exponential decay factor given by $\Psi^{+}_{2}(r,\lambda).$  We will mainly focus on estimating the first term with no decaying factor. We only focus on the first term for $k=2,$ the remaining cases follow in the same way.
  \begin{align*}
&  \int_{0}^{\delta_0}  \left| \xi^{\frac{1}{4}} < \upvarphi_{21}(\cdot,\lambda)   \chi_0( \sqrt{\xi} \cdot), \phi_2(\cdot) >   \frac{1}{\xi^{\frac{1}{4}} } <   \Psi_{12}^{+}(\cdot,-\lambda) \chi_1( \sqrt{\xi} \cdot) ,  \psi_2(\cdot) >  \right|  d\xi  
\\
&\lesssim  \left(\int_{0}^{\delta_0}  \sqrt{\xi} \left| \int_0^{\infty}  \upvarphi_{21}(r,\lambda)\chi_0( \sqrt{\xi} r)  \phi_2(r)dr  \right|^2 d \xi \right)^{\frac{1}{2}}
\left(\int_{0}^{\delta_0} \frac{1}{\sqrt{\xi}} \left|  \Psi_{12}^{+}(r,-\lambda) \chi_1( \sqrt{\xi} r)   \psi_2(r)  dr \right|^2 d \xi \right)^{\frac{1}{2}}
\end{align*}
By applying duality together with Schur's test or using \eqref{L^2-boundsPhi21chi0} from Case 1, we obtain  \begin{align*}
   \left(\int_{0}^{\delta_0}  \sqrt{\xi} \left| \int_0^{\infty}  \upvarphi_{21}(r,\lambda)\chi_0( \sqrt{\xi} r)  \phi_2(r)dr  \right|^2 d \xi \right)^{\frac{1}{2}} \lesssim \left\|\phi_2\right\|_{L^2_r}
\end{align*}
Moreover, by Lemma \ref{L^2-L^2-boud-Psi-1} we have
\begin{align*}
    \left(\int_{0}^{\delta_0} \frac{1}{\sqrt{\xi}} \left|  \Psi_{12}^{+}(r,-\lambda) \chi_1( \sqrt{\xi} r)   \psi_2(r)  dr \right|^2 d \xi \right)^{\frac{1}{2}} \lesssim \left\|\psi_2 \right\|_{L^2_r}. 
\end{align*}

This concludes the proof of Claim \ref{claim:case3-nonresonance-small-lambda}.

\end{proof}

\textbf{Step 2: Large $\xi \in (\Lambda_0,\infty).$}\\
\textbf{Case 1:} Contribution of $\mathcal{I}_0(\lambda)\mathcal{J}_0(\lambda).$
\begin{claim} We have
\begin{align*}
  & \int_{\Lambda_0}^{\infty}  \left|   \frac{\kappa(\lambda)}{d^{+}(\lambda) d^{-}(\lambda)} \mathcal{I}_0^1(\lambda) \mathcal{J}_0^k(\lambda) \right|  d\xi \lesssim   \left\| \phi_2 \right\|_{L^2_r}  \left\| \psi_2 \right\|_{L^2_r}, \quad \text{ for } \;  k=1,2. \\
   & \int_{\Lambda_0}^{\infty} \left|   \frac{\kappa(\lambda)}{d^{+}(\lambda) d^{-}(\lambda)} \mathcal{I}_0^2(\lambda) \mathcal{J}_0^k(\lambda) \right| d\xi \lesssim   \left\| \phi_2 \right\|_{L^2_r}  \left\| \psi_2 \right\|_{L^2_r}, \quad \text{ for } \;  k=1,2. 
\end{align*}
\end{claim}

\begin{proof}
Similarly to step 1, we focus on proving the estimate for $\mathcal{I}_0^1(\lambda)\mathcal{J}_0^1(\lambda)$, while the bounds for the other terms follow by analogous arguments. We have 
    \begin{align*}
    & \int_{\Lambda_0}^{\infty}    \frac{\kappa(\lambda)}{d^{+}(\lambda) d^{-}(\lambda)} \mathcal{I}_0^1(\lambda) \mathcal{J}_0^1(\lambda) d\xi  \\
 &= \int_{\Lambda_0}^{\infty}   \frac{\kappa(\lambda)}{d^{+}(\lambda) d^{-}(\lambda)} < \omega_{22}^{+} (\lambda)\upvarphi_{11}(\cdot,\lambda)   \chi_0( \sqrt{\xi} \cdot), \phi_2(\cdot) > < \omega_{22}^{+}(\lambda) \upvarphi_{12}(\cdot,\lambda)   \chi_0( \sqrt{\xi} \cdot), \psi_2(\cdot) > d\xi \\
& + \int_{\Lambda_0}^{\infty}   \frac{\kappa(\lambda)}{d^{+}(\lambda) d^{-}(\lambda)} < \omega_{22}^{+} (\lambda)\upvarphi_{11}(\cdot,\lambda)   \chi_0( \sqrt{\xi} \cdot), \phi_2(\cdot) > <   w_{21}^{+}(\lambda) \upvarphi_{22}(\cdot,\lambda)  \chi_0( \sqrt{\xi} \cdot), \psi_2(\cdot) >  d\xi \\
& + \int_{\Lambda_0}^{\infty}   \frac{\kappa(\lambda)}{d^{+}(\lambda) d^{-}(\lambda)} <   w_{21}^{+}(\lambda) \upvarphi_{21}(\cdot,\lambda)  \chi_0( \sqrt{\xi} \cdot), \phi_2(\cdot) >  < \omega_{22}^{+}(\lambda) \upvarphi_{12}(\cdot,\lambda)   \chi_0( \sqrt{\xi} \cdot), \psi_2(\cdot) > d\xi \\
&+  \int_{\Lambda_0}^{\infty}   \frac{\kappa(\lambda)}{d^{+}(\lambda) d^{-}(\lambda)}  <   w_{21}^{+}(\lambda) \upvarphi_{21}(\cdot,\lambda)  \chi_0( \sqrt{\xi} \cdot), \phi_2(\cdot) >  <   w_{21}^{+}(\lambda) \upvarphi_{22}(\cdot,\lambda)  \chi_0( \sqrt{\xi} \cdot), \psi_2(\cdot) > d\xi \\
&:=\mathrm{I}+ \mathrm{II}+ \mathrm{III} + \mathrm{IV}
\end{align*}
By Lemma \ref{lem:omega_j-behavior-large-xi}, \ref{lem:parity-of-omega-j}, and the fact that $|\kappa(\lambda)|\simeq \sqrt{|\xi|},$  we have
\begin{align*}
 \left| \frac{\kappa(\lambda) (\omega_{22}^{+}(\lambda))^2(\lambda) }{d^{+}(\lambda) d^{-}(\lambda)}  \right|  \lesssim |\xi| , \quad
 \left| \frac{\kappa(\lambda)\omega_{22}^{+}(\lambda) \omega_{21}^{+}(\lambda) }{d^{+}(\lambda) d^{-}(\lambda)}  \right|   \lesssim |\xi| , \quad  \left| \frac{\kappa(\lambda)( \omega_{21}^{+}(\lambda))^2 }{d^{+}(\lambda) d^{-}(\lambda)}  \right|   \lesssim|\xi|.
\end{align*} 
Recall that $\upvarphi_1(r,\lambda)$ and $\upvarphi_2(r,\lambda)$ satisfies, 
\begin{align*}
    | \upvarphi_1(r,\lambda) | &\lesssim r^{\frac{3}{2}}, \qquad  \quad   | \upvarphi_2(r,\lambda) | \lesssim r^{\frac{3}{2}} \quad \text{ for small } r , \\ 
   | \upvarphi_1(r,\lambda) | &\lesssim e^{\frac{3}{\sqrt{2}}r}, \qquad   | \upvarphi_2(r,\lambda) | \lesssim r  \quad \text{ for large  } r \leq r_{\varepsilon} . 
\end{align*}

We will only estimate the first term  $\mathrm{I}, $ and all other terms can estimated similarly.  We have 
\begin{align*}
| \mathrm{I}|   &\lesssim  \int_{\Lambda_0}^{\infty} \left|  |\xi|  <  \upvarphi_{11}(\cdot,\lambda)  \chi_0( \sqrt{\xi} \cdot),  \phi_2(\cdot) >  <   \upvarphi_{12}(\cdot,\lambda)  \chi_0( \sqrt{\xi} \cdot), \psi_2(\cdot) > \right| d\xi \\
& \lesssim  \left\| \xi^{\frac{1}{2}}   <  \upvarphi_{11}(\cdot,\lambda)  \chi_0( \sqrt{\xi} \cdot),  \phi_2(\cdot) >    \right\|_{L^2_\xi(0,\delta_0)} 
\left\| \xi^{\frac{1}{2}}    <   \upvarphi_{12}(\cdot,\lambda)  \chi_0( \sqrt{\xi} \cdot), \psi_2(\cdot) >   \right\|_{L^2_\xi(0,\delta_0)}  
\end{align*}
Next, we estimate the first term separately $\left\| \xi^{\frac{1}{2}}   <  \upvarphi_{11}(\cdot,\lambda)  \chi_0( \sqrt{\xi} \cdot),  \phi_2(\cdot) >    \right\|_{L^2_\xi(0,\delta_0)} .$ Using a change of variable $\zeta=\sqrt{\xi}$, we obtain 
\begin{align*}
    \int_{\Lambda_0}^{\infty} \xi \left| \int_0^{\infty} \upvarphi_{11}(r,\lambda)  \chi_0( \sqrt{\xi} r)  \phi_2(r) dr \right|^2 d\xi = 2 \int_{\sqrt{\Lambda_0}}^{\infty} \zeta^3 \left| \int_0^{\infty} \upvarphi_{11}(r,\lambda(\zeta^2))  \chi_0( \zeta r)  \phi_2(r) dr \right|^2 d\zeta 
\end{align*}
We denote by 
\begin{align*}
T v(\zeta)  :=  \int_0^{\infty}  \zeta^{\frac{3}{2}}  \upvarphi_{11}(r,\lambda(\zeta^2))  \chi_0( \zeta r) v(r) dr   , \quad 
T^{\ast} w (r):=\int_{\sqrt{\Lambda_0}}^{\infty}  \sigma^{\frac{3}{2}}   \upvarphi_{11}(r,\lambda(\sigma^2))   \chi_0( \sigma r)  w(\sigma) d\sigma.
\end{align*}
Thus, we have
\begin{align*}
TT^{\ast} w( \zeta  )&= \int_0^{\infty}  \zeta^{\frac{3}{2}}   \upvarphi_{11}(r,\lambda(\zeta^2))  \chi_0( \zeta r)  \int_{\sqrt{\Lambda_0}}^{\infty}  \sigma^{\frac{3}{2}}   \upvarphi_{11}(r,\lambda(\sigma^2))   \chi_0( \sigma r)  w(\sigma) d\sigma  dr     \\
&= \int_{\sqrt{\Lambda_0}}^{\infty} \int_0^{\infty}  \zeta^{\frac{3}{2}} \sigma^{\frac{3}{2}}  \upvarphi_{11}(r,\lambda(\zeta^2)) \upvarphi_{11}(r,\lambda(\sigma^2))    \chi_0( \zeta r)  \chi_0( \sigma r)  dr w(\sigma) d\sigma \\
&= \int_{\sqrt{\Lambda_0}}^{\infty}  K(\zeta,\sigma) w(\sigma) d\sigma
\end{align*}
where, 
 \begin{align*}
     K(\zeta,\sigma)= \int_0^{\infty}  \zeta^{\frac{3}{2}} \sigma^{\frac{3}{2}}  \upvarphi_{11}(r,\lambda(\zeta^2)) \upvarphi_{11}(r,\lambda(\sigma^2))    \chi_0( \zeta r)  \chi_0( \sigma r)  dr.
 \end{align*} 
Moreover, we have 
\begin{align*}
   \int_{\sqrt{\Lambda_0}}^{\infty} |K(\zeta,\sigma) | d\sigma &\lesssim  \int_{\sqrt{\Lambda_0}}^{\infty}  
    \int_0^{\min(\frac{1}{\zeta},\frac{1}{\sigma})}  \zeta^{\frac{3}{2}} \sigma^{\frac{3}{2}}  r^{3} \chi_0( \zeta r)  \chi_0( \sigma r)    dr d\sigma  \\
    &\lesssim  \int_{\sqrt{\Lambda_0}}^{\zeta} \zeta^{\frac{3}{2}} \sigma^{\frac{3}{2}} \int_0^{\frac{1}{\zeta}} r^{3} dr d\sigma+ \int_{\zeta}^{\infty} \zeta^{\frac{3}{2}} \sigma^{\frac{3}{2}} \int_0^{\frac{1}{\sigma}} r^{3} dr d\sigma \lesssim 1 .
\end{align*}
By symmetry, one can obtain  $\int_{\sqrt{\Lambda_0}}^{\infty} |K(\zeta,\sigma) | d\zeta \lesssim 1.$ Therefore, $\left\|K(\zeta,\sigma)\right\|_{L^1(d\sigma)}$  and $\left\|K(\zeta,\sigma)\right\|_{L^1(d\zeta)}$ are bounded uniformly, then by Schur's test we obtain $\left\| TT^{\ast} \phi_2 \right\|_{L^2_r} \lesssim \left\|\phi_2 \right\|_{L^2_r} . $ Thus, by duality we obtain 
\begin{align}
\label{L^2-boundsPhi21chi0-large-lambda}
  \left\| \xi^{\frac{1}{2}}   <  \upvarphi_{11}(\cdot,\lambda)  \chi_0( \sqrt{\xi} \cdot),  \phi_2(\cdot) >    \right\|_{L^2_\xi(\Lambda_0,\infty)} \lesssim   \left\|\phi_2 \right\|_{L^2_r}
\end{align}
Similarly, one can check that 
\begin{align*}
    \left\| \xi^{\frac{1}{2}}    <   \upvarphi_{12}(\cdot,\lambda)  \chi_0( \sqrt{\xi} \cdot), \psi_2(\cdot) >   \right\|_{L^2_\xi(\Lambda_0,\infty)}  \lesssim \left\|\psi_2 \right\|_{L^2_r}  .
\end{align*}
which yields, 
\begin{align*}
    | \mathrm{I}|   & \lesssim \left\|\phi_2 \right\|_{L^2_r}   \left\|\psi_2 \right\|_{L^2_r}  
\end{align*}
The other three terms $\mathrm{II},\mathrm{III}$ and $\mathrm{IV}$ can be estimated analogously. Similarly, one can estimate all the other terms for $\mathcal{I}_0^2(\lambda) \mathcal{J}_0^2(\lambda).$ 
\end{proof}
\textbf{Case 2: } Contribution of $\mathcal{I}_1(\lambda)\mathcal{J}_1(\lambda).$
 \begin{claim} We have
\begin{align*}
  & \int_{\Lambda_0}^{\infty}   \left|  \frac{\kappa(\lambda)}{d^{+}(\lambda) d^{-}(\lambda)} \mathcal{I}_1^1(\lambda) \mathcal{J}_1^k(\lambda) \right| d\xi \lesssim   \left\| \phi_2 \right\|_{L^2_r}  \left\| \psi_2 \right\|_{L^2_r}, \quad \text{ for } \;  k=1,2.  \\
   & \int_{\Lambda_0}^{\infty}     \left|  \frac{\kappa(\lambda)}{d^{+}(\lambda) d^{-}(\lambda)} \mathcal{I}_1^2(\lambda) \mathcal{J}_1^k(\lambda) \right| d\xi  \lesssim   \left\| \phi_2 \right\|_{L^2_r}  \left\| \psi_2 \right\|_{L^2_r}, \quad \text{ for } \;  k=1,2. 
\end{align*}

\end{claim}
\begin{proof}
We only focus on the estimate for $\mathcal{I}_1^1(\lambda)\,\mathcal{J}_1^1(\lambda)$, since the proofs of the other cases proceed in the same way. By Lemma \ref{Resol-interms-psi}, we have
\begin{align*}
       & \int_{0}^{\delta_0}    \frac{\kappa(\lambda)}{d^{+}(\lambda) d^{-}(\lambda)} \mathcal{I}_1^1(\lambda) \mathcal{J}_1^1(\lambda) d\xi \\
&  = \int_{\Lambda_0}^{\infty}   \frac{\kappa(\lambda)}{d^{+}(\lambda) d^{-} (\lambda)}   < \gamma_1(\lambda) \left( \omega^{+}_{11}(\lambda)   \Psi_{11}^{+}(r,-\lambda)  +   \omega^{+}_{11}(-\lambda)  \Psi^{+}_{11}(r,\lambda) \right)  \chi_1( \sqrt{\xi} \cdot), \phi_2(\cdot) >  \\
& \qquad \qquad \qquad \qquad  \quad  \cdot < \gamma_1(\lambda) \left( -\omega^{+}_{11}(\lambda)   \Psi_{12}^{+}(r,-\lambda)  +   \omega^{+}_{11}(-\lambda)  \Psi^{+}_{12}(r,\lambda) \right)  \chi_1( \sqrt{\xi} \cdot), \psi_2(\cdot) > d\xi  \\
&+ \int_{\Lambda_0}^{\infty}     \frac{\kappa(\lambda)}{d^{+}(\lambda) d^{-} (\lambda)}  < \gamma_1(\lambda) \left( \omega^{+}_{11}(\lambda)  \Psi_{11}^{+}(r,-\lambda)  +   \omega^{+}_{11}(-\lambda)  \Psi^{+}_{11}(r,\lambda) \right)  \chi_1( \sqrt{\xi} \cdot), \phi_2(\cdot) >  \\
& \qquad  \qquad \qquad  \qquad \quad  \cdot < \left( - \gamma_2(\lambda)   \Psi_{12}^{+}(r,-\lambda)  + \gamma_3(\lambda)  \Psi^{+}_{12}(r,\lambda) + \gamma_4(\lambda)  \Psi^{+}_{22}(r,\lambda) \right) \chi_1( \sqrt{\xi} \cdot) ,  \psi_2(\cdot) > d\xi \\
&+ \int_{\Lambda_0}^{\infty}     \frac{\kappa(\lambda)}{d^{+}(\lambda) d^{-} (\lambda)}    < \left(  \gamma_2(\lambda)   \Psi_{11}^{+}(r,-\lambda)  + \gamma_3(\lambda)  \Psi^{+}_{11}(r,\lambda) + \gamma_4(\lambda)  \Psi^{+}_{21}(r,\lambda) \right) \chi_1( \sqrt{\xi} \cdot) ,  \phi_2(\cdot) >  \\
&  \qquad \qquad \qquad \qquad  \quad  \cdot  < \gamma_1(\lambda) \left( - \omega^{+}_{11}(\lambda)   \Psi_{12}^{+}(r,-\lambda)  +   \omega^{+}_{11}(-\lambda)  \Psi^{+}_{12}(r,\lambda) \right)  \chi_1( \sqrt{\xi} \cdot), \psi_2(\cdot) > d\xi   \\
& + \int_{\Lambda_0}^{\infty}   \frac{\kappa(\lambda)}{d^{+}(\lambda) d^{-} (\lambda)}  < \left(  \gamma_2(\lambda)    \Psi_{11}^{+}(r,-\lambda)  + \gamma_3(\lambda)  \Psi^{+}_{11}(r,\lambda) + \gamma_4(\lambda)  \Psi^{+}_{21}(r,\lambda) \right) \chi_1( \sqrt{\xi} \cdot) ,  \phi_2(\cdot) >  \\
& \qquad \qquad \qquad \qquad  \quad \cdot < \left( - \gamma_2(\lambda)  \Psi_{12}^{+}(r,-\lambda)  + \gamma_3(\lambda)  \Psi^{+}_{12}(r,\lambda) + \gamma_4(\lambda)  \Psi^{+}_{22}(r,\lambda) \right) \chi_1( \sqrt{\xi} \cdot) ,  \psi_2(\cdot) >  d\xi \\
&:= \mathrm{I}+ \mathrm{II}+ \mathrm{III} + \mathrm{IV}.
\end{align*}
By Lemma \ref{lem:omega_j-behavior-large-xi}, \ref{lem:parity-of-omega-j}, \ref{Resol-interms-psi} and the fact that $|\kappa(\lambda)|\simeq \sqrt{|\xi|},$ we have 
\begin{align*}
     \left| \frac{\kappa(\lambda) \gamma_1(\lambda)^2 }{d^{+}(\lambda) d^{-}(\lambda)}  \right|  \lesssim 1 , \quad  \left| \frac{\kappa(\lambda) \gamma_2(\lambda)^2 }{d^{+}(\lambda) d^{-}(\lambda)}  \right|  \lesssim \frac{1}{\sqrt{|\xi|}} ,  \quad  \left| \frac{\kappa(\lambda) \gamma_3(\lambda)^2 }{d^{+}(\lambda) d^{-}(\lambda)}  \right|  \lesssim \frac{1}{\sqrt{|\xi|}}, \quad  \left| \frac{\kappa(\lambda) \gamma_4(\lambda)^2 }{d^{+}(\lambda) d^{-}(\lambda)}  \right|  \lesssim \frac{1}{\sqrt{|\xi|}} 
\end{align*}
\begin{align*}
         \left| \frac{\kappa(\lambda) \gamma_1(\lambda) \gamma_2(\lambda) }{d^{+}(\lambda) d^{-}(\lambda)}  \right|  \lesssim \frac{1}{|\xi|^{\frac{1}{4}}}  , \quad   \left| \frac{\kappa(\lambda) \gamma_1(\lambda) \gamma_3(\lambda) }{d^{+}(\lambda) d^{-}(\lambda)}  \right|  \lesssim \frac{1}{|\xi|^{\frac{1}{4}}}  , 
         \quad   \left| \frac{\kappa(\lambda) \gamma_1(\lambda) \gamma_4(\lambda) }{d^{+}(\lambda) d^{-}(\lambda)}  \right|  \lesssim \frac{1}{|\xi|^{\frac{1}{4}}}  
\end{align*}
\begin{align*}
      \left| \frac{\kappa(\lambda) \gamma_2(\lambda) \gamma_3(\lambda) }{d^{+}(\lambda) d^{-}(\lambda)}  \right|  \lesssim \frac{1}{\sqrt{|\xi|}} , \quad   \left| \frac{\kappa(\lambda) \gamma_2(\lambda) \gamma_4(\lambda) }{d^{+}(\lambda) d^{-}(\lambda)}  \right|  \lesssim  \frac{1}{\sqrt{|\xi|}}   , \quad  \left| \frac{\kappa(\lambda) \gamma_3(\lambda) \gamma_4(\lambda) }{d^{+}(\lambda) d^{-}(\lambda)}  \right|  \lesssim  \frac{1}{\sqrt{|\xi|}} 
\end{align*}
Notice that the main terms are the one that has $\left| \frac{\kappa(\lambda)  \gamma_1(\lambda)^2 }{d^{+}(\lambda) d^{-}(\lambda)}  \right| \lesssim 1 $ and $\left| \frac{\kappa(\lambda) \gamma_1(\lambda) \gamma_j(\lambda) }{d^{+}(\lambda) d^{-}(\lambda)}  \right| \lesssim \frac{1}{|\xi|^{\frac{1}{4}}},   $ for some $j=2,3,4.$ Notice that the first term carries an additional factor of $|\omega_{11}^{+}(\lambda)|^2$, while the second is multiplied by $|\omega_{11}^{+}(\lambda)|$.  By Lemma \ref{lem:omega_j-behavior-large-xi}, we have $|\omega_{11}^{+}(\pm \lambda)| \lesssim |\lambda|^{-1/4}$. Hence, all coefficients are bounded by $|\xi|^{-1/2}$.  Therefore, it suffices to estimate the first term $\mathrm{I}$ involving $\Psi^+_1,$ as the remaining terms follow by a similar argument. \\

By Lemma \ref{L^2-bounds-of-Psi-large-lambda}, we have the following decomposition for  $\Psi_1(r,\lambda)$
   \begin{align*}
    \Psi_1(r,\lambda)&=    \Uppsi^{+}_{(E,1)} (r,\lambda) (1+O(\frac{1}{\sqrt{|\xi|}}) )+  \Upupsilon_1(r,\lambda,\Psi^{+}_{1} (r,\lambda)),
\end{align*}
where,
\begin{align}
\label{eq:Upupsilon-1-Psi-1-plus}
\left\|  \Upupsilon_1(\cdot,\lambda,\Psi^{+}_{1} (\cdot,\lambda))  \chi_1( \sqrt{\xi} \cdot) \right\|_{L^2_r(0,\infty)}  \lesssim_{\tvarepsilon} \frac{1}{\sqrt{|\xi|} }.
\end{align}
We first focus on the contribution of $    \Uppsi^{+}_{(E,1)} (r,\lambda) = \begin{pmatrix}
     h_{+}( k_1(\lambda)r) \\ c_1(\lambda) h_{+}( k_1(\lambda)r)
  \end{pmatrix}  .$ Using Cauchy-Schwarz inequality in $\xi$ and the fact that $|\omega_{11}^{+}(\pm \lambda)|  \lesssim |\xi|^{-\frac{1}{4}}$ together with Lemma \ref{L^2-L^2-boud-Psi-1-large-lambda}, we obtain 
\begin{align*}
  |I| \lesssim  &  \int_{\Lambda_0}^{\infty}   \bigg|  <  \left( \omega^{+}_{11}(\lambda)   h_{-}(k_1(\lambda)r)  +   \omega^{+}_{11}(-\lambda)  h_{+}(k_1(\lambda)r)  \right)  \chi_1( \sqrt{\xi} \cdot), \phi_2(\cdot) >  \\
& \qquad  <  \left( \omega^{+}_{11}(\lambda)   c_1(-\lambda)    h_{-}(k_1(\lambda)r) +   \omega^{+}_{11}(-\lambda)  c_1(\lambda)    h_{+}(k_1(\lambda)r)  \right)  \chi_1( \sqrt{\xi} \cdot), \phi_2(\cdot) > \bigg| d\xi \\
& \lesssim  \left\| \phi_2 \right\|_{L^2_r}  \left\| \phi_2 \right\|_{L^2_r} 
\end{align*}
The estimate for $\Upupsilon_1$ can be obtained similarly, using the Cauchy–Schwarz inequality and either \eqref{eq:Upupsilon-1-Psi-1-plus} or the second assertion of Lemma \ref{L^2-L^2-boud-Psi-1-large-lambda}.
\end{proof}
\textbf{Case 3:} Contribution of $\mathcal{I}_0(\lambda)\mathcal{J}_1(\lambda).$
 \begin{claim} 
 \label{claim:case3-nonresonance-large-lambda}
We have
\begin{align*}
  & \int_{\Lambda_0}^{\infty}    \left| \frac{\kappa(\lambda)}{d^{+}(\lambda) d^{-}(\lambda)} \mathcal{I}_0^1(\lambda) \mathcal{J}_1^k(\lambda) \right| d\xi \lesssim   \left\| \phi_2 \right\|_{L^2_r}  \left\| \psi_2 \right\|_{L^2_r} , \quad \text{ for } \;  k=1,2.    \\
   & \int_{\Lambda_0}^{\infty}  \left|   \frac{\kappa(\lambda)}{d^{+}(\lambda) d^{-}(\lambda)} \mathcal{I}_0^2(\lambda) \mathcal{J}_1^k(\lambda) \right| d\xi  \lesssim   \left\| \phi_2 \right\|_{L^2_r}  \left\| \psi_2 \right\|_{L^2_r}, \quad \text{ for } \;  k=1,2. 
\end{align*} 
\end{claim}

\begin{proof}
By Lemma \ref{Resol-interms-phi} and \ref{Resol-interms-psi}, we have 
    \begin{align*}
    & \int_{\Lambda_0}^{\infty}   \frac{\kappa(\lambda)}{d^{+}(\lambda) d^{-}(\lambda)} \mathcal{I}_0^1(\lambda) \mathcal{J}_1^1(\lambda) d\xi  \\
   & =  \int_{\Lambda_0}^{\infty}    \frac{\kappa(\lambda)}{d^{+}(\lambda) d^{-}(\lambda)} <\Theta_1(\cdot, \lambda) \chi_0( \sqrt{\xi} \cdot), \phi_2(\cdot) >  <\Theta_2(\cdot, \lambda) \chi_1( \sqrt{\xi} \cdot), \psi_2(\cdot) >d\xi   \\
&= \int_{\Lambda_0}^{\infty}   \frac{\kappa(\lambda)}{d^{+}(\lambda) d^{-}(\lambda)}  \bigg( < (\omega_{22}^{+} (\lambda)\upvarphi_{11}(\cdot,\lambda) - \omega_{21}^{+}(\lambda) \upvarphi_{21}(\cdot,\lambda) )  \chi_0( \sqrt{\xi} \cdot), \phi_2(\cdot) >  \bigg) \\
& \qquad \qquad \qquad  \times \bigg( < \gamma_1(\lambda) \left(- \omega^{+}_{11}(\lambda) \Psi_{12}^{+}(r,-\lambda)  +   \omega^{+}_{11}(-\lambda)  \Psi^{+}_{12}(r,\lambda) \right)  \chi_1( \sqrt{\xi} \cdot), \psi_2(\cdot) > \\
   & \qquad \qquad \qquad  \;   + < \left( - \gamma_2(\lambda)   \Psi_{12}^{+}(r,-\lambda)  + \gamma_3(\lambda)  \Psi^{+}_{12}(r,\lambda) + \gamma_4(\lambda)  \Psi^{+}_{22}(r,\lambda) \right) \chi_1( \sqrt{\xi} \cdot) ,  \psi_2(\cdot) > 
   \bigg) d\xi
\end{align*}

From Lemma \ref{lem:omega_j-behavior-large-xi}, \ref{Resol-interms-psi},  and the fact that $|\kappa(\lambda)|\simeq \sqrt{|\xi|} $ for large $|\xi|,$ we have
\begin{align*}
     \left| \frac{\kappa(\lambda)\omega_{22}^{+}(\lambda)  \omega_{11}^{+}(\lambda)  \gamma_1(\lambda) }{d^{+}(\lambda) d^{-}(\lambda)}  \right|  & \lesssim \lambda^{\frac{1}{4}} , \quad   \left| \frac{\kappa(\lambda)\omega_{21}^{+}(\lambda)  \omega_{11}^{+}(\lambda)  \gamma_1(\lambda) }{d^{+}(\lambda) d^{-}(\lambda)}  \right|   \lesssim \lambda^{\frac{1}{4}} , \\
       \left| \frac{\kappa(\lambda)\omega_{22}^{+}(\lambda)    \gamma_k(\lambda) }{d^{+}(\lambda) d^{-}(\lambda)}  \right|  & \lesssim \lambda^{\frac{1}{4}} , \qquad \quad 
        \left| \frac{\kappa(\lambda)\omega_{21}^{+}(\lambda)  \gamma_k(\lambda) }{d^{+}(\lambda) d^{-}(\lambda)}  \right|   \lesssim \lambda^{\frac{1}{4}} , \quad \text{ for } k=2,3,4.
\end{align*}

Notice that all coefficients exhibit the same bounds. Therefore, we will focus on estimating the first term without a decay factor, as all other terms, with or without decay, can be handled in a similar manner.

  \begin{align*}
&  \int_{\Lambda_0}^{\infty}\left|  \xi^{\frac{1}{2}} < \upvarphi_{11}(\cdot,\lambda)   \chi_0( \sqrt{\xi} \cdot), \phi_2(\cdot) >   \frac{1}{\xi^{\frac{1}{4}} } <   \Psi_{12}^{+}(\cdot,-\lambda) \chi_1( \sqrt{\xi} \cdot) ,  \psi_2(\cdot) > \right|  d\xi 
\\
&\lesssim  \left( \int_{\Lambda_0}^{\infty} \xi \left| \int_0^{\infty}  \upvarphi_{11}(r,\lambda)\chi_0( \sqrt{\xi} r)  \phi_2(r)dr  \right|^2 d \xi \right)^{\frac{1}{2}}
\left( \int_{\Lambda_0}^{\infty} \frac{1}{\sqrt{\xi}} \left|  \Psi_{12}^{+}(r,-\lambda) \chi_1( \sqrt{\xi} r)   \psi_2(r)  dr \right|^2 d \xi \right)^{\frac{1}{2}}
\end{align*}

By applying duality together with Schur's test or using \eqref{L^2-boundsPhi21chi0-large-lambda} from Case 1, we obtain  \begin{align*}
   \left( \int_{\Lambda_0}^{\infty}  \sqrt{\xi} \left| \int_0^{\infty}  \upvarphi_{11}(r,\lambda)\chi_0( \sqrt{\xi} r)  \phi_2(r)dr  \right|^2 d \xi \right)^{\frac{1}{2}} \lesssim \left\|\phi_2\right\|_{L^2_r}
\end{align*}
Moreover, by Lemma \ref{L^2-L^2-boud-Psi-1-large-lambda} we have
\begin{align*}
    \left( \int_{\Lambda_0}^{\infty} \frac{1}{\sqrt{\xi}} \left|  \Psi_{12}^{+}(r,-\lambda) \chi_1( \sqrt{\xi} r)   \psi_2(r)  dr \right|^2 d \xi \right)^{\frac{1}{2}} \lesssim \left\|\psi_2 \right\|_{L^2_r}. 
\end{align*}

This concludes the proof of Claim \ref{claim:case3-nonresonance-large-lambda}.

\end{proof}
\subsubsection{Proof of the resonance case} As in the non-resonant case, the proof is split into two parts, we treat separately the regimes of small and large $\xi$.
 \\
 
\textbf{Step 1: Small $\xi \in (0,\delta_0).$}\\
\textbf{Case 1:} Contribution of $\mathcal{I}_0(\lambda)\mathcal{J}_0(\lambda).$
\begin{claim}
We have
\begin{align*}
  & \int_{0}^{\delta_0}  \left|   \frac{\kappa^{R}(\lambda)}{d^{R,+}(\lambda) d^{R,-}(\lambda)} \mathcal{I}_0^1(\lambda) \mathcal{J}_0^1(\lambda) \right| d\xi \lesssim   \left\| \phi_2 \right\|_{L^2_r}  \left\| \psi_2 \right\|_{L^2_r}\\
   & \int_{0}^{ \delta_0}    \left|  \frac{\kappa^{R}(\lambda)}{d^{R,+}(\lambda) d^{R,-}(\lambda)} \mathcal{I}_0^2(\lambda) \mathcal{J}_0^2(\lambda) \right| d\xi \lesssim   \left\| \phi_2 \right\|_{L^2_r}  \left\| \psi_2 \right\|_{L^2_r}
\end{align*}

\end{claim}
\begin{proof}
As in the proof of Claim \ref{case1-contribution-I_0J_0}, it suffices to establish the estimate for $\mathcal{I}_0^1(\lambda)\mathcal{J}_0^1(\lambda)$ and the corresponding bound for $\mathcal{I}_0^2(\lambda)\mathcal{J}_0^2(\lambda)$ follows in the same way. We denote $\upvarphi_1^R(r,\lambda)=\begin{pmatrix}
    \upvarphi_{11}(r,\lambda) \\
    \upvarphi_{12}(r,\lambda)
\end{pmatrix}$
and $\upvarphi_2^R(r,\lambda)=\begin{pmatrix}
    \upvarphi_{21}(r,\lambda) \\
    \upvarphi_{22}(r,\lambda)
\end{pmatrix}.$ By Lemma \ref{Resol-interms-phi-resonance}, we have $
        \Theta^R(r,\lambda) := \omega_{22}^{R,+}(\lambda) \upvarphi_1^R(r,\lambda) - w_{21}^{R,+}(\lambda) \upvarphi_2^R(r,\lambda). $ Thus, we have 
    \begin{align*}
    & \int_{0}^{\delta_0}    \frac{\kappa^{R}(\lambda)}{d^{R,+}(\lambda) d^{R,-}(\lambda)} \mathcal{I}_0^1(\lambda) \mathcal{J}_0^1(\lambda) d\lambda \\
 &= \int_{0}^{\delta_0}   \frac{\kappa^{R}(\lambda)}{d^{R,+}(\lambda) d^{R,-}(\lambda)} < \omega_{22}^{R,+} (\lambda)\upvarphi_{11}(\cdot,\lambda)   \chi_0( \sqrt{\xi} \cdot), \phi_2(\cdot) > < \omega_{22}^{R,+}(\lambda) \upvarphi_{12}(\cdot,\lambda)   \chi_0( \sqrt{\xi} \cdot), \psi_2(\cdot) > d\lambda \\
& + \int_{0}^{\delta_0}   \frac{\kappa^{R}(\lambda)}{d^{R,+}(\lambda) d^{R,-}(\lambda)} < \omega_{22}^{R,+} (\lambda)\upvarphi_{11}(\cdot,\lambda)   \chi_0( \sqrt{\xi} \cdot), \phi_2(\cdot) > <   \omega_{21}^{R,+}(\lambda) \upvarphi_{22}(\cdot,\lambda)  \chi_0( \sqrt{\xi} \cdot), \psi_2(\cdot) >  d\lambda \\
& + \int_{0}^{\delta_0}   \frac{\kappa^{R}(\lambda)}{d^{R,+}(\lambda) d^{R,-}(\lambda)} <   \omega_{21}^{R,+}(\lambda) \upvarphi_{21}(\cdot,\lambda)  \chi_0( \sqrt{\xi} \cdot), \phi_2(\cdot) >  < \omega_{22}^{R,+}(\lambda) \upvarphi_{12}(\cdot,\lambda)   \chi_0( \sqrt{\xi} \cdot), \psi_2(\cdot) > d\lambda \\
&+  \int_{0}^{\delta_0}   \frac{\kappa^{R}(\lambda)}{d^{R,+}(\lambda) d^{R,-}(\lambda)}  <   \omega_{21}^{R,+}(\lambda) \upvarphi_{21}(\cdot,\lambda)  \chi_0( \sqrt{\xi} \cdot), \phi_2(\cdot) >  <   \omega_{21}^{R,+}(\lambda) \upvarphi_{22}(\cdot,\lambda)  \chi_0( \sqrt{\xi} \cdot), \psi_2(\cdot) > d\lambda\\
&:=\mathrm{I}+ \mathrm{II}+ \mathrm{III} + \mathrm{IV}
\end{align*}
By Lemma \ref{omega_j-behavior-resonance}, \ref{lem:parity-of-omega-j-resonance}, Proposition \ref{jump-resol-resonance}, we have 
\begin{align*}
 &\left| \frac{\kappa^R(\lambda) (\omega_{22}^{R,+}(\lambda))^2(\xi) }{d^{R,+}(\lambda) d^{R,-}(\lambda)}  \right|  \lesssim \frac{1}{\sqrt{|\xi|}}  e^{-\frac{6}{\sqrt{2}} r_{\varepsilon}} , \quad
 \left| \frac{\kappa^R(\lambda)\omega_{22}^{R,+}(\lambda) \omega_{21}^{R,+}(\lambda) }{d^{R,+}(\lambda) d^{R, -}(\lambda)}  \right|   \lesssim \frac{1}{\sqrt{|\xi|}}e^{-\frac{3}{\sqrt{2}} r_{\varepsilon}},\\
& \quad  \left| \frac{\kappa^R(\lambda)( \omega_{21}^{R,+}(\lambda))^2 }{d^{R,+}(\lambda) d^{R,-}(\lambda)}  \right|   \lesssim \frac{1}{\sqrt{|\xi|}}.
\end{align*} 
Recall that $\upvarphi_1^R(r,\lambda)$ and $\upvarphi_1^R(r,\lambda)$ satisfies, 
\begin{align*}
    | \upvarphi_1^R(r,\lambda) | &\lesssim r^{\frac{3}{2}}, \qquad  \quad   | \upvarphi_2^R(r,\lambda) | \lesssim r^{\frac{3}{2}} \quad \text{ for small } r , \\ 
   | \upvarphi_1^R(r,\lambda) | &\lesssim e^{\frac{3}{\sqrt{2}}r}, \qquad   | \upvarphi_2^R(r,\lambda) | \lesssim 1  \quad \text{ for large  } r \leq r_{\varepsilon} . 
\end{align*}

We will only estimate the most delicate term $\mathrm{IV}.$ We have 
\begin{align*}
| \mathrm{IV}|   &\lesssim  \int_{0}^{\delta_0} \frac{1}{\sqrt{\xi}}  <  \upvarphi_{21}(\cdot,\lambda)  \chi_0( \sqrt{\xi} \cdot),  \phi_2(\cdot) >  <   \upvarphi_{22}(\cdot,\lambda)  \chi_0( \sqrt{\xi} \cdot), \psi_2(\cdot) > d\xi \\
& \lesssim  \left\| \frac{1}{\xi^{\frac{1}{4}}}   <  \upvarphi_{21}(\cdot,\lambda)  \chi_0( \sqrt{\xi} \cdot),  \phi_2(\cdot) >    \right\|_{L^2_\xi(0,\delta_0)} 
\left\| \frac{1}{\xi^{\frac{1}{4}}}    <   \upvarphi_{22}(\cdot,\lambda)  \chi_0( \sqrt{\xi} \cdot), \psi_2(\cdot) >   \right\|_{L^2_\xi(0,\delta_0)}  
\end{align*}

Next, we estimate the first term separately $\left\| \frac{1}{\xi^{\frac{1}{4}}}   <  \upvarphi_{21}(\cdot,\lambda)  \chi_0( \sqrt{\xi} \cdot),  \phi_2(\cdot) >    \right\|_{L^2_\xi(0,\delta_0)} .$ Using a change of variable $\zeta=\sqrt{\xi}$, we obtain 
\begin{align*}
    \int_0^{\delta_0}  \frac{1}{\sqrt{\xi}} \left( \int_0^{\infty} \upvarphi_{21}(r,\lambda)  \chi_0( \sqrt{\xi} r)  \phi_2(r) dr \right)^2 d\xi = 2 \int_0^{\sqrt{\delta_0}}   \left( \int_0^{\infty} \upvarphi_{21}(r,\lambda(\zeta^2))  \chi_0( \zeta r)  \phi_2(r) dr \right)^2 d\zeta .
\end{align*}
We define the operator  
\begin{align*}
T v (\zeta)  :=  \int_0^{\infty}   \upvarphi_{21}(r,\lambda(\zeta^2))  \chi_0( \zeta r) v(r) dr   . 
\end{align*}
Setting $\chi_0(x)= \displaystyle \sum_{j=0}^{\infty} \eta_j(x),$ where $supp(\eta_j) \subseteq [2^{-(j+1)},2^{-j}],$ we decompose the operator $T=\sum_{j=0}^{\infty} T_j$ where 
\begin{align*}
    T_j v (\zeta)  :&=  \int_0^{\infty}   \upvarphi_{21}(r,\lambda(\zeta^2))  \eta_j( \zeta r) v(r) dr, \quad 0\leq \zeta \leq \sqrt{\delta_0}, \\
   T^{\ast}_j w (r):&= \int_0^{\sqrt{\delta_0}}    \upvarphi_{21}(r,\lambda(\sigma^2))   \eta_j( \sigma r)  w(\sigma) d\sigma.
\end{align*}
Next, we estimate the operator norms $ ||  T^{\ast}_j  T_k ||_{L^2_{\lambda} \to L^2_{\lambda}}.$ Indeed, we have 
\begin{align*}
   T^{\ast}_j  T_k v(r) &= \int_0^{\sqrt{\delta_0}}\int_0^{\infty}   \upvarphi_{21}(r,\lambda(\sigma^2))  \eta_j( \sigma r)      \upvarphi_{21}(s,\lambda(\zeta^2))  \eta_k( \sigma s) v(s) ds d\sigma \\ 
   &= \int_0^{\infty}  K(r,s) v(s)ds, 
\end{align*}
where $K(r,s):= \displaystyle \int_0^{\sqrt{\delta_0}} \upvarphi_{21}(r,\lambda(\sigma^2))        \upvarphi_{21}(s,\lambda(\zeta^2)) \eta_j( \sigma r)   \eta_k( \sigma s) d\sigma.$
Using the asymptotic behavior $\upvarphi_{21}$ for small and large $r,$ we obtain  
\begin{align*}
   \int_0^{\infty} |K(r,s)| ds & \lesssim    \int_0^{\sqrt{\delta_0}}    \eta_j( \sigma r) \int_0^{\infty}   \eta_k( \sigma s) ds  d\sigma\\
&\lesssim  \int_{\frac{2^{-(j+1)}}{r}}^{\frac{2^{-j}}{r}}  \int_{\frac{2^{-(k+1)}}{\sigma}}^{\frac{2^{-k}}{\sigma}}  ds d\sigma \\
& \lesssim  2^{-k} \int_{\frac{2^{-(j+1)}}{r}}^{\frac{2^{-j}}{r}}  \frac{1}{\sigma}  d\sigma \lesssim 2^{-k} .
\end{align*}
By symmetry, we obtain $  \int_0^{\infty} |K(r,s)| dr \lesssim 2^{-j}.$ Therefore the norms $\left\| K(r,\cdot) \right|_{L^2(ds)}$ and $\left\| K(\cdot,s) \right|_{L^2(dr)}$ are uniformly bounded by $2^{-k}$ and $2^{-j},$ respectively. Then by by Shur's test, we obtain 
\begin{equation*}
  \left\|  T^{\ast}_j  T_k \right\|_{L^2_{\lambda} \to L^2_{\lambda}} \lesssim 2^{-\frac{k}{2}} 2^{-\frac{j}{2}}.
\end{equation*}
Then by duality we obtain, 
\begin{align*}
    <T \phi_2 ,T\phi_2>= \sum_{j,k=0}^{\infty} <\phi_2,  T^{\ast}_j  T_k  \phi_2>   \lesssim \sum_{j,k=0}^{\infty} 2^{-\frac{k}{2}} 2^{-\frac{j}{2}} \left\| \phi_2 \right\|_{L^2}^2 \lesssim  \left\| \phi_2 \right\|_{L^2}^2 
\end{align*}
Therefore, we get 
\begin{align}
\label{eq:L^2-boundsPhi21chi0-samll-lambda-resonance}
   \left\| \frac{1}{\xi^{\frac{1}{4}}}   <  \upvarphi_{21}(\cdot,\lambda)  \chi_0( \sqrt{\xi} \cdot),  \phi_2(\cdot) >    \right\|_{L^2_\xi(0,\delta_0)}   \lesssim \left\| \phi_2 \right\|_{L^2}
\end{align}
Similarly, we obtain 
\begin{align*}
   \left\| \frac{1}{\xi^{\frac{1}{4}}}   <  \upvarphi_{22}(\cdot,\lambda)  \chi_0( \sqrt{\xi} \cdot),  \phi_2(\cdot) >    \right\|_{L^2_\xi(0,\delta_0)}   \lesssim \left\| \phi_2 \right\|_{L^2}.
\end{align*}
This concludes the estimates for $\mathrm{IV}.$ All the other terms come with an additional decaying factor and can be estimated similarly, we omit the details.
\end{proof}
\textbf{Case 2: } Contribution of $\mathcal{I}_1(\lambda)\mathcal{J}_1(\lambda).$
 \begin{claim}
 \label{Claim-case_2-resonance} 
 We have
\begin{align*}
  & \int_{0}^{\delta_0} \left|   \frac{\kappa^{R}(\lambda)}{d^{R,+}(\lambda) d^{R,-}(\lambda)} \mathcal{I}_1^1(\lambda) \mathcal{J}_1^k(\lambda) \right| d\xi \lesssim   \left\| \phi_2 \right\|_{L^2_r}  \left\| \psi_2 \right\|_{L^2_r}, \quad \text{ for } \;  k=1,2. \\
   & \int_{0}^{ \delta_0}    \left| \frac{\kappa^{R}(\lambda)}{d^{R,+}(\lambda) d^{R,-}(\lambda)} \mathcal{I}_1^2(\lambda) \mathcal{J}_1^k(\lambda) \right| d\xi \lesssim   \left\| \phi_2 \right\|_{L^2_r}  \left\| \psi_2 \right\|_{L^2_r} , \quad \text{ for } \;  k=1,2. 
\end{align*}

\end{claim}
\begin{proof}
We only need to establish the estimate for $\mathcal{I}_1^1(\lambda)\,\mathcal{J}_1^1(\lambda),$ the other cases can be treated analogously. Recall that from Lemma \ref{Theta-interms-psi-resonance}, we have 
    \begin{align*}
     \Theta^R(r,\lambda)& =\gamma_1^{R}(\lambda) \left( \omega^{R,+}_{11}(\lambda) \sigma_3  \Psi_1^{R,+}(r,-\lambda)  +   \omega^{R,+}_{11}(-\lambda)  \Psi^{R,+}_1(r,\lambda) \right)
  \\
 & + \gamma_2^{R}(\lambda)   (\sigma_3) \Psi_1^{R,+}(r,-\lambda)  + \gamma_3^{R}(\lambda)  \Psi^{R,+}_1(r,\lambda) + \gamma_4^{R}(\lambda)  \Psi^{R,+}_2(r,\lambda).
\end{align*}
By a direct computation, similar to that in the proof of Claim \ref{Claim-step1-case2}, we obtain the following decomposition
\begin{align*}
   \int_{0}^{\delta_0}    \frac{\kappa^{R}(\lambda)}{d^{R,+}(\lambda) d^{R,-}(\lambda)} \mathcal{I}_1^1(\lambda) \mathcal{J}_1^1(\lambda) d\xi=:= \mathrm{I}^{R}+ \mathrm{II}^{R}+ \mathrm{III}^{R} + \mathrm{IV}^{R}.
\end{align*}
Note that by Lemma \ref{omega_j-behavior-resonance} and \ref{Theta-interms-psi-resonance} together with the fact that $|\kappa^{R}(\lambda)|\simeq \sqrt{|\xi|} $ we have 
\begin{align*}
     \left| \frac{\kappa^{R}(\lambda) \gamma_1^R(\lambda)^2 }{d^{R,+}(\lambda) d^{R,-}(\lambda)}  \right|  &\lesssim \frac{1}{|\xi|^{\frac{3}{2}}}e^{-\frac{6}{\sqrt{2}} r_{\varepsilon}}  , \quad  \left| \frac{\kappa^{R}(\lambda) \gamma_2^R(\lambda)^2 }{d^{R,+}(\lambda) d^{R,-}(\lambda)}  \right|  \lesssim \frac{1}{\sqrt{|\xi|}} ,  \quad  \left| \frac{\kappa^{R}(\lambda) \gamma_3^R(\lambda)^2 }{d^{R,+}(\lambda) d^{R,-}(\lambda)}  \right|  \lesssim \frac{1}{\sqrt{|\xi|}}, \\
     \left| \frac{\kappa^{R}(\lambda) \gamma_4^R(\lambda)^2 }{d^{R,+}(\lambda) d^{R,-}(\lambda)}  \right|  &\lesssim \frac{1}{\sqrt{|\xi|}}  e^{\frac{3}{\sqrt{2}} \frac{4 r_{\epsilon}}{5}}  
         \left| \frac{\kappa^{R}(\lambda) \gamma_1^R(\lambda) \gamma_2^R(\lambda) }{d^{R,+}(\lambda) d^{R,-}(\lambda)}  \right|  \lesssim  \frac{1}{|\xi|}  e^{-\frac{3}{\sqrt{2}} r_{\varepsilon}}  , \quad   \left| \frac{\kappa^R(\lambda) \gamma_1^R(\lambda) \gamma_3^R(\lambda) }{d^{R,+}(\lambda) d^{R,-}(\lambda)}  \right|  \lesssim \frac{1}{|\xi|} e^{-\frac{3}{\sqrt{2}} r_{\varepsilon}}  , \\
          \left| \frac{\kappa^R(\lambda) \gamma_1^R(\lambda) \gamma_4^R(\lambda) }{d^{R,+}(\lambda) d^{R,-}(\lambda)}  \right|  &\lesssim \frac{1}{|\xi|} e^{-\frac{3}{\sqrt{2}} \frac{3}{5} r_{\varepsilon}} 
      \left| \frac{\kappa^R(\lambda) \gamma_2^R(\lambda) \gamma_3^R(\lambda) }{d^{R,+}(\lambda) d^{R,-}(\lambda)}  \right|  \lesssim \frac{1}{\sqrt{|\xi|}}, \quad   \left| \frac{\kappa^R(\lambda) \gamma_2^R(\lambda) \gamma_4^R(\lambda) }{d^{R,+}(\lambda) d^{R,-}(\lambda)}  \right|  \lesssim \frac{1}{\sqrt{|\xi|}}  e^{\frac{3}{\sqrt{2}} \frac{2r_{\varepsilon}}{5}}  , \\   \left| \frac{\kappa^R(\lambda) \gamma_3^R(\lambda) \gamma_4^R(\lambda) }{d^{R,+}(\lambda) d^{R,-}(\lambda)}  \right|  & \lesssim  \frac{1}{\sqrt{|\xi|}} e^{\frac{3}{\sqrt{2}} \frac{2r_{\varepsilon}}{5}} 
\end{align*}

Notice that the main terms are the one that has $\left| \frac{\kappa^{R}(\lambda) \gamma_i^{R}(\lambda) \gamma_j^{R}(\lambda) }{d^{R,+}(\lambda) d^{R,-}(\lambda)}  \right| \lesssim \frac{1}{\sqrt{|\xi|}}, $ and $\left| \frac{\kappa^{R}(\lambda) \gamma_i^{R}(\lambda) \gamma_4^{R}(\lambda) }{d^{R,+}(\lambda) d^{R,-}(\lambda)}  \right| \lesssim \frac{1}{\sqrt{|\xi|}} e^{\frac{3}{\sqrt{2}} \frac{4r_{\varepsilon}}{5}},   $ for some $i,j \in \{1,2,3,4 \}.$ Theses terms can be estimates exactly as in Claim \ref{Claim-step1-case2} using Lemma \ref{omega_j-behavior-resonance} together with Lemma \ref{L^2-bounds-of-Psi}and \ref{L^2-L^2-boud-Psi-1}.
\end{proof}
\textbf{Case 3:} Contribution of $\mathcal{I}_0(\lambda)\mathcal{J}_1(\lambda).$
 \begin{claim}
 \label{claim:case3-nonresonance-small-lambda-resonance}
 We have
\begin{align*}
  & \int_{0}^{\delta_0}   \bigg| \frac{\kappa^{R}(\lambda)}{d^{R,+}(\lambda) d^{R,-}(\lambda)} \mathcal{I}_0^1(\lambda) \mathcal{J}_1^k(\lambda)  \bigg| d\xi \lesssim   \left\| \phi_2 \right\|_{L^2_r}  \left\| \psi_2 \right\|_{L^2_r} , \quad \text{ for } \;  k=1,2.    \\
   & \int_{0}^{ \delta_0}   \bigg| \frac{\kappa^R(\lambda)}{d^{R,+}(\lambda) d^{R,-}(\lambda)} \mathcal{I}_0^2(\lambda) \mathcal{J}_1^k(\lambda) \bigg| d\xi  \lesssim   \left\| \phi_2 \right\|_{L^2_r}  \left\| \psi_2 \right\|_{L^2_r}, \quad \text{ for } \;  k=1,2. 
\end{align*} 
\end{claim}

\begin{proof}
It is enough to prove the estimate for $\mathcal{I}_0^1(\lambda)\,\mathcal{J}_1^1(\lambda)$, since the remaining cases follow by the same reasoning. By Lemma \ref{Resol-interms-phi-resonance} and \ref{Theta-interms-psi-resonance}, we have
 \begin{align*}
    & \int_{0}^{\delta_0}    \frac{\kappa^R(\lambda)}{d^{R,+}(\lambda) d^{R,-}(\lambda)} \mathcal{I}_0^1(\lambda) \mathcal{J}_1^1(\lambda)  d\xi \\
   & =  \int_{0}^{\delta_0}    \frac{\kappa^R(\lambda)}{d^{R,+}(\lambda) d^{R,-}(\lambda)}  <\Theta_1^R(\cdot, \lambda) \chi_0( \sqrt{\xi} \cdot), \phi_2(\cdot) >  <\Theta_2^R(\cdot, \lambda) \chi_1( \sqrt{\xi} \cdot), \psi_2(\cdot) >  d\xi  \\
&= \int_{0}^{\delta_0}   \frac{\kappa^R(\lambda)}{d^{R,+}(\lambda) d^{R,-}(\lambda)}   \bigg( < (\omega_{22}^{R,+} (\lambda)\upvarphi_{11}(\cdot,\lambda) - \omega_{21}^{R,+}(\lambda) \upvarphi_{21}(\cdot,\lambda) )  \chi_0( \sqrt{\xi} \cdot), \phi_2(\cdot) >  \bigg) \\
& \qquad \qquad \qquad   \times \bigg( < \gamma_1^{R}(\lambda) \left(- \omega^{R,+}_{11}(\lambda)   \Psi_{12}^{+}(r,-\lambda)  +   \omega^{R,+}_{11}(-\lambda)  \Psi^{+}_{12}(r,\lambda) \right)  \chi_1( \sqrt{\xi} \cdot), \psi_2(\cdot) > \\
   & \qquad \qquad \qquad \quad    + < \left( - \gamma_2^{R}(\lambda)  \Psi_{12}^{+}(r,-\lambda)  + \gamma_3^{R}(\lambda)  \Psi^{+}_{12}(r,\lambda) + \gamma_4^{R}(\lambda)  \Psi^{+}_{22}(r,\lambda) \right) \chi_1( \sqrt{\xi} \cdot) ,  \psi_2(\cdot) > 
   \bigg) d\xi 
\end{align*}
Using  Lemma \ref{omega_j-behavior-resonance}, \ref{lem:parity-of-omega-j-resonance} and  \ref{Theta-interms-psi-resonance} with the fact that $\kappa^{R}(\lambda)$
Note that, using  Lemma \ref{omega_j-behavior-resonance}, \ref{lem:parity-of-omega-j-resonance} and  \ref{Theta-interms-psi-resonance} with the fact that $|\kappa^{R}(\lambda)| \simeq \sqrt{|\xi|}$
\begin{align*}
     \left| \frac{\kappa^{R}(\lambda)\omega_{22}^{R,+}(\lambda)  \omega_{11}^{R,+}(\lambda)  \gamma_1^{R}(\lambda) }{d^{R,+}(\lambda) d^{R,-}(\lambda)}  \right|  & \lesssim
     \frac{1}{|\xi |} e^{-\frac{3}{\sqrt{2}} \frac{9}{5} r_{\varepsilon}}  , \quad   \left| \frac{\kappa^R(\lambda)\omega_{21}^{R,+}(\lambda)  \omega_{11}^{R,+}(\lambda)  \gamma_1^{R}(\lambda) }{d^{R,+}(\lambda) d^{R,-}(\lambda)}  \right|   \lesssim    \frac{1}{|\xi|}
     e^{-\frac{3}{\sqrt{2}} \frac{4}{5} r_{\varepsilon}} , \\
       \left| \frac{\kappa^{R}(\lambda)\omega_{22}^{R,+}(\lambda)    \gamma_k^{R}(\lambda) }{d^{R,+}(\lambda) d^{R,-}(\lambda)}  \right|  & \lesssim   \frac{1}{\sqrt{|\xi|}} e^{-\frac{3}{\sqrt{2}}  r_{\varepsilon}}  , \qquad 
        \left| \frac{\kappa^{R}(\lambda)\omega_{21}^{R,+}(\lambda)  \gamma_k(\lambda) }{d^{R,+}(\lambda) d^{R,-}(\lambda)}  \right|   \lesssim \frac{1}{\sqrt{ |\xi| }} , \quad \text{ for } k=2,3 \\    
         \left| \frac{\kappa^{R}(\lambda)\omega_{22}^{R,+}(\lambda)    \gamma_4^{R}(\lambda) }{d^{R,+}(\lambda) d^{R,-}(\lambda)}  \right|  & \lesssim \frac{1}{\sqrt{ |\xi| }}  e^{-\frac{3}{\sqrt{2}} \frac{3}{5}  r_{\varepsilon}}  , \quad  \left| \frac{\kappa^{R}(\lambda)\omega_{21}^{R,+}(\lambda)    \gamma_4^{R}(\lambda) }{d^{R,+}(\lambda) d^{R,-}(\lambda)}  \right|  \lesssim \frac{1}{\sqrt{ |\xi| }} e^{\frac{3}{\sqrt{2}} \frac{2 r_{\epsilon}}{5}}.
\end{align*}

Notice that the main terms are the one that has $\left| \frac{\kappa^R(\lambda) \omega_{21}^{R,+}(\lambda)  \gamma_k^R(\lambda)}{d^{R,+}(\lambda) d^{R,-}(\lambda)}  \right| \lesssim \frac{1}{\sqrt{ |\xi| }}, $ for $k=2,3$ and $\left| \frac{\kappa^R(\lambda) \omega_{21}^{R,+}(\lambda)    \gamma_4^R(\lambda) }{d^{R,+}(\lambda) d^{R,-}(\lambda)}  \right|  \lesssim \frac{1}{\sqrt{ |\xi| }} e^{\frac{3}{\sqrt{2}} \frac{2 r_{\epsilon}}{5}}   .$ The second term involves $\gamma_4^R,$ and therefore it includes an exponential decay factor given by $\Psi^{+}_{2}(r,\lambda).$  We will mainly focus on estimating the first term with no decaying factor. We only focus on the first term for $k=2,$ the remaining cases follow in the same way.
  \begin{align*}
&   \int_{0}^{\delta_0} \left|  \frac{1}{\xi^{\frac{1}{4}} }  < \upvarphi_{21}(\cdot,\lambda)   \chi_0( \sqrt{\xi} \cdot), \phi_2(\cdot) >   \frac{1}{\xi^{\frac{1}{4}} } <   \Psi_{12}^{+}(\cdot,-\lambda) \chi_1( \sqrt{\xi} \cdot) ,  \psi_2(\cdot) >\right|   d\xi  
\\
&\lesssim  \left(\int_{0}^{\delta_0}  \frac{1}{\sqrt{\xi}} \left| \int_0^{\infty}  \upvarphi_{21}(r,\lambda)\chi_0( \sqrt{\xi} r)  \phi_2(r)dr  \right|^2 d \xi \right)^{\frac{1}{2}}
\left(\int_{0}^{\delta_0} \frac{1}{\sqrt{\xi}} \left|  \Psi_{12}^{+}(r,-\lambda) \chi_1( \sqrt{\xi} r)   \psi_2(r)  dr \right|^2 d \xi \right)^{\frac{1}{2}}
\end{align*}

By applying duality together with Schur's test or using \eqref{eq:L^2-boundsPhi21chi0-samll-lambda-resonance} from Case 1, we obtain  \begin{align*}
   \left(\int_{0}^{\delta_0}  \frac{1}{\sqrt{\xi}} \left| \int_0^{\infty}  \upvarphi_{21}(r,\lambda)\chi_0( \sqrt{\xi} r)  \phi_2(r)dr  \right|^2 d \xi \right)^{\frac{1}{2}} \lesssim \left\|\phi_2\right\|_{L^2_r}
\end{align*}
Moreover, by Lemma \ref{L^2-L^2-boud-Psi-1} we have
\begin{align*}
    \left(\int_{0}^{\delta_0} \frac{1}{\sqrt{\xi}} \left|  \Psi_{12}^{+}(r,-\lambda) \chi_1( \sqrt{\xi} r)   \psi_2(r)  dr \right|^2 d\xi \right)^{\frac{1}{2}} \lesssim \left\|\psi_2 \right\|_{L^2_r}. 
\end{align*}

This concludes the proof of Claim \ref{claim:case3-nonresonance-small-lambda-resonance}.

\end{proof}
\textbf{Step 2: Large $\xi \in (\Lambda_0,\infty).$} This case is analogous to Step 2 for non-resonance case and we omit the details.

\appendix
\section{Proof of Theorem 1} \label{sec:appOde}
In this section, we present the proof of Theorem \ref{th:ode}. The argument is essentially the same as in the Euclidean case, and we follow the approach in \cite{ChenElliott94}. We prove the existence and uniqueness of the solution to the equation \eqref{eq:ode2} and \eqref{eq:boundary_condition-ode2}, that is, 
\begin{align}
\label{eq:ode2-app}
&\partial_r^2 u  + \coth(r) \partial_r u  - \frac{n^2}{\sinh^2( r)}u + u - u^3=0, \\
\label{eq:boundary_condition-ode2-app}
&\lim_{r \to 0} u(r)=0, \quad \lim_{r \to \infty} u(r) =1 , \quad u(r) \geq 0, \quad \forall r >0.
\end{align}

Notice that a solution to equation \eqref{eq:ode2-app},\eqref{eq:boundary_condition-ode2-app} defined on $(0,R)$ can be written near $0$ as power series $u(r)=\sum_{k=1}^{\infty} a_k r^k.$ Plugging this into the equation for small $r$, i.e, 
$$
(\partial_r^2 u + (\frac{1}{r} + O(r) )\partial_r u  - \frac{n^2}{r^2+O(r^4) } ) u  + u - u^3 =0
$$
One can see that $a_k=0,$ for all $k<n,$ and we have $u(r)=a r^{n} + O(r^{(n+2)}),$ where $a=a_n.$ Therefore,  any bounded solution of \eqref{eq:ode2-app} near $r=0$ is a solution to the following problem 
\begin{align}
\label{eq:ode-v}
\frac{1}{2} \partial_r^2 v(\alpha,r) + \frac{1}{2} \coth(r)& \partial_r v(\alpha,r)   - \frac{n^2}{2 \sinh^2( r)} v(\alpha,r) + v(\alpha,r) -  v^3(\alpha,r)=0, \quad r>0 \\
\label{eq:bc-v}
& v(\alpha,r) \sim \alpha r^n \quad \text{as } r \to 0.
\end{align}

To prove Theorem~\ref{th:ode}, we use a shooting argument with respect to the parameter $\alpha$. As by product, we have the following results: 
\begin{theorem}
\label{thm:v}
    Let $v(\alpha,r)$ be a solution of \eqref{eq:ode-v}-\eqref{eq:bc-v} in the maximal existence interval $(0,R_{\alpha})$ and let $a$ be constant of Theorem \ref{th:ode}. Then, we have 
    \begin{itemize}
        \item if $\alpha \in (0,a),$ then $R_{\alpha}=\infty$ and $|u|<1$ and oscillating. 
        \item if $\alpha=a$ then $R_{\alpha}=\infty$ and $ \displaystyle \lim_{r \to \infty} u(r) = 1. $
        \item  if $\alpha \in (a,\infty),$ then $R_{\alpha}< \infty$ and 
        $ \displaystyle \lim_{r \to R_{\alpha}} v(\alpha,r) = \infty. $
    \end{itemize}
    \end{theorem}

First, we prove that the system \eqref{eq:ode2-app}-\eqref{eq:boundary_condition-ode2-app} admits at least one solution.  Our strategy is to employ a shooting argument. To this end, for each $\alpha > 0$ we consider the auxiliary problem \eqref{eq:ode-v}-\eqref{eq:bc-v}.

\begin{lemma}
    For any $\alpha \geq 0 $ there exists a unique solution near $r=0$ for the problem \eqref{eq:ode-v}-\eqref{eq:bc-v} such that 
    \begin{equation*}
       v(\alpha,r)=\alpha r^n -  \frac{\alpha }{2n+2} r^{n+2}+ O(r^{n+4}), \quad \text{as } \; r \to 0. 
    \end{equation*}
    
    The unique extension of this solution to the maximal interval $[0,R_{\alpha})$ satisfy either $R_{\alpha}=\infty$ or $ \displaystyle \lim_{r \to R_{\alpha}} v(\alpha,r)=\infty.$ 
\end{lemma}
\begin{proof} 
    The result follows directly from a power series expansion combined with standard theory for ordinary differential equations.  For brevity, we omit the details.
\end{proof}
\begin{remark}
Note that the dominant term in the asymptotic expansion of $v$ or $u$ as $r \to 0$ is the same as in the Euclidean setting. The difference appears in the coefficient of $r^{n+2}$, which is slightly altered due to working in the hyperbolic space with curvature $-\frac12$. 
More generally, in a hyperbolic space of curvature $\kappa$, we have  
\begin{equation*}
    v(\alpha,r) = \alpha r^n + \frac{\alpha}{\kappa} \frac{r^{n+2}}{4n+4} + O(r^{n+4}), 
    \quad \text{as } r \to 0.
\end{equation*}
\end{remark}

Next, we introduce the following three disjoint sets:
\begin{align*}
   \mathcal{A}&:=\{ \alpha >0 \, | \;  \exists y \in (0,R_{\alpha}) \text{ such that } v_r(\alpha,y)=0\} \\ 
\mathcal{B}&:=\{ \alpha >0  \, | \;  \forall r \in (0,R_{\alpha}) \text{ and } \max_{r \in (0,R_{\alpha})  } v(\alpha,r) >1    \} \\
\mathcal{C} &:=  \{  \alpha >0 \, | \;  v_r(\alpha,r) >0, \; v(\alpha,r) \leq 1\; \forall r > 0 \}.
\end{align*}

Notice that the three sets are disjoints and that 
\begin{equation}
\label{eq:ABC}
    \mathcal{A} \cup \mathcal{B} \cup \mathcal{C} = (0,\infty) .
\end{equation}
\begin{lemma}
\label{lem:C-sol}
    If $\alpha \in \mathcal{C},$ then $v(\alpha,r)$ is a solution to \eqref{eq:ode2-app}, \eqref{eq:boundary_condition-ode2-app}.
\end{lemma}
\begin{proof}
Since for $\alpha \in \mathcal{C}$ we have that $v(\alpha,r)$ is strictly increasing and bounded by $1$, it follows that the limit exists 
\[
\lim_{r \to \infty} v(\alpha,r) =: l \in (0,1]
\]
 It remains to prove that $l = 1$.  Assume by contradiction, that $l < 1$. Then, from \eqref{eq:ode-v}, we obtain $\frac{1}{2}(\sinh(r) v_r)_r := \sinh(r) ( l^3 -l +  +\frac{n^2}{2\sinh^2(r)} v +(v^3-l^3) - (v-l)  ) ,$ which implies that 
 $\frac{1}{2 } v_r =  \frac{\cosh(r)-1}{\sinh(r)}(l^3 -l )+ o(1).  $ Since $v_r > 0$ for all $r > 0$, this leads to a contradiction. Therefore, $l=1$ and $v(\alpha,r)$ is a solution to \eqref{eq:ode2-app}, \eqref{eq:boundary_condition-ode2-app}.
\end{proof}

Note that, by Lemma~\ref{lem:C-sol}, in order to prove Theorem~\ref{thm:v} it suffices to show that $\mathcal{C}$ is non-empty.  
From \eqref{eq:ABC}, it is sufficient to prove that the sets $\mathcal{A}$ and $\mathcal{B}$ are both open and non-empty. First, we study the set $\mathcal{A}.$

\begin{lemma}
    The set $\mathcal{A}$ is not empty and there exists a positive constant $m$ such that $(0,m) \subset \mathcal{A}.$   
\end{lemma}
\begin{proof}
    For any $\alpha >0,$ define $w(\alpha,r)=\frac{v(\alpha,r)}{\alpha}, $ for $r \in (0,R_{\alpha}).$ Then $w$ satisfies 
    \begin{align*}
    \frac{1}{2} \partial_r^2 w + \frac{1}{2} \coth(r) \partial_r w - \frac{n^2}{2 \sinh(r)} w + w = \alpha^2 w^3 , \quad r \in (0,R_{\alpha}),     \\
    w(\alpha,r) = r^n - \frac{1}{2n+2} r^{n+2} +O(r^{n+4}), \quad \text{as } \; r \to \infty.
    \end{align*}
It follows that as $\alpha \to 0 ,$ $w(\alpha,r) \to w(0,r),$ where $w(0,r)$ is the solution of 
\begin{align}
\label{eq:w0}
     \frac{1}{2} \partial_r^2 w + \frac{1}{2} \coth(r) \partial_r w - \frac{n^2}{2 \sinh(r)} w + w=0, \quad r\in (0,\infty), \\
      w(\alpha,r)=  r^n - \frac{1}{2n+2} r^{n+2} +O(r^{n+4}), \quad \text{as } \; r \to 0.
\end{align}
Notice that the solution \eqref{eq:w0} are the spherical functions. Also, one can see that for large $r,$ the equation \eqref{eq:w0} is approximates by  $(\frac{1}{2} \partial_r^2 + \partial_r +1 )w  =0,$  whose solutions are oscillating. Thus the solution \eqref{eq:w0} must exhibit oscillatory behavior. Consequently, for small values of $\alpha$, the function $w(\alpha,r)$ oscillates, and hence so does $v(\alpha,r) = \alpha w(\alpha,r)$. As a result, there exists some $m > 0$ such that for all $\alpha \in (0,m)$, we have $\alpha \in \mathcal{A}$.
\end{proof}

For any $\alpha \in \mathcal{A},$ we define 
\begin{equation}
\label{eq:defy}
y_0(\alpha) := \inf \{ r \in (0,R_ \alpha) |  v_r(\alpha,r)=0   \}.
\end{equation}

\begin{lemma}
\label{lem:Aopen}
    If $\alpha \in \mathcal{A},$ then 
    \begin{align}
        \label{eq:y>n}
    \sinh^{-1}(\frac{n}{\sqrt{2}})< y_0(\alpha) \\
       \label{eq:v<y_0}
       v(\alpha,y_0(\alpha)) < \left( 1 - \frac{n^2}{2 \sinh^2(y_0(\alpha))} \right)^{\frac{1}{2}} \\
       \label{eq:v_rr<0}
       v_{rr}(\alpha,y_0(\alpha))<0.
    \end{align}
 In addition, $\mathcal{A}$ is open and $y_0(\cdot) \in \mathcal{C}^{\infty}(\mathcal{A}).$   
\end{lemma}
\begin{proof}
    Since $\alpha \in \mathcal{A},$ and by definition of $y_0(\alpha)$ we have $v_r(\alpha,y_0(\alpha))=0  $ and $v_{rr}(\alpha,y_0(\alpha)) \leq 0,$ because $v \sim \alpha r^n$ near $0.$ By the equation \eqref{eq:ode-v}, we have
    \begin{equation}
    \label{eq:ode-v-y_0}
    \frac{1}{2} v_{rr}(\alpha,y_0(\alpha))=v(\alpha,y_0(\alpha)) \left( \frac{n^2}{2\sinh^2(y_0(\alpha))} + v^2(\alpha,y_0(\alpha)) -1  \right) \leq 0.
    \end{equation}
which yields, $0 \leq v^2(\alpha,y_0(\alpha)) \leq 1 - \frac{n^2}{2 \sinh^2(y_0(\alpha))}.$ Therefore, $\sinh^{-1}(\frac{n}{\sqrt{2}})< y_0(\alpha)$ and     $0 \leq v(\alpha,y_0(\alpha)) \leq \left( 1 - \frac{n^2}{2 \sinh^2(y_0(\alpha))} \right)^{\frac{1}{2}}.$ To obtain \eqref{eq:v<y_0}, we need to show that the inequality is strict. Assume $v^2(\alpha,y_0(\alpha)) + \frac{n^2}{2 \sinh^2(y_0(\alpha))}-1 =0.$ Then 
$$\frac{d}{dr} \left(v^2(\alpha,r) + \frac{n^2}{2 \sinh^2(r)}-1 \right)\big|_{r=y_0(\alpha)}= -\frac{n^2 \cosh(y_0(\alpha))}{\sinh^3(y_0(\alpha))}  <0, $$
which implies that $v^2+ \frac{n^2}{\sinh(r)} -1 >0$ in $(y_0(\alpha)-\delta,y_0(a)),$ for some $\delta >0.$ It follows that $\frac{1}{2}(\sinh(r) v_r)_r= \sinh(r)(v^2+ \frac{n^2}{\sinh^2(r)} -1 )>0$ for $r \in (y_0(\alpha)-\delta,y_0(a)).$ Integrating between $y_0(\alpha)-\delta$ and $y_0(a),$ we obtain $y_0(\alpha) v_r(\alpha,y_0(\alpha)) - (y_0(\alpha)-\delta) v_r(\alpha, y_0(\alpha)-\delta)>0.$ Using the fact that $v_r(\alpha,y_0(\alpha))=0,$ we get $v_r(\alpha,y_0(\alpha)-\delta)=0,$ which contradict the definition of $y_0(\alpha).$ This concludes the proof of \eqref{eq:v<y_0}. Finally, using the equation \eqref{eq:ode-v-y_0} with \eqref{eq:v<y_0}, we obtain \eqref{eq:v_rr<0}. The last assertion of the lemma is obtain by Implicit Function Theorem and \eqref{eq:v_rr<0}. Let $\alpha_0 \in \mathcal{A},$ then we have 
$v_r(\alpha_0,y_0(\alpha_0)) =0 $ and $v_{rr}(\alpha_0,y_0(\alpha_0))<0.$ By Implicit Function Theorem, there exists a smooth function defined in neighborhood of $\alpha_0$ to a neighborhood of $y_0(\alpha_0),$ i.e., $f:\mathcal{N}(\alpha_0) \longrightarrow \mathcal{N}(y_0(\alpha_0)), \; \alpha \mapsto f(\alpha),$ such that $f(\alpha_0)=y_0(\alpha_0)$ and $v_r(\alpha,f(\alpha))=0.$ Hence, $\mathcal{A}$ is open and $y_0$ is smooth. 
\end{proof}

Next, we prove that the set $\mathcal{B}$ is open and nonempty. First, we show that if a solution reaches the value $1$ at some finite $r$, then it must blow up.

\begin{lemma}
\label{lem:blowupsv}
    Assume that for some $y_1 \in (0,R_{\alpha}),$ we have $v(\alpha,y_1)=1$ and $v(\alpha,r)<1$ for all $r \in (0,y_1).$ Then, $v_r(\alpha,r) >0$ for all $r \in [y_1,R_{\alpha}).$
\end{lemma}
\begin{proof}
First, we show that $v_r(\alpha,y_1)>0.$ Assume $v_r(\alpha,y_1)=0,$ then by \eqref{eq:ode-v}  we have  $ \frac{1}{2}v_{rr} \big|_{r=y_1} >0.$ Using Taylor expansion at $y_1$ and $y_1-h,$ we obtain 
\begin{align*}
   v_{rr}(\alpha,y_1)= 2 \lim_{h \to 0^{+}} \frac{v_r(\alpha,y_1) h - (v(\alpha,y_1) -v(\alpha,y_1-h )}{h^2} \leq 0,
\end{align*}
which is a contraction. Then $v_r(\alpha,y_1)>0.$ Assume $v_r(\alpha,r) >0$ for all $r \in (y_1,R_{\alpha})$ is not true. Then, there exist $y_2 \in (y_1,R_{\alpha})$ such that $v(\alpha,y_2)=0$ and $v_r(\alpha,r)>0,$ for all $r \in [y_1,y_2).$ It then follows that $v_rr(\alpha,y_2) \leq 0,$ which contradict the equation \eqref{eq:ode-v}. Therefore, the assertion of the lemma follows. 
\end{proof}
\begin{lemma}
   The set $\mathcal{B}$ is nonempty, there exists $M>0,$ such that $(M,\infty) \subset \mathcal{B}.$ 
\end{lemma}
\begin{proof}
    By Lemma \ref{lem:Aopen},we have for any $\alpha \in (0,\infty)$
\begin{equation}
\label{eq:v_r}
    v_r(\alpha,r) >0 , \quad \forall r \in (0,\min(\sinh^{-1} (\frac{n}{\sqrt{2}}),R_{\alpha}))
\end{equation}
It follows that $v(\alpha,r)>0$ for all $\alpha>0$ and $r \in (0,\min(\sinh^{-1} (\frac{n}{\sqrt{2}}),R_{\alpha})).$     For any $\alpha >0,$ define $w(\alpha,r)=\frac{v(\alpha,r)}{\alpha}, $ and set $w^{\alpha}:=\frac{d}{d\alpha} w(\alpha,r),$ it satisfies 
\begin{align}
\label{eq:dw}
    \frac{1}{2} w^{\alpha}_{rr}  + \frac{1}{2} \coth(r) w^{\alpha}_r &- w^{\alpha}\left( 3 \alpha^2 w^{2} + \frac{n^2}{2 \sinh^2(r) } -1 \right) = 2 \alpha w^3>0, \quad  r \in (0,\min(\sinh^{-1} (\frac{n}{\sqrt{2}}),R_{\alpha})), \\
    w^{\alpha}(\alpha,r)&= r^n + O(r^{n+2}) , \quad \text{ as } r \to 0.
\end{align}
By series expansion for small $r,$ we have 
\begin{align*}
  w^{\alpha}(\alpha,r)= \frac{\alpha}{2n^2 + 3n +1} r^{3n+2} + O(r^{3n+3}) .  
\end{align*}
Since, $w^{\alpha}>0$ near $0,$ then $w^{\alpha}$ can not attain a positive local maximum in $(0,\min(\sinh^{-1} (\frac{n}{\sqrt{2}}),R_{\alpha})).$ If not then there exists a critical point $r_0<\sinh^{-1} (\frac{n}{\sqrt{2}}),$ (i.e., $ 1 < \frac{1}{2 \sinh^2(r_0) }$) such that $w^{\alpha}_r (\alpha,r_0)=0$ and $ w^{\alpha}_{rr}(\alpha,r_0) \leq 0.$ Using the equation \eqref{eq:dw}, we obtain $\frac{1}{2} w^{\alpha}_rr(\alpha,r_0)>0,$ which is a contraction. Therefore, we have $w^{\alpha}(\alpha,r) > 0$, for all $r \in (0, \min(\sinh^{-1}(\frac{n}{\sqrt{2}}), R_{\alpha}))$. Hence $w(\alpha,r)$ is strictly increasing in $\alpha$, for fixed $r \in (0, \min(\sinh^{-1}(\frac{n}{\sqrt{2}}), R_{\alpha}))$. Since $v = \alpha w$, there exists $M > 0$ such that if $a > M$ we have $v(a,r) > 1$. If $R_{\alpha} < n$, then by \eqref{eq:v_r} we have $a \in \mathcal{B}$. If $R_{\alpha} > n$, then $v(a,n) > 1$ and by \eqref{eq:v_r} we have $v_r(a,r) > 0$ for all $r \in (0,n)$. Thus, there exists $r_1$ such that $v(\alpha,r_1) = 1$ and $v(\alpha,r) < 1$ for $r \in (0,r_1)$. Then, by Lemma~\ref{lem:blowupsv} we have $v_r(\alpha,r) > 0$ for all $r \in [r_1, R_{\alpha})$. Therefore, if $a > M$ then $a \in \mathcal{B}$. 
\end{proof}

For any $\alpha \in \mathcal{B},$ define 
\begin{align}
\label{eq:defy1}
   y_1(\alpha):= \inf \{  r \in (0,R_{\alpha}) \, |\; v(\alpha,r)=1  \} .
\end{align}

\begin{lemma}
    The set $\mathcal{B} $ is open and $y_1(\cdot) \in \mathcal{C}^{\infty}(\mathcal{B}).$
\end{lemma}

\begin{proof}
    By continuous dependence of $v$ on $\alpha,$ we have for every $\alpha_0$ 
    \begin{align*}
      \exists \delta>0, \text{ such that } v_r(\alpha,r)>0, \quad \forall \alpha \in (\alpha_0-\delta_0,\alpha_0 + \delta_0) ,\text{ and }  \forall r \in (0,y_1(\alpha_0)+\delta_1).
    \end{align*}
 Thus, by Lemma \ref{lem:blowupsv} for some $\delta_2\in(0,\delta_1)$ we have $(\alpha_0 -\delta_1 ,\alpha_0+\delta_1) \subset \mathcal{B}.$ Moreover, using Implicit Function Theorem for $v(a,y_1)-1=0,$ for $y_1=y_1(\alpha),$ we obtain that $y_1(\alpha)$   is smooth and this concludes the proof of the Lemma.
\end{proof}

\textbf{Proof of Theorem \ref{thm:v}.} Recall that $ \mathcal{A} \cup \mathcal{B} \cup \mathcal{C} = (0,\infty)$ and $\mathcal{A}$ and $\mathcal{B}$ are open sets then $\mathcal{C}$ is not empty. Hence by Lemma \ref{lem:C-sol} we obtain the existence of solution to the problem \ref{eq:ode2-app},\eqref{eq:boundary_condition-ode2-app}.

Next, we proof properties of the solution to the problem \ref{eq:ode2-app},\eqref{eq:boundary_condition-ode2-app}.

\begin{lemma}
    Let $u$ be a solution to the problem \ref{eq:ode2-app},\eqref{eq:boundary_condition-ode2-app}, 
    then 
    \begin{align*}
   0<u(r)<1 , \quad   u^{\prime}(r)>0 \quad \text{for all } \; r >0.      
    \end{align*}
\end{lemma}

\begin{proof}
    The assertion of this lemma is consequence of the following claim 
    \begin{claim}
        $u$ be a solution to the problem \ref{eq:ode2-app},\eqref{eq:boundary_condition-ode2-app} if and only if there is $\alpha \in \mathcal{C}$ such that $u(r)=v(\alpha,r).$
    \end{claim}
    \begin{proof}
    Notice that, by Lemma \ref{lem:C-sol}, we have if $a \in \mathcal{C}$ then $u(\cdot)=v(\alpha,\cdot)$ is a solution to the equation \eqref{eq:ode2-app},\eqref{eq:boundary_condition-ode2-app}. Let $u$ be a solution of \eqref{eq:ode2-app} and \eqref{eq:boundary_condition-ode2-app}. Then $u$ is bounded, and there exists $\alpha \in (0,\infty)$ such that $u(r) = v(\alpha,r)$. Since $u$ is bounded, we have $a \notin \mathcal{B}$, and it suffices to show that $a \notin \mathcal{A}$. We argue by contradiction, assume $a \in \mathcal{A}.$ Recall that by Lemma \ref{lem:Aopen}
   we have $  \sinh^{-1}(\frac{n}{\sqrt{2}})< y_0(\alpha),$ $v(\alpha,y_0(\alpha)) < \left( 1 - \frac{n^2}{2 \sinh^2(y_0(\alpha))} \right)^{\frac{1}{2}}$ and $ v_{rr}(\alpha,y_0(\alpha))<0.$
Since $v(\alpha,r) \geq 0,$ and $\lim_{r \to \infty } v(\alpha,r)=1,$ then $v$ must attains local minimum at some point $r_0 > y_0(\alpha),$ such that 
\begin{align}
\label{eq:vr0}
   0\leq v^2(\alpha,r_0) < v^2(\alpha,y_0(\alpha)) < 1 -\frac{n^2}{2 \sinh^2(y_0(\alpha))} < 1 -\frac{n^2}{2 \sinh^2(r_0)} . 
\end{align}
 As $v$ attains local minimum at $r=r_0,$ we have $v_r(\alpha,r_0)=0$ and $v_{rr}(\alpha,r_0) \geq 0.$ Then using \eqref{eq:vr0} and the equation \eqref{eq:ode2-app},  $v_{rr}(\alpha,r_0)=v(\alpha,r_0)(v(\alpha,r_0)^2 + \frac{n^2}{2 \sinh^2(r_0)} -1 ) \geq 0, $ one can see that the only possibility is $v(\alpha,r_0)=0.$ Howoever, the only solution to second order equation with $v(\alpha,r_0)=v_r(\alpha,r_0)=0$ is $v \equiv 0,$ which can not be a solution as $u\to 1,$ as $r \to \infty.$ Thus $a \notin A.$ and this concludes the proof of the claim. 
   \end{proof}
 The proof of the lemma follows from the claim above. 
\end{proof}

\begin{lemma}
    If $u$ solves \eqref{eq:ode2-app} and \eqref{eq:boundary_condition-ode2-app}, then  
    $u(r)=1 - 2 n^2 e^{-2r} + o(e^{-2r})$ as $r \to \infty.$
\end{lemma}
\begin{proof}
Since the asymptotics of $u$ as $r \to \infty$ follow from standard ODE arguments, we omit the proof.
\end{proof}
Next, we show that the solution is unique.  
For that, it suffices to prove that $\mathcal{C}$ consists of a single point.  
We first show that both $\mathcal{A}$ and $\mathcal{B}$ are connected, therefore, $\mathcal{C}$ consists of an interval $[a, A]$. We then prove that $a = A$. 

We first start by proving a monotonicity lemma, whcih will play a key role in the uniqueness proof.

\begin{lemma}
\label{lem:uniqKey}
 Assume that $v(\alpha,r)>0$ in $(0,R_1)$ for some $R_1 \leq R_{\alpha}.$ Then 
 \begin{align*}
   v^{\alpha}  :=\frac{d}{d\alpha} v(\alpha,r) > \frac{1}{\alpha} v(\alpha,r), \quad \forall r \in (0,R_1).
 \end{align*}
If in addition we assume that $v_r(\alpha,r)>0$ for $r\in (0,R_1)$ then 
\begin{align}
\label{eq:v_ralpha>0}
    v_r^{\alpha}(\alpha,r) :=\frac{d}{dr} v^{\alpha}(\alpha,r) >0, \quad \forall r \in (0,R_1).
\end{align}
\end{lemma}
\begin{proof}
    Notice that $v^{\alpha}$ satisifies 
    \begin{align}
        \frac{1}{2} v^{\alpha}_{rr} +\frac{1}{2} &\coth(r) v^{\alpha}_r - v^{\alpha}( 3 v^2+ \frac{n^2}{2 \sinh^2(r)} -1 ) = 0, \quad \forall r \in (0,R_1), \\ 
        v^{\alpha}(\alpha,r) &= r^n - \frac{1}{2n+2} r^{n+2}+O(r^{n+4}), \quad \text{ as } r \to 0. 
    \end{align}
and $w(\alpha,r)= \frac{v(\alpha,r)}{\alpha}$ satisfies 
\begin{align}
    \frac{1}{2} w_{rr} + \frac{1}{2} &\coth(r) w_r - w(  3 v^2 + \frac{n^2}{\sinh^2(r)} -1 ) = -2 \alpha w^3 <0, \quad \forall r \in (0,R_1), \\
    w(\alpha,r)&=r^n - \frac{1}{2n+2} r^{n+2}+O(r^{n+4}), \quad \text{ as } r \to 0. 
\end{align}
First, we show that $v^{\alpha} > w$ for small $r.$ Recall that $w^{\alpha}(\alpha,r) \sim \frac{\alpha}{2n^2+3n+1} r^{3n+2}.$  Note that $w^{\alpha}(\alpha,r)=\frac{v^{\alpha}(\alpha,r)-w(\alpha,r)}{\alpha}>0,$ for small $r.$ Thus, $v^{\alpha}(\alpha,r)>w(\alpha,r)$ for small $r.$ Next, we argue by contradiction. Assume that $v^{\alpha}(\alpha,r) > w(a,r)=\frac{w(\alpha,r)}{\alpha}$ for all $r \in (0,R_1)$ is not true. Then there exist $R_2 \in (0,R_1)$ such that $v^{\alpha}(\alpha,R_2) = w( \alpha,R_2).$ Define
\begin{equation}
\label{eq:defkap}
   \kappa:=\sup \{ c >0 \,| \, c v^{\alpha}(\alpha,r) \leq w(\alpha,r) , \; \text{ for } r \in (0,R_2)  \}.
\end{equation}
Note that for small $r,$ we have $c v^{\alpha} - w \sim (\kappa - 1) r^{n} <0,$ for small $r.$ We conclude that, $\exists R_3 \in (0,R_2)$ such that 
\begin{align*}
   0=c v^{\alpha}(\alpha,R_3) - w(\alpha,R_3) = \max_{ r \in [0,R_2]} \{ \kappa v^{\alpha}(\alpha,r) - w(\alpha,r) \} .
\end{align*}
Then it follows that 
\begin{align}
\label{eq:vv_rv_rr-r3}
  \kappa v^{\alpha}=w, \quad \kappa v_r^{\alpha}= w_r, \quad \kappa v_{rr}^{\alpha} \leq w_{rr} , \quad \text{ for } r=R_3.   
\end{align}
However, this is impossible otherwise we have 
\begin{align*}
    0&=\kappa \left( \frac{1}{2} v_{rr}^{\alpha} + \frac{1}{2} \coth(r) v_r^{\alpha} - v^{\alpha} (3v^2+ \frac{n^2}{2 \sinh^2(r)} -1 ) \right) \bigg|_{r=R_3} \\
    & \leq \frac{1}{2} w_{rr} + \frac{1}{2} w_r - w \left( 3v^2 + \frac{n^2}{2 \sinh^2(r)} -1 \right) \bigg|_{r=R_3} = -2 \alpha^2 w^2(\alpha,R_3)<0,
\end{align*}
which is a contradiction. Therefore, we have $\frac{d}{d\alpha} v(\alpha,r)> \frac{1}{\alpha} v(\alpha,r), $ for all $r\in (0,R_1).$ \\

Next, we prove \eqref{eq:v_ralpha>0}. Note that $v_r^{\alpha}>0$ when $r$ is small. Similarly, we argue by contradiction. Suppose that $\eqref{eq:v_ralpha>0}$ is not true, then there exists $R_2 \in (0,R_1)$ such that $v^{\alpha}_r(\alpha,R_2)=0.$ Define $\kappa$ as in \eqref{eq:defkap}. Since $\kappa v^{\alpha}_r- w_r \big|_{r=R_2} <0,$ we can deduce that there exists $R_3 \in (0,R_2)$ satisfying \eqref{eq:vv_rv_rr-r3}, which by the same arguments leads to a contradiction. 

\end{proof}
 
\begin{lemma} \label{lem:B-connected}
    $\mathcal{B}$ is connected and there exists a constant $A>0$ such that $\mathcal{B}=(A,\infty).$
\end{lemma}

\begin{proof}
    Since $\mathcal{B}$ is open, we need only to show that if $(\alpha_1,\alpha_2) \in \mathcal{B},$ then $\alpha_2 \in \mathcal{B}.$ Define $y_1(\alpha)$ as in \eqref{eq:defy1}, then we have $v(\alpha,y_1(\alpha))=1$ and $v(\alpha,r)<1$ for all $r\in (0,y_1(\alpha)).$ Thus, by Lemma \ref{lem:blowupsv} we have $v_r(\alpha,r)>0$ for all $r \in [y_1(\alpha),R_{\alpha})$ and by Lemma \ref{lem:uniqKey} we have $v^{\alpha}(\alpha,y_1(\alpha))>0.$ Using Implicit Function Theorem for $v(\alpha,y_1(\alpha))=1,$ we have 
    \begin{align*}
        \frac{d}{d \alpha } y_1(\alpha)= -\frac{v^{\alpha}(\alpha,y_1(\alpha))}{v_r(\alpha,y_1(\alpha))} <0.
    \end{align*}
It follows that $\lim_{ \alpha \to \alpha_2} y_1(a)=l_0 > 0 $ exists. Using continuous dependence of $v$ on $\alpha,$ we have $v(\alpha_2,l_0)=1$ and $v_r(\alpha_2,r) \geq 0 $ for all $r\in (0,R_{{\alpha}_2})$. Then by Lemma \ref{lem:Aopen}  and \eqref{lem:blowupsv}, we conclude that $\alpha_2 \in \mathcal{B}.$ 
\end{proof}
\begin{lemma} \label{lem:A-connected}
   $\mathcal{A}$ is connected and there is a constant $a>0,$ such that $\mathcal{A}=(0,a).$ 
\end{lemma}
\begin{proof}
    Let $y_0(\alpha)$ defined as \eqref{eq:defy}, then by \eqref{eq:v_ralpha>0} we have $v^{\alpha}_r(\alpha,y_0(\alpha)) \geq 0.$ By Lemma \ref{lem:Aopen} and Implicit Function Theorem, we have 
    \begin{align*}
       \frac{d}{d\alpha}y_0(\alpha)=- \frac{v^{\alpha}_r(\alpha,y_0(\alpha))}{v_{rr}(\alpha,y_0(\alpha))} \geq 0.
    \end{align*}
By an argument analogous to the one in the proof of the previous lemma, we obtain the desired result.
\end{proof}
Finally, we prove the uniqueness of the solution of \eqref{eq:ode2-app},\eqref{eq:boundary_condition-ode2-app}.
\begin{theorem}
    The problem \eqref{eq:ode2-app},\eqref{eq:boundary_condition-ode2-app} admits a unique solution.
\end{theorem}
\begin{proof}
 In view of \eqref{eq:ABC} and Lemma \ref{lem:B-connected}, \ref{lem:A-connected}, we have $\mathcal{C}=[a,A].$ Thus, it suffices to prove that $a = A$. Note that for any $r>0,$ we have 
 \begin{align*}
     v(A,r)=v(a,r)+ \int_a^A  v^{\alpha}(\alpha,r) d\alpha
 \end{align*}
 By definition of the set $\mathcal{C}$ and Lemma \ref{lem:uniqKey}, we have 
  \begin{align*}
     v(A,r)=v(a,r)+ \int_a^A  \frac{1}{\alpha} v(\alpha,r) d\alpha
 \end{align*}
Since for every $\alpha \in \mathcal{C},$ $v(\alpha,r) \to 1 ,$ as $r \to \infty,$ taking the limit $r \to \infty,$ yields $1 \geq 1 + \ln(\frac{a}{A}),$ which implies $a=A.$ This completes the proof of the theorem.   
\end{proof}

\section{Green's kernel for $i\calL_0-z$ } \label{sec:resol-iL0}
In this section we present the proof of Lemma \ref{lem::kernel-resol-L0}. 
We seek to compute the resolvent kernel of $(i\calL_0 -z)^{-1},$ for $\pm \im(z)>0.$ Therefore, we define the resolvent kernel as 
\begin{align}
    \mathcal{R}_0^{\pm}(r,s,z):=\Psi_{(0)}^{\pm}(r,z) S(s,z)\mathbb{1}_{0\leq s \leq r}+ \PPsi_{(0)}^{\pm}(r,z) T(s,z) \mathbb{1}_{r \leq s \leq \infty}. 
\end{align}

Let $\mathfrak{D}^{\pm}(z)=W(\Psi_{(0)}^{\pm}(\cdot,z),\PPsi_{(0)}^{\pm}(\cdot,z) ).$
Using similar argument as in previous sections in particular \S \ref{subsec:GreenKernel} and \ref{subsec:GreenKernel-resonance}, we obtain 
\begin{align*}
   \mathcal{R}_0^{\pm}(r,s,z)&:= i  \Psi_{(0)}^{\pm}(r,z) (\mathfrak{D}^{\pm}(z))^{-t}  \PPsi_{(0)}^{\pm}(s,z)^{t} \sigma_1 \mathbb{1}_{0\leq s \leq r} \\
 &  + i \PPsi_{(0)}^{\pm}(r,z) (\mathfrak{D}^{\pm}(z))^{-1}  \Psi_{(0)}^{\pm}(s,z)^{t} \sigma_1 \mathbb{1}_{r \leq s \leq \infty}.
\end{align*}
Denote by $  \mu_{ij}^{0,\pm}(z)=W(\Psi_{(0,i)}^{\pm}(\cdot,z),\Psi_{(0,j)}^{\pm}(\cdot,z)).$ Then we get 
\begin{align*}
  \mathfrak{D}^{\pm}(z)=\begin{pmatrix}
      \mu_{13}^{0,\pm}(z) &   \mu_{14}^{0,\pm}(z) \\
        \mu_{23}^{0,\pm}(z) &  \mu_{24}^{0,\pm}(z)
  \end{pmatrix}  , \quad \mathfrak{d}^{\pm}(z):=\det(  \mathfrak{D}^{\pm}(z))
\end{align*}
Note that by definition of $k_j,$ we have $1-c_1(z)c_2(z)=0,$ see Section \ref{sec:nearinfty-small-xi}. Then we have $ \mu_{14}^{0,\pm}(z)=\mu_{23}^{0,\pm}(z) =0.$ Therefore, $\mathfrak{d}^{\pm}(z)= \mu_{13}^{0,\pm}(z)  \mu_{24}^{0,\pm}(z).$ Denote by 
\begin{align*}
     \mathfrak{R}_{ij}^{\pm}(r,s,z):=i\bigg( \Psi_{(0,i)}^{\pm}(r,z) \Psi_{(0,j)}^{\pm }(s,z)^t \mathbb{1}_{\{ 0\leq s \leq r\} } + \Psi_{(0,j)}^{\pm}(r,z) \Psi_{(0,i)}^{\pm  }(s,z)^t \mathbb{1}_{\{ r\leq s \leq \infty \} } \bigg) \sigma_1.
\end{align*}
Hence, by direct computation, we obtain 
\begin{align}
    \mathcal{R}_0^{\pm}(r,s,z) &:=\frac{1}{\mathfrak{d}^{\pm}} \left(  \mu_{24}^{0,\pm}(z) \mathfrak{R}_{13}^{\pm}(r,s,z) + \mu_{13}^{0,\pm}(z)   \mathfrak{R}_{24}^{\pm}(r,s,z)\right) \\
    &=\frac{1}{\mu_{13}^{0,\pm}(z)}  \mathfrak{R}_{13}^{\pm}(r,s,z) + \frac{1}{\mu_{24}^{0,\pm}(z)}  \mathfrak{R}_{24}^{\pm}(r,s,z).
\end{align}
Finally, we estimate $\mu_{13}^{0,\pm}(z)$ and $\mu_{24}^{0,\pm}(z):$  
\begin{align*}
    \mu_{13}^{0,\pm}(z):=-(1-c_1^2(z)) W( \psi_0(r,k_1(z)), \phi_0(r,k_1(z))) \\
     \mu_{24}^{0,\pm}(z):=-(1-c_2^2(z)) W( \psi_0(r,k_2(z)), \phi_0(r,k_2(z))) 
\end{align*}
Using the asymptotics  near $0$ for $\phi_0(r)= c r^{\frac{3}{2}}+ O(r^\frac{5}{2})$ and $\psi_0(r)= \tc_1 r^{-\frac{1}{2}} + \tc_2 r^{\frac{3}{2}}+O(1),$ we have
$W( \psi_0(r,k_i(z)), \phi_0(r,k_j(z)))=2 c k_i^{-\frac{1}{2}} k_j^{\frac{3}{2}},$ for $i,j=1,2.$ Recall that for large $|z|,$  we have $|k_1(z)|\simeq\sqrt{|z|},$ $|k_2(z)|\simeq\sqrt{|z|},$  $c_1(z)=-i +O(|z|^{-1})  $ and $c_2(z)=i+O(|z|^{-1}),$ (see Remark \ref{K_j-C_j-D_j-behavior-large-xi}). Thus, we have 
\begin{align*}
      |\mu_{13}^{0,\pm}(z)|\simeq|\mu_{24}^{0,\pm}(z)|\simeq \sqrt{|z|}
\end{align*}
This concludes the proof of Lemma \ref{lem::kernel-resol-L0}.

\bibliographystyle{acm}
\bibliography{references.bib}
\end{document}